\documentclass{amsart}
\usepackage[original]{imakeidx}
\makeindex[title={Index of notation}]
\usepackage[hidelinks]{hyperref}
\usepackage{amssymb, graphics, color, enumitem, mathrsfs}
\usepackage[all]{xy}

 \usepackage[applemac]{inputenc}
\usepackage[cyr]{aeguill}
\usepackage{tikz}
\usetikzlibrary{cd}

\usepackage[margin = 1in]{geometry}

\newtheorem{lemma}{Lemma}[section]
\newtheorem{proposition}[lemma]{Proposition}
\newtheorem{corollary}[lemma]{Corollary}
\newtheorem{theorem}[lemma]{Theorem}

\newtheorem{example}[lemma]{Example}
\newtheorem{definition}[lemma]{Definition}

\newtheorem{remark}[lemma]{Remark}

\newtheorem*{Acknowledgement}{Acknowledgements}

\newtheorem{Theorem}{Theorem}
\newtheorem{assumption}[lemma]{Assumption}

\newtheorem{Corollary}[Theorem]{Corollary}
\theoremstyle{remark}

\newtheorem*{Remark}{Remark}

\newcommand\kridx[3]{\index{#2@$#1$}#1}

\newcommand\cf{cf.\@ }
\newcommand\ie{i.e.\@ }

\newcommand\pa{ \partial}

\newcommand\bbC{\mathbb C}

\newcommand\bbN{\mathbb N}

\newcommand\bbR{\mathbb R}
\newcommand\bbS{\mathbb S}

\newcommand\ol[1]{\overline{#1}}

\newcommand\set[1]{\left\{#1\right\}}
\newcommand\abs[1]{\left|#1\right|}
\newcommand\pd[1]{\frac{\pa}{\pa #1}}
\newcommand\rst{|}
\newcommand\inv{{-1}}
\newcommand\QFBT{{}^\QFB T}

\newcommand\QbT{{}^\Qb T}
\newcommand\OmegaQFB{{}^\QFB \Omega}

\newcommand\OmegaQb{{}^\Qb \Omega}
\newcommand\Omegab{{}^b \Omega}

\renewcommand\Re{\operatorname{Re}}

\usepackage{color}

\newcommand\lrp[1]{\left( #1 \right)}

\newcommand\CI{\mathcal{C}^{\infty}}

\newcommand\Diff{\operatorname{Diff}}

\newcommand\cC{\mathcal{C}}
\newcommand\cO{\mathcal{O}}
\newcommand\cA{\sA}
\newcommand\cE{\mathcal{E}}
\newcommand\cD{\mathcal{D}}
\newcommand\cF{\mathcal{F}}
\newcommand\cL{\mathcal{L}}

\newcommand\cV{\mathcal{V}}
\newcommand\cU{\mathcal{U}}
\newcommand\cI{\mathcal{I}}

\newcommand\End{\operatorname{End}}

\newcommand\pr{\operatorname{pr}}

\newcommand\phg{\operatorname{phg}}
\newcommand\Spec{\operatorname{Spec}}
\newcommand\Id{\operatorname{Id}}

\newcommand\bM{\overline{M}}

\renewcommand\sc{\operatorname{sc}}

\newcommand\ff{\operatorname{ff}}
\newcommand\cH{\mathcal{H}}

\newcommand\cN{\mathcal{N}}
\newcommand\cM{\mathcal{M}}

\newcommand\AC{\operatorname{AC}}

\newcommand\res{\operatorname{res}}
\newcommand\SU{\operatorname{SU}}

\newcommand{\cP}{\mathcal{P}}

\newcommand\QAC{{\operatorname{QAC}}}

\newcommand\Qb{{\operatorname{Qb}}}
\newcommand\QFB{{\operatorname{QFB}}}
\newcommand\QFC{\operatorname{QFC}}
\newcommand\ALE{\operatorname{ALE}}
\newcommand\QALE{\operatorname{QALE}}

\newcommand\bV{\overline{V}}

\newcommand\cK{\mathcal{K}}
\newcommand\sus{\operatorname{sus}}

\newcommand\Hom{\operatorname{Hom}}

\newcommand\cS{\mathcal{S}}

\newcommand\sA{\mathscr{A}}
\newcommand\sB{\mathscr{B}}
\newcommand\bd{\operatorname{b}}

\newcommand\w{\mathfrak{w}}

\newcommand\cQ{\mathcal{Q}}
\newcommand\cR{\mathcal{R}}

\newcommand\cB{\mathcal{B}}
\newcommand\cG{\mathcal{G}}

\renewcommand\k{\mathscr{K}}
\newcommand\cl{\operatorname{cl}}

\newcommand\op{\operatorname{op}}
\newcommand\rp{\operatorname{rp}}
\newcommand\ttimes{\mathbin{\hbox to -0.07em{\raisebox{0.5em}{$\scriptstyle\sim$}}\times}}
\newcommand\rttimes{\mathbin{\hbox to -0.07em{\raisebox{0.5em}{$\scriptstyle\sim$}}\times_r}}
\newcommand\jtimes{\mathbin{\hbox to 0.07em{\raisebox{0.5em}{$\scriptstyle\sim$}}\star}}
\newcommand\rjtimes{\mathbin{\hbox to 0.07em{\raisebox{0.5em}{$\scriptstyle\sim$}}\star_r}}\newcommand\diag{\mathfrak{D}}
\newcommand\cn{\operatorname{cn}}

\newcommand\bvarrho{\overline{\varrho}}

\begin{document}
\title[$\QFB$ operators]
{Quasi-fibered boundary pseudodifferential operators}

\author{Chris Kottke}
\address{New College of Florida}
\email{ckottke@ncf.edu}


\author{Fr\'ed\'eric Rochon}
\address{Département de Mathématiques, Universit\'e du Qu\'ebec \`a Montr\'eal}
\email{rochon.frederic@uqam.ca}

\maketitle

\begin{abstract}
We develop a pseudodifferential calculus for differential operators associated to quasi-fibered boundary metrics ($\QFB$ metrics), a class of metrics including the quasi-asymptotically conical metrics ($\QAC$ metrics) of Degeratu-Mazzeo and the quasi-asymptotically locally Euclidean metrics ($\QALE$ metrics) of Joyce.  Introducing various principal symbols, we introduce the notion of fully elliptic  $\QFB$ operators and show that those are Fredholm when acting on $\QFB$ Sobolev spaces.  For $\QAC$ metrics, we also develop a pseudodifferential calculus for the conformally related class of $\Qb$ metrics.  We use these calculi to construct a parametrix for the Hodge-deRham operator of certain $\QFB$ metrics, allowing us to show that it is Fredholm on suitable Sobolev spaces and that  the space of $L^2$ harmonic forms is finite dimensional.  Our parametrix is obtained by inverting certain model operators at infinity, inversions that we achieve in part through a fine understanding of the low energy limit of the resolvent of the Hodge-deRham operator.  Our parametrix also implies that $L^2$ harmonic forms decay faster at infinity than an arbitrary $L^2$ form, the extra decay being quantified in terms of a small negative power of the distance function.  This decay of $L^2$ harmonic forms is used in a companion paper to study the $L^2$ cohomology of some $\QFB$ metrics.  \end{abstract}

\tableofcontents

\section*{Introduction}

In \cite{Joyce}, Joyce introduced the notion of quasi-asymptotically locally Euclidean ($\QALE$) metrics and produced examples that are Calabi-Yau.  Such metrics are a generalization of the well-known class of asymptotically locally Euclidean ($\ALE$) metrics, and are defined on a suitable crepant resolution $\pi : X \to \bbC^m/\Gamma$ where $\Gamma$ is a finite subgroup of $\SU(m)$.  When the action of $\Gamma$ is free on $\bbC^m\setminus \{0\}$, the associated $\QALE$ metric on $X$ is in fact an $\ALE$ metric with tangent cone at infinity $\bbC^m/\Gamma$; in particular, a cocompact set of $X$ is diffeomorphic to $(1,\infty)\times S^{2m-1}/\Gamma$ and the metric is asymptotic along with its derivatives to the model metric $dr^2 + r^2 g_{S}$, where $g_S$ is the image of the standard spherical metric on the quotient $S^{2m-1}/\Gamma$.

If instead the action of $\Gamma$ is not free on $\bbC^m\setminus \{0\}$, then $\bbC^m/\Gamma$ is still the tangent cone at infinity, but now it has rays of singularities going off to infinity and the resolution must be performed on a non-compact set. Away from this problematic region, $\QALE$ metrics behave like $\ALE$ metrics. 
On the other hand, near 
a singular point $p$ of $(\bbC^m\setminus \{0\})/\Gamma$, a neighborhood is diffeomorphic
to a subset $V$ of $\bbC^{m-1}/\Gamma_p \times \bbC$,
where $\Gamma_p$ is the stabilizer of $p$, and the resolution of $V$ is determined by
\begin{equation}
  \pi : Y_p \times \bbC \to \bbC^{m-1}/\Gamma_p \times \bbC
\label{int.2}\end{equation}
where $Y_p$ is a crepant resolution of $\bbC^{m-1}/\Gamma_p$.
This is a prototypical example of a local product resolution, as discussed in \cite{Joyce}, and allows $\QALE$ metrics to be defined iteratively by {\em depth}.  Namely, a depth 1 $\QALE$ metric is\footnote{A slightly different convention is used in \cite{DM2018}.}  simply an $\ALE$ metric, while in general
one can define the quasi-isometry class of a $\QALE$ metric on 
the subset $\pi^{-1}(V)$ to be represented by the Cartesian product
\begin{equation}
  g_{Y_p}+ g_{\bbC},
\label{int.3}\end{equation}
where $g_{\bbC}$ is the Euclidean metric and $g_{Y_p}$ is itself a $\QALE$ metric of strictly lower depth on $Y_p$. In particular, if the action of $\Gamma_p$ is free on $\bbC^m\setminus \{0\}$, then $g_{Y_p}$ has depth 1 and is an $\ALE$ metric.   

For many applications, it is necessary to go beyond quasi-isometry classes and impose stronger conditions on $\QALE$ metrics and their asymptotics with respect to the local models \eqref{int.3}.
In particular, Joyce in \cite{Joyce} relies on a refined definition to solve a complex Monge-Amp\`ere equation and produce examples of Calabi-Yau $\QALE$ metrics.  One of the key steps in this construction uses an isomorphism property of the scalar Laplacian of a $\QALE$ metric acting on suitable H\"older spaces.  

We consider here the larger class of quasi-fibered boundary ($\QFB$) metrics as introduced in \cite{CDR}, which includes the class of $\QALE$ metrics along with the more general class of quasi-asymptotically conic ($\QAC$) metrics originally introduced in \cite{DM2018}. The latter are themselves a generalization of the notion of asymptotically conic ($\AC$) metrics, which in turn also generalize $\ALE$ metrics but in a different way. The relationships between these types of metrics can be summarized in the following diagram:
\[
\begin{tikzpicture}
\matrix (m) [matrix of math nodes, row sep=0, column sep=2ex] {
	{} & \QALE & {} & {} \\
	\ALE &{} & \QAC & \QFB. \\
	{} & \AC & {} & {} \\
};
\path (m-2-1) edge[draw=none] node[auto=false,sloped] {$\subset$} (m-1-2)
      (m-2-1) edge[draw=none] node[auto=false,sloped] {$\subset$} (m-3-2)
      (m-1-2) edge[draw=none] node[auto=false,sloped] {$\subset$} (m-2-3)
      (m-3-2) edge[draw=none] node[auto=false,sloped] {$\subset$} (m-2-3)
      (m-2-3) edge[draw=none] node[auto=false,sloped] {$\subset$} (m-2-4);
\end{tikzpicture}
\]
To give a rough idea of what is a $\QFB$ metric, let us first recall that $\AC$ metrics, as the name suggests, are modelled outside a compact set on a Riemannian cone
\begin{equation}
  dr^2+ r^2g_{Y} \quad \mbox{on} \quad (0,\infty)\times Y
\label{int.4}\end{equation}
with $(Y,g_Y)$ a closed Riemannian manifold.  From this point of view, an $\ALE$ metric is an $\AC$ metric for which $(Y,g_{Y})$ is a spherical space form.  Now, the larger class of $\QAC$ metrics of Degeratu-Mazzeo \cite{DM2018} admits an iterative definition where the local model \eqref{int.4} is replaced more generally by a Cartesian product
\begin{equation}
   g_Z+ (dr^2+r^2g_Y) 
\label{int.5}\end{equation}
on a subset of $Z\times ((0,\infty)\times Y)$ of the form
\begin{equation}
   W= \{ (z,r,y)\in Z\times ((0,\infty)\times Y) \; |  \;  r>\delta,  \; d_{g_Z}(p,z)<  \frac{r}{\delta}\} \quad \mbox{for some} \; \delta>0,
\label{int.5b}\end{equation}
 where $g_{Z}$ is a  $\QAC$ metric defined on the lower dimensional space $Z$, $p\in Z$ is a fixed point, $d_{g_{Z}}$ is the Riemannian distance on $Z$ and  $((0,\infty)\times Y,dr^2+r^2g_Y)$ is a Riemannian cone with $Y$ a (possibly non-compact) smooth cross-section.  
Compared to a $\QALE$ metric, the tangent cone at infinity of a $\QAC$ metric is more complicated, of the form 
\begin{equation}
   dr^2+ r^2g_{\infty} \quad \mbox{on} \quad (0,\infty)\times S,
\label{int.6}\end{equation}
where $S$, instead of being a singular quotient of a sphere by a finite group of isometries, can more generally be taken to be a (smoothly) stratified space with $g_{\infty}$ a wedge metric \cite{AG} (also called incomplete iterated edge metric in \cite{ALMP2012}).    

More generally, $\QFB$ metrics can be defined iteratively like $\QAC$ metrics, but with the difference that in the local models \eqref{int.5}, $g_Z$ is now a $\QFB$ metric, while away from those local models the metric looks like a fibered boundary metric in the sense of \cite{Mazzeo-MelrosePhi} instead of an asymptotically conical metric.  A simple example of $\QFB$ metric would consist in taking the Cartesian product of a closed Riemannian manifold with a complete non-compact Riemannian manifold equipped with $\QAC$ metric.   In particular, $\QFB$ metrics that are not $\QAC$ have a strictly lower dimensional tangent cone at infinity and non-maximal volume growth.  Natural examples of $\QFB$ metrics arise in the ongoing work \cite{KS,FKS}, in which the authors compactify the moduli spaces of $\SU(2)$ monopoles on $\bbR^3$ and show that the natural $L^2$ metric on these moduli spaces is of $\QFB$ type.

Now, a $\QFB$ metric on a manifold $M$ naturally induces certain elliptic differential operators, including the scalar and Hodge Laplacians, the Hodge-deRham operator and so on.
More generally, we can define the set $\Diff^*_\QFB(M;E,F)$ of $\QFB$ differential operators as the compositions of vector fields which are uniformly bounded with respect to the metric, with coefficients in $\Hom(E,F)$ for vector bundles $E$ and $F$ over $M$. A more satisfactory definition of $\QFB$ structures and their associated differential operators is discussed below in terms of manifolds with fibered corners and Lie structures at infinity.

Ellipticity is defined as usual for a $\QFB$ differential operator in terms of invertibility of its principal symbol. In addition, a $\QFB$ operator admits a set of {\em normal operators}, by which it is modelled asymptotically in various regimes at infinity. These normal operators are themselves suspended versions of $\QFB$ operators, though of strictly lower depth. As in \cite{Mazzeo-MelrosePhi,DLR}, we say an elliptic operator is {\em fully elliptic} (c.f.\ Definition~\ref{sm.31} below) provided each of its normal operators is appropriately invertible. One of our first results is a Fredholm result for such operators, which follows from the construction of a parametrix within a calculus of $\QFB$ pseudodifferential operators.

\begin{Theorem}[Proposition~\ref{mp.3} and Theorem~\ref{mp.32} below]
  A fully elliptic pseudodifferential $\QFB$ operator $P\in \Psi^m_{\QFB}(M;E,F)$ induces a Fredholm operator
  \[
         P: x^{\mathfrak{t}}H^{m}_{\QFB}(M;E)\to x^{\mathfrak{t}}H^0_{\QFB}(M;F)
  \]
  for all multiweights $\mathfrak{t}$, where $H^{m}_{\QFB}(M;E)$ is the Sobolev space of order $m$ associated to a $\QFB$ metric $g$ with norm defined in \eqref{mp.28} below.  Moreover, elements in the kernel $P$ decay rapidly at infinity.  In fact, if $P$ happens to be a differential operator, then elements in the kernel even decay exponentially fast at infinity.    
   \label{int.7}\end{Theorem}
 
This result applies in particular to $\Delta_g-\lambda$, where $\Delta_g$ is the scalar Laplacian (with positive spectrum) of an exact (see Definition~\ref{pt.6}) $\QFB$ metric $g$ and $\lambda\in \bbC\setminus [0,\infty)$.  It also applies to some Dirac operators as described in Example~\ref{su.3e} below.  As explained in Remark~\ref{nbs.1} below, the exponential decay at infinity for elements in the kernel of a fully elliptic differential $\QFB$ operators can be seen as a generalization of results of Froese-Herbst \cite{Froese-Herbst} and Vasy \cite{Vasy2004}. Unfortunately, the above result does not apply to the scalar Laplacian $\Delta_g$ itself, as this operator has continuous spectrum all the way down to $0$ and cannot be Fredholm on $\QFB$ Sobolev spaces, hence is not fully elliptic. Note that this is not in contradiction with the isomorphism theorems of \cite{Joyce} and \cite{DM2018}, since as pointed out in \cite{CDR}, those results are obtained with respect to the weighted Sobolev spaces and weighted H\"older spaces of the different (but conformally related) class of $\Qb$ metrics.  

Briefly, using the manifold with fibered corners description, a $\Qb$ metric $g_{\Qb}$ on a manifold $M$ with fibered corners is of the form
 \begin{equation}
     g_{\Qb}= x_{\max}^2 g_{\QAC}
 \label{int.8}\end{equation} 
where $x_{\max}$ is a product of boundary defining functions for all the boundary hypersurfaces which are maximal with respect to the partial order induced by the fibered corners structure (roughly speaking, these represent the regions at infinity where the metric behaves asymptotically like an AC).
The $\Qb$ metrics are a generalization of the $b$ metrics of Melrose \cite{MelroseAPS} and the relation \eqref{int.8} is a natural generalization of the conformal 
relationship between $\AC$ metrics and $b$ metrics. As explained in \cite{CDR} and below, $\Qb$ metrics also admit a description in terms of a Lie structure at infinity.  

In addition to the pseudodifferential calculus for $\QFB$ operators, we develop here a pseudodifferential calculus for $\Qb$ operators. The normal operators for the latter
are more complicated and harder to invert in general. For this reason, 
rather than studying general elliptic Qb operators, we instead focus our effort on the analysis of a specific geometric operator, namely the Hodge-deRham operator of a $\QAC$ metric thought as a weighted $\Qb$ operator.  Referring to Theorem~\ref{HdR.6} as well as Corollaries~\ref{HdR.8} and \ref{HdR.9} for the detailed statement, we will in fact focus directly on the more general case of the Hodge-deRham operator of a $\QFB$ metric and construct a parametrix as follows.

\begin{Theorem}
If $(M,g)$ is a $\QFB$ manifold satisfying suitable hypotheses, then the corresponding Hodge-deRham operator admits a parametrix with nice compact error term.
\label{para.1}
\end{Theorem}
\begin{Remark}
For simplicity, in this introduction, we focus on the Hodge-deRham operator, but in the bulk of the paper, the various parametrix constructions and their implications are formulated more generally for Dirac $\QFB$ operators satisfying appropriate assumptions. 
\end{Remark}
\noindent  The hypotheses on the $\QFB$ manifold in Theorem~\ref{para.1} are of two sorts.  One is a mild geometric assumption on the form of the $\QFB$ metric ensuring nice asymptotic models for the Hodge-deRham operator; see Assumption~\ref{do.1} below.  To enforce the invertibility of these asymptotic models, we need also to assume that some of the model metrics describing the behavior of $g$ at infinity are sufficiently small to scale away some indicial roots.  However, some indicial roots are of cohomological nature and cannot be scaled away by tuning the metric.  To rule those out, we need to impose some cohomological conditions on the $\QFB$ manifold $(M,g)$.  This other sort of hypotheses is definitely restrictive, but nevertheless allows to cover a wide class of interesting examples.  For instance, according to Corollary~\ref{HdR.8}, if $g$ is a $\QALE$ metric, these conditions hold provided $\dim Y_p>4$ for each local product resolutions as in \eqref{int.2}.  

One important consequence of Theorem~\ref{para.1} is the following.
\begin{Corollary}[Corollary~\ref{do.53} with $\delta=-\frac12$]
Let $v= \prod_i x_i $ be a product of of the boundary defining functions for all the boundary hypersurfaces of $M$.  If $(M,g)$ is a $\QFB$ manifold as in Theorem~\ref{para.1}, then the associated Hodge-deRham operator $\eth_{\QFB}$ induces a Fredholm operator
 $$
    v^{-\frac12}\eth_{\QFB}v^{-\frac12}: \mathcal{D}_{\max}\to L^2_{\QFB}(M;\Lambda^*(T^*M)),
 $$
 with $\cD_{\max}$  the maximal extension of $v^{-\frac12}\eth_{\QFB}v^{-\frac12}$ seen as an unbounded operator acting on $L^2$ forms.  In particular, the space of $L^2$ harmonic forms of $(M,g)$ is finite dimensional.
\label{fredh.1}\end{Corollary}
In terms of vector fields, the operator $v^{-\frac12}\eth_{\QFB}v^{-\frac12}$ is more associated to the conformally related metric $g_{\QFC}:=v^2 g$.  By analogy with the depth one case \cite{HHM2004}, we say that $g_{\QFC}$ is a \textbf{quasi-fibered cusp metric}.  Because we are dealing with $v^{-\frac12}\eth_{\QFB}v^{-\frac12}$ instead of  $\eth_{\QFB}$, notice that it is not quite sufficient that the error term of the parametrix be compact.  We need in fact good decay of the error term at some of the boundary hypersurfaces of the $\QFB$ double space.  In turn, to obtain this good decay, we rely crucially on fine compositions results for $\QFB$ pseudodifferential operators.  This is also the case for the corresponding result for fibered boundary metrics \cite[Corollary~{3.17}]{KR0}, that is, the depth 1 case of Corollary~\ref{fredh.1}.

Another important consequence of Theorem~\ref{para.1} is the following result about the decay of harmonic forms.
\begin{Corollary}
For $(M,g)$ as in Theorem~\ref{para.1}, the space of $L^2$ harmonic forms is contained in
\[
	v^\epsilon L^2_{\QFB}(M; \Lambda^*(T^* M)),
\]
for some $\epsilon>0$.
\label{decay.1}\end{Corollary}
In the companion paper \cite{KR2}, this decay of $L^2$ harmonic forms is central in proving the Vafa-Witten conjecture on the Hilbert scheme of $n$ points on $\bbC^2$, as well as in proving the Sen conjecture for the monopole moduli space of charge $3$ using the announcement of \cite{FKS}.  Compared to the depth 1 case \cite[Corollary~{3.16}]{KR0}, notice however that we  cannot show that $L^2$ harmonic forms admit a polyhomogeneous expansion at infinity.  In fact, we are not even able to show that $L^2$ harmonic forms are conormal  in the sense of \cite[\S~3]{Melrose1992}, \ie that they stay in $L^2$ under the repeated action of $b$-vector fields.  We introduce instead the weaker notion of $\QFB$ conormal sections, that is, space of sections stable under the action of $\QFB$ vector fields.  This weaker notion of conormality allows us to give a sufficiently good pseudodifferential characterization of the inverse of the Fredholm operator of Corollary~\ref{fredh.1}.
\begin{Corollary} [Corollary~\ref{hd.4}]
The inverse of the Fredholm operator of Corollary~\ref{fredh.1}, that is, the operator $G$ such that
$$
           G(v^{-\frac12}\eth_{\QFB}v^{-\frac12})=(v^{-\frac12}\eth_{\QFB}v^{-\frac12})G=\Id -P
$$ 
with $P$ the $L^2$ orthogonal projection onto the kernel of $(v^{-\frac12}\eth_{\QFB}v^{-\frac12})$, is a $\QFB$ pseudodifferential operator of order $-1$.
\label{pch.1}\end{Corollary}

\medskip

We now proceed to discuss the above results in more detail and  put them in proper context.  
In \cite{DM2018}, Degeratu and Mazzeo were able to extend and generalize the isomorphism theorem of Joyce \cite[Corollary~9.5.14]{Joyce} to Laplace-type operators associated to a $\QAC$ metric.  While Joyce's approach was through the construction of a suitable barrier function, the results of \cite{DM2018} were obtained via careful heat kernel estimates based on the work of Grigor'yan and Saloff-Coste \cite{GSC2005}.  This allowed them to improve the isomorphism theorem of Joyce even for the scalar Laplacian of $\QALE$ metrics.  
This linear isomorphism of \cite[(7.1) and Theorem~7.6]{DM2018} was subsequently used in \cite{CDR} to produce examples of Calabi-Yau $\QAC$ metrics that are not $\QALE$.  More recently, relying on an ansatz of Hein-Naber \cite{Naber},  Yang Li \cite{Li2019} on $\bbC^3$, and independently  \cite{CR} and \cite{Gabor2019} on $\bbC^n$ for $n\ge 3$, produced examples of Calabi-Yau metrics of maximal volume growth that are not flat, giving counter-examples in all dimensions to a conjecture of Tian \cite{Tian2006}.  The tangent cone at infinity of these metrics has singular cross-section.  However, as explained in \cite{CR}, they are not quite $\QAC$ metrics, but are conformal to $\QAC$ metrics.  Other examples of Calabi-Yau metrics of maximal volume growth with tangent cone at infinity having a singular cross-section have been obtained by Biquard-Delcroix \cite{Biquard-Delcroix} via a generalization of a construction of Biquard-Gauduchon \cite{Biquard-Gauduchon}.

In all these examples, a key role was played by the mapping properties of the scalar Laplacian.  Unfortunately however,  unless strong assumptions on the curvature are made, neither the approach of Joyce nor the approach of Degeratu-Mazzeo seems to generalize to more complicated operators like the Hodge Laplacian or the Hodge-deRham operator.  On the other hand, one important ingredient in the construction of Calabi-Yau $\QAC$ metrics in \cite{CDR} was an alternative description of $\QAC$ metrics in terms of a compactification by a manifold with corners whose boundary hypersurfaces correspond to the various models of the form \eqref{int.5} at infinity.  These induce a fiber bundle on the total space of each boundary hypersurface, inducing an iterated fibration structure in the sense \cite{AM2011, ALMP2012}.  In other words, in the terminology of \cite{DLR}, the compactification is a manifold with {\em fibered corners} (defined below in section~\ref{vf.0}).  

To describe $\QAC$ metrics on such a compactification $M$, one associates a natural Lie structure at infinity in the sense of \cite{ALN04}.  Such a structure gives a Lie algebroid ${}^{\QAC}TM\to M$ with anchor map $a: {}^{\QAC}TM\to TM$ inducing an isomorphism on the interior of $M$.  A $\QAC$ metric on the interior of $M$ is then the restriction of a bundle metric on ${}^{\QAC}TM$ via the isomorphism $a_{M\setminus \pa M}: {}^{\QAC}TM|_{M\setminus \pa M}\to T(M\setminus \pa M)$.  This point of view naturally led the authors of \cite{CDR} to introduce the more general class of $\QFB$ metrics in terms of a Lie structure at infinity.

It is precisely in this setting of manifolds with corners and Lie structures at infinity that
one can hope to apply the microlocal approach of Melrose to study operators like the Hodge-deRham operator.  Roughly speaking, this approach consists of introducing a suitable manifold with corners---the {\em double space}---on which to define Schwartz kernels of pseudodifferential operators adapted to the geometry. Such operators include, among other things, the inverses (Green's functions) of Laplace-type operators. On the double space, the restriction of pseudodifferential operators to the various boundary hypersurfaces meeting the diagonal gives models---the normal operators alluded to above---for their asymptotic behaviors at infinity. One can then hope to invert a Laplace-type operator, or at least show it is Fredholm, by inverting these models.  

This microlocal approach has been carried out with great success in many situations where the Lie structure at infinity is defined on a compact manifold with boundary:  by Mazzeo-Melrose \cite{MM1987} for asymptotically hyperbolic metrics, by Epstein-Melrose-Mendoza \cite{EMM} for asymptotically complex hyperbolic metrics, by Melrose \cite{MelroseAPS, MelroseGST}  for asymptotically cylindrical metrics and asymptotically conical metrics,  by Mazzeo \cite{MazzeoEdge} for edge metrics and by Mazzeo-Melrose \cite{Mazzeo-MelrosePhi} for fibered boundary metrics.  These various pseudodifferential calculi can also be used to study operators associated to conformally related problems.  For instance, in \cite{GKM}, operators associated to wedge metrics (also called incomplete edge metrics)  can be studied using the edge calculus of \cite{MazzeoEdge}, while in his thesis \cite{Vaillant}, Vaillant study operators associated to fibered cusps metrics via the $\Phi$ calculus of \cite{Mazzeo-MelrosePhi}.  However, when the Lie structure at infinity is defined over a manifold with {\em corners} (beyond a simple boundary), it has only been in the last few years that there have been attempts to carry the microlocal approach of Melrose, notably in \cite{DLR} for fibered corners metrics and in \cite{AG} for edge metrics on stratified spaces of depth 2 or higher.  

In the present paper, we begin to develop the microlocal approach of Melrose to $\QFB$ operators.  Now, as discussed in \cite{CDR}, the definition of $\QFB$ metrics is closely related to the one of fibered corners metrics of \cite{DLR}.  One important difference however, is that, compared to  fibered corners metrics or even to edge metrics, 
 the various local models of $\QFB$ metrics are more intricate and do not admit a nice warped product presentation.  To be more precise, the local models \eqref{int.5} are of course nice Cartesian products of metrics, but the problem is that the subsets \eqref{int.5b} on which they are considered are not.  In terms of Lie structures at infinity, this corresponds to a more intricate description of $\QFB$ vector field in local coordinates, \cf \eqref{ds.4} for $\QFB$ vector fields with \cite[(2.5)]{DLR} for fibered corners vector fields or \eqref{E:VF_ie} for edge vector fields.  This has important implications for the construction of the $\QFB$ double space.  In particular, in contrast to \cite{DLR,AG}, it is not possible to solely blow up $p$-submanifolds intersecting the lifted diagonal.  This is because some of the submanifolds that  `need' to be blown-up are not initially $p$-submanifolds.  Thus it is only after performing preliminary blow-ups away from the lifted diagonal that those become $p$-submanifolds and can subsequently be blown up.  The order in which those blow-ups are performed turn out to be crucial, especially for the subsequent construction of a triple space needed to establish composition results.  We need in fact to rely on the companion paper \cite{KR3}, where a systematic way of blowing up, adapted to QFB metrics, is introduced. 
 
 As alluded  in Corollary~\ref{pch.1}, we need also to replace the conormality of \cite[\S~3]{Melrose1992} with the weaker notion of $\QFB$ conormality.  Because of this, we cannot directly rely on the standard pushforward theorem to establish composition, but fortunately, straightforward modifications of the proof of \cite[Theorem~4]{Melrose1992} suffice to provide the composition result that we need.  Notice that such a weaker notion of conormality can naturally be formulated for other calculi, for instance for the edge calculus \cite{MazzeoEdge,AG}, allowing to deal with situations where polyhomogeneity results are out of reach or not expected.  Except for this weaker notion of conormality, one can then define corresponding $\QFB$ pseudodifferential operators following the usual approach.  In particular, we can introduce suitable normal operators that correspond to suspended versions of $\QFB$ operators.  
This allows us, as in \cite{Mazzeo-MelrosePhi,DLR}, to introduce the notion of fully elliptic $\QFB$ operator in Definition~\ref{sm.31} below and to prove Theorem~\ref{int.7} above.
As discussed above, to analyze certain non-fully elliptic operators such as the Hodge-deRham operator, we develop a calculus of $\Qb$ operators. The double space for these is essentially the $\QAC$ double space, but with the last blow-up omitted.  

If we consider the Hodge-deRham operator associated to a $\QAC$ metric, there are two types of model operators to consider at infinity.  For a maximal boundary hypersurface (with respect to the partial order induced by the iterated fibration structure) the model operator essentially corresponds to the Hodge-deRham operator on the tangent cone at infinity of the $\QAC$ metric, which is a cone with cross-section a stratified space equipped with a wedge metric.  We can therefore use the edge calculus of \cite{MazzeoEdge,AG}  as well as the $\Qb$ calculus developed here to invert this model.  
For a non-maximal boundary hypersurface, the model to invert is  an operator of the form $\eth_{Z}+\eth_{V}$, where $\eth_{Z}$ is a Hodge-deRham operator of a $\QAC$ metric $g_Z$ on a manifold $Z$ of lower depth and $\eth_{V}$ is the Hodge-deRham operator of some Euclidean space $V$.  To invert it, we take a Fourier transform in the Euclidean factor and obtain a family of $\QAC$ operators parametrized by the dual variable  $\xi\in V^*$.  The key difficulty in trying to invert this family is that when $\xi\ne 0$, it is invertible as a $\QAC$ operator, while it is not when $\xi=0$. However, when $(Z,g_Z)$ is in fact an $\AC$ manifold, this is exactly the setting considered in our companion paper \cite{KR0}, where a generalization of  the works of Guillarmou-Hassell \cite{GH1,GH2} and Guillarmou-Sher \cite{GS} is obtained for the low energy limit of the resolvent of some Dirac operators associated to an $\AC$ metric and more generally a fibered boundary metric.  Using this result, we can then invert the model at $\xi=0$ as a weighted $b$ operator and combine it with the $\AC$ inverses for $\xi\ne 0$.   Thanks to the detailed description of  inverse Fourier transform of that inverse in \cite{KR0},  one can check that this inverse fits nicely as a conormal distribution on the corresponding boundary face of the $\QAC$ double space. Doing so however, the inverse of the model operator creates an error term at another boundary hypersurface away from the diagonal.  This forces us to invert a model at this boundary face as well.  Fortunately, this model is again given by the Hodge-deRham operator of a cone with cross-section a stratified space with a wedge metric, so we can invert it in the same way that it was done for the model associated to a maximal boundary hypersurface.  Thus, if we are with the Hodge-deRham operator of a $\QAC$ metric on a manifold with corners of depth 2, this completes the inversion of each model at infinity and produces a nice parametrix in the large $\QAC$ pseudodifferential calculus.  Since \cite{KR0} also works for fibered boundary metrics, a nice parametrix can also be produced for $\QFB$ manifolds of depth $2$ following roughly the same strategy.  

To extend this parametrix construction to $\QFB$ manifolds of higher depth, we can employ the same strategy provided one can obtain a suitable higher depth version of \cite{KR0}.  This turns out to be possible, namely, we manage to construct a suitable double space and a corresponding $\k,\QFB$ calculus in  which the low energy limit of the resolvent of a Hodge-deRham $\QFB$ operator admits a pseudodifferential characterization.  
\begin{Corollary}
If $(M,g)$ is a $\QFB$ manifold as in Theorem~\ref{para.1}, then the low energy limit of the resolvent of the corresponding Hodge Laplacian admits a pseudodifferential characterization in the sense of Corollary~\ref{HdR.7} below.
\label{para.2}\end{Corollary}         
The proof of this corollary relies in particular on Corollary~\ref{pch.1}.  We can also characterize the inverse Fourier transform of such a pseudodifferential characterization and check that it fits nicely where it should on a $\QFB$ double space of higher depth.  This allows to prove Theorem~\ref{para.1} by induction on the depth of the $\QFB$ manifold with Corollary~\ref{para.2} corresponding to the inductive step.  For this inductive argument to work, let us note however that we also need more generally a suspended version $\k,\QFB,\sus$ of the $\k,\QFB$ calculus, that is, a version of the calculus that allows to give a pseudodifferential characterization of the low energy limit of the resolvent of an operator of the form $\eth_Z+\eth_V$ with $\eth_Z$  the Hodge-deRham operator of a $\QFB$ metric $g_Z$ on a manifold $Z$ and $\eth_V$ the Hodge-deRham operator of some Euclidean space $V$.  This is because such resolvents occur as models at infinity (with $Z$ of lower depth) when trying to obtain a pseudodifferential characterization of the low energy limit of the resolvent of the Hodge-deRham operator of a $\QFB$ metric.    We can summarize how all these results combine to achieve the inductive step of the proof of Theorem~\ref{para.1} in the following diagram:
\begin{equation}
\xymatrix{
\mbox{depth} \; k & \QFB \ar[dd]\ar[rrrr]^-{Corollary~\ref{pch.1}} & & & & \k,\QFB \ar[d]_-{\mbox{Inverse Fourier transform}}  \ar[rrrr] & & & & \k,\QFB, \sus \ar[lllddl]\ar[dd] \\
 &  &  & & & \QFB, \sus \ar[dllll]\ar[urrrr] & & & &  \\
 \mbox{depth} \; k+1 & \QFB \ar[rrrr]_-{Corollary~\ref{pch.1}} & & & & \k,\QFB  \ar[rrrr] & & & & \k,\QFB, \sus 
}
\label{indstep.1}\end{equation}  
In this diagram, the first column corresponds to Theorem~\ref{para.1}, the items $\k,\QFB$ in the middle column correspond to Corollary~\ref{para.2} while the last column corresponds to a suspended version of Corollary~\ref{para.2}, namely, Theorem~\ref{ksm.38} below.  The item $ \QFB, \sus$  in the middle of the diagram corresponds to a suspended version of Corollary~\ref{pch.1}, namely a pseudodifferential characterization of the inverse of an operator of the form $\eth_Z+\eth_V$ with $\eth_Z$ the Hodge-deRham operator of a $\QFB$ metric $g_Z$ and $\eth_V$ the Hodge-deRham operator of some Euclidean space.  An arrow from one item to another indicates that the first item is used in a proof establishing the second item.  For instance, the arrow from $\QFB,\sus$ to $\QFB$ indicates that the proof of Theorem~\ref{para.1} on a $\QFB$ manifold of depth $k+1$  uses the suspended version of Corollary~\ref{pch.1} in depth $k$.

The paper is organized as follows.  In \S~\ref{vf.0} we review the notion of $\QFB$ metrics, $\QAC$ metrics and $\Qb$ metrics.  This is used in \S~\ref{ds.0} to construct the $\QFB$ double space and the $\Qb$ double space.  We then define pseudodifferential $\QFB$ operators and $\Qb$ operators in \S~\ref{ts.0}.  The $\QFB$ and $\Qb$ triple spaces are introduced in \S~\ref{tris.0}, yielding corresponding composition results in \S~\ref{com.0}.  In \S~\ref{sm.0}, we introduce the principal symbol of pseudodifferential $\QFB$ operators, as well as various normal operators, leading to the notion of fully elliptic $\QFB$ operator.  In \S~\ref{mp.0}, we obtain criteria for $\QFB$ operators and $\Qb$ operators to act on weighted $L^2$ spaces and show that fully elliptic differential $\QFB$ operators are Fredholm when acting on suitable $\QFB$ Sobolev spaces.  In \S~\ref{smb.0},  we introduce the principal symbol of $\Qb$ operators as well as normal operators and explain how the edge double space of \cite{MazzeoEdge, AG} naturally arises on the front faces of the $\Qb$ double space associated to maximal boundary hypersurfaces.  We take this as an opportunity to describe the results about edge operators that we will need.   In  \S~\ref{do.0}, we construct a parametrix in the $\QFB$ calculus for certain Dirac operators associated to a  $\QFB$ metric modulo some assumptions and prove Theorem~\ref{para.1} when the underlying $\QFB$ manifold is of depth $2$.  To extend this result to arbitrary depth, we introduce in \S~\ref{kqfb.0} $\k,\QFB$ operators, a pseudodifferential calculus adapted to the study of the low energy limit of the resolvent of a Dirac $\QFB$ operator.  After introducing a triple space for this calculus in \S~\ref{tkqfb.0}, we obtain in \S~\ref{ckqfb.0} good composition results.  After introducing symbol maps in \S~\ref{ksm.0}, we can finally provide in \S~\ref{qfble.0} a pseudodifferential characterization of the low energy limit of the resolvent of a Dirac $\QFB$ operator modulo some inductive assumption.  The inverse Fourier transform of this pseudodifferential characterization is given in \S~\ref{ift.0}, while \S~\ref{sle.0} provides a suspended version of the results of \S~\ref{qfble.0}.  Finally, all these results are combined in \S~\ref{HdR.0} to derive our main result Theorem~\ref{para.1}.

\begin{Acknowledgement}
The authors are grateful to Rafe Mazzeo, Richard Melrose and Michael Singer for stimulating discussions related to their project, as well as to two anonymous referees for many helpful suggestions to improve the manuscript. 
CK was supported by NSF Grant No.\ DMS-1811995. In addition, this material is based in part on work 
supported by the NSF
under Grant No.\ DMS-1440140 while CK was in residence at the
Mathematical Sciences Research Institute in Berkeley, California, during the
Fall 2019 semester.  FR was supported by NSERC and a Canada Research chair.
\end{Acknowledgement}


\numberwithin{equation}{section}

\section{$\QFB$ structures} \label{vf.0}

In this section, we will review the notion of quasi-fibered boundary metrics introduced in \cite{CDR} and clarify how such a notion depends on a choice of total boundary defining function.  
To this end, recall first that a {\bf manifold with fibered corners} (also called iterated space or manifold with corners equipped with an iterated fibration structure) \cite{AM2011,ALMP2012, DLR} is a manifold, $M$, with corners, all of whose boundary hypersurfaces $H$ are equipped with surjective submersions
\[
	\phi : H \to S, \quad \text{with fiber $Z$},
\]
satisfying a compatibility condition described below.
Here the base $S$ and fiber $Z$ are themselves manifolds with corners, and by the Ehresmann lemma \cite{Ehresmann} in the category of manifolds with corners, $\phi$ is a locally trivial fiber bundle.
It follows that boundary hypersurfaces of $H$ itself come in two types, either the restriction $\phi \rst_{\phi^\inv(T)}$ of the fiber bundle over a boundary hypersurface $T$ of the base $S$, or a smaller fiber bundle over $S$ with fiber a boundary hypersurface $W$ of the fiber $Z$.
Note that disjoint\footnote{by the usual condition that boundary hypersurfaces of a manifold with corners be embedded, only disjoint hypersurfaces of $Z$ may belong to a common boundary hypersurface of  $H$.}  boundary hypersurfaces of $Z$ may belong to the same boundary hypersurface of  $H$, so we allow for the possibility that $W \subset Z$ consists of a union of disjoint boundary hypersurfaces.

The compatibility condition is as follows: whenever $H_1 \cap H_2 \neq \emptyset$, 
we require that $\dim(S_1) \neq \dim(S_2)$ (say without loss of generality that $\dim(S_1) < \dim(S_2)$, and hence $\dim(Z_1) > \dim(Z_2)$), and then, on $H_1 \cap H_2$, $\phi_2$ restricts to a fiber bundle $\phi_2: H_1\cap H_2\to \pa_1 S_{2}$ onto a boundary hypersurface $\pa_{1} S_2$ of $S_2$ on one hand, while $\phi_1$ restricts to be a fiber bundle over $S_1$ with fiber a boundary hypersurface $\pa_{2} Z_1$ on the other hand.
We further require the existence of a fibration $\phi_{21} : \pa_1 S_2 \to S_1$ such that
$\phi_1 = \phi_{21} \circ \phi_2$; in other words,
\begin{equation}
	H_1 \cap H_2 \stackrel{\phi_2} \to \pa_{1} S_2 \stackrel{\phi_{21}} \to S_1
\label{iterated.1}\end{equation}
is an `iterated fibration'.

This fibered corners structure induces a partial order on hypersurfaces, where $H < H'$ if $H \cap H' \neq \emptyset$ and $\dim(S) < \dim(S')$, and likewise induces a fibered corners structure on each fiber space $Z$ and base space $S$.
By convention, we choose an enumeration $H_1,\ldots, H_\ell$ of the boundary hypersurfaces which is consistent with the partial order, with $H_i < H_j$ implying $i < j$, and we denote the collection of fibrations by $\phi = (\phi_1,\ldots, \phi_\ell)$; then we say that $(M,\phi)$ is a manifold with fibered corners.
It follows that $H_{i_1} \cap \cdots \cap H_{i_k} \neq \emptyset$ if and only if $H_{i_1} < \cdots < H_{i_k}$ is a totally ordered chain, and the fiber bundles induce an iterated fibration
\[
	H_{i_1}\cap \cdots \cap H_{i_k} \stackrel{\phi_{i_k}} \to \pa_{i_1 \cdots i_{k-1}} S_{i_k} \stackrel{ \phi_{i_k i_{k-1}}} \to \pa_{i_1 \cdots i_{k-2}} S_{i_{k-1}} \to \cdots \stackrel{ \phi_{i_2 i_1}} \to S_{i_1},
\]
the composition of the first $j+1$ of which constitute the restriction of $\phi_{i_{k-j}}$ to $H_{i_1}\cap \cdots \cap H_{i_k}$, which surjects onto the boundary face $\pa_{i_1 \cdots i_{k-j-1}} S_{i_{k-j}}$ of $S_{i_{k-j}}$, and has fiber
the boundary face $\pa_{i_k \cdots i_{k-j+1}} Z_{i_{k-j}}$ of $Z_{i_{k-j}}$.

As shown in \cite[Lemma~1.4]{DLR}, it is possible to choose global boundary defining functions $x_i$ for each $H_i$ which are {\em compatible} with $\phi$ in the sense that each $x_i$ is constant along the fibers of $\phi_j$ for $H_i<H_j$.  Unless otherwise stated, we will always assume that the boundary defining functions considered are compatible with the structure of fibered corners.  
By \cite[Lemma~1.10]{CDR}, recall that, given compatible boundary defining functions
$x_1,\ldots, x_{\ell}$, there is for each $i$ a tubular neighborhood 
\begin{equation}
 c_i: H_i\times [0,\epsilon)\hookrightarrow M
\label{pt.1}\end{equation}  
such that 
\begin{enumerate}
\item $c_i^*x_i=\pr_2$;
\item $c_i^*x_j=x_j\circ \pr_1$ for $j\ne i$;
\item $\phi_j \circ c_i(h,t)=\phi_j(h)\in S_j$ for $h\in H_i\cap H_j$, $H_j<H_i$;
\item For $H_j>H_i$, there is a corresponding tubular neighborhood $c_{ij}: \pa_i S_j\times [0,\epsilon)\to S_j$ such that $c_{ij}^{-1}\circ \phi_j\circ c_i(h,t)=(\phi_j(h),t)\in \pa_i S_j\times [0,\epsilon)$ for $h\in H_i\cap H_j$;
\end{enumerate}
where $\pr_1: H_i\times [0,\epsilon)\to H_i$ and $\pr_2: H_i\times [0,\epsilon)\to [0,\epsilon)$ are the projections on the first and second factors.  Similarly, in a sufficiently small neighborhood of a point $p$ in a component of the intersection $H_{i_1} \cap \cdots \cap H_{i_k}$, which we can relabel for simplicity as $H_1 \cap \cdots \cap H_k$, it is always possible to choose local coordinates
\begin{equation}
   (x_1,y_1,x_2,y_2,\ldots, x_{k},y_k,z)
\label{coor.1}\end{equation} 
where $y_i=(y_i^1,\ldots, y_i^{\ell_i})$ and $z\in(z^1,\ldots, z^q)$ are tuples,
such that 
\begin{itemize}
\item $x_i$ is a local boundary defining function for $H_i$,
\item $(x_1,\ldots, x_{i-1}, y_1,\ldots, y_i)$ are local coordinates for $S_i$ near $\phi_i(p)$, and
\item the fiber bundle $\phi_i$ is given in local coordinates by the projection
\begin{equation}
	\phi_i : (x_1,\ldots,\widehat{x}_i,\ldots,x_k,y_1,\ldots,y_k,z)\mapsto (x_1,\ldots,x_{i-1},y_1,\ldots,y_i),
\label{coor.2}\end{equation}
with the hat denoting omission.
In particular, the coordinates $(y_{i+1},\ldots, y_k, z)$ restrict to give coordinates on the fibers of $\phi_i$.  
\end{itemize}

We remark that to a manifold with fibered corners $(M,\phi)$,  we can canonically associated  a certain stratified space $X$ by iteratively collapsing the fibers of each boundary fibration. %
Thus $X$ is a disjoint union of strata $s_0, s_1, \ldots, s_\ell$ where $s_0 = M^\circ$ is dense, $s_i = S_i^\circ$, and the closure $\ol{s_i}$ of each stratum is the space obtained from $S_i$ by collapsing its boundary fibers. 
The partial order on strata where $s_i < s_j$ if $s_i \subset \ol{s_j}$ is the same as the order on hypersurfaces of $M$ discussed above, and corresponds to the reverse of the order by depth.
The resulting space $X$ carries additional data (not discussed here; see \cite{ALMP2012}) with respect to which $M$ may be recovered.

Let $(M,\phi)$ be a manifold with fibered corners as above, with compatible global boundary defining functions $x_1,\ldots, x_{\ell}$, and let $\kridx{v}{v_tot}{Total boundary defining function}=\prod_{i=1}^{\ell} x_i$ be the corresponding total boundary defining function.  
Then recall from\ \cite{CDR} that this choice of $v$, combined with the structure of manifold with fibered corners, induces a Lie algebra of quasi-fibered boundary vector fields $\cV_{\QFB}(M)$ ($\QFB$ vector fields for short).  More precisely, it is a subalgebra of the Lie algebra of $b$ vector fields 
$$
\kridx{\cV_b}{Vb}{$b$ vector fields}(M):= \{ \xi \in \CI(M,TM) \; | \; \xi x_i\in x_i\CI(M) \forall i\}
$$
defined by imposing the extra conditions that $\xi\in \cV_{b}(M)$ is in $\kridx{\cV_{\QFB}}{VQFB}{QFB vector fields}(M)$ if 
\begin{enumerate}
\item[(QFB1)] $\left.\xi\right|_{H_i}$ is tangent to the fibers of $\phi_i:H_i\to S_i$, and
\item[(QFB2)] $\xi v\in v^2\CI(M)$.
\end{enumerate}
With respect to local coordinates \eqref{coor.1}, it is straightforward to verify in light of \eqref{coor.2} that
the Lie subalgebra of vector fields satisfying (QFB1) alone, the edge vector fields $\kridx{\cV_e}{Ve}{edge vector fields}(M)$  of \cite{MazzeoEdge,ALMP2012, AG},  is generated by 
\begin{equation}
	v_1 \pd{x_1},\, v_1 \pd{y_1},\, v_2 \pd{x_2},\, v_2 \pd{y_2},\, \ldots,\, v_k \pd{x_k},\, v_k \pd{y_k},\, \pd z,
	\label{E:VF_ie}
\end{equation}
where $\kridx{v_i}{v_i}{partial boundary defining function} = \prod_{j=i}^k x_j$ and $\pd{y_i}$ and $\pd z$ are shorthand
for the sets of vector fields $\pd{y_i^1}, \ldots, \pd{y_i^{\ell_i}}$ and $\pd{z^1}, \ldots \pd{z^q}$, respectively.
Incorporating the additional condition (QFB2) leads to the subalgebra generated in 
local coordinates by
\begin{equation}
v_1x_1\frac{\pa}{\pa x_1}, v_1\frac{\pa}{\pa y_1}, v_2\left( x_2\frac{\pa}{\pa x_2}- x_1\frac{\pa}{\pa x_1} \right), v_2\frac{\pa}{\pa y_2}, \ldots, v_k\left( x_{k}\frac{\pa}{\pa x_{k}}- x_{k-1}\frac{\pa}{\pa x_{k-1}} \right), v_k\frac{\pa}{\pa y_k}, \frac{\pa}{\pa z}.
\label{ds.4}
\end{equation}
Indeed, the natural modification $v_i\, x_i\pd{x_i}$ of the iterated edge vector fields normal to the boundary hypersurfaces satisfy condition (QFB2); however the differences $x_i \pd{x_i} - x_j\pd{x_j}$ annihilate $v$ altogether, and multiplied by the factor of $v_{\min(i,j)+1}$ are seen to lie in the subalgebra of \eqref{E:VF_ie}, so \eqref{ds.4} provides a generating set.

This local computation shows that the condition (QFB2) above is equivalent to the condition appearing in \cite{CDR} that $\xi v_i \in v_i^2 \CI(M)$ for each $1\leq i \leq \ell$; in particular, the condition for $i > 2$ follows from the condition for $i = 1$ alone.  Compared to $\cV_e(M)$, the definition of $\cV_{\QFB}(M)$ depends in principle on the choice of the total boundary defining function $v$.   We say that two total boundary defining functions are \textbf{$\QFB$ equivalent} if they yield the same Lie algebra of $\QFB$ vector fields.  The following criterion will be useful to determine when two total boundary defining functions are $\QFB$ equivalent.  
\begin{lemma}
Two total boundary defining functions $v,v'\in \CI(M)$ are $\QFB$ equivalent if and only if the function 
$$
   f:= \log\left( \frac{v'}{v} \right)
$$
is such that for all boundary hypersurface $H_i$, $\left.f  \right|_{H_i}= \phi_i^* h_i$ for some $h_i\in \CI(S_i)$.  
\label{qfbs.1}\end{lemma}
\begin{proof}
We essentially follow the proof of \cite[Lemma~1.16]{CDR}.  Clearly, $v$ and $v'$ are $\QFB$ equivalent if and only if for all $\xi\in \cV_e(M)$, 
\begin{equation}
      \frac{dv}{v^2}(\xi)\in \CI(M) \; \Longleftrightarrow \; \frac{dv'}{(v')^2}(\xi)\in \CI(M).
\label{qfbs.2}\end{equation}
Now, by the definition of $f$, we have that $v'= e^f v$, so
\begin{equation}
  \frac{dv'}{(v')^2}= e^{-f}\left(  \frac{dv}{v^2}+ \frac{df}{v}\right).
\label{qfbs.3}\end{equation}
In particular, if for each boundary hypersurface $H_i$,  $\left.f  \right|_{H_i}= \phi_i^* h_i$ for some $h_i\in \CI(S_i)$, then \eqref{qfbs.2} clearly holds for each $\xi\in \cV_e(M)$.  Conversely, if for some $H_i$, $f|_{H_i}$ is not the pull-back of a function on $S_i$, then we can find $\xi\in \cV_e(M)$ such that $\frac{dv}{v^2}(\xi)\in \CI(M)$ and $df(\xi)|_{H_i}\ne 0$, in which case \eqref{qfbs.3} shows that \eqref{qfbs.2} cannot hold.  
\end{proof}
In this paper, we will usually assume that a Lie algebra of $\QFB$ vector fields has been fixed once and for all.  If $v$ is a total boundary defining function inducing this Lie algebra of $\QFB$ vector fields, we will therefore also say that another total boundary defining function $v'$ is \textbf{compatible} with this Lie algebra of $\QFB$ vector fields if it is $\QFB$ equivalent to $v$.

The sheaf of sections of $\cV_\QFB(M)$ is locally free of rank $\dim M$, and so by the Serre-Swan theorem \cite[Theorem~2.1]{ALN04}, there is a vector bundle, the $\QFB$ tangent bundle $\kridx{\QFBT}{TQFB}{QFB tangent bundle}M \to M$, with the defining property that  $\cV_\QFB(M) = C^\infty(M; \QFBT M)$.  A local frame (nonvanishing up to and including the boundary) for this vector bundle is given by the vector fields in \eqref{ds.4}.
The inclusion $\cV_\QFB(M) \subset \cV(M)$ induces a natural bundle map $\QFBT M \to TM$
which in coordinates amounts to simply evaluating \eqref{ds.4} as ordinary vector fields.  In particular, the bundle map is an isomorphism over the interior but has kernel
at a point in any boundary face $H_i$.
Composition of QFB vector fields as operators on $\CI(M)$ leads to the enveloping algebra
of  differential $\QFB$ operators, $\Diff_\QFB^\ast(M)$, which is filtered by degree and characterized by 
\[
	\kridx{\Diff^k_\QFB}{DiffQFB}{QFB differential operators}(M) = \set{\xi_1\cdots \xi_k : \xi_i \in \cV_\QFB(M)} + \Diff^{k-1}_\QFB(M),
	\quad \Diff_\QFB^0(M) = \CI(M).
\]
By tensoring with sections of $\Hom(E,F)$ for vector bundles $E$ and $F$ over $M$
we obtain the spaces of  differential QFB operators $\Diff_\QFB^\ast(M;
E, F)$ acting between sections of these respective bundles.

The QFB cotangent bundle $\QFBT^\ast M \to M$ is defined to be the dual bundle
to $\QFBT M$.  In local coordinates, notice that a basis of local sections is given by the basis
\begin{equation}
\frac{dv_1}{v_1^2}, \frac{dy_1}{v_1}, \frac{dv_2}{v_2^2}, \frac{dy_2}{v_2}, \ldots, \frac{dv_k}{v_k^2}, \frac{dy_k}{v_k}, dz
\label{ds.4b}\end{equation}
dual to \eqref{ds.4}.
Before introducing the notion of $\QFB$ metric, it will be useful to have a better understanding of the $\QFB$ tangent $\QFBT M$ near a boundary hypersurface $H_i$.
Since the fibers of $\phi_i: H_i\to S_i$ are naturally manifolds with fibered corners, we can consider the Lie algebra of vector fields $\cV_{\phi}(H_i)$ consisting of vector fields
$\xi\in \CI(H_i;TH_i)$ such that 
\begin{itemize}
\item $\xi$ is tangent to the fibers of $\phi_i:H_i\to S_i$;
\item for $H_j>H_i$, $\xi|_{H_j\cap H_i}$ is tangent to the fibers of $\phi_j: H_j\cap H_i\to \pa_iS_{j}$;
\item $\xi v_{H_i}\in v^2_{H_i}\CI(H_i)$, where $\displaystyle v_{H_i}= \prod_{H_j>H_i} x_j$.
\end{itemize}
\noindent Notice that this is just a fibered version of the QFB algebra, defining
a smooth algebra of QFB vector fields on the fibers $Z_i$ of $\phi_i: H_i \to S_i$
which vary smoothly with respect to the base $S_i$.

By the Serre-Swan theorem, to $\cV_{\phi}(H_i)$  corresponds a  vector bundle $\kridx{{}^{\phi}T(H_i/S_i)}{TphiHS}{relative QFB tangent bundle}\to H_i$ with a canonical isomorphism
$$
      \CI(H_i; {}^{\phi}T(H_i/S_i))=\cV_{\phi}(H_i).
$$
In terms of this vector bundle, the natural map $\cV_{\QFB}(M)|_{H_i}\to \cV_{\phi}(H_i)$ becomes a surjective morphism of vector bundles 
\begin{equation}
          {}^{\QFB}TM|_{H_i}\to {}^{\phi_i}T(H_i/S_i).
\label{vf.1b}\end{equation}
Let $\kridx{{}^{\phi}NH}{NphiH}{relative QFB normal bundle}_i\to H_i$ denote the kernel of this map, so that there is the following short exact sequence of vector bundles
\begin{equation}
\xymatrix{
0\ar[r] & {}^{\phi}NH_i \ar[r] & {}^{\QFB}TM|_{H_i} \ar[r] & {}^{\phi_i}T(H_i/S_i)\to 0.
}
\label{vf.1}\end{equation}
Hence, a choice of splitting ${}^{\phi}T(H_i/S_i)\hookrightarrow {}^{\QFB}TM|_{H_i}$ induces identifications
$$
       {}^{\QFB}TM|_{H_i}= {}^{\phi}NH_i\oplus {}^{\phi}T(H_i/S_i) \quad \mbox{and} \quad {}^{\phi}NH_i\cong {}^{\QFB}TM|_{H_i}/{}^{\phi}T(H_i/S_i).
$$

The vector bundle ${}^{\phi}NH_i$ is in fact naturally isomorphic to the pull-back of a vector bundle on $S_i$.  To see this, notice first that we can equip $S_i\times [0,1)$ with a structure of manifold with fibered corners.  Indeed, the boundary hypersurfaces of $S_i\times [0,1)$ are given by $S_i\times \{0\}$ and $\pa_jS_{i}\times [0,1)$ whenever $H_j<H_i$.  On $S_i$, we can put the fiber bundle $S_i\times \{0\}\to S_i\times \{0\}$ given by the identity map.  On $\pa_jS_{i}\times [0,1)$, we can instead consider the fiber bundle 
$$
        \phi_{ij}\circ \pr_1: \pa_jS_{i}\times [0,1)\to S_j
$$
obtained by composing the projection $\pr_1: \pa_jS_{i}\times [0,1)\to \pa_jS_{i}$ with $\phi_{ij}:\pa_jS_{i}\to S_j$.  As can be readily checked, this induces a structure of manifold with fibered corners on $S_{i}\times [0,1)$ with partial order on the boundary hypersurfaces given by
$$
      S_i\times \{0\}> \pa_jS_{i}\times [0,1)\quad \forall \ H_j<H_i  \quad\mbox{and} \quad \pa_jS_{i}\times [0,1) >\pa_kS_{i}\times [0,1)\; \Longleftrightarrow \; H_i>H_j>H_k.
$$
Since the boundary defining functions of $M$ are assumed to be compatible with the structure of manifold with fibered corners, we see that for $H_j<H_i$, $x_j$ descends to give a boundary defining function of $\pa_jS_{i}$ in $S_i$.  Of course, the projection $\pr_2:S_i\times [0,1)\to [0,1)$ is a boundary defining function for $S_i\times\{0\}$ in $S_i\times [0,1)$.  Taking these as boundary defining functions, we get a corresponding $\QFB$ tangent bundle, in fact a $\QAC$ tangent bundle that we will denote by ${}^{\QAC}T(S_i\times [0,1))$.  We will denote its restriction to $S_i$ by 
\begin{equation}
  {}^{\phi}NS_i:= {}^{\QAC}T(S_i\times [0,1))|_{S_i\times \{0\}}.
\label{nb.1}\end{equation}
\begin{lemma}
There is a natural isomorphism
\begin{equation}
  {}^{\phi}NH_i\cong \phi_i^*({}^{\phi}NS_i).
\label{vf.2}\end{equation}
\label{nb.2}\end{lemma}
\begin{proof}
Consider the function $v_i= x_i \prod_{H_j>H_i}x_j$.  Then in terms of the local coordinates \eqref{coor.1}, the isomorphism 
$$
       \phi_i^*({}^{\phi}NS_i)\to {}^{\phi}NH_i
$$
is given by multiplication by $\frac{v_i}{\pr_2\circ \phi_i}$ where $\pr_2: S_i\times [0,1)\to [0,1)$ is the boundary defining function of $S_i\times \{0\}$ in $S_i\times [0,1)$.   
\end{proof}

Now, by a {\bf QFB manifold}, we mean a manifold $(M,\phi)$ with fibered corners, together with a choice of compatible boundary defining functions, which is equipped with a non-degenerate, symmetric inner product $g$ on $\QFBT M$, hereafter called a {\bf QFB metric}. 
Such a metric determines a complete Riemannian metric on the interior of $M$ with respect
to the isomorphism of $\QFBT M^\circ$ and $TM^\circ$, but is necessarily singular at the boundary.  In terms of the local basis of sections \eqref{ds.4b}, an example of $\QFB$ metric is given by
\begin{equation}
   \sum_{i=1}^k \frac{dv_i^2}{v_i^4}+ \sum_{i=1}^k\sum_{j=1}^{\ell_i} \frac{(dy^j_i)^2}{v_i^2} + \sum_{i=1}^q dz_i^2.
\label{ds.4c}\end{equation}
As a special case, a {\bf QAC manifold} is a QFB manifold for which $\phi_i = \Id : H_i \to S_i = H_i$ is a bundle with trivial fiber for each maximal boundary hypersurface $H_i$ (with respect to the partial order coming from the fibered corners structure), and we use the adornment QAC instead of QFB in this situation.
When we choose a complete list $H_1,\ldots, H_{\ell}$ of boundary hypersurfaces, we will usually assume that for such a manifold, the maximal hypersurfaces (with respect to the partial order induced by the iterated fibration structure) are enumerated last, so for instance $H_1,\ldots, H_k$ will denote the non-maximal boundary hypesurfaces while the remaining $\ell-k$ boundary hypersurfaces $H_{k+1}, \ldots, H_\ell$ will be all maximal.  As explained in \cite[\S~2.3.5]{DM2018} and \cite[Example~1.22]{CDR}, a simple example of a QAC metric is given by the Cartesian product of two asymptotically conical metrics.  

In order to give a better geometric intuition of what $\QFB$ metrics are,  we can give an example in a tubular neighborhood as in \eqref{pt.1}.  However, instead of the decomposition in terms of the level sets of $x_i$, it is better to pick a total boundary defining function $\prod_j x_j$ compatible with the Lie algebra of $\QFB$ vector fields  and decompose in terms of its level sets.  To emphasize that we could take different total boundary defining functions compatible with the Lie algebra of vector fields for different choices of boundary hypersurfaces in \eqref{pt.1}, we will denote this total boundary defining function by $u_i$.  With this understood, a decomposition in terms of the level sets of such a total boundary defining function is given by
\begin{equation}
  \kridx{\cU_i}{U}{modified normal neighborhood} := \{ (p,c)\in H_i\times [0,\epsilon)\; | \; \prod_{j\ne i} x_j(p)> \frac{c}{\epsilon} \} \subset H_i\times [0,\epsilon)
\label{pt.2}\end{equation} 
with natural diffeomorphism 
\begin{equation}
    \begin{array}{lccc}\psi_i: &(H_i\setminus \pa H_i)\times [0,\epsilon) & \to & \cU_i \\
           & (p,t) & \mapsto & (p, t \prod_{j\ne i} x_j(p)).
           \end{array}
\label{pt.3}\end{equation}
On $\cU_i$ seen as a subset of $H_i\times [0,\epsilon)$, denote by $\widetilde{\pr}_1$ and $\widetilde{\pr}_2$ the restriction of the projections $H_i\times [0,\epsilon)\to H_i$ and 
$H_i\times [0,\epsilon)\to [0,\epsilon)$.  Since $\widetilde{\pr}_i$ is defined in terms of the level sets of $u_i$, notice that it does not correspond to $\pr_i$ under the diffeomorphism \eqref{pt.3}.  A choice of connection for the fiber bundle $\phi_i: H_i\to S_i$ then induces a natural map of vector bundles
\begin{equation}
   \left(\widetilde{\pr}_1^* {}^{\phi}T^*(H_i/S_i)\right)|_{\cU_i}\hookrightarrow {}^{\QFB}T^*M|_{\cU_i}.
\label{pt.4}\end{equation}
In fact, on $\overset{\circ}{H}_i\times \{0\}\subset \cU_i$, notice that this map does not depend on the choice of connection and corresponds to the dual of the map \eqref{vf.1b}.  Using the inclusion \eqref{pt.4}, an example of $\QFB$ metric in $\cU_i$ is then given by
\begin{equation}
  \frac{du_i^2}{u_i^4}+ \frac{\widetilde{\pr_1}^*\phi_i^*g_{S_i}}{u_i^2}+ \widetilde{\pr}_1^*\kappa,
\label{pt.5}\end{equation}       
where $g_{S_i}$ is a wedge metric on $S_i\setminus \pa S_i$ in the sense of \cite{ALMP2012, AG} (see also Definition~\ref{smb.18} below) and $\kappa \in \CI(H_i; S^2({}^{\phi}T^*(H_i/S_i)))$ induces a $\QFB$ metric on each fiber of $\phi_i: H_i\to S_i$.  A $\QFB$ metric of the form \eqref{pt.5} in $\cU_i$ is said to be of \textbf{product-type near $H_i$}.  In terms of the coordinates \eqref{coor.1}, we can take $u_i=v_1$ and use the coordinates \eqref{coor.1} with the coordinate $x_i$ replaced with $u_i=v_1$.  Under this change of coordinates, the vector fields \eqref{ds.4} become
\begin{equation}
\begin{gathered}
v_1\left( x_1\frac{\pa}{\pa x_1}+ v_1\frac{\pa}{\pa v_1}\right), v_1\frac{\pa}{\pa y_1}, v_2\left( x_2\frac{\pa}{\pa x_2}-x_1\frac{\pa}{\pa x_1}\right) v_2\frac{\pa}{\pa y_2},\ldots, v_{i-1}\left( x_{i-1}\frac{\pa}{\pa x_{i-1}}- x_{i-2}\frac{\pa}{\pa x_{i-2}}\right), v_{i-2}\frac{\pa}{\pa y_{i-2}}, \\
-v_{i-1}\frac{\pa}{\pa x_{i-1}}, v_i\frac{\pa}{\pa y_i}, v_{i+1}x_{i+1}\frac{\pa}{\pa x_{i+1}}, v_{i+1}\frac{\pa}{\pa y_{i+1}}, \\
 v_{i+2}\left( x_{i+2}\frac{\pa}{\pa x_{i+2}}-x_{i+1}\frac{\pa}{\pa x_{i+1}}\right), v_{i+2}\frac{\pa}{\pa y_{i+2}}, \ldots,  
v_k\left( x_k\frac{\pa}{\pa x_k}- x_{k-1}\frac{\pa}{\pa x_{k-1}} \right), v_k\frac{\pa}{\pa y_k}, \frac{\pa}{\pa z}.
\end{gathered}
\label{pt.5b}\end{equation}  
Making a change of basis with respect to $\CI(M)$, these vector fields are equivalent to
\begin{equation}
\begin{gathered}
v_1^2\frac{\pa}{\pa v_1}, v_1\frac{\pa}{\pa x_1}, v_1\frac{\pa}{\pa y_1}, v_2\frac{\pa}{\pa x_2}, v_2\frac{\pa}{\pa y_2}, \ldots, v_{i-1}\frac{\pa}{\pa x_{i-1}}, v_{i-1}\frac{\pa}{\pa y_{i-1}}, v_i\frac{\pa}{\pa y_i}, v_{i+1}x_{i+1}\frac{\pa}{\pa x_{i+1}}, v_{i+1}\frac{\pa}{\pa y_{i+1}}, \\
v_{i+2}\left(x_{i+2}\frac{\pa}{\pa x_{i+2}}-x_{i+1}\frac{\pa}{\pa x_{i+1}}  \right), v_{i+2}\frac{\pa}{\pa y_{i+2}}, \ldots, v_k\left(x_k\frac{\pa}{\pa x_k}- x_{k-1}\frac{\pa}{\pa x_{k-1}}\right), v_k\frac{\pa}{\pa y_k}, \frac{\pa}{\pa z}.
\end{gathered}
\label{pt.5c}\end{equation}
In particular, on $H_i$, the horizontal vector fields are edge vector fields, while the vertical vector fields are $\QFB$ vector fields, in agreement with the model metric \eqref{pt.5}.

\begin{definition}
An \textbf{exact} $\QFB$ metric is a $\QFB$ metric which is product-type near $H_i$ up to a term in $x_i\CI(M, S^2(\QFBT M))$ for each boundary hypersurface $H_i$.  In other words, a $\QFB$ metric $g_{\QFB}$ is exact if, for each boundary hypersurface $H_i$, there is a tubular neighborhood as in \eqref{pt.1}, a choice of total boundary defining function $u_i$ compatible with the Lie algebra of $\QFB$ vector fields and a choice of connection on $\phi_i: H_i\to S_i$ inducing an inclusion \eqref{pt.4}, such that on $\cU_i$ as given by \eqref{pt.2}, we have that 
$$
g_{\QFB}- (c_i)_*\psi_i^*\left( \frac{du_i^2}{u_i^4}+ \frac{\widetilde{\pr_1}^*\phi_i^*g_{S_i}}{u_i^2}+ \widetilde{\pr}_1^*\kappa \right)  \in x_i\CI(c_i(H_i\times[0,\epsilon)), S^2(\QFBT M)).
$$
\label{pt.6}\end{definition}

A $\QFB$ metric naturally defines a volume density on $M^\circ$.  More generally, we define the line bundle of QFB densities  by
\[
	(\kridx{\OmegaQFB}{OQFB}{QFB density bundle})_p = \big\{\abs{\rho} :  \rho \in \textstyle\bigwedge^n \QFBT^\ast_p M\big\}, \quad p \in M,
\]
where $n = \dim(M)$. 
From the vanishing factors on the generating sections of $\QFBT M$, it follows that QFB densities are related to ordinary densities (and $b$ densities) by
\begin{equation}
	\OmegaQFB(M) = \prod_i x_i^{-\dim(S_i) - 2}\; \Omega(M) = \prod_i x_i^{-\dim(S_i)-1}\; \Omegab(M).
\label{E:QFB_density_relation}
\end{equation}

On a QAC manifold, on can also consider the conformally related Lie algebra 
of vector fields
\[
	\kridx{\cV_\Qb}{VQb}{Qb vector fields}(M) = v_{\max}^\inv \cV_\QAC(M), \quad v_{\max} = v_{k+1} = \prod_{i = k+1}^\ell x_i,
\]
that is, the Lie algebra of $\Qb$ vector fields introduced in \cite{CDR}.
The condition that the maximal boundary fibrations $\phi_i$ for $i > k$ have no nontrivial fibers is essential in order that $\cV_\Qb(M)$ forms a Lie algebra; otherwise it would contain elements which are singular at certain boundary faces when considered as ordinary vector fields, and would fail to be closed under Lie bracket.

The algebra of  differential Qb operators, $\Diff_\Qb^\ast(M),$ is defined as above from composing Qb vector fields as operators on $C^\infty(M)$, and the associated vector bundles
$\kridx{\QbT }{TQb}{Qb tangent bundle}M$ and $\QbT^\ast M$ are likewise defined in a similar manner.
Finally, we record for later use the fact that Qb densities, defined in a similar way to QFB densities, are given by
\begin{equation}
	\kridx{\OmegaQb}{OQb}{Qb density bundle}(M) = \prod_{i=1}^k x_i^{-\dim(S_i) - 1}\; \Omegab(M)
	\label{E:Qb_density_relation}
\end{equation}
where the product is taken only over indices for non-maximal boundary hypersurfaces.

\section{The $\QFB$ double space}\label{ds.0}
To define a natural class of pseudodifferential operators associated to $\QFB$ metrics, we will follow the microlocal approach of Melrose \cite{MelroseAPS} and define them as conormal distributions on an adapted double space which will be the subject of this section.  Compared to the fibered boundary case \cite{Mazzeo-MelrosePhi}, there are however important new features to construct such a double space, the main one being that we need to rely on the ordered product and the reverse ordered product of \cite{KR3}.  For the convenience of the reader, along the way, we will go over the definition of the ordered product in details and rederive some of its key features by methods slightly different from those of \cite{KR3}.  

Before that though, let us first recall from \cite{hmm} the following general lemmas about blow-ups and $b$-fibrations, since those will be used repeatedly in what follows. 
By definition, an interior $b$-map $f : X \to Y$ satisfies
\[
	f^\ast \rho_G = a_G \prod_{H \in \cM_1(X)} \rho_H^{e(H,G)}\quad \text{for each $G \in \cM_1(Y)$,}
\]
where $a_G \in \CI(X)$ is strictly positive and the exponents $e(H,G)$ are integers; here $\cM_1(X)$ denotes the set of boundary hypersurfaces of a manifold with corners $X$. A $b$-map is said to be {\bf simple} provided each $e(H,G) \in \set{0,1}$.

\begin{lemma}[{\cite[Lemma~2.5]{hmm}}] Let $f:X\to Y$ be a surjective $b$-submersion between compact manifolds with corners.  Let $T\subset Y$ be a closed $p$-submanifold.  Let $\rho_T$ be the product of the boundary defining functions  of the boundary hypersurfaces of $Y$ containing $T$ and suppose that 
$$
\displaystyle f^*\rho_T= a\prod_{i\in \cI} \rho_i
$$ 
for some $\cI\subset M_1(X)$ and a positive smooth function $a\in \CI(X)$.  Then there is a minimal collection $\cS$ of $p$-submanifolds of $X$ into which the lift of $T$ under $f$ decomposes such that $f$ lifts to a surjective $b$-submersion
\begin{equation}
      f_T: [X;\cS]\to [Y;T]
\label{bf.4}\end{equation}
for any order of blow-up of the elements of $\cS$.  Moreover, if $f$ is a $b$-fibration, then so is $f_T$.  
\label{bf.3}\end{lemma}
\begin{proof}
If $f$ is a $b$-fibration, then this is exactly the statement of \cite[Lemma~2.5]{hmm}.  If instead $f$ is not a $b$-fibration, then the assumptions on $T$ ensure that we can still choose coordinates as in \cite[(2.17)]{hmm}, so that the proof of \cite[Lemma~2.5]{hmm} still applies to show at least that $f$ lifts to a surjective $b$-submersion $f_T: [X;\cS]\to [Y;T]$.  
\end{proof}

\begin{lemma}[{\cite[Lemma~2.7]{hmm}}]  Let $f: X\to Y$ be a $b$-fibration of compact manifolds with corners and suppose that $S\subset X$ is a closed $p$-submanifold which is $b$-transversal to $f$ in the sense that 
\begin{equation}
          \ker({}^b\! f_* | {}^{b}T_pM)+ {}^{b}T_p(S,X)= {}^{b}T_pX \quad \forall \; p\in S,
\label{bf.2}\end{equation}
where $^{b}T_p(S,X)$ is the subspace of $^{b}T_pX$ consisting of vectors tangent to $S$.  Suppose also that $f(S)$ is not contained in any boundary face of $Y$ of codimension $2$.  Then the composition of the blow-down map $[X;S]\to X$ with $f$ induces a $b$-fibration
$$
           f': [X;S]\to Y.  
$$
\label{bf.1}\end{lemma}

To construct the $\QFB$ double space associated to $(M,\phi)$ and $v$, we start with the Cartesian product $M\times M$.
As a first step, instead of considering the $b$-double space $M^2_b$ as in \cite{DLR}, we will use the ordered product of \cite{KR3}.  Precisely, recalled from \cite{KR3} that a \textbf{manifold with ordered corners $X$} is a manifold with corners together with a partial order on its set $\cM_1(X)$ of boundary hypersurfaces such that with respect to this partial order, two boundary hypersurfaces are comparable whenever they intersect.  For instance, a manifold with fibered corners is automatically a manifold with ordered corners.  Given two manifolds with ordered corners $X$ and $Y$, there is an induced partial order on $\cM_1(X)\times \cM_1(Y)$ by requiring that $(H,G)\le (H',G')$ in $\cM_1(X)\times \cM_1(Y)$ if and only if $H\le H'$ and $G\le G'$.  The ordered product 
$$
   X\kridx{\ttimes}{\times}{ordered product} Y= [X\times Y; \cM_1(X)\times \cM_1(Y)]
$$
of $X$ and $Y$ introduced in \cite{KR3} is then obtained from $X\times Y$ by blowing up codimension $2$ corners of the form $H\times G$ for $H\in \cM_1(X)$ and $G\in\cM_1(Y)$ in an order compatible with the partial order on $\cM_1(X)\times \cM_1(Y)$, that is, if $(H,G)<(H',G')$, then $H\times G$ must be blown up before $H'\times G'$. 

\begin{lemma}[\cite{KR3}]
The definition of the ordered product $X\ttimes Y$  is independent of the choice of blow-up order compatible with the partial order on $\cM_1(X)\times \cM_1(Y)$.  
\label{op.2a}\end{lemma}
\begin{proof}
If $H\times G$ and $H'\times G'$ are incomparable, then either $H$ is incomparable to $H'$ or $G$ is incomparable to $G'$, in which case $H\times G$ and $H'\times G'$ are disjoint and their blow-ups commute, or else $H<H'$ and $G>G'$ (or vice-versa), in which case $H\times G'$ has to be blown up first.  Since the lift of $H\times G$ and $H'\times G'$ are disjoint after the blow-up of $H\times G'$, this shows again that their blow-ups commute.   
\end{proof}

When $X$ and $Y$ are manifold with fibered corners, it is also important to consider the ordered product obtained by reversing the partial order on $\cM_1(X)$ and $\cM_1(Y)$, thus reversing the partial order for the structure of manifolds with ordered corners of $X$ and $Y$.  We refer to the ordered product with respect to this reverse partial order as the reverse ordered product and denote it $X\rttimes Y$.  When $X$ and $Y$ are manifolds with fibered corners, both $X\ttimes Y$ and $X\ttimes_r Y$ are again naturally manifolds with fibered corners.  In fact, as shown in \cite{KR3}, if $g_X$ and $g_Y$ are $\QFB$ metrics on $X$ and $Y$, then their Cartesian product is naturally a QFB metric on $X\kridx{\rttimes}{\times_r}{reverse ordered product} Y$ \cite[Theorem~6.6]{KR3}, while if $g_1$ and $g_2$ are wedge metrics on $X$ and $Y$ in the sense of Definition~\ref{smb.18} below, then their Cartesian product is naturally a wedge metric on $X\ttimes Y$ \cite[Theorem~6.8]{KR3}.  Heuristically, this is the main reason why the reverse ordered product should be the starting point of the construction of the QFB double space.  As we will see however, the ordered product itself will also play a non-trivial role in this construction, the intuitive reason being that, within the construction of the $\QFB$ double space, Cartesian products of wedge metrics naturally occur.

 Thus, coming back to our manifold with fibered corners $M$, we define the reverse ordered product double space of $M$ to be  the reverse ordered product of $M$ with itself, 
\begin{equation}
   \kridx{M^2_{\rp}}{M2rp}{reverse ordered product double space}:= M\rttimes M \quad \mbox{with blow-down map} \quad \beta_{\rp}: M^2_{\rp}\to M^2.
\label{op.1}\end{equation}
Similarly, we define the ordered product double space by $\kridx{M^2_{\op}}{M2op}{ordered product double space}=M\ttimes M$.    

\begin{remark}
If $H_1,\ldots, H_{\ell}$ is an exhaustive list of the boundary hypersurfaces of $M$ such that 
$$
      H_i<H_j \; \Longrightarrow \; i<j,
$$
then notice that one way to define $M^2_{\rp}$ is to blow up the codimension $2$ corners $H_{i}\times H_j$ in reverse lexicographic order with respect to $(i,j)$, that is, we first blow up $H_\ell\times H_\ell,\ldots, H_\ell\times H_{1}$, then $H_{\ell-1}\times H_\ell,\ldots, H_{\ell-1}\times H_{1}$, and so on until we blow up $H_{1}\times H_\ell, \ldots, H_{1}\times H_{1}$.  Alternatively, we could use the reverse lexicographic order from right to left instead of from left to right, namely we can instead first blow up $H_\ell\times H_\ell,\ldots, H_{1}\times H_{\ell}$, then $H_\ell\times H_{\ell-1},\ldots, H_{1}\times H_{\ell-1}$, and so on until we blow up $H_\ell\times H_{1}$, \ldots, $H_{1}\times H_{1}$.  These two equivalent ways of defining $M^2_{\rp}$ indicate that the involution $\cI: M\times M\to M\times M$ given by $\cI(q,q')= (q',q)$ naturally lifts to an involution on $M^2_{\rp}$.  In other words, in $M^2_{\rp}$, there is still a symmetry with respect to the left and right factors.  Clearly, a similar discussion holds for $M^2_{\op}$.     
\label{ps.0a}\end{remark}

Now, on $M\times M$, let us denote by $v$ and $v'$ the lifts from the left and from the right of the total boundary defining function $v$.  On $M^2_{\rp}$, consider then the function $s:= \frac{v}{v'}$.  On the lift $H_{ii}^{\rp}$ of $H_i\times H_i$ to $M^2_{\rp}$, consider the submanifold
$$
      \kridx{\diag_i}{Di}{$s = 1$ submanifold of $H^{\rp}_{ii}$}= \{ p\in H_{ii}^{\rp} \; | \; s(p)=1\}.
$$
\begin{lemma}[\cite{KR3}]
The submanifold $\diag_i$ is a $p$-submanifold of $M^2_{\rp}$.  
\label{ps.0}\end{lemma}
\begin{proof}
The space $\diag_i$ is given by the equation $s=1$ in $H_{ii}^{\rp}$.  Notice first that near $\pa(M^2)$ the equation $s=1$ does not define a $p$-submanifold.    Indeed, consider a corner of $M^2$, say where  $H_{i_1}\times M,\ldots, H_{i_{\sigma}}\times M, M\times H_{j_1},\ldots, M\times H_{j_{\tau}}$ intersect, where indices are chosen such that $1\le i_1<\cdots<i_{\sigma}\le \ell$ and $1\le j_1<\cdots<j_{\tau}\le \ell$. 
Near such a corner, we can choose boundary defining functions $x_{i_r}$ of $H_{i_r}\times M$ and $x_{j_r}'$ of $M\times H_{j_r}$ in such a way that 
\begin{equation}
       s= \frac{\prod_{r=1}^{\sigma} x_{i_r}}{\prod_{r=1}^{\tau}x_{j_r}'}.
\label{ps.1}\end{equation}
Clearly then, the equation $s=1$ does not yield a $p$-submanifold near that corner, in particular near for instance $x_{i_1}=x_{j_1}'=0$.  

However, near that corner, the blow-ups performed to obtain $M^2_{\rp}$ out of $M^2$ correspond to blowing up $x_{i_r}= x_{j_t}'=0$ for $r\in\{1,\ldots,\sigma\}$ and  $t\in\{1,\ldots,\tau\}$.  Such a blow-up corresponds to introducing the coordinate $\frac{x_{i_r}}{x_{j_t}'}$ (or $\frac{x_{j_t}'}{x_{i_r}}$), so after such a blow-up is performed, the effect is to allow us to choose boundary defining functions on the new blown up space in such a way that $s$ is locally of the form \eqref{ps.1}, but with $\sigma$ or $\tau$ decreased by $1$ depending on whether we are near $\frac{x_{i_r}}{x_{j_t}'}=0$ or $\frac{x_{j_t}'}{x_{i_r}}=0$ on the new blown-up space.  

Thus, after all the blow-ups are performed, we are left with a situation where for suitable choices of boundary defining functions, near the problematic regions where $\frac{1}{s}$ or $s$ are not bounded, $s$ is locally of the form \eqref{ps.1} with either $\sigma=1$ and  $\tau=0$ or $\sigma=0$ and $\tau=1$, a form in which the equation $s=1$ clearly defines a $p$-submanifold.  Hence, we see that the equation $s=1$ does indeed define a $p$-submanifold near $\pa (M^2_{\rp})$, and hence a $p$-submanifold on the boundary hypersurface $H^{\rp}_{ii}$.      
       
\end{proof}

To describe the next submanifolds that need to be blown up, we need to discuss certain geometric structures on $\diag_i$ and $H_{ii}^{\rp}$.  First, to describe the geometry of $H_{ii}^{\rp}$, it is convenient to use a description of $M^2_{\rp}$ proceeding as follows:
\begin{itemize}
\item[Step 1:]  Blow up first the $p$-submanifolds $H_j\times H_k$ for $j>i$ and $k>i$; 
\item[Step 2:]Then blow up the lifts of $H_j\times H_i$ and $H_i\times H_j$ for $j>i$;
\item[Step 3:] Blow up the lift of $H_i\times H_i$; 
\item[Step 4:] Blow up the lifts of $H_j\times H_k$ and $H_k\times H_j$ for $j<i<k$;  
\item[Step 5:]  Finally,  blow up the lifts of $H_k\times H_i$ and $H_i\times H_k$ for $k<i$, and then blow up the lifts of $H_j\times H_k$ for $j<i$ and $k<i$.   
\end{itemize}
Let $H_{ii}^{(k)}$ denotes the lift of $H_i\times H_i$ after Step $k$.   Then notice first that $H_{ii}^{(1)}$ is still a codimension $2$ corner.  Moreover, it is  naturally the total space of a fiber bundle 
\begin{equation}
\xymatrix{
   Z_i\rttimes Z_i \ar[r] & H^{(1)}_{ii} \ar[d]  \\ & S_i\times S_i
}
\label{op.2}\end{equation}   
with vertical arrow the natural lift of $\phi_i\times \phi_i: H_i\times H_i\to S_i\times S_i$ to $H^{(1)}_{ii}$.  This is because on $H_i\times H_i$, all the $p$-submanifolds that are blown-up in Step 1 are of the form $H\times G$ for $H$ and $G$ boundary hypersurfaces such that $\phi_i(H)=\phi_i(G)=S_i$.  

Now, on $H^{(1)}_{ii}$, the submanifolds blown up in Step 2 corresponds to boundary hypersurfaces, hence do not alter the intrinsic geometry of $H^{(1)}_{ii}$, that is, there is a natural diffeomorphism  $H^{(2)}_{ii}\cong H^{(1)}_{ii}$.  However, those blow-ups have important consequences on the extrinsic geometry of $H^{(1)}_{ii}$ in that $H^{(1)}_{ii}$ becomes disjoint from the lift of $H_j\times M$ and $M\times H_j$ after the lifts of $H_j\times H_i$ and $H_i\times H_j$ are blown-up.  

Since Step 3 is the blow up of $H^{(2)}_{ii}$, this means that there is a natural diffeomorphism  
$$H^{(3)}_{ii}\cong H^{(2)}_{ii}\times [0,\frac{\pi}2]_{\theta_i},
$$ 
where the angular variable can be taken to be 
$$
  \theta_i= \arctan \left( \prod_{j\ge i} \frac{x_j}{x_j'} \right)=\arctan\left( \frac{v_i}{v_i'} \right).
$$
Under this identification, the fiber bundle \eqref{op.2} lifts to a fiber bundle 
\begin{equation}
        \xymatrix{
   Z_i\rttimes Z_i \ar[r] & H^{(3)}_{ii} \ar[d]  \\ & S^2_i\times [0,\frac{\pi}2]_{\theta_i}.
}\label{op.3}\end{equation}
  
  Moving to Step 4, notice that by Step 2, it corresponds to the blow-up of $p$-submanifolds disjoint from $H^{(3)}_{ii}$, so that $H^{(4)}_{ii}=H^{(3)}_{ii}$.  
  
  If $H_i$ is a minimal boundary hypersurface, Step 5 is vacuous so that $H_{ii}^{\rp}=H^{(3)}_{ii}=H^{(2)}_{ii}\times [0,\frac{\pi}2]_{\theta_i}$.  This is not true however if $H_i$ is not a minimal boundary hypersurface.  In these cases, the blow-ups of Step~5 correspond to blow-ups of the base $S_i^2\times [0,\frac{\pi}2]_{\theta_i}$ of \eqref{op.3}, so that on $H^{\rp}_{ii}=H^{(5)}_{ii}$  the fiber bundle \eqref{op.3} lifts to a fiber bundle 
  \begin{equation}
        \xymatrix{
   Z_i\rttimes Z_i \ar[r] & H^{\rp}_{ii} \ar[d]  \\ & S_i\rjtimes S_i,
}
\label{ps.5}\end{equation}  
where $S_i\kridx{\rjtimes}{\star}{reverse join product} S_i$ is the reverse join product of \cite{KR3} given by 
\begin{equation}
S_i\rjtimes S_i = [ S_i\times S_i\times [0,\frac{\pi}2]_{\theta_i}; S_i\times \cM^r_1(S_i)\times \{0\}, \cM^r_1(S_i)\times S_i\times \{\frac{\pi}2\}, \cM^r_1(S_i)\times \cM^r_1(S_i)\times [0,\frac{\pi}2]  ],
\label{ps.7}\end{equation} 
where $\cM^r_1(S_i)$ denotes $\cM_1(S_i)$ with the reverse partial order and indicates that the blow-ups are performed in an order compatible with the reverse partial order on $\cM(S_i)$ and $\cM(S_i)\times \cM(S_i)$.    

In terms of this description, the equation $s=1$ becomes 
\begin{equation}
  \left(\frac{v_i}{v_i'}\right)\left( \frac{\rho_i}{\rho_i'}\right)=1,
\label{ps.7b}\end{equation}  
where $\rho_i=\prod_{H_j<H_i}x_j$ is a total boundary defining function for $S_i$.  In particular, the equation behaves well with respect to the fiber bundle \eqref{ps.5} in that it descends to an equation on the base $S_i\rjtimes S_i$.  By \cite[Corollary~4.15]{KR3}, the equation \eqref{ps.7b} on $S_i\rjtimes S_i$ is naturally identified with $(S_i)^2_{\op}=S_i\ttimes S_i$, so that on the $p$-submanifold $\diag_i\subset H^{\rp}_{ii}$, the fiber bundle \eqref{ps.5} induces a fiber bundle 
\begin{equation}
        \xymatrix{
   Z_i\rttimes Z_i \ar[r] & \diag_i \ar[d]^{\phi_i\rttimes \phi_i}  \\ & (S_i)^2_{\op}.
}
\label{ps.6}\end{equation}  
Referring to \cite{KR3} for details, let us briefly explain geometrically why the base of \eqref{ps.6} is given by $(S_i)^2_{\op}$ instead of $(S_i)^2_{\rp}$.  Without loss of generality, we can clearly suppose that $H_i$ is maximal with $H_i=S_i$ and  that $\phi_i: H_i\to S_i$ is the identity map, so that a $\QFB$ metric is just a $\QAC$ metric.  In this case, a model of $\QAC$ metric near $H_i$ is given by a cone metric 
$$
             dr^2+ r^2g_w  \quad \mbox{on} \quad (0,\infty)\times S_i,  \quad r=\frac1{v}.
$$
Then, on $M^2_{\rp}$, a model metric near the maximal boundary hypersurface $H^{\rp}_{ii}$ is given by the Cartesian product of the cone metric with itself.  This is again a cone metric 
$$
       dr^2+ r^2g_{ii}, \quad r= \sqrt{\frac1{v^2}+\frac1{(v')^2}},
$$
this times with $g_{ii}$ the wedge metric on $H^{\rp}_{ii}=S_i\rjtimes S_i$ corresponding to the wedge metric of the geometric suspension of $g_w$ with itself,
$$
   g_{ii}= d\theta^2+ \sin^2\theta g_w+ \cos^2\theta g_w \quad \mbox{on} \quad  (S_i\setminus\pa S_i)\times (S_i\setminus\pa S_i)\times (0,\frac{\pi}2)_{\theta}.
$$ 
In terms of this link, the equation $s=1$ corresponds to $\theta=\frac{\pi}{4}$, yielding a cone metric with link the Cartesian product of $g_w$ with itself, that is, the wedge metric $g_w\times g_w$ on $S_i\ttimes S_i$.  This geometric way of interpreting \eqref{ps.6} can be made rigorous using the explicit isomorphism  of \cite[(7.2)]{KR3}.

Having clarified this point about the base of the fiber bundle \eqref{ps.6}, let $\diag^2_{S_i,\op}$ denote the lift of the diagonal in $S_i\times S_i$ to $(S_i)^2_{\op}$ and consider the $p$-submanifold of $\diag_i$ given by 
\begin{equation}
   \kridx{\Phi_i}{Phi}{fiber diagonal in $H_{ii}^{\rp}$}:= (\phi_i\rttimes \phi_i)^{-1} (\diag^2_{S_i,\op}).
\label{ps.8}\end{equation}
When $M$ is a manifold with fibered boundary, notice that $\Phi_1$ is just the fiber diagonal $\Phi$ of \cite{Mazzeo-MelrosePhi}.  More generally, by the discussion above, notice that $\Phi_i$ comes naturally with a fiber bundle structure $\Phi_i\to \diag^2_{S_i,\op}$ with typical fiber $(Z_i)^2_{\rp}$.  
\begin{remark}
As the fiber diagonal $\Phi$ in \cite{Mazzeo-MelrosePhi}, the definition of the $p$-submanifold $\Phi_i$ depends on the choice of a total boundary defining function $v$.  However, by Lemma~\ref{qfbs.1}, \eqref{ps.7b} and \eqref{ps.8}, two  $\QFB$ equivalent total boundary defining functions induce the same $p$-submanifolds $\Phi_i$, so the dependence is only through the choice of a Lie algebra of $\QFB$ vector fields.
\label{ps.8new}\end{remark}

The $p$-submanifolds $\Phi_i$ in $M^2_{\rp}$ are precisely those that need to be blown-up to construct the $\QFB$ double space, namely
\begin{equation}
 \kridx{M^2_{\QFB}}{M2QFB}{QFB double space}:= [M^2_{\rp}; \Phi_1,\ldots, \Phi_{\ell}] 
\label{ds.1}\end{equation} 
with corresponding blow-down maps
\begin{equation}
 \beta_{\QFB-\rp}: M^2_{\QFB}\to M^2_{\rp}, \quad  \beta_{\QFB}: M^2_{\QFB}\to M^2.
 \label{ds.1b}\end{equation}
In this definition, the order in which we blow up is important, since changing the order typically gives, up to diffeomorphism, a different manifold with corners.  
We denote the `new' boundary hypersurface coming from the blow-up of $\Phi_i$ by 
$$
      \kridx{\ff_{i}}{ffi}{front face of $\Phi_i$ in $M^2_{\QFB}$}\subset M^2_{\QFB}.
$$
We say that $\ff_{i}$ is the \textbf{front face} associated to the boundary hypersurface $H_i$.   Denote by $H_{i0}$, $H_{0i}$ and $H_{ij}$ the lifts
of $H_i\times M$, $M\times H_i$ and $H_i\times H_j$.   Then $\ff_i$ for $i\in \{1,\ldots,\ell\}$ and $H_{ij}$ for $i,j\in\{0,1,\ldots,\ell\}$ with $(i,j)\ne (0,0)$ give a complete list of the boundary hypersurfaces of $M^2_{\QFB}$.  Though we will not need it until \S~\ref{kqfb.0}, it is a good moment to state and prove the following lemma.
\begin{lemma}
For $i,j\in\{1,\ldots,\ell\}$, $\ff_i\cap H_{j0}=\emptyset$ and $\ff_i\cap H_{0j}=\emptyset$.  
\label{new.2}\end{lemma}
\begin{proof}
Since $\Phi_i\subset \diag_i$, it suffices to show that $\diag_i\cap H_{j0}^{\rp}=\emptyset$ and $\diag_i\cap H_{0j}^{\rp}=\emptyset$, which follows from the fact that $s=1$ on $\diag_i$, $s=0$ on $H_{j0}^{\rp}$ and $s^{-1}=0$ on $H_{0j}^{\rp}$.
\end{proof}

Let 
$$
     \kridx{\diag_{\QFB}}{DQFB}{diagonal in $M^2_{\QFB}$}:= \overline{\beta_{\QFB}^{-1}(\overset{\circ}{\diag}_{M})}
$$ 
denote the lift of the diagonal $\diag_M\subset M^2$ to $M^2_{\QFB}$, where $\overset{\circ}{\diag}_{M}$ is the interior of the diagonal.  Clearly, the lifted diagonal $\diag_{\QFB}$ is a $p$-submanifold of $M^2_{\QFB}$.
The next lemma will describe how $\diag_{\QFB}$ behaves in terms of the lifts of the Lie algebra $\cV_{\QFB}(M)$ with respect to the natural maps
$$
     \pi_L= \pr_L\circ\beta_{\QFB}: M^2_{\QFB}\to M, \quad  \pi_R= \pr_R\circ\beta_{\QFB}: M^2_{\QFB}\to M,
$$
where $\pr_L: M\times M\to M$ is the projection onto the left factor and $\pr_R$ is the projection onto the right factor.
\begin{lemma}
The lifts to $M^2_{\QFB}$ of the Lie algebra $\cV_{\QFB}(M)$ via the maps $\pi_L$ and $\pi_R$ are transversal to the lifted diagonal $\diag_{\QFB}$.
\label{ds.2}\end{lemma}
\begin{proof}
The result is trivial away from the boundary, so let $p\in \diag_{\QFB}\cap \pa M^2_{\QFB}$ be given.  We need to show that the lemma holds in a neighborhood of $p$ in $M^2_{\QFB}$.   By the symmetry of Remark~\ref{ps.0a}, we only need to establish the result for the lift of $\cV_{\QFB}(M)$ with respect to the map $\pi_L$.  After relabeling the boundary hypersurfaces of $M$ if necessary, we can assume that $H_1,\ldots, H_k$ are the boundary hypersurfaces containing $\pi_L(p)$ and that 
\begin{equation}
      H_1<\cdots < H_k.
\label{ds.2b}\end{equation}
Near $\pi_L(p)\in \pa M$, let $(x_i,y_i, z)$ be coordinates as in \eqref{coor.1}, so that in particular $x_i$ is a boundary defining function for $H_i$ and the total boundary defining function is given locally by $v=\prod_{i=1}^k x_i$.  On $M^2$, we can then consider the coordinates 
\begin{equation}
      x_i,y_i, z, x_i',y_i',z'
\label{ds.2c}\end{equation} 
where $(x_i,y_i,z)$ denotes the pull-back of the coordinates $(x_i,y_i,z)$ from the left factor and $(x_i',y_i',z')$ the pull-back from the right factor.  In principle, we need to lift these coordinates through all the blow-ups involved in the construction of $M^2_{\rp}$ near $\beta_{\QFB-\rp}(p)$.  Using Lemma~\ref{op.2a}, we can blow up in the following order,
\begin{multline}
H_k\times H_k, H_k\times H_{k-1}, H_{k-1}\times H_k, \ldots, H_k\times H_1, H_1\times H_k, \\
H_{k-1}\times H_{k-1}, H_{k-1}\times H_{k-2}, H_{k-2}\times H_{k-1},\ldots, H_{k-1}\times H_1, H_{1}\times H_{k-2}, \ldots, H_1\times H_1.
\end{multline}
Notice then that for $i<j$ the lifts of $H_i\times H_j$ and $H_j\times H_i$ are disjoint from the lifted diagonal after the blow-up of the lift of $H_j\times H_j$.  Hence, effectively, near the lifted diagonal $\diag_{M,\rp}$ in $M^2_{\rp}$, we only need to consider the blow-up of $H_i\times H_i$ for $i\in\{1,\ldots k\}$, so that  
\begin{equation}
  s_i:=\frac{x_i}{x_i'}, x_i', y_i,y_i', z,z'
\label{ds.3}\end{equation}
are coordinates near $\beta_{\QFB-\rp}(p)$.
In terms of these coordinates, we see that
$$
     x_i= x_i's_i, \quad v_i= v_i'\sigma_i, \quad\mbox{where}\quad v_i:=\prod_{j=i}^k x_j,\quad v_i':=\prod_{j=i}^k x_j'\quad \mbox{and}  \quad \sigma_i:= \prod_{j=i}^k s_j,
$$
hence, under the map $\pr_L\circ \beta_{\rp}$, the basis of $\QFB$ vector fields \eqref{ds.4}
lifts to
\begin{equation}
v_1'\sigma_1s_1\frac{\pa}{\pa s_1}, v_1' \sigma_1\frac{\pa}{\pa y_1^{n_{1}}}, v_2'\sigma_2\left( s_2\frac{\pa}{\pa s_2}- s_1\frac{\pa}{\pa s_1} \right), v_2'\sigma_2\frac{\pa}{\pa y_2^{n_2}}, \ldots, v_k'\sigma_k\left( s_{k}\frac{\pa}{\pa s_{k}}- s_{k-1}\frac{\pa}{\pa s_{k-1}} \right), v_k'\sigma_k\frac{\pa}{\pa y_k^{n_k}}, \frac{\pa}{\pa z_q}.\label{ds.5}\end{equation}
In fact, instead of \eqref{ds.3}, we can use the alternative system of coordinates
\begin{equation}
  \sigma_i= \prod_{j=i}^k s_j, \; x_i', y_i, y_i', z,z',
\label{ds.6}\end{equation}  
which gives a simpler description of the lift, namely
\begin{equation}
v_i'\sigma_i^2\frac{\pa}{\pa \sigma_i}, v_i' \sigma_i\frac{\pa}{\pa y_i^{n_{i}}}, \frac{\pa}{\pa z_q}.
\label{ds.7}\end{equation}
Finally, to lift this basis to $M^2_{\QFB}$, one can use the system of coordinates 
\begin{equation}
  S_i:= \frac{\sigma_i-1}{v_i'}, x_i', Y_i:= \frac{y_i-y_i'}{v_i'}, y_i', z,z'
\label{ds.8}\end{equation}
near $p$, which gives
\begin{equation}
(1+v_i'S_i)^2\frac{\pa}{\pa S_i}, (1+v_i'S_i)\frac{\pa}{\pa Y_i^{n_{i}}}, \frac{\pa}{\pa z_q}.
\label{ds.9}\end{equation}
In the coordinates \eqref{ds.8}, the lifted diagonal $\diag_{\QFB}$ is given by 
\begin{equation}
    \{ (S_i, x_i',Y_i,y_i',z,z') \; | \; z=z', S_i=0, Y_i=0 \; \forall \; i\},
\label{ds.9c}\end{equation}
so that the basis of vector fields in \eqref{ds.9} is clearly transversal to $\diag_{\QFB}$.  

\end{proof}

\begin{corollary}
The natural diffeomorphism $\diag_{\QFB}\cong M$ induced by the map $\pi_L$ (or $\pi_R$) also yields a natural identification
$$
      N\diag_{\QFB}\cong ^{\QFB}\!TM, \quad N^*\diag_{\QFB}\cong ^{\QFB}\!T^*M,
$$
where $N\diag_{\QFB}$ is the normal bundle of $\diag_{\QFB}$ in $M^2_{\QFB}$.
\label{ds.9b}\end{corollary}


For a QAC manifold with the Lie algebra $\cV_\Qb(M)$ of Qb vector fields, there is of course a corresponding $\Qb$ double space.  If we assume as before that our list $H_1,\ldots, H_{\ell}$ of boundary hypersurfaces is chosen such that $H_{1},\ldots, H_{\ell'}$ correspond to the non-maximal boundary hypersurfaces with respect to the partial order, so that for instance $x_{\max}= \prod_{i=\ell'+1}^{\ell}x_i$, then the $\Qb$ double space is given by
\begin{equation}
    \kridx{M^2_{\Qb}}{M2QB}{Qb double space}:= [M^2_{\rp}; \Phi_1,\ldots, \Phi_{\ell'}]  
\label{ds.10}\end{equation}
with corresponding blow-down maps
$$
\beta_{\Qb-\rp}: M^2_{\Qb}\to M^2_{\rp}, \quad \beta_{\Qb}=\beta_{\rp}\circ\beta_{\Qb-\rp}: M^2_{\Qb}\to M^2.$$
Let $H^{\Qb}_{i0}$, $H^{\Qb}_{0i}$ and $H^{\Qb}_{ij}$ be the lifts of $H_{i}\times M$, $M\times H_i$ and $H_i\times H_j$ to $M^2_{\Qb}$ for $i,j\in\{1,\ldots,\ell\}$.   For $H_i$ not maximal, denote also by $\ff^{\Qb}_i$ the front face created by the blow-up of $\Phi_i$.  Then, these boundary hypersurfaces  give a complete list of the boundary hypersurfaces of $M^2_{\Qb}$.  Let us denote by $\kridx{\diag_{\Qb}}{DQb}{diagonal in $M^2_{\Qb}$}$ the lift of the diagonal $\diag_M$ to $M^2_{\Qb}$.  It is again a $p$-submanifold and we have the following analog of Lemma~\ref{ds.2}.
\begin{lemma}
The lifts to $M^2_{\Qb}$ of the Lie algebra $\cV_{\Qb}(M)$ via the maps $\pr_L\circ\beta_{\Qb}$ and $\pr_R\circ\beta_{\Qb}$ are transversal to the lifted diagonal $\diag_{\Qb}$.
\label{ds.11}\end{lemma}
\begin{proof}
As in the proof of Lemma~\ref{ds.2}, by symmetry, it suffices to prove the result for the map $\pr_L\circ\beta_{\Qb}$ near a point $p\in \diag_{\Qb}\cap \pa M^2_{\Qb}$.  As in \eqref{ds.2b}, let $H_1,\ldots, H_k$ be the boundary hypersurfaces containing $\pr_L\circ \beta_{\Qb}(p)$.  Clearly, if $H_k$ is not maximal with respect to the partial order of the structure of manifold with fibered corners,  we can proceed exactly as in the proof of Lemma~\ref{ds.2}.  If instead $H_k$ is maximal, then we can still use the coordinates \eqref{ds.6} near $\beta_{\Qb-\rp}(p)$, but with the precision that there is no $(z,z')$ coordinates since the fibers of $\phi_k:H_k\to H_k$ are points.  In terms of these coordinates, the lift of the basis of $\Qb$ vector fields  
\begin{equation}
\frac{v_1x_1}{x_k}\frac{\pa}{\pa x_1}, \frac{v_1}{x_k}\frac{\pa}{\pa y_1^{n_{1}}}, \frac{v_2}{x_k}\left( x_1\frac{\pa}{\pa x_1}- x_2\frac{\pa}{\pa x_2} \right), \frac{v_2}{x_k}\frac{\pa}{\pa y_2^{n_2}}, \ldots, \frac{v_k}{x_k}\left( x_{k-1}\frac{\pa}{\pa x_{k-1}}- x_{k}\frac{\pa}{\pa x_{k}} \right), \frac{v_k}{x_k}\frac{\pa}{\pa y_k^{n_k}}
\label{ds.12}\end{equation}
via the map $\pr_L\circ \beta_{\rp}$ is given by 
\begin{equation}
\frac{v_i'\sigma_i^2}{x_k'\sigma_k}\frac{\pa}{\pa \sigma_i}, \frac{v_i' \sigma_i}{x_k'\sigma_k}\frac{\pa}{\pa y_i^{n_{i}}}.
\label{ds.13}\end{equation}
To lift to $M^2_{\Qb}$, we use, instead of \eqref{ds.8}, the system of coordinates 
\begin{equation}
  S_i:= \frac{\sigma_i-1}{w_i'}, x_i', Y_i:= \frac{y_i-y_i'}{w_i'}, y_i' \quad \mbox{for}\; i<k \quad \mbox{together with} \quad \sigma_k, x_k', y_k, y_k', \quad \mbox{where} \quad w_i':=\frac{v_i'}{x_k'},\label{ds.14}\end{equation}
near $p\in M^2_{\Qb}$.  In terms of these coordinates, the basis of vector fields \eqref{ds.13} lifts to
\begin{equation}
\frac{(1+w_i'S_i)^2}{\sigma_k}\frac{\pa}{\pa S_i}, \frac{(1+w_i'S_i)}{\sigma_k}\frac{\pa}{\pa Y_i^{n_{i}}}\quad \mbox{for}\; i<k \quad \mbox{and} \quad \sigma_k\frac{\pa}{\pa \sigma_k}, \frac{\pa}{\pa y_k}.
\label{ds.15}\end{equation}
Since the lifted diagonal $\diag_{\Qb}$ is locally given by the equations
$$
S_i=0, Y_i=0 \quad \mbox{for} \; i<k \quad \mbox{and} \quad  \sigma_k=1, y_k=y_k'
$$
in the coordinates \eqref{ds.14}, we clearly see that the basis of vector fields \eqref{ds.15} is transversal to $\diag_{\Qb}$.
\end{proof}

\begin{corollary}
The natural diffeomorphism $\diag_{\Qb}\cong M$ induced by $\pr_L\circ\beta_{\Qb}$ (or $\pr_R\circ\beta_{\Qb}$) yields a natural identification
$$
     N\diag_{\Qb}\cong {}^{\Qb}TM, \quad N^*\diag_{\Qb}\cong {}^{\Qb}T^*M,
$$
where $N\diag_{\Qb}$ is the normal bundle of $\diag_{\Qb}$ in $M^2_{\Qb}$.
\label{ds.15b}\end{corollary}

\section{$\QFB$ pseudodifferential operators} \label{ts.0}

Thanks to the double space of the previous section, we can now introduce the notion of pseudodifferential $\QFB$ operators as conormal distributions on $M^2_{\QFB}$.  Compared to \cite{Mazzeo-MelrosePhi}, one important new features is that we will allow these conormal distributions to have a weaker asymptotic behavior on the boundary hypersurfaces of the double space $M^2_{\QFB}$, namely they will be $\QFB$ conormal in the sense of Definitions~\ref{con.1} and \ref{ext.8} below instead of $b$ conormal.  This will give us flexibility needed to construct parametrices later on.      

However, before introducing pseudodifferential $\QFB$ operators, let us first explain how differential $\QFB$ operators admit a simple description in terms of conormal distributions on $M^2_{\QFB}$.  The simplest one is the identity operator.  In the coordinates \eqref{ds.2c} near $\diag_M\cap (\pa M\times \pa M)$, its Schwartz kernel takes the form
$$
   \kappa_{\Id}= \left( \prod_{i=1}^k \delta(x_i'-x_i) \delta(y_i'-y_i)dx_i'dy_i'\right) \delta(z'-z)dz'
$$
with $\delta(\zeta)$ denoting the Dirac $\delta$ distribution supported at $\zeta=0$.
In terms of the coordinates \eqref{ds.3}, this becomes
\begin{equation}
\begin{aligned}
  \kappa_{\Id}&= \left(  \prod_{i=1}^k {x'_i}^{-1}\delta(1-s_i)\delta(y_i'-y_i)dx_i'dy_i'\right) \delta(z'-z)dz' \\
                      &= \left(  \prod_{i=1}^k \delta(1-s_i)\delta(y_i'-y_i)\frac{dx_i'}{x_i'}dy_i'\right) \delta(z'-z)dz'. 
\end{aligned}  
\label{pdo.1}\end{equation}
Thus, in terms of the coordinates \eqref{ds.6} and with the convention that $\sigma_{k+1}:=1$, this gives
\begin{equation}
\begin{aligned}
  \kappa_{\Id}&= \left(  \prod_{i=1}^k \sigma_{i+1}\delta(\sigma_{i+1}-\sigma_i)\delta(y_i'-y_i)\frac{dx_i'}{x_i'}dy_i'\right) \delta(z'-z)dz' \\
                      &= \left(  \prod_{i=1}^k \delta(1-\sigma_i)\delta(y_i'-y_i)\frac{dx_i'}{x_i'}dy_i'\right) \delta(z'-z)dz',
\end{aligned}  
\label{pdo.2}\end{equation}
so that finally, in terms of the coordinates \eqref{ds.8}, we have that
\begin{equation}
\begin{aligned}
  \kappa_{\Id}&= \left(  \prod_{i=1}^k \delta(-S_i)\delta(-Y_i)\frac{dx_i'}{x_i'(v_i')^{1+k_i}}dy_i'\right) \delta(z'-z)dz' \\
                      &= \left(  \prod_{i=1}^k \delta(-S_i)\delta(-Y_i)\right) \delta(z'-z) \pi_R^*(\nu_{\QFB}),
\end{aligned}  
\label{pdo.3}\end{equation}
where $k_i=\dim S_i-\dim S_{i-1}-1$ if $i>1$, $k_1=\dim S_1$ otherwise and $\nu_{\QFB}\in \CI(M;^{\QFB}\Omega)$ is a non-vanishing $\QFB$ density on $M$, the last equality following from \eqref{E:QFB_density_relation}. In other words, 
$$
    \kappa_{\Id}\in \cD^0(\diag_{\QFB})\cdot \pi^*_R(\nu_{\QFB})
$$
where $\cD^0(\diag_{\QFB})$ is the space of Dirac delta distributions supported (and varying smoothly) on the $p$-submanifold $\diag_{\QFB}\subset M^2_{\QFB}$.  More generally, the Schwartz kernel of an operator $P\in \Diff^k_{\QFB}(M)$ is 
$$
   \kappa_P= \pi^*_LP \cdot \kappa_{\Id}\in \cD^k(\diag_{\QFB})\cdot \pi^*_R(\nu_{\QFB})
$$ 
where $\cD^k(\diag_{\QFB})$ is the space of smooth Dirac delta distributions of order at most $k$ supported (and varying smoothly) on $\diag_{\QFB}$, namely
$$
     \cD^k(\diag_{\QFB}):= \Diff^k(M^2_{\QFB})\cdot \cD^0(\diag_{\QFB}).
$$
Since $\pi_L^*(\cV_{\QFB}(M))$ is transversal to $\diag_{\QFB}$ by Lemma~\ref{ds.2}, we see that the space of Schwartz kernels corresponding to operators in 
$\Diff^k_{\QFB}(M)$ is precisely $\cD^k(\diag_{\QFB})\cdot \pi^*_R(\nu_{\QFB})$.  If $E$ and $F$ are smooth complex vector bundles on $M$, then the same discussion shows that the space of Schwartz kernels corresponding to the operators in $\Diff^k_{\QFB}(M;E,F)$ is precisely 
$$
       \cD^k(\diag_{\QFB})\cdot \CI(M^2_{\QFB}; \beta_{\QFB}^*(\Hom(E,F)\otimes \pr_R^*{}^{\QFB}\Omega))
$$
where $\Hom(E,F):=\pr^*_{L}F\otimes \pr_R^*(E^*)$ on $M^2$.  This preliminary discussion suggests to define the {\bf small calculus} of $\QFB$ pseudodifferential operators as the union over $m\in \bbR$ of the spaces
\begin{equation}
\kridx{\Psi^m_{\QFB}}{PsiQFB}{QFB pseudodifferential operators (small calculus)}(M;E,F):=\{\kappa\in I^m(M^2_{\QFB}; \diag_{\QFB};\beta_{\QFB}^*(\Hom(E,F)\otimes \pr_R^*{}^{\QFB}\Omega)) \; | \; \kappa\equiv 0 \; \mbox{at} \; \pa M^2_{\QFB}\setminus \ff_{\QFB} \},
\label{pdo.4}\end{equation}
where $\ff_{\QFB} = \cup_i \ff_i$ is the union of the front faces of $M^2_{\QFB}$ and 
$$I^m(M^2_{\QFB};\diag_{\QFB};\beta_{\QFB}^*(\Hom(E,F)\otimes \pr_R^*{}^{\QFB}\Omega))
$$ 
is the space of conormal distributions of order $m$ at $\diag_{\QFB}$.

To define $\Qb$ pseudodifferential operators, we can proceed in a similar way.  One can first check that in the coordinates \eqref{ds.14},
\begin{equation}
\begin{aligned}
\kappa_{\Id}&= \left( \prod_{i=1}^{k-1} \delta(-S_i)\delta(-Y_i) \frac{dx_i'dy_i'}{x_i'(w_i')^{1+k_i}} \right)\left( \delta(1-\sigma_k)\delta(y_k'-y_k)\frac{dx_k'}{x_k'}dy_k' \right) \\
                    &=\left( \prod_{i=1}^{k-1}\delta(-S_i)\delta(-Y_i) \right)\left(  \delta(1-\sigma_k)\delta(y_k'-y_k) \right)  \beta_{\Qb}^*\pr_R^*\nu_{\Qb}
\end{aligned}
\label{pdo.5}\end{equation}
for some non-vanishing $\Qb$ density $\nu_{\Qb}$.  Hence, the space of Schwartz kernels associated $\Qb$ differential operators of order $k$ is just
$$
       \cD^k(\diag_{\Qb})\cdot \CI(M^2_{\Qb}; \beta_{\Qb}^*\left( \Hom(E,F)\otimes \pr_R^* {}^{\Qb}\Omega \right)),
$$
and the {\bf small calculus} of $\Qb$ pseudodifferential operators is the union over $m\in \bbR$ of the spaces
\begin{equation}
\kridx{\Psi^m_{\Qb}}{PsiQB}{Qb pseudodifferential operators (small calculus)}(M;E,F):= \{\kappa \in I^m(M^2_{\Qb};\diag_{\Qb}; \beta_{\Qb}^*(\Hom(E,F)\otimes \pr_R^*{}^{\Qb}\Omega)) \; | \; \kappa\equiv 0 \; \mbox{at} \; \pa M^2_{\Qb}\setminus \ff_{\Qb}\},
\label{pdo.6}\end{equation} 
where $\ff_{\Qb}$ is the union of the boundary hypersurfaces of $M^2_{\Qb}$ intersecting the lifted diagonal $\diag_{\Qb}$.

To construct good parametrices, one needs often to enlarge slightly the space of pseudodifferential operators to allow non-trivial behavior at other faces.  To this end, let us first recall some notation from \cite{Melrose1992}.  Thus, let $X$ be a manifold with corners with boundary hypersurfaces $B_1,\ldots, B_{\ell}$ and corresponding boundary defining functions $\rho_i$.  For $\mathfrak{s}\in \bbR^{\ell}$ a multiweight, we denote the conormal functions with multiweight $\mathfrak{s}$ by 
\begin{equation}
   \kridx{\sA^{\mathfrak{s}}_-}{As-}{conormal functions with multiweight $\mathfrak{s}$}(X)= \bigcap_{\mathfrak{t}<\mathfrak{s}} \rho^{\mathfrak{t}} H^{\infty}_b(X)
\label{pdo.7}\end{equation}
where $\rho^{\mathfrak{t}}= \rho_1^{\mathfrak{t}_1}\cdots \rho_{\ell}^{\mathfrak{t}_{\ell}}$, $\mathfrak{t} < \mathfrak{s}$ means that $t_i < s_i$ for every $i$, and 
$$
     H^{\infty}_b(X)= \bigcap_k H^k_b(X)
$$
with 
$$
   \kridx{H^k_b}{Hkb}{$b$ Sobolev space}(X)= \{u\in L^2_b(X)\; | \; Pu\in L^2_b(X)\; \forall P\in \Diff^k_b(X)\},
$$
the $b$ Sobolev space of order $k$ associated to the space $L^2_b(X)$ of square integrable functions with respect to a $b$ density.  Given an index family $\cE$ for the boundary hypersurfaces of $X$, we denote by $\sA^{\cE}_{\phg}(X)$ the space of smooth functions on the interior of $X$ admitting at each boundary hypersurface a polyhomogeneous expansion for the corresponding index set of the index family $\cE$.  More generally, for $\mathfrak{s}>\mathfrak{s}'$, following  \cite[(22)]{Melrose1992}, 
we can consider the space 
\begin{equation}
          \kridx{\sB^{\cE/\mathfrak{s}}_{\phg}\sA^{\mathfrak{s}'}_-}{BEsA}{partially polyhomogeneous conormal functions}(X)
\label{pdo.8}\end{equation}
of functions  in  $\sA^{\mathfrak{s}'}_-(X)$ 
admitting a partial expansion at each boundary hypersurface $H$ with exponents in $\cE(H)$ and remainder term in $\sA^{\mathfrak{s}_H}_-(X)$ with
$$
     \mathfrak{s}_H(H')= \left\{ \begin{array}{ll} \mathfrak{s}(H), & H=H' \\
              \mathfrak{s}'(H'), & \mbox{otherwise.}   \end{array}  \right.
$$
In particular, 
$$
\sB^{\emptyset/\mathfrak{s}}_{\phg}\sA^{\mathfrak{s}'}_-(X)=\sA^{\mathfrak{s}}_-(X).
$$
When $\mathfrak{s}' = 0$ and $\mathfrak{s} > 0$, we use the simpler notation
\[
	\kridx{\sA_{\phg}^{\cE/\mathfrak{s}}}{AEs}{partially polyhomogeneous conormal functions}(X) := \sB_{\phg}^{\cE/\mathfrak{s}} \sA_-^0(X).
\]

In particular, $\sA^{\cE/\infty}_{\phg}(X)=\sA^{\cE}_{\phg}(X)$ and $\sA^{\emptyset/\mathfrak{s}}_{\phg}(X) = \sA^{\mathfrak{s}}_-(X)$. Since these spaces are $\CI$-modules, for $E\to X$ a smooth vector bundle, we can also consider the space of partially polyhomogeneous sections of $E$
$$
	\sA^{\cE/\mathfrak{s}}_{\phg}(X; E) := 
	\sA^{\cE/\mathfrak{s}}_{\phg}(X) \otimes_{\CI(X)}\CI(X;E).
$$

Coming back to $\QFB$ operators, given an index family $\cE$ for $M^2_{\QFB}$ and a multiweight $\mathfrak{s}$, we can consider the spaces
of pseudodifferential operators
\begin{equation}
\begin{gathered}
  \Psi^{-\infty,\cE/\mathfrak{s}}_{\QFB}(M;E,F):=\sA^{\cE/\mathfrak{s}}_{\phg}(M^2_{\QFB};\beta^*_{\QFB}(\Hom(E,F)\otimes\pr_R^*\Omega_{\QFB})), \\
  \kridx{\Psi^{m,\cE/\mathfrak{s}}_{\QFB}}{PsiQFBEs}{QFB pseudodifferential operators (large calculus)}(M;E,F):=  \Psi^{m}_{\QFB}(M;E,F)+  \Psi^{-\infty,\cE/\mathfrak{s}}_{\QFB}(M;E,F), \quad \mbox{for}\; m\in \bbR.  \end{gathered}
\label{pdo.9}\end{equation}
  This will be particularly interesting when $\cE$ is a \textbf{$\QFB$ nonnegative index family}, which is an index family $\cE$ such that $\inf \Re\cE(\ff_i)\ge 0$ and $\inf\Re\cE(H_{ij})>h_j$ for each $i$ and $j$ in $\{0,1,\ldots,\ell\}$ with the convention that $H_0=M$, $h_0=0$ and $h_j=\dim S_j+1$ (c.f.~\eqref{E:QFB_density_relation} and notice that the asymmetry between for instance $\cE(H_{i0})$ and $\cE(H_{0i})$ is due to the fact that we are using right densities).  An index family $\cE$ is said to be \textbf{$\QFB$ positive} if it is $\QFB$ nonnegative and $\inf\Re \cE(\ff_i)>0$ for all $i$.  Similarly, a multiweight $\mathfrak{q}$ is said to be \textbf{$\QFB$ positive} (respectively $\QFB$ \textbf{nonnegative}) if $\mathfrak{q}(\ff_i)> 0$ (respectively $\mathfrak{q}(\ff_i)\ge 0$) and $\mathfrak{q}(H_{ij})>h_j$ for each $i$ and $j$ in $\{0,1,\ldots,\ell\}$ such that $H_i\cap H_j\ne \emptyset$.
In particular, we will consider the spaces   
\begin{equation}
\Psi^{-\infty,\mathfrak{r}}_{\QFB}(M;E,F):= \Psi^{-\infty,\emptyset/\mathfrak{r}}_{\QFB}(M;E,F)
\label{co.11b}\end{equation} 
for $\mathfrak{r}$ a $\QFB$ positive multiweight.

Similarly, if $\cE$ is an index family on $M^2_{\Qb}$ and a multiweight $\mathfrak{s}$, there are corresponding spaces of $\Qb$ pseudodifferential operators
\begin{equation}
\begin{gathered}
  \Psi^{-\infty,\cE/\mathfrak{s}}_{\Qb}(M;E,F):=\sA^{\cE/\mathfrak{s}}_{\phg}(M^2_{\Qb};\diag_{\Qb};\beta^*_{\Qb}(\Hom(E,F)\otimes\pr_R^*\Omega_{\Qb})), \\ 
  \kridx{\Psi^{m,\cE/\mathfrak{s}}_{\Qb}}{PsiQbEs}{Qb pseudodifferential operators (large calculus)}(M;E,F):= \Psi^{m}_{\Qb}(M;E,F)+ \Psi^{-\infty,\cE/\mathfrak{s}}_{\Qb}(M;E,F).
      \end{gathered}
\label{pdo.10}\end{equation} 

We will be interested the case where $\cE$ is a \textbf{$\Qb$ nonnegative index family}, that is, an index family $\cE$ such that $\inf \Re\cE(H^{\Qb}_{ii})\ge 0$ for $H_i$ maximal, $\Re\cE(\ff_i^{\Qb})\ge 0$ for $H_i$ non-maximal and otherwise such that 
$$
  \Re\cE(H^{\Qb}_{ij})> \widetilde{h}_j \quad \mbox{for}  \;  i,j\in \{0,\ldots\ell\},
$$
 where 
\begin{equation}
\widetilde{h}_j= \left\{  \begin{array}{ll} 0, &   j>0, H_j \; \mbox{is maximal}, \\
   h_j, & \mbox{ otherwise}.
   \end{array}  \right.
\label{max.1}\end{equation}
We will also say that a multiweight $\mathfrak{q}$ is \textbf{$\Qb$ positive} (respectively $\Qb$ \textbf{nonnegative}) if  $\mathfrak{q}(H^{\Qb}_{ii})> 0$ (respectively $\mathfrak{q}(H^{\Qb}_{ii})\ge 0$) for $H_i$ maximal, $\mathfrak{q}(\ff_i^{\Qb})> 0$ (respectively $\mathfrak{q}(\ff_i^{\Qb})\ge 0$) for $H_i$ non-maximal and otherwise such that 
$$
  \mathfrak{q}(H^{\Qb}_{ij})> \widetilde{h}_j \quad \mbox{for}  \;  i,j\in \{0,\ldots\ell\}.
$$
Special cases are obtained by considering the space
\begin{equation}
\Psi^{-\infty,\mathfrak{r}}_{\Qb}(M;E,F):= \Psi^{-\infty,\emptyset/\mathfrak{r}}_{\Qb}(M;E,F),  
\label{co.15}\end{equation} 
where $\mathfrak{r}$ is a $\Qb$ positive multiweight.

To describe how $\QFB$ pseudodifferential operators act, it is convenient to first check that the projections $\pi_L$ and $\pi_{R}$ from $M^2_{\QFB}$ to $M$ are $b$-fibrations.

\begin{lemma}
The projections $\pi_L$ and $\pi_R$ from $M^2_{\QFB}$ to $M$ are $b$-fibrations.  Similarly, the projections $\pr_L\circ \beta_{\Qb}$ and $\pr_R\circ \beta_{\Qb}$ are $b$-fibrations from $M^2_{\Qb}$ to $M$.
\label{ao.1}\end{lemma}
\begin{proof}
Starting with the $b$-fibrations $\pr_L: M^2\to M$ and $\pr_R: M^2\to M$, one can first apply iteratively Lemma~\ref{bf.1} for each blow-up leading to $M^2_{\rp}$  to show that $\pr_L\circ \beta_{\rp}$ and $\pr_R\circ \beta_{\rp}$ are $b$-fibrations.  Indeed, for each of these blow-ups, $^{b}T_p(S,X)= {}^{b}T_pX$, so condition \eqref{bf.2} is automatically satisfied, while $f(S)$ always corresponds to a boundary hypersurface of $Y$.  From there, the result follows by iteratively applying Lemma~\ref{bf.1} to the blow-ups of $\Phi_1,\ldots,\Phi_{\ell}$.  In this case, each $\Phi_i$ is sent inside the boundary hypersurface $H_i\subset M$ and no other.  On the other hand, we no longer have that $^{b}T(S,X)=^{b}TX$, but it is easy to see from \eqref{ps.8} that \eqref{bf.2} still holds.  
\end{proof}

Given $f\in \cA^{\cF/\mathfrak{t}}_{\phg}(M;E)$, we would like to define the action of $P\in\Psi^{m,\cE/\mathfrak{s}}_{\QFB}(M;E,F)$ by
\begin{equation}
 Pf:= (\pi_L)_*(P\cdot \pi_R^*f))= \left[(\pi_L)_*(\cdot P\cdot \pi_R^*f\otimes\pi_L^*\varpi )\right]\otimes\varpi^{-1}
\label{pdo.11}\end{equation}
for $\varpi$ any non-vanishing $b$ density on $M$.   To describe in greater details what we obtain, recall from \eqref{E:QFB_density_relation} that if $\nu_b$ is a non-vanishing $b$ density, then
\begin{equation}
  \kridx{\nu_{\QFB}}{nuQFB}{non-vanishing QFB density}=\left(\prod_{i=1}^{\ell} x_i^{1+\dim S_i}\right)^{-1}\nu_b
\label{pdo.12}\end{equation} 
is a non-vanishing $\QFB$ density.    Let $\rho_{\ff_i}$ and $\rho_{ij}$ be boundary defining functions for $\ff_i$ and $H_{ij}$ respectively.  To describe the lift from the right of a $\QFB$ density on the double space, we will need the following lemma due to Melrose, the proof of which is a simple local
coordinate computation.
\begin{lemma}
Let $Y$ be a $p$-submanifold of a manifold with corners $X$.  Let $w$ the codimension of $Y$ within the smallest boundary face of $X$ containing $Y$.  Let $\beta$ be the blow-down map from $[X;Y]$ to $X$.  If  $\rho_Y\in\CI([X;Y])$ is a boundary defining function for the new boundary hypersurface created by the blow-up of $Y$, then
$$
  \beta^* {}^{b}\Omega(X)= (\rho_Y^w){}^{b}\Omega([X;Y]).
$$
\label{pdo.12a}\end{lemma}  
Using this lemma, one computes that
\begin{equation}
    \beta_{\QFB}^{*}{}^{b}\Omega(M^2)= \left(\prod_i \rho_{\ff_i}^{1+\dim S_i} \right){}^{b}\Omega(M^2_{\QFB}).
\label{pdo.12b}\end{equation}
Since 
\begin{equation}
\pr_L^*{}^{b}\Omega(M)\otimes \pr_R^*{}^{b}\Omega(M)= {}^{b}\Omega(M^2),
\label{pdo.12c}\end{equation} 
we deduce from \eqref{pdo.12} and \eqref{pdo.12b} that 
\begin{equation}
\pi_L^*\nu_b \cdot\pi_R^*\nu_{\QFB}\in \frac{{}^{b}\Omega(M^2_{\QFB})}{\rho^{\mathfrak{r}}}
\label{pdo.13}\end{equation}
with $\kridx{\mathfrak{r}}{r}{}$ the multiweight such that 
\begin{equation}
	\rho^{\mathfrak{r}} = \prod_{i=1}^\ell \prod_{j=0}^\ell \rho_{ji}^{1 + \dim S_i}
\label{pdo.14}\end{equation}
with $\rho_{ji}$ denoting a boundary defining function for $H_{ji}$.
To formulate how $\QFB$ operators act on sections of a bundle, the following notation will be useful.
\begin{definition}
Given weights $\mathfrak{s},\mathfrak{t}$ with associated index sets $\cE$ and $\cF$ respectively, we denote by $\mathfrak{s}\kridx{\dot{+}}{+}{multiweight sum with associated index sets}\mathfrak{t}$ the weight given by
$$
      \min\{\mathfrak{s}+\mathfrak{t}, \mathfrak{s}+\mathfrak{f}, \mathfrak{e}+\mathfrak{t}\}
$$
with associated index set $\cE+\cF$, where 
$$
       \mathfrak{e}:= \min\{ \Re a\; | (a,k)\in \cE\} \quad \mbox{and} \quad \mathfrak{f}:= \min\{ \Re a\; | (a,k)\in \cF\}.
$$
The operator $\dot{+}$ is clearly commutative.  It is also associative.  That is, if  $\mathfrak{r}$ is another weight with associated index set $\cD$, then
\begin{equation}
       (\mathfrak{r}\dot{+}\mathfrak{s})\dot{+}\mathfrak{t}= \mathfrak{r}\dot{+}(\mathfrak{s}\dot{+}\mathfrak{t})
\label{co.8c}\end{equation}
with associated index set $\cD+\cE+\cF$.  We can thus unambiguously use the notation $\mathfrak{r}\dot{+}\mathfrak{s}\dot{+}\mathfrak{t}$ to denote the weight \eqref{co.8c} with associated index set $\cD+\cE+\cF$.  

\label{co.8b}\end{definition}

Using the pushforward theorem of \cite[\S8]{Melrose1992} and the notation therein  yields the following.

\begin{theorem}
If $P\in \Psi^{m,\cE/\mathfrak{s}}_{\QFB}(M;E,F)$ and $f\in \sA^{\cF/\mathfrak{t}}_{\phg}(M;E)$ are such that
$$
       \min\{ \mathfrak{s}(H_{0i}),\min\Re(\cE(H_{0i}^{\QFB}))\}+\min\{\mathfrak{t}(H_{i}),\min\Re(\cF(H_{i}))\}>1+\dim S_i, \quad \forall \ i,
$$
then $Pf:=(\pi_L)_*(P\cdot \pi_R^*f)$ is well-defined and such that 
$$
         Pf\in \sA^{\cK/\mathfrak{k}}_{\phg}(M;F)
$$
with $\cK= (\pi_L)_{\#}(\cE-\mathfrak{r}+ \pi_R^{\#}\cF)$ and $\mathfrak{k}= (\pi_L)_{\#}(\mathfrak{s}-\mathfrak{r}\dot{+}\pi_R^{\#}\mathfrak{t})$, that is,
$$
\cK(H_i)= \cE(H^{\QFB}_{i0}) \overline{\cup} ((\cE)(\ff_i)+\cF(H_i)) \overline{\cup}\left(\underset{H_{j}\in\cM_1(M)}{\overline{\cup}} ((\cE-\mathfrak{r})(H^{\QFB}_{ij})+\cF(H_j))\right)
$$
and 
$$
\mathfrak{k}(H_i)= \min\{ \mathfrak{s}(H_{i0}^{\QFB}), (\mathfrak{s}-\mathfrak{r})(\ff_i)\dot{+} \mathfrak{t}(H_i),(\mathfrak{s}-\mathfrak{r})(H_{ij}^{\QFB}) \dot{+}\mathfrak{t}(H_j) \; \mbox{for} \;
   H_j\in \cM_1(M)\},
$$
where $\mathfrak{r}$ is the multiweight defined in \eqref{pdo.14}.
\label{pdo.15}\end{theorem}
\begin{proof}
Since $\pi_L$ is transverse to the lifted diagonal, notice that the singularities of the conormal distribution $P\cdot \pi_R^*f$ are integrated out.  Hence, we can apply the pushforward theorem as if it were a conormal function.    
\end{proof}

\begin{corollary}
An operator $P\in \Psi^m_{\QFB}(M;E,F)$ induces a continuous linear map
$$
             P: \CI(M;E)\to \CI(M;F)
$$
which restricts to a continuous linear map
$$
      P: \dot{\cC}^{\infty}(M;E)\to \dot{\cC}^{\infty}(M;E).
$$
\label{pdo.16}\end{corollary}
\begin{proof}
It suffices to apply Theorem~\ref{pdo.15}  with $\mathfrak{s}=\infty$, $\cE$ the index family with all index sets $\bbN_0$ at each front face $\ff_i$ and $\emptyset$ elsewhere,  and $\cF$ the index family with all index sets given by $\bbN_0$ for the first assertion and $\emptyset$ for the second assertion.
\end{proof}

There is a similar result for $\Qb$ pseudodifferential operators.   Let $\rho_{\ff^{\Qb}_i}$, $\rho_{i0}$, $\rho_{0i}$ and $\rho_{ij}$ be boundary defining functions for $\ff^{\Qb}_i, H^{\Qb}_{i0}, H^{\Qb}_{0i}$ and $H^{\Qb}_{ij}$.  
Now,  let $\cI$ be the set associated to non-maximal boundary hypersurfaces,
i.e., $H_i$ is maximal if and only if $i\notin \cI$.   Given a non-vanishing $b$ density $\nu_b$ on $M$, we recall from \eqref{E:Qb_density_relation} that
\begin{equation}
  \kridx{\nu_{\Qb}}{nuQb}{non-vanishing Qb density}:= \left(  \prod_{i\in\cI}x_i^{1+\dim S_i}\right)^{-1}\nu_b
\label{pdo.17}\end{equation}
is a non-vanishing $\Qb$ density.  Using Lemma~\ref{pdo.12a}, we compute that 
\begin{equation}
  \beta^*_{\Qb}  {}^{b}\Omega(M^2)= \left( \prod_{i\in \cI} \rho^{1+\dim S_i}_{\ff^{\Qb}_i} \right) {}^{b}\Omega(M^2_{\Qb}).
\label{pdo.18}\end{equation}
Hence, we deduce from \eqref{pdo.12c}, \eqref{pdo.17} and \eqref{pdo.18} that 
$$
    \beta_{\Qb}^*(\pr_L^*\nu_b\cdot \pr_R^*\nu_{\Qb})\in \frac{{}^{b}\Omega(M^2_{\Qb})}{\rho^{\mathfrak{r}'}}
$$ 
with $\kridx{\mathfrak{r}'}{rprime}{}$ is the multiweight such that 
\begin{equation}
    \rho^{\mathfrak{r}'}= \prod_{i\in \cI} \prod_{j=0}^{\ell} \rho_{ji}^{1+\dim S_i}.
\label{mwp.1}\end{equation}
Hence, the pushforward theorem of \cite[\S~8]{Melrose1992} yields the following.

\begin{theorem}
If $P\in \Psi^{m,\cE/\mathfrak{s}}_{\Qb}(M;E,F)$ and $f\in \sA^{\cF/\mathfrak{t}}_{\phg}(M;E)$ are such that for each $H_i\in\cM_1(M)$, 
$$
         \min\{\mathfrak{s}(H^{\Qb}_{0i}), \min\Re( \cE(H_{0i}^{\Qb}))\}  + \min\{ \mathfrak{t}(H^{\Qb}_{0i}),\min\Re(\cF(H_i))\}> \left\{  \begin{array}{ll} 1 +\dim S_i, & i\in\cI, \\
                                                                                                                      0, & \mbox{otherwise},
                                                                                                                                      \end{array}  \right.                                                                                                                          
$$
then $Pf:= (\pr_L\circ \beta_{\Qb})_*(P\cdot (\pr_R\circ\beta_{\Qb})^* f)$ is well-defined and such that 
$$
    Pf\in \sA^{\cK/\mathfrak{k}}_{\phg}(M;F)
$$
with 
$$
  \cK=(\pr_L\circ\beta_{\Qb})_{\#}(\cE-\mathfrak{r}' + (\pr_R\circ\beta_{\Qb})^{\#}\cF)  \quad \mbox{and}  \quad \mathfrak{k}=(\pr_L\circ\beta_{\Qb})_{\#}(\mathfrak{k}-\mathfrak{r}' \dot{+} (\pr_R\circ\beta_{\Qb})^{\#}\mathfrak{t}),$$
that is,
$$
    \cK(H_i)= \cE(H_{i0}^{\Qb})  \overline{\cup} (\cE(\ff_i^{\Qb})+\cF(H_i))\overline{\cup} \left( \underset{H_j\in\cM_1(M)}{\overline{\cup}} ((\cE-\mathfrak{r}')(H^{\Qb}_{ij})+\cF(H_j)) \right)
$$
and 
$$
\mathfrak{k}(H_i)=\min\{\mathfrak{s}(H_{i0}^{\Qb}), \mathfrak{s}(\ff_i^{\Qb})\dot{+}\mathfrak{t}(H_i), (\mathfrak{s}-\mathfrak{r}')(H^{\Qb}_{ij})\dot{+} \mathfrak{t}(H_j) \; \mbox{for} \; H_i\cap H_j\ne \emptyset  \}
$$
if $i\in \cI$, and otherwise
$$
    \cK(H_i)= \cE(H_{i0}^{\Qb}) \overline{\cup} \left( \underset{H_j\in\cM_1(M)}{\overline{\cup}} ((\cE-\mathfrak{r}')(H^{\Qb}_{ij})+\cF(H_j)) \right)
$$
and 
$$
\mathfrak{k}(H_i)=\min\{\mathfrak{s}(H_{i0}^{\Qb}), (\mathfrak{s}-\mathfrak{r}')(H^{\Qb}_{ij})\dot{+} \mathfrak{t}(H_j) \; \mbox{for} \; H_i\cap H_j\ne \emptyset  \},
$$
where $\mathfrak{r}'$ is the multiweight defined in \eqref{mwp.1}.

\label{pdo.19}\end{theorem}

For the small $\Qb$ calculus, this result specializes to the following.

\begin{corollary}
An operator $P\in \Psi^m_{\Qb}(M;E,F)$ induces a continuous linear map 
$$
       P: \CI(M;E)\to \CI(M;F)
$$
which restricts to a continuous linear map
$$
      P: \dot{\cC}^{\infty}(M;E)\to \dot{\cC}^{\infty}(M;F).
$$
\label{pdo.20}\end{corollary}

In fact, to be able to construct good parametrices for manifolds with fibered corners of depth higher than 2, we will need to further enlarge the calculus by giving up on conormality, that is, by replacing the usual notion of conormality of functions by the notions of $\QFB$ conormal functions and $\Qb$ conormal functions.  To describe this, let $\cV$ denote the  Lie structure at infinity on $M$ corresponding to a Lie algebra of $\Qb$ vector fields or $\QFB$ vector fields.  
\begin{definition}
For $E\to M$ a vector bundle over $M$, the space of $\cV$ conormal sections is defined by 
$$
\kridx{\cA_{\cV}}{AV}{$L^\infty$ conormal sections associated to Lie structure $\cV$}(M;E)=\{ \sigma \in L^\infty(M;E) \; | \;
\forall k \in \bbN, \; \forall X_1,\ldots, X_k\in \cV_{\cV}(M),  \; \nabla_{X_1}\cdots\nabla_{X_k}\sigma\in L^\infty(M;E)\},
$$
where $\nabla$ is a choice of connection for $E$.  More generally, we can consider the weighted space of $\cV$ conormal sections 
\begin{equation}
\begin{aligned}
x^{\mathfrak{t}}\cA_{\cV}(M;E) &= \{ \sigma \in x^{\mathfrak{t}}L^\infty(M;E) \; | \;  x^{-\mathfrak{t}}\sigma\in\cA_{\cV}(M;E)\}, \\
&=\{ \sigma \in x^{\mathfrak{t}}L^\infty(M;E) \; | \;
\forall k \in \bbN, \; \forall X_1,\ldots, X_k\in \cV_{\cV}(M),  \; \nabla_{X_1}\cdots\nabla_{X_k}\sigma\in x^{\mathfrak{t}}L^\infty(M;E)\}.
\end{aligned}
\end{equation}
The space $x^{\mathfrak{t}}\cA_{\cV}(M;E)$ is a Fr\'echet space with semi-norms induced by the norm of $x^{\mathfrak{t}}L^\infty(M;E)$ and the $x^{\mathfrak{t}}L^\infty$ norms of the $\QFB$ derivatives.  
\label{con.1}\end{definition}

More generally, we can consider the slightly enlarged space of $\cV$ conormal sections
\begin{equation}
  \kridx{\sA^{\mathfrak{s}}_{\cV,-}}{AVs}{elarged space of ($L^\infty$) $\cV$ conormal sections}(M;E)=  \bigcap_{\mathfrak{t}<\mathfrak{s}} x^{\mathfrak{t}}\cA_{\cV}(M;E).
\label{con.1d}\end{equation}
for $\mathfrak{s}$ a multiweight for the manifold with corners $M$.  It will be also useful to consider the $L^2_b$ version of Definition~\ref{con.1}, namely the space
\begin{equation}
\kridx{\cA_{\cV,2}}{AV2}{$L^2$ conormal sections associated to Lie structure $\cV$}(M;E):=\{ \sigma \in L^2_b(M;E) \; | \;
\forall k \in \bbN, \; \forall X_1,\ldots, X_k\in \cV_{\cV}(M),  \; \nabla_{X_1}\cdots\nabla_{X_k}\sigma\in L^2_b(M;E)\}
\label{con.1b}\end{equation}
and its weighted version $x^{\mathfrak{t}}\cA_{\cV,2}(M;E)$.  
\begin{remark}
When $\cV=b$ stands for the Lie structure at infinity associated to $b$-vector fields, notice that \eqref{con.1d} can be alternatively defined by
$$
 \sA^{\mathfrak{s}}_{b,-}(M;E)=  \bigcap_{\mathfrak{t}<\mathfrak{s}} x^{\mathfrak{t}}\cA_{b,2}(M;E)
$$
thanks to the Sobolev embedding theorem applied to $b$-metrics.  However, for $\cV=\QFB$ or $\cV=\Qb$, this argument involving the Sobolev embedding theorem does not work.  In fact, using $\cA_{\cV,2}(M;E)$ instead of $\cA_{\cV}(M;E)$ yields a different space,
$$
\sA^{\mathfrak{s}}_{\cV,-}(M;E) \ne  \bigcap_{\mathfrak{t}<\mathfrak{s}} x^{\mathfrak{t}}\cA_{\cV,2}(M;E).
$$

\label{con.1c}\end{remark}

Now, we will be also interested in this construction for  $(M^2,\cV\times \cV)$ where $\cV\times \cV$ is the natural Lie structure at infinity obtained by taking the Cartesian product of $(M,\cV)$ with itself.  In this case, $\cA_{\cV\times \cV}(M^2)=\cA_{\cV\times \cV}(M^2;E)$ with $E$ a trivial line bundle corresponds to the space  
 \begin{equation}
      \cA_{\cV}(M^2):= \{ \kappa\in L^\infty(M^2) \; | \; (\pr_R^*P)\kappa\in L^\infty(M^2), \; (\pr_L^* P)\kappa\in L^\infty(M^2)  \quad \forall P\in \Diff^*_{\cV}(M)\}.
\label{st.1}\end{equation}      
We can also denote the space $\sA^{\mathfrak{s}}_{\cV\times \cV,-}(M^2)$ by
$$
  \sA^{\mathfrak{s}}_{\cV,-}(M^2)=  \bigcap_{\mathfrak{t}<\mathfrak{s}} \rho^{\mathfrak{t}}\cA_{\cV}(M^2)
$$
for $\mathfrak{s}$ a multiweight for the manifold with corners $M^2$.
\begin{remark}
Alternatively, multiplying by the pull-back from the right of a non-vanishing $b$-density $\nu_b$, we can consider the space
$$
       \cA_{\cV}(M^2; \pr_R^*(\Omega_b(M)))= \cA_{\cV}(M^2)\cdot \pr_R^*\nu_b,
$$
which can be seen as a space of Schwartz kernels of operators.   From this point of view, the space in \eqref{st.1} corresponds to Schwartz kernels in $L^{\infty}(M^2; \pr_R^*\Omega_b)$ remaining in that space under composition from the left or from the right by $\cV$ differential operators, that is, 
\begin{equation}
\cA_{\cV}(M^2; \pr_R^*(\Omega_b(M)))= \{  K\in L^\infty(M^2; \pr_R^*\Omega_b(M)) \; |  \; P\circ K, K\circ P\in L^\infty(M^2; \pr_R^*\Omega_b(M)) \quad \forall P\in \Diff^*_{\cV}(M)\}.
\label{st.3}\end{equation}
\label{st.2}\end{remark}  
For $E\to M^2$ a vector bundle on $M^2$, we can consider the space of sections
$$
\sA^{\mathfrak{s}}_{\cV,-}(M^2;E):= \sA^{\mathfrak{s}}_{\cV,-}(M^2)\otimes_{\CI(M^2)} \CI(M^2;E).
$$

 Using Lemmas~\ref{ds.2} and \ref{ds.11}, it is not hard to infer that $\cV$ vector fields lifts from the left and from the right to be $b$-vector fields on $M^2_{\cV}$. This means that we can naturally replace $M^2$ by $M^2_{\cV}$ within \eqref{st.1}, that is, we can consider the space 
\begin{equation}
  \cA_{\cV}(M^2_{\cV}):= \{ \kappa\in L^\infty(M^2_{\cV})\;  | \; ((\pr_R\circ \beta_{\cV})^*P)\kappa\in L^\infty(M^2_{\cV}), \; ((\pr_L\circ \beta_{\cV})^*P)\kappa\in L^\infty(M^2_{\cV})
  \; \forall\ P\in \Diff^*_{\cV}(M)\}
\label{ext.1}\end{equation} 
and correspondingly the space
\begin{equation}
\sA^{\mathfrak{s}}_{\cV,-}(M^2_{\cV})=  \bigcap_{\mathfrak{t}<\mathfrak{s}}  \rho^{\mathfrak{t}}\cA_{\cV}(M^2_{\cV})
\label{ext.2}\end{equation}
for $\mathfrak{s}$ a multiweight for the manifold with corners $M^2_{\cV}$.  

In fact, if $\cB$ is any boundary face of $M^2_{\cV}$, we can similarly consider the space 
\begin{equation}
\cA_{\cV}(\cB)= \{\kappa\in L^\infty(\cB)\; |\; ((\pr_R\circ \beta_{\cV})^*P)\widetilde{\kappa} \in L^\infty(M), \; ((\pr_L\circ \beta_{\cV})^*P)\widetilde{\kappa} \in L^\infty(M)
  \; \forall\ P\in \Diff^*_{\cV}(M)\},
\label{ext.3}\end{equation}
where $\widetilde{\kappa}$ is a smooth extension of $\kappa$ off $\cB$, and correspondingly the space
\begin{equation}
\sA^{\mathfrak{s}}_{\cV,-}(\cB)= \bigcap_{\mathfrak{t}<\mathfrak{s}} \rho^{\mathfrak{t}} \cA_{\cV}(\cB)
\label{ext.4}\end{equation}
for $\mathfrak{s}$ a multiweight for the manifold with corners $\cB$.

Let $\cM$ be $M^2_{\cV}$ or any boundary face of $M^2_{\cV}$.  If $\cE$ is an index set on $\cM$ and $\mathfrak{s}$ is a multiweight such that $\inf\Re\cE< \mathfrak{s}$, we can consider the space
$\cA^{\cE/\mathfrak{s}}_{\cV,\phg}(\cM)$ admitting at each boundary hypersurface $\cB$ of $\cM$  a partial polyhomogeneous expansion up to order $\mathfrak{s}(\cB)$.  More precisely,  consider the multiweight $\mathfrak{t}$ defined by 
$$
\mathfrak{t}(\cB)= \min\{\mathfrak{s}(\cB),\min\{  \Re \sigma\; | \;  (\sigma,k)\in \cE(\cB)\}\}
$$
for each boundary hypersurfaces $\cB$ of $\cM$.  Then the space  $\kridx{\cA^{\cE/\mathfrak{s}}_{\cV,\phg}}{AVEs}{partially polyhomogeneous conormal sections with Lie structure $\cV$}(\cM)$ is defined recursively on the depth of $\cM$ to consist of functions in $\kappa\in \cA^{\mathfrak{t}}_{\cV,-}(\cM)$ such that 
 for each boundary hypersurface $\cB$ of $\cM$, there is  for each $(\sigma,k)\in \cE(\cB)$ with $\Re \sigma< \mathfrak{s}(\cB)$ an element $\kappa_{\cB,(\sigma,k)}\in \cA^{(\cE|_{\cB})/(\mathfrak{s}|_{\cB})}_{\cV,\phg}(\cB)$ and a smooth extension $\widetilde{\kappa}_{\cB,(\sigma,k)}$ off $\cB$ such that 
 \begin{equation}
           \left(\kappa- \underset{\Re\sigma<\mathfrak{s}(\cB)}{\sum_{(\sigma,k)\in \cE(\cB)}} (\rho_{\cB}^{\sigma}(\log \rho_{\cB})^k) \widetilde{\kappa}_{\cB,(\sigma,k)} \right)\in \cA^{\mathfrak{t}_{\cB}}_{\cV,-}(\cM),
 \label{ext.5}\end{equation}
where 
\begin{equation}
 \mathfrak{t}_{\cB}(\cH)= \left\{  \begin{array}{ll} \mathfrak{t}(\cH), & \cH\ne \cB, \\ \mathfrak{s}(\cB), & \cH=\cB.  \end{array} \right.
\label{ext.6}\end{equation}
Moreover, we require that each partial polyhomogeneous expansions at the various boundary hypersurfaces are compatible in the sense that if $\cB'$ is another boundary hypersurface in $M^2_{\cV}$ with $\cB'\cap \cB$, then the terms $\widetilde{\kappa}_{\cB,(\sigma,k)}$ have a corresponding partial polyhomogeneous expansion at $\cB\cap \cB'$.   Of course, if $\cB$ has no boundary, then we use the convention $\cA^{(\cE|_{\cB})/(\mathfrak{s}|_{\cB})}_{\cV,\phg}(\cB)=\cA_{\cV}(\cB)$, so each space is ultimately well-defined proceeding by induction on the depth of $\cM$.  We are mostly interested in the case $\cM=M^2_{\cV}$.  Since $\cA^{\cE/\mathfrak{s}}_{\cV,\phg}(M^2_{\cV})$ is a $\CI(M^2_{\cV})$-module, we can for $E$ a vector bundle on $M^2_{\cV}$ associate the space of sections
\begin{equation}
 \cA^{\cE/\mathfrak{s}}_{\cV,\phg}(M^2_{\cV};E):= \cA^{\cE/\mathfrak{s}}_{\cV,\phg}(M^2_{\cV})\otimes_{\CI(M^2_{\cV})} \CI(M^2_{\cV};E).
\label{ext.7}\end{equation} 
With this notation, we can now introduce the larger space of pseudodifferential $\QFB$ operators that we need.

\begin{definition}
Let $E,F$ be vector bundles on $M$.  For $\cE$ an index family for $M^2_{\QFB}$ and $\mathfrak{s}$ a multiweight, we can consider the spaces of weakly conormal $\QFB$ pseudodifferential operators
\begin{equation}
\begin{gathered}
    \Psi^{-\infty,\cE/\mathfrak{s}}_{\QFB,\cn}(M;E,F):= \cA^{\cE/\mathfrak{s}}_{\QFB,\phg}(M^2_{\QFB};\beta_{\QFB}^*(\Hom(E,F)\otimes \pr_R^*{}^{\QFB}\Omega)), \\
    \kridx{\Psi^{m,\cE/\mathfrak{s}}_{\QFB,\cn}}{PsiQFBEscn}{weakly conormal QFB pseudodifferential operators}(M;E,F):=\Psi^{m}_{\QFB}(M;E,F)  + \Psi^{-\infty,\cE/\mathfrak{s}}_{\QFB,\cn}(M;E,F)  \quad m\in \bbR. \end{gathered}
\label{ext.8b}\end{equation}

\label{ext.8}\end{definition}
In particular, if $\mathfrak{s}$ is a $\QFB$ positive multiweight, we can consider the space of operators
\begin{equation}
      \Psi^{-\infty,\mathfrak{s}}_{\QFB,\cn}(M;E,F):= \Psi^{-\infty,\emptyset/\mathfrak{s}}_{\QFB,\cn}(M;E,F).
\label{ext.10}\end{equation}
We will be particularly interested in the case where $\mathfrak{s}=\beta_{\QFB}^{\#}\mathfrak{r}$ in the sense of \cite[(3.2)]{Melrose1992} for some multiweight $\mathfrak{r}$ associated to $M^2$, in which case we denote this space of operators by 
\begin{equation}
  \kridx{\Psi^{-\infty,\mathfrak{r}}_{\QFB,\res}}{PsiQFBres}{QFB $\mathfrak{r}$ residual pseudodifferential operators}(M;E,F):= \Psi^{-\infty,\beta_{\QFB}^{\#}\mathfrak{r}}_{\QFB,\cn}(M;E,F)
\label{ext.11}\end{equation}
and call it the space of \textbf{$\QFB$ $\mathfrak{r}$ residual pseudodifferential operators}.  

There is a corresponding definition for $\Qb$ operators.
\begin{definition}
Let $E,F$ be vector bundles on $M$.  For $\cE$ an index family for $M^2_{\Qb}$ and $\mathfrak{s}$ a multiweight, we can consider the spaces of \textbf{weakly conormal $\Qb$ pseudodifferential operators}
\begin{equation}
\begin{gathered}
\Psi^{-\infty,\cE/\mathfrak{s}}_{\Qb,\cn}(M;E,F):= \cA^{\cE/\mathfrak{s}}_{\Qb,\phg}(M^2_{\QFB};\beta_{\Qb}^*(\Hom(E,F)\otimes \pr_R^*{}^{\QFB}\Omega)), \\
    \kridx{\Psi^{m,\cE/\mathfrak{s}}_{\Qb,\cn}}{PsiQbEscn}{weakly conormal Qb pseudodifferential operators}(M;E,F):=\Psi^{m}_{\Qb}(M;E,F)  + \Psi^{-\infty,\cE/\mathfrak{s}}_{\Qb,\cn}(M;E,F)  \quad m\in \bbR.     \end{gathered}
\label{ext.9b}\end{equation}
 \label{ext.9}\end{definition}   
 Similarly, if $\mathfrak{s}$ is a multiweight, we can consider the space
 $$
    \Psi^{-\infty,\mathfrak{s}}_{\Qb,\cn}(M;E,F)= \Psi^{-\infty,\emptyset/\mathfrak{s}}_{\Qb,\cn}(M;E,F),
 $$
 and if $\mathfrak{r}$ is a multiweight on $M^2$ such that $\beta^*_{\Qb}\mathfrak{r}$ is $\Qb$ positive, we use the notation
 \begin{equation}
 \kridx{\Psi^{-\infty,\mathfrak{r}}_{\Qb,\res}}{PsiQbres}{Qb $\mathfrak{r}$ residual pseudodifferential operators}(M;E,F):= \Psi^{-\infty,\beta_{\Qb}^{\#}\mathfrak{r}}_{\Qb,\cn}(M;E,F)
 \label{ext.12}\end{equation}
 and refer to such a space as the space of \textbf{$\Qb$ $\mathfrak{r}$ residual pseudodifferential operators}.
 In all these cases, we have obviously natural inclusions
 $$
    \Psi^{-\infty,\cE/\mathfrak{s}}_{\cV}(M;E,F)\subset \Psi^{-\infty,\cE/\mathfrak{s}}_{\cV,\cn}(M;E,F) 
 $$
 for $\cV=\QFB$ and $\cV=\Qb$.
 
 \begin{remark}
For $\cV=\QFB$ or $\Qb$, notice that by Remark~\ref{st.2}, Lemma~\ref{ds.2} and the fact that $\cV$ vector fields lift from the left and from the right to $b$-vector fields on $M^2_{\cV}$, the space $\Psi^{-\infty,\cE/\mathfrak{s}}_{\cV,\cn}(M;E,F)$ is stable under left or right composition with $\cV$ differential operators.   
\label{ext.8c}\end{remark}

\section{The $\QFB$ triple space} \label{tris.0}

To obtain nice composition results for $\QFB$ operators, we need to construct a $\QFB$ triple space.  This is where the use of the ordered product of \cite{KR3} will be particularly useful.  The key property that will make the whole construction works is that the ordered product is associative.  When we consider ordered product or the reverse ordered product of manifolds with fibered corners, this is consistent with the fact that the ordered product corresponds geometrically to the Cartesian product of wedge metrics, while the reverse ordered product corresponds to the Cartesian product of $\QFB$ metrics.  

To see that the ordered product is associative directly from the definition, recall first from \cite{KR3} that the ordered product $X\ttimes Y$ of two manifolds with ordered corners is again a manifold with ordered corners.  Indeed, if $\cM_{1,0}(X)=\cM_1(X)\cup\{X\}$, then the partial order on $\cM_{1}(X)$ can be extended to $\cM_{1,0}(X)$ by declaring the new element $X$ to be such that $X>H$ for all $H\in \cM_1(X)$.  With this notation, there is a natural identification
$$
  \cM_{1,0}(X\ttimes Y)=\cM_{1,0}(X)\ttimes\cM_{1,0}(Y).
$$   
Under this identification, the partial order on $\cM_{1,0}(X\ttimes Y)$ making $X\ttimes Y$ a manifold with ordered corners is then the one such that  $(H,G)\le(H',G')$ if and only if $H\le G$ and $H'\le G'$ for $H,H'\in \cM_{1,0}(X)$ and $G,G'\in\cM_{1,0}(Y)$.

Since $X\ttimes Y$ is itself a manifold with ordered corners, we can consider its ordered product $(X\ttimes Y)\ttimes Z$ with another manifold with ordered corners $Z$.  The ordered product is in fact associative in the following sense.    
\begin{lemma}[\cite{KR3}] 
If $X$, $Y$ and $Z$ are manifolds with ordered corners, the identity map in the interior of $X\times Y\times Z$ extends uniquely to a diffeomorphism
$$
           (X\ttimes Y)\ttimes Z\cong X\ttimes(Y\ttimes Z).
$$ 
\label{op.6}\end{lemma}
\begin{proof}
The proof of \cite{KR3} relies on a certain universal property of the ordered product which is established using results of \cite{KM}.  For the convenience of the reader, we will present an alternative proof of the associativity of the ordered product relying instead directly on Lemmas~\ref{bf.3} and \ref{bf.1}.  

First, as in the proof of Lemma~\ref{ds.2}, we can use Lemma~\ref{bf.1} to see that the projections from $X\times Y$ onto $X$ and $Y$ lifts to $b$-fibrations
\begin{equation}
    X\ttimes Y\to X \quad \mbox{and} \quad X\ttimes Y\to Y.
\label{op.7}\end{equation}
Similarly, the projections $(X\ttimes Y)\times Z$ onto $X\ttimes Y$ and $Z$ induce $b$-fibrations
\begin{equation}
   (X\ttimes Y)\ttimes Z \to X\ttimes Y \quad \mbox{and} \quad (X\ttimes Y)\ttimes Z\to Z.  
\label{op.8}\end{equation}
In particular, combining the $b$-fibrations $X\ttimes Y\to Y$ with the identity map $Z\to Z$ yields a $b$-fibration
\begin{equation}
        (X\ttimes Y)\times Z\to Y\times Z.  
\label{op.9}\end{equation}
Let $\cM_1^Y(X\ttimes Y)$ be the subset of $\cM_1(X\ttimes Y)$ consisting of boundary hypersurfaces not mapped surjectively onto $Y$ under the $b$-fibration $X\ttimes Y\to Y$.  Then applying Lemma~\ref{bf.3}, we see that the $b$-fibration \eqref{op.9} lifts to a $b$-fibration
\begin{equation}
  [(X\ttimes Y)\times Z; \cM^Y_1(X\ttimes Y)\times \cM_{1}(Z)]\to Y\ttimes Z,  
\label{op.10}\end{equation}
where the order in which we blow up $p$-submanifolds in $(X\ttimes Y)\times Z$ is compatible with the partial order on $\cM_1(X\ttimes Y)\times \cM_1(Z)$.  Using Lemma~\ref{bf.1}, this can be further lifted to a $b$-fibration
\begin{equation}
  (X\ttimes Y)\ttimes Z\to Y\ttimes Z.    
\label{op.11}\end{equation}
On the other hand, composing the first map in  \eqref{op.8} with the first map in \eqref{op.7} yields a $b$-fibration
\begin{equation}
    (X\ttimes Y)\ttimes Z\to X.
\label{op.12}\end{equation}
Combining the $b$-fibrations \eqref{op.11} and \eqref{op.12} then yields a surjective $b$-submersion
\begin{equation}
   (X\ttimes Y)\ttimes Z\to X\times (Y\ttimes Z).
\label{op.13}\end{equation}
However, this is not a $b$-fibration, since for $H\in \cM_1(X)$ and $(G,F)\in (\cM_{1,0}(Y)\times\cM_{1,0}(Z))\setminus\{(Y,Z)\}$, the boundary hypersurface corresponding to the lift of $H\times G\times F$ in $(X\ttimes Y)\ttimes Z$ is mapped by \eqref{op.13} into a codimension $2$ corner of $X\times (Y\ttimes Z)$.  The codimension $2$ corners obtained in this way corresponds in fact exactly to those that need to be blown-up to obtain $X\ttimes (Y\ttimes Z)$.    If $T$ is the first codimension 2 corner that is blown up to obtain $X\ttimes (Y\ttimes Z)$ from $X\times (Y\ttimes Z)$, then we can apply Lemma~\ref{bf.3} to lift \eqref{op.13} to a surjective $b$-submersion
\begin{equation}
       (X\ttimes Y)\ttimes Z \to [X\times (Y\ttimes Z);T].
\label{op.14}\end{equation}
Indeed, if $T$ corresponds to the lift of $H\times G\times F$, then the inverse image of $T$ in $(X\ttimes Y)\ttimes Z$ is the $p$-submanifold $S$ corresponding to the lift of $H\times G\times F$ to $(X\ttimes Y)\ttimes Z$.  In particular, $S\in \cM_1((X\ttimes Y)\ttimes Z)$ is a boundary hypersurface, so blowing it up does not alter the geometry of $(X\ttimes Y)\ttimes Z$ and the conclusion of Lemma~\ref{bf.3} yields the claimed surjective $b$-submersion \eqref{op.14}.  Notice moreover that under the map \eqref{op.14}, the boundary hypersurface $S$ is no longer sent into a codimension 2 corner, but onto the boundary hypersurface created by the blow-up of $T$.    

Continuing in this way, we can apply Lemma~\ref{bf.3} repeatedly to see that the map \eqref{op.13} lifts to a surjective $b$-submersion
\begin{equation}
       \psi_1:(X\ttimes Y)\ttimes Z\to X\ttimes(Y\ttimes Z).  
\label{op.15}\end{equation} 
By the discussion above, none of the boundary hypersurfaces of $(X\ttimes Y)\ttimes Z$ is sent onto a codimension $2$ corner of $X\ttimes (Y\ttimes Z)$, so that \eqref{op.15} is in fact a $b$-fibration.  

Of course, by symmetry, a similar argument yields a $b$-fibration
$$
       \psi_2: X\ttimes (Y\ttimes Z)\to (X\ttimes Y)\ttimes Z.
$$      
Since $\psi_1\circ \psi_2$ and $\psi_2\circ\psi_1$ restrict to be the identity on the interior of $X\ttimes (Y\ttimes Z)$ and $(X\ttimes Y)\ttimes Z$, we must have by continuity that 
$$
    \psi_1\circ\psi_2=\Id \quad \mbox{and} \quad \psi_2\circ\psi_1=\Id.  
$$
Hence, $\psi_1$ is a diffeomorphism with inverse $\psi_2$, from which the result follows.  
\end{proof}

Thanks to this lemma, we can unambiguously use the notation 
$$
X\ttimes Y\ttimes Z=(X\ttimes Y)  \ttimes Z= X\ttimes (Y\ttimes Z)\cong  (X\ttimes Z)\ttimes Y
$$ 
and similarly for the reverse ordered product.

Now, if $M$ is a manifold with fibered corners, this discussion indicates that we can define the ordered product triple space of $M$ by 
\begin{equation}
   \kridx{M^3_{\op}}{M3op}{ordered product triple space}= M\ttimes M \ttimes M
\label{op.16b}\end{equation}
and the reverse ordered triple space by
\begin{equation}
    \kridx{M^3_{\rp}}{M3rp}{reverse ordered product triple space}= M\rttimes M \rttimes M.
\label{op.16}\end{equation} 
In the later case, the partial order on $\cM_{1,0}(X)$ extending the reverse order is the one obtained by declaring $X$  to be such that $X>_r H$ for all $H\in \cM_1(X)$, that is, by declaring $X$ to be maximal with respect to the reverse order.

\begin{proposition}
Let $\alpha$ be a permutation of $\{1,2,3\}$ and consider the diffeomorphism $A: M^3\to M^3$ given by $A(m_1,m_2,m_3)= (m_{\alpha(1)}, m_{\alpha(2)}, m_{\alpha(3)}).$  Then $A$ lifts to  diffeomorphisms $A_{\rp}: M^3_{\rp}\to M^3_{\rp}$ and $A_{\op}: M^3_{\op}\to M^3_{\op}$
\label{ts.7}\end{proposition}
\begin{proof}
This follows from Remark~\ref{ps.0a} and Lemma~\ref{op.6}.  
\end{proof}
Let $\pr_L,\pr_C,\pr_R: M^3\to M^2$  be the projections given by
\begin{equation}
    \pr_L(m,m',m'')= (m,m'), \quad \pr_C(m,m',m'')= (m,m''), \quad \pr_R(m,m',m'')= (m',m''). 
\label{ts.1a}\end{equation}
\begin{proposition}
For $o=L,C,R$, the projection $\pr_{o}: M^3\to M^2$ lifts to a $b$-fibration $\pi^{\rp}_{o}:M^3_{\rp}\to M^2_{\rp}$ inducing the commutative diagram
\begin{equation}
\xymatrix{
    M^3_{\rp} \ar[d]^{\pi^{\rp}_{o}} \ar[r] & M^3 \ar[d]^{\pr_o} \\
     M^2_{\rp} \ar[r] & M^2,
}
\label{ts.8a}\end{equation} 
where the horizontal arrows are given by the natural blow-down maps.  Similarly, the projection $\pr_{o}$ lifts to $b$-fibration $\pi^{\op}_o: M^3_{\op}\to M^2_{\op}$ inducing the commutative diagram \eqref{ts.8a} with $\rp$ replaced by $\op$.  
\label{ts.8}\end{proposition}
\begin{proof}
By Proposition~\ref{ts.7}, it suffices to prove the proposition for $o=L$.  In this case, we need to show that the projection
$$
     (M\rttimes M)\times M\to M\rttimes M
$$
lifts to a $b$-fibration
$$
   \pi^{\rp}_L: (M\rttimes M)\rttimes M\to M\rttimes M,
$$
which is a simple consequence of Lemma~\ref{bf.1}.
\end{proof}

Let $H_1,\ldots, H_{\ell}$ be an exhaustive list of the boundary hypersurfaces of $M$ listed in an order compatible with the partial order on $\cM_1(M)$.  Using the convention that $H_0:=M$, denote by $H_{ijk}^{\rp}$ the lift of $H_i\times H_j\times H_k$ to $M^3_{\rp}$ for $i,j,k\in\{0,1,\ldots,\ell\}$.  Thus, $H^{\rp}_{000}=M^3_{\rp}$, and otherwise $H_{ijk}^{\rp}$ corresponds to a boundary hypersurface of $M^3_{\rp}$ and all boundary hypersurfaces of $M^3_{\rp}$ are obtained in this way.  

Let $\diag_{\rp}$ be the lift of the diagonal $\diag_M\subset M^2$ in $M^2_{\rp}$.  Let also $\diag^3_{\rp}\subset M^3_{\rp}$ be the lift of the triple diagonal 
$$
\diag^3_M=\{ (m,m,m)\in M^3\; | \; m\in M\}.
$$
As the next lemma shows, the triple space $M^3_{\rp}$ behaves well with respect to the $p$-submanifolds
$$
   \diag_{\rp,o}= (\pi^{\rp}_o)^{-1}(\diag_{\rp}), \quad o\in \{L,C,R\}.
$$
\begin{lemma}
For $o\ne o'$, the $b$-fibration $\pi^{\rp}_o$ is transversal to $\diag_{\rp,o'}$ and induces a diffeomorphism $\diag_{\rp,o'}\cong M^2_{\rp}$ sending $\diag^3_{\rp}$ onto $\diag_{\rp}$ in $M^2_{\rp}$.
\label{ts.8b}\end{lemma}
\begin{proof}
By Proposition~\ref{ts.7}, we can assume that $o=R$ and $o'=L$.  Starting with $M^2_{\rp}\times M$, consider the $p$-submanifold $\diag_{\rp}\times M$.  On this space, notice that the lift of the projection $\pr_R: M^3\to M^2$ induces a $b$-fibration $M^2_{\rp}\times M\to M^2$ transversal to $\diag_{\rp}\times M$ which induces a diffeomorphism 
\begin{equation}
    \diag_{\rp}\times M\cong M\times M.
\label{ts.8c}\end{equation}
When performing the blow-ups constructing $M^3_{\rp}$ out of $M^2_{\rp}\times M$, the transversality of the lift of $\pr_R$ with respect to the lift of $\diag_{\rp}\times M$ is preserved after each blow-up, so that after the final blow-up, we obtain that $\pi^{\rp}_R$ is transversal to $\diag_{\rp,L}$ and induces the claimed diffeomorphism.   
\end{proof}

As in \cite{Mazzeo-MelrosePhi}, to define the $\QFB$ triple space, we need then to blow-up the lifts of the $p$-submanifolds $\Phi_i$ with respect to $\pi^{\rp}_{o}$ for $o\in \{L,C,R\}$.  Now, each $p$-submanifold $\Phi_i\subset M^2_{\rp}$ is such that its lift with respect to $\pi^{\rp}_{L}$ is the normal family $\{\Phi^L_{ij}\; | \; j\in \{0,1,\ldots,\ell\} \}$ with $\Phi^L_{ij}$ inside $H^{\rp}_{iij}$.    Similarly, let $\{\Phi^C_{ij}\; \; | \; j\in\{0,1,\ldots,\ell\}\}$ and $\{\Phi^R_{ij}\; |\; j\in \{0,\ldots \ell\}\}$ be the normal families corresponding to the lifts of $\Phi_i$ with respect to $\pi^{\rp}_{C}$ and $\pi^{\rp}_R$.

For $o$ and $o'$ distinct, the $p$-submanifolds $\Phi_{ii}^{o}$ and $\Phi^{o'}_{ii}$ intersect.  To describe their intersection, let $x_i$, $x_i'$ and $x_i''$ denote the pull-back of $x_i$ to $M^3$ with respect to the projection on the first, second and third factors, so that 
\begin{equation}
   \pr_L^*s= \prod_i \left( \frac{x_i}{x_i'} \right), \quad \pr_C^*s= \prod_i \left( \frac{x_i}{x_i''} \right), \quad \pr_R^* s= \prod_i \left(\frac{x_i'}{x_i''} \right).
\label{ts.13}\end{equation} 
 In $M^3_{\rp}$, consider then the subset $\diag_{i,oT}$ of $H^{\rp}_{iii}$ where $(\pi^{\rp}_o)^* s$ lifts to be equal to one.  Proceeding as in the proof of Lemma~\ref{ps.0}, one can show that $\diag_{i,o,T}$ is a $p$-submanifold of $M^3_{\rp}$.  From \eqref{ts.13}, we clearly see that the intersection of any pair of 
 $\diag_{i,LT}$, $\diag_{i,CT}$ and $\diag_{i,RT}$ is the $p$-submanifold 
 $$
      \diag_{i,T}= \diag_{i,LT}\cap \diag_{i,CT}\cap\diag_{i,RT}.
 $$
 Since $\Phi^o_{ii}\subset \diag_{i,oT}$, we see that $\Phi^{o}_{ii}\cap\Phi^{o'}_{ii}\subset \diag_{i,T}$ for $o\ne o'$.  To give a more precise description of this intersection, we need to discuss natural geometric structures on $\diag_{i,oT}$ and $\diag_{i,T}$.  First, by Proposition~\ref{ts.7}, we can assume that $o=L$.  In this case, $\diag_{i,LT}$ corresponds to the lift of $\diag_i\times H_i\subset M^2_{\rp}\times M$ to $M^3_{\rp}=M^2_{\rp}\rttimes M$, where $\diag_i$ is the $p$-submanifold of Lemma~\ref{ps.0}.  Hence, to describe $\diag_{i,LT}$, we can keep track of the effect on $\diag_i\times H_i$ of each blow-up performed to obtain $M^2_{\rp}\rttimes M$ from $M^2_{\rp}\times M$.  Since $\diag_i\subset M^2_{\rp}$ only intersects boundary hypersurfaces $H^{\rp}_{jk}$ for $j,k\in\{0,1\ldots,\ell\}$ with $(H_j,H_k)$ comparable with $(H_i,H_i)$, we will only keep track of the blow-ups involving those.  Keeping this in mind, we can follow essentially the same steps as in the description of $\diag_i$ in \S~\ref{ds.0}, namely the blow-ups can be performed as follows:
 \begin{itemize}
 \item[Step 1:]  Blow up the $p$-submanifolds $H^{\rp}_{jk}\times H_q$ for $q>i$ and $(j,k)>(i,i)$ using lexicographic order;
 \item[Step 2:]  Blow up the lifts of $H^{\rp}_{jk}\times H_i$ and $H^{\rp}_{ii}\times H_q$ for $(j,k)>(i,i)$ and $q>i$;
 \item[Step 3:]  Blow up the lift of $H^{\rp}_{ii}\times H_i$;
 \item[Step 4:]  Blow up the lifts of $H^{\rp}_{jk}\times H_q$ for $(j,k)>(i,i)$ and $q<i$, as well as the lifts of $H^{\rp}_{jk}\times H_q$ for $(j,k)<(i,i)$ and $q>i$;
 \item[Step 5:]  Blow up the lifts of $H^{\rp}_{jk}\times H_i$ and $H^{\rp}_{ii}\times H_q$ for $(j,k)<(i,i)$ and $q<i$, then blow up the lifts of $H^{\rp}_{jk}\times H_q$ for $(j,k)<(i,i)$ and $q<i$.
 \end{itemize} 
 Let $\diag^{(k)}_{i,LT}$ be the lift of $\diag_i\times H_i$ after Step $k$.  Combining the fiber bundle \eqref{ps.6} with the fiber bundle $\phi_i: H_i\to S_i$ yields the fiber bundle 
\begin{equation}
\xymatrix{
  (Z_i)^2_{\rp}\times Z_i \ar[r]  & \diag_i\times H_i \ar[d]^{(\phi_i\rttimes\phi_i)\times \phi_i} \\
   &   (S_i)^2_{\op}\times S_i.
 }
 \label{ts.27}\end{equation}
 Clearly, after Step 1, $\diag^{(1)}_{i,LT}$ is still a codimension 3 $p$-submanifold contained in a codimension 2 corner and the fiber bundle \eqref{ts.27} lifts to a fiber bundle 
 \begin{equation}
\xymatrix{
  (Z_i)^3_{\rp} \ar[r]  & \diag^{(1)}_{i,LT} \ar[d] \\
   &   (S_i)^2_{\op}\times S_i.
 }
 \label{ts.28}\end{equation} 
  Now, the $p$-submanifolds blown up in Step 2 correspond to boundary hypersurfaces in $\diag^{(1)}_{i,LT}$, so the blow-down map induces a natural diffeomorphism $\diag^{(2)}_{i,LT}\cong \diag^{(1)}_{i,LT}$.  However, notice that after Step 2 is completed, the lifts of $M^2_{\rp}\times H_q$ and $H^{\rp}_{jk}\times M$ are disjoint from $\diag^{(2)}_{i,LT}$ for $q>i$ and $(j,k)>(i,i)$.  The effect of Step 3 is to transform the codimension 3 $p$-submanifold $\diag^{(2)}_{i,LT}$ into a codimension 2 $p$-submanifold contained in a boundary hypersurface, namely 
 $$
        \diag^{(3)}_{i,LT}=\diag^{(2)}_{i,LT}\times [0,\frac{\pi}2]_{\theta_i},
 $$
 where the angular variable can be taken to be
 $$
      \theta_i= \arctan\left(  \frac{\prod_{j\ge i} \prod_{k\ge i}x_{jk}}{\prod_{j\ge i }x_j''}\right)=\arctan\left( \frac{v_{ii}}{v_i''}\right)
 $$
 with $x_{jk}$ a boundary defining function for $H^{\rp}_{jk}$ in $M^2_{\rp}$ and 
 $$
   v_{ii}:= \prod_{j\ge i} \prod_{k\ge i}x_{jk}.
 $$ 
Under this identification, the fiber bundle \eqref{ts.28} lifts to a fiber bundle 
 \begin{equation}
\xymatrix{
  (Z_i)^3_{\rp} \ar[r]  & \diag^{(3)}_{i,LT} \ar[d] \\
   &   (S_i)^2_{\op}\times S_i\times [0,\frac{\pi}2]_{\theta_i}.
 }
 \label{ts.29}\end{equation} 
Now, by Step 2, the $p$-submanifold blown up in Step 4 are all disjoint from $\diag^{(3)}_{i,LT}$, so in fact $\diag^{(4)}_{i,LT}=\diag^{(3)}_{i,LT}.$ 

Finally, if $H_i$ is a minimal boundary hypersurface, then Step 5 is vacuous and $\diag^{(5)}_{i,LT}=\diag^{(4)}_{i,LT}$ with $(S_i)^2_{\op}=S_i^2$.  This is not the case however if $H_i$ is not minimal.   In this case,  the blow-ups of Step~5 correspond to blow-ups of the base $(S_i)^2_{\op}\times S_i\times [0,\frac{\pi}2]_{\theta_i}$ of \eqref{ts.29}, so that $\diag_{i,LT}=\diag^{(5)}_{i,LT}$ and the fiber bundle \eqref{ts.29} lifts to a fiber bundle
\begin{equation}
\xymatrix{
  (Z_i)^3_{\rp} \ar[r]  & \diag_{i,LT} \ar[d] \\
   &   (S_i)^2_{\op}\rjtimes S_i
 }
 \label{ts.29b}\end{equation} 
with  $(S_i)^2_{\op}\rjtimes S_i$ the reverse join product
$$
[(S_i)^2_{\op}\times S_i\times [0,\frac{\pi}2]_{\theta_i}; (S_i)^2_{\op}\times \cM_1^r(S_i)\times \{0\}, \cM_1^r((S_i)^2_{\op})\times S_i\times \{\frac{\pi}2\}, \cM_1^r((S_i)^2_{\op})\times \cM^r_1(S_i)\times [0,\frac{\pi}2]_{\theta_i}]
$$
of $(S_i)^2_{\op}$ and $S_i$.  In terms of this description, the equation for $\diag^{\rp}_{i,CT}$ is given by
\begin{equation}
       \left(\frac{v_{ii}}{v_i''}\right)\left(\frac{\rho_{ii}}{\rho_i''}\right)=1,
\label{ts.30}\end{equation}
where $\rho_i''=\prod_{H_j<H_i}x_j''$ is a total boundary defining function for $S_i$ (on the right factor) and
$$
    \rho_{ii}= \prod_{H_j<H_i}\prod_{H_k<H_i} x_{jk}
$$
is a total boundary defining function for $(S_i)^2_{\op}$.
  In particular, this equation behaves well with respect to the fiber bundle \eqref{ts.29} in the sense that it descends to an equation on the base $(S_i)^2_{\op}\rjtimes S_i$.  By \cite[Corollary~4.15]{KR3}, the equation \eqref{ts.30} on $(S_i)^2_{\op}\rjtimes S_i$ is naturally identified with $(S_i)^3_{\op}= (S_i)^2_{\op}\ttimes S_i$.  Hence, on the submanifold $\diag_{i,T}=\diag_{i,LT}\cap \diag_{i,CT}$, the fiber bundle \eqref{ts.29} restricts to give a fiber bundle
\begin{equation}
\xymatrix{
  (Z_i)^3_{\rp} \ar[r]  & \diag_{i,T} \ar[d]^{\phi_i\rttimes\phi_i\rttimes\phi_i} \\
   &   (S_i)^3_{\op}.
 }
 \label{ts.31}\end{equation}

If $\pi^{\op}_{L,S_i}: (S_i)^3_{\op}\to (S_i)^2_{\op}$ is the $b$-fibration induced by the projection $S^3_i\to S^2_i$ on the first two factors, then notice that with respect to the fiber bundle \eqref{ts.31}, 
\begin{equation}
   \Phi^L_{ii}\cap \diag_{i,T}= (\phi_i\rttimes \phi_i\rttimes \phi_i)^{-1}((\pi^{\op}_{L,S_i})^{-1}(\diag^2_{S_i,\op})),
\label{ts.32}\end{equation}
where $\diag^2_{S_i,\op}$ is the lift of the diagonal in $S_i^2$ to $(S_i)^2_{\op}$.  By symmetry, that is, more specifically by Proposition~\ref{ts.7}, we have more generally for $o\in\{L,C,R\}$ that
\begin{equation}
   \Phi^o_{ii}\cap \diag_{i,T}= (\phi_i\rttimes \phi_i\rttimes \phi_i)^{-1}((\pi^{\op}_{o,S_i})^{-1}(\diag^2_{S_i,\op})),
\label{ts.33}\end{equation}

We are now ready to describe how $\Phi^L_{ii}$, $\Phi^C_{ii}$ and $\Phi^R_{ii}$ intersect.

\begin{proposition}
In terms of the fiber bundle \eqref{ts.31}, the intersection of any pair of $\Phi^L_{ii}$, $\Phi^C_{ii}$ and $\Phi^R_{ii}$ is the $p$-submanifold 
$$
     \Phi_{i,T}= \Phi^L_{ii}\cap \Phi^C_{ii}\cap \Phi^R_{ii}\cong (\phi_i\rttimes\phi_i\rttimes\phi_i)^{-1}(\diag^3_{S_i,\op}),
$$
where $\diag^3_{S_i,\op}\subset (S_i)^3_{\op}$ is the lift of the triple diagonal
$$
      \diag^3_{S_i}= \{ (s,s,s)\in(S_i)^3 \; | \; s\in S_i \}.
$$ 
\label{ts.18}\end{proposition}
\begin{proof}
Since $\Phi^o_{ii}\subset \diag_{i,oT}$, we see that for $o\ne o'$,
$$
     \Phi^o_{ii}\cap \Phi^{o'}_{ii}\subset \diag_{i,oT}\cap \diag_{io'T}=\diag_{i,T},
$$
so the result follows from \eqref{ts.33}.
\end{proof}

As the next lemma shows,  many potential intersections between the $p$-submanifolds $\Phi^o_{ij}$ and $\Phi^{o'}_{ik}$ for $o\ne o'$ can be ruled out.   

\begin{lemma}
For $o,o'\in \{L,C,R\}$ distinct  and $j,k\in\{0,1,\ldots,\ell\}$, we have that 
$$
\Phi^o_{ij}\cap \Phi^{o'}_{ik}\ne\emptyset\; \Longrightarrow \; j\ge k=i \quad \mbox{or} \quad k\ge j=i \quad \mbox{and} \quad \Phi^o_{ij}\cap \Phi^{o'}_{ik}\subset \Phi_{i,T}.  
$$
 
\label{ts.12}\end{lemma}
\begin{proof}
By Proposition~\ref{ts.7}, we can assume that $o=L$ and $o'=R$ and that $j\le k$.  Suppose first that $j=0$.  Then $\pr_R^* s$ lifts to be zero on the interior of $H_{ii0}$, which means it lifts to be zero on $H^{\rp}_{ii0}$.  In particular,
$$
     \Phi^R_{ik}\cap H^{\rp}_{ii0}=\emptyset \quad \forall k\in \{0,1,\ldots,\ell\},
$$
since $\pi_R^* s$ lifts to be equal to $1$ on $\Phi^R_{ik}$.  Since $\Phi^L_{i0}$ is included in $H^{\rp}_{ii0}$, this implies that
$$
     \Phi^L_{i0}\cap \Phi^R_{ik}=\emptyset \quad \forall k\in \{0,1,\ldots,\ell\}.  
$$ 
Thus, suppose instead that $0<j\le k$ and that $\Phi^L_{ij}\cap \Phi^R_{ik}\ne \emptyset $.  Since $\Phi^L_{ij}\subset H^{\rp}_{iij}$ and $\Phi^R_{ik}\subset H^{\rp}_{kii}$, this means that
$$
     H^{\rp}_{iij}\cap H^{\rp}_{kii}\ne \emptyset.  
$$
In particular, $H^{\rp}_{iij}$ and $H^{\rp}_{kii}$ must be comparable with respect to the partial order on the boundary hypersurfaces of $M^3_{\rp}= M^2_{\rp}\rttimes M$.  Since $j\le k$, we must have therefore that $H^{\rp}_{iij}\le H^{\rp}_{kii}$, which implies that $(i,i)\le (k,i)$ and $j\le i$, that is
$$
       j\le i \le k.
$$
First, if $j<i\le k$, then the function $(\pr_{R}^*s_j)^{-1}$ lifts to be equal to zero on $H^{\rp}_{iij}$, while it lifts to be equal to $1$ on $\Phi^R_{ik}$, so that 
$$
      \Phi^L_{ij}\cap \Phi^R_{ik}\subset H^{\rp}_{iij}\cap \Phi^R_{ik}=\emptyset
$$       
as claimed.

If instead $j=i$, then $\Phi^L_{ij}=\Phi^L_{ii}\subset H^{\rp}_{iii}$.  Since $\Phi^{o}_{iq}\cap H^{\rp}_{iii}\subset\Phi^o_{ii}$ for all $q$ and $o$, we must have that
\begin{equation}
  \Phi^L_{ij}\cap \Phi^{R}_{ik}= \Phi^L_{ij}\cap \Phi^{R}_{ik}\cap H^{\rp}_{iii}\subset  \Phi^L_{ii}\cap \Phi^R_{ii}=\Phi_{i,T}
\label{ts.34}\end{equation}
as claimed.  \end{proof}

These preliminary results allow us to define the $\QFB$ triple space as follows.

\begin{definition}
The $\QFB$ triple space $\kridx{M^3_{\QFB}}{M3QFB}{QFB triple space}$ is obtained from $M^3_{\rp}$ by blowing up the sequence of families of $p$-submanifolds 
$$
   \{ \Phi_{i,T}, \{\Phi^L_{ij}\}_{j=0}^{\ell},\{\Phi^C_{ij}\}_{j=0}^{\ell},\{\Phi^R_{ij}\}_{j=0}^{\ell}\}
$$
for $i\in\{1,\ldots, \ell\}$ with $i$ and $j$ increasing.  We denote by $\beta^3_{\QFB}: M^3_{\QFB}\to M^3$ the corresponding blow-down map. 
\label{ts.11}\end{definition}
Similarly, in the $\QAC$ setting,  we can then define the $\Qb$ triple space as follows.
\begin{definition}
Suppose the maximal boundary hypersurfaces of $M$ are given by $H_{k+1},\ldots, H_{\ell}$ and are such that $S_i=H_i$ with $\phi_i$ the identity map.  In this case, the $\Qb$ triple space $\kridx{M^3_{\Qb}}{M3Qb}{Qb triple space}$ is obtained by blowing up the sequence of families of $p$-submanifolds
$$
   \{ \Phi_{i,T}, \{\Phi^L_{ij}\}_{j=0}^{\ell},\{\Phi^C_{ij}\}_{j=0}^{\ell},\{\Phi^R_{ij}\}_{j=0}^{\ell}\}
$$
for $i\in\{1,\ldots, k\}$ with $i$ and $j$ increasing.
\label{ts.22}\end{definition}

\begin{proposition}
The projection $\pr_o: M^3\to M^2$ for $o\in \{L,C,R\}$ lifts to a $b$-fibration 
$$\pi^3_{\QFB,o}: M^3_{\QFB}\to M^2_{\QFB}
$$  
Similarly, when the $\Qb$ double space is defined, its lifts to a $b$-fibration
 $$
 \pi^3_{\Qb,o}: M^3_{\Qb}\to M^2_{\Qb}.
 $$
\label{ts.23}\end{proposition}
\begin{proof}
By Proposition~\ref{ts.8}, $\pr_o$ lifts to a $b$-fibration $\pi^{\rp}_{o}: M^3_{\rp}\to M^2_{\rp}$, so it suffices to show that $\pi^{\rp}_{o}$ lifts to a $b$-fibration $\pi^3_{\QFB,o}: M^3_{\QFB}\to M^2_{\QFB}$ and similarly for the $\Qb$ triple space.  

Now, in the sequence of blow-ups
\begin{equation}
   \{ \Phi_{i,T}, \{\Phi^L_{ij}\}_{j=0}^{\ell},\{\Phi^C_{ij}\}_{j=0}^{\ell},\{\Phi^R_{ij}\}_{j=0}^{\ell}\},
\label{ts.23b}\end{equation}
notice that by Proposition~\ref{ts.18}, the first blow-up separates $\Phi^{L}_{ii}, \Phi^{C}_{ii}$ and $\Phi^{R}_{ii}$.  Combined with Lemma~\ref{ts.12}, this means that in this sequence, the blow-ups after the one of $\Phi_{i,T}$ commute when they are associated to distinct $o,o'\in \{L,C,R\}$.
Thus, without loss of generality, we can assume that $o=L$.  Looking first at the sequence of blow-ups for $i=1$, set
\begin{equation}
   M^3_{\QFB,1}:= [M^3_{\rp}; \Phi_{1,T}, \{\Phi^L_{1j}\}_{j=0}^{\ell},\{\Phi^C_{1j}\}_{j=0}^{\ell},\{\Phi^R_{1j}\}_{j=0}^{\ell}].
\label{ts.24}\end{equation}
At this point, we would like to apply Lemma~\ref{bf.3}, but since $\Phi^{L}_{ij}\cap \Phi_{i,T}\ne \emptyset$ for $j>i$, their blow-ups does not seem to commute, so we need to proceed slightly differently than in \cite{Mazzeo-MelrosePhi}, where such a problem do not arise.  More precisely, we will check that the proof of Lemma~\ref{bf.3}, that is the proof of \cite[Lemma~2.5]{hmm}, can still be applied even if we first blow up $\Phi_{i,T}$.  To see this, notice that in our setting, the starting point of the proof of \cite[Lemma~2.5]{hmm} is to notice that there are local coordinates $(x'_1,\ldots, x'_{k'}, y_1',\ldots, y_{n'-k'}')$ on $M^2_{\rp}$ with $\Phi_1$ corresponding to the $p$-submanifold 
$$
     x'_1=0,  \quad y'_{1}=\cdots=y'_{m}=0, 
$$
and coordinates $(x_1,\ldots x_k, y_1,\ldots, y_{n-k})$ on $M^{3}_{\rp}$ such that
$$
    (\pi^{\rp}_{L})^*x'_1= \prod_{j=1}^{q} x_j,  \quad (\pi^{\rp}_{L})^*y'_j= y_j.
$$
In these coordinate charts, this gives $q$ $p$-submanifolds to blow up on $M^3_{\rp}$ described locally by
$$
    S_j=\{  x_j=y_1=\cdots=y_m=0\},  \quad j\in\{1,\ldots,q\}.
$$
We can assume that $\Phi^L_{ii}$ corresponds to $S_1$, in which case we can assume $\Phi_{i,T}$ corresponds to 
$$
    S_0= \{x_1=y_1=\cdots=y_m=\cdots= y_{m+1+\dim S_i}=0\},
$$
since $\Phi_{i,T}$ is of codimension $1+\dim S_i$ in $\Phi^{L}_{ii}$.  The proof of \cite[Lemma~2.5]{hmm} would then consists in lifting the coordinates $(x_1,\ldots,x_k,y_1,\ldots,y_{n-k})$ after each of the subsequent blow-ups of $S_1,\ldots, S_q$ and checking that the corresponding map $\pi^{\rp}_{L}$ lifts to a $b$-fibration.  Now, blowing-up first $S_0$ obviously adds a layer of complication to the proof, but the point is that it does not compromise the final result, namely, if we successively blow up $S_0, S_1,\ldots, S_{q}$, then the map $\pi^{\rp}_{L}$ still lifts to a $b$-fibration.  Indeed, the blow-up of $S_0$ amounts to add a superfluous step in the inductive construction of the coordinates $R_k, x_{j}^{(k)}, y_j^{(k)}$ of the proof of \cite[Lemma~2.5]{hmm} by adding a new set of such coordinates at the beginning of the inductive construction.  In greater details,
we first introduce the functions
$$
\begin{gathered}
 R_0 = x_1+ \left( y_1^2+\cdots+ y_{m+1+\dim S_i}^2 \right)^{\frac12},  \\
  x_1^{(0)}=\frac{x_1}{R_0},  \quad x^{(0)}_j=x_j,  \; j>1, \\
 y^{(0)}_j=  \left\{ \begin{array}{ll}  \frac{y_j}{R_0}, & j\le m+1+\dim S_i, \\
                      y_j, & \mbox{otherwise}.   \end{array} \right.   
\end{gathered}  
$$ 
which can be used to construct local coordinates on $[M^3_{\rp}; \Phi_{i,T}]$.  For $\nu\in\{1,\ldots,q\}$, one can then recursively construct the functions
$$
\begin{gathered}
 R_{\nu} = x_{\nu}^{(\nu-1)}+ \left( (y^{(\nu-1)}_1)^2+\cdots+ (y^{(\nu-1)}_{m})^2 \right)^{\frac12},  \\
  x_\nu^{(\nu)}=\frac{x^{(\nu-1)}_\nu}{R_\nu},  \quad x^{(\nu)}_j=x^{(\nu-1)}_j,  \; j\ne \nu, \\
 y^{(\nu)}_j=  \left\{ \begin{array}{ll}  \frac{y^{(\nu-1)}_j}{R_\nu}, & j\le m, \\
                      y^{(\nu-1)}_j, & \mbox{otherwise},   \end{array} \right.   
\end{gathered}  
$$ 
which can be used to obtain local coordinates on $[M^3_{\rp};S_0,\ldots,S_{\nu}]$.  On the other hand, on $[M^2_{\rp};\Phi_{i}]$, we can consider the functions
$$
\begin{gathered}
R'=x_1'+ ((y'_1)^2+\cdots (y'_m)^2)^{\frac12}, \\
x_1''= \frac{x_1'}{R'}, \quad x_j''=x_j', \quad j>1,
y_j''= \left\{ \begin{array}{ll}  \frac{y'_j}{R'}, & j\le m, \\
                      y'_j, & \mbox{otherwise},   \end{array} \right. 
\end{gathered}
$$
which can be used to obtain local coordinates.  In terms of these functions, 
$$
\begin{gathered}
(\pi^{\rp}_L)^*R'= aR_0R_1\ldots R_q, \\
(\pi^{\rp}_L)^*x_1''= a^{-1}\prod_{j=1}^q x_{j}^{(q)}, \\
(\pi^{\rp}_L)^*y_j''=a^{-1}y^{(q)}_j,
\end{gathered}
$$
where, as in the proof of \cite[Lemma~2.5]{hmm}, 
$$
a= \prod_{j=1}^q x_j^{(q)}+ \sqrt{(y^{(q)}_1)^2+\cdots (y^{(q)}_m)^2}
$$ 
is a smooth positive function on $[M^3_{\rp};S_0,\ldots,S_q]$.  This shows that $\pi^{\rp}_L$ lifts to a $b$-fibration
\begin{equation}
         [M^3_{\rp}; \Phi_{1,T},\{\Phi^L_{1j}\}^{\ell}_{j=0}]\to [M^2_{\rp};\Phi_1]
\label{ts.25b}\end{equation}
as claimed.
Applying Lemma~\ref{bf.1}, this further lifts to a $b$-fibration
$$
        M^3_{\QFB,1}\to [M^2_{\rp};\Phi_1].
$$
Defining recursively
$$
M^3_{\QFB,i}:= [M^3_{\QFB,i-1}; \Phi_{i,T}, \{\Phi^L_{ij}\}_{j=0}^{\ell},\{\Phi^C_{ij}\}_{j=0}^{\ell}, \{\Phi^R_{ij}\}_{j=0}^{\ell}],
$$
we can then show by induction on $i$ using the same argument for each $i$ that  $\pi^{\rp}_{L}$ lifts to a $b$-fibration 
$$
           M^3_{\QFB,i}\to [M^2_{\rp};\Phi_1,\Phi_2,\ldots,\Phi_i],
$$
which gives the result.  

\end{proof}

As in \cite{Mazzeo-MelrosePhi}, the $b$-fibrations $\pi^3_{\QFB,o}$ and $\pi^3_{\Qb,o}$ behave well with respect to the lifted diagonals.  More precisely, for $o\in\{L,C,R\}$, let $\diag_{\QFB,o}$ (respectively $\diag_{\Qb,o})$ be the lift of $\pi_o^{-1}(\diag_M)$ to $M^3_{\QFB}$ (respectively $M^3_{\Qb}$).  These are clearly $p$-submanifolds.  For $o\ne o'$, the intersection $\diag_{\QFB,o}\cap \diag_{\QFB,o'}$ (respectively $\diag_{\Qb,o}\cap\diag_{\Qb,o'}$) is the $p$-submanifold $\diag_{\QFB,T}$ (respectively $\diag_{\Qb,T}$) given by the lift of the triple diagonal
$$
      \diag^3_M:= \{(m,m,m)\in M^3\; | \; m\in M\} \subset M^3
$$  
to $M^3_{\QFB}$ (respectively $M^3_{\Qb}$).  

\begin{lemma}
For $o\ne o'$,  $\pi^{3}_{\QFB,o}$ is transversal to $\diag_{\QFB,o'}$ and induces a diffeomorphism $\diag_{\QFB,o'}\cong M^2_{\QFB}$ sending $\diag_{\QFB,T}\subset \diag_{\QFB,o'}$ onto $\diag_{\QFB}\subset M^2_{\QFB}$.  Furthermore, an analogous statement holds for the $Qb$ triple space.  
\label{ts.26}\end{lemma}  
\begin{proof}
By Lemma~\ref{ts.8b}, we know that the corresponding statement for $M^3_{\rp}$ holds.  By Proposition~\ref{ts.7}, we can assume that $o=L$ and $o'=R$.  Performing the blow-ups to construct $M^3_{\QFB}$ out of $M^3_{\rp}$ in the order specified by \eqref{ts.23b} with $i$ increasing, we can check at each step that the transversality is preserved, so that $\pi^{3}_{\QFB,L}$ is transversal to $\diag_{\QFB,R}$ and induces the claimed diffeomorphism
$$
         \diag_{\QFB,R}\cong M^2_{\QFB}.  
$$
Not performing the blow-ups \eqref{ts.23b} for $H_i$ maximal, we obtain  the corresponding result for the $\Qb$ triple space.  
\end{proof}

\section{Composition of $\QFB$ operators and $\Qb$ operators} \label{com.0}

Using the triple space of \S~\ref{tris.0}, we can now obtain composition results for $\QFB$ pseudodifferential operators using standard technics combined with small adjustments due to the fact that our operators are in general only weakly conormal in the sense of Definition~\ref{ext.8}. To this end,  recall first that on $M^2_{\QFB}$, $H_{ij}$ denotes the lift of $H^{\rp}_{ij}$ for $i,j\in \{0,1,\ldots, \ell\}$ while $\ff_i$ corresponds to the boundary hypersurface introduced by the blow-up of $\Phi_i\subset M^2_{\rp}$.  Similarly, on $M^3_{\QFB}$, for $i,j,k\in \{0,1,\ldots,\ell\}$ and with the convention that $H_0:=M$, let us denote by $H_{ijk}$ the lift of $H_i\times H_j\times H_k$ to $M^3_{\QFB}$.   Let us denote by $\ff_{i,T}$ and $\ff^o_{ij}$   the boundary hypersurfaces of $M^3_{\QFB}$ created by the blow-ups of $\Phi_{i,T}$ and $\Phi^o_{ij}$ for $o\in\{L,C,R\}$, $i\in\{1,\ldots,\ell\}$ and $j\in\{0,1,\ldots,\ell\}$.  Using this notation, we can then describe how the boundary hypersurfaces behave with respect to the three maps of Proposition~\ref{ts.23}.  

For the middle map $\pi^3_{\QFB,C}$, one can check that $H_{0i0}$ for $i\in\{1,\ldots,\ell\}$ are the only boundary hypersurfaces sent surjectively onto $M^{2}_{\QFB}$, while for the other boundary hypersurfaces, they are sent to the boundary as follows for $i,j\in\{1,\ldots,\ell\}$,
\begin{equation}
\begin{aligned}
(\pi^{3}_{\QFB,C})^{-1}(H_{i0})&= \ff^L_{i0}\cup \bigcup_{k=0}^{\ell}H_{ik0}, \quad 
(\pi^{3}_{\QFB,C})^{-1}(H_{0i})= \ff^R_{i0} \cup \bigcup_{k=0}^{\ell} H_{0ki}, \\ 
(\pi^{3}_{\QFB,C})^{-1}(H_{ij})&= \ff^L_{ij}\cup \ff^R_{ji}\cup \bigcup_{k=0}^{\ell}H_{ikj}, \quad     
(\pi^{3}_{\QFB,C})^{-1}(\ff_i)= \ff_{i,T} \cup \bigcup_{k=0}^{\ell} \ff^C_{ik}.
\end{aligned}
\label{co.1}\end{equation} 
Similarly, for the map $\pi^{3}_{\QFB,L}: M^3_{\QFB}\to M^2_{\QFB}$, the only boundary hypersurfaces sent onto $M^{2}_{\QFB}$ are given by  $H_{00i}$ for $i\in\{1,\ldots,\ell\}$.  For the other boundary hypersurfaces, they are sent to the boundary of $M^2_{\QFB}$ as follows for $i,j\in\{1,\ldots,\ell\}$,
\begin{equation}
\begin{aligned}
(\pi^{3}_{\QFB,L})^{-1}(H_{i0})&= \ff^C_{i0}  \cup \bigcup_{k=0}^{\ell} H_{i0k}, \quad 
(\pi^{3}_{\QFB,L})^{-1}(H_{0i})=  \ff^R_{i0}  \cup  \bigcup_{k=0}^{\ell} H_{0ik}, \\
(\pi^{3}_{\QFB,L})^{-1}(H_{ij})&= \ff^C_{ij} \cup \ff^R_{ji} \cup \bigcup_{k=0}^{\ell} H_{ijk}, \quad
(\pi^{3}_{\QFB,L})^{-1}(\ff_i)= \ff_{i,T}\cup \bigcup_{k=0}^{\ell} \ff^{L}_{ik}. 
\end{aligned}
\label{co.2}\end{equation} 
For the $b$-fibration  $\pi^{3}_{\QFB,R}$, the only boundary hypersurfaces sent onto $M^2_{\QFB}$ are given by $H_{i00}$ for each $i$.  For the other boundary hypersurfaces, they are sent to the boundary of $M^2_{\QFB}$ as follows for $i,j\in\{1,\ldots,\ell\}$,
\begin{equation}
\begin{aligned}
(\pi^{3}_{\QFB,R})^{-1}(H_{i0})&=\ff^L_{i0}\cup \bigcup_{k=0}^{\ell} H_{ki0}, \quad
(\pi^{3}_{\QFB,R})^{-1}(H_{0i})= \ff^C_{i0}\cup \bigcup_{k=0}^{\ell}   H_{k0i}, \\
(\pi^{3}_{\QFB,R})^{-1}(H_{ij})&=\ff^L_{ij}\cup \ff^C_{ji}\cup \bigcup_{k=0}^{\ell} H_{kij}, \quad
(\pi^{3}_{\QFB,R})^{-1}(\ff_i)= \ff_{i,T}\cup \bigcup_{k=0}^{\ell} \ff^R_{ik}. 
\end{aligned}
\label{co.3}\end{equation} 

Now, using Lemma~\ref{pdo.12a}, one computes that 
\begin{equation}
   (\beta^{3}_{\QFB})^*({}^{b}\Omega(M^3))=  \prod_{i=1}^{\ell}  \left( \rho^2_{\ff_{i,T}} \prod_{j=0}^{\ell}(\rho_{\ff^L_{ij}}\rho_{\ff^{C}_{ij}}\rho_{\ff^R_{ij}} ) \right)^{h_i}  {}^{b}\Omega(M^3_{\QFB}),
\label{co.4}\end{equation}
where $\beta^{3}_{\QFB}: M^3_{\QFB}\to M^3$ is the natural blow-down map, $h_i= 1+\dim S_i$ and $\rho_{H}$ denotes a boundary defining function for the boundary hypersurface $H$ of $M^3_{\QFB}$.   On the other hand, on the double space $M^2_{\QFB}$, recall that 
$$
   \pi_R^* x_i= \rho_{\ff_i}\prod_{j=0}^{\ell} \rho_{ji},
$$ 
where $\rho_{ji}$ denotes a boundary defining function for $H_{ji}$ in $M^2_{\QFB}$.  
Pulling it back to the $\QFB$ triple space via $\pi^3_{\QFB,L}$ and $\pi^{3}_{\QFB,R}$, we obtain from \eqref{co.2} and \eqref{co.3} that 
\begin{equation}
\begin{aligned}
(\pi^{3}_{\QFB,L})^*\pi_R^*x_i &= a_i \rho_{\ff_{i,T}} \prod_{j=0}^{\ell} \left( \rho_{\ff^L_{ij}} \rho_{\ff^R_{ij} } \prod_{k=0}^{\ell}\left(\rho_{jik} \right)  \right)\left( \prod_{j=1}^{\ell}\rho_{\ff^C_{ji}} \right), \\
(\pi^{3}_{\QFB,R})^*\pi_R^*x_i&= b_i \rho_{\ff_{i,T}} \prod_{j=0}^{\ell} \left( \rho_{\ff^R_{ij}} \rho_{\ff^C_{ij} } \prod_{k=0}^{\ell}\left(\rho_{kji} \right)  \right)\left( \prod_{j=1}^{\ell}\rho_{\ff^L_{ji}} \right),
\end{aligned}
\label{co.5}\end{equation}
where $a_i$ and $b_i$ are smooth positive functions and  $\rho_{ijk}$ denotes a boundary defining function for $H_{ijk}$.  Hence, in terms of \eqref{pdo.12}, we see that
\begin{equation}
(\pi^{3}_{\QFB,L})^*(\pi_L^*{}^{b}\Omega(M)\cdot \pi_R^* {}^{\QFB}\Omega(M))   \cdot(\pi^{3}_{\QFB,R})^* (\pi_R^* {}^{\QFB}\Omega(M))= (\rho^{\mathfrak{a}})\ {}^{b}\Omega(M^3_{\QFB})
\label{co.7}\end{equation}
with $\mathfrak{a}$ the multiweight such that 
\begin{equation}
\rho^{\mathfrak{a}}= \prod_{i=1}^{\ell} \left(  \left( \prod_{j=1}^{\ell} \rho_{\ff^C_{ji}}\rho_{\ff^L_{ji}} \right) \left( \prod_{j=0}^\ell \rho_{\ff^R_{ij}} \prod_{k=0}^{\ell} \rho_{jik}\rho_{kji}\right) \right)^{-h_i}.
\label{co.8}\end{equation}

We are now in a position to state and prove the following composition result.  

\begin{theorem}
Let $E,F$ and $G$ be vector bundles over $M$.  Suppose that $\cE$ and $\cF$ are index families and $\mathfrak{s}$ and $\mathfrak{t}$ are multiweights  such that for each $i$, 
\begin{equation}
   \min\{\mathfrak{s}_{0i},\min\Re\cE_{0i}\}+\min\{\mathfrak{t}_{i0},\min\Re \cF_{i0}\}> h_i=1+\dim S_i.
\label{co.9c}\end{equation}
Then given $A\in\Psi^{m,\cE/\mathfrak{s}}_{\QFB}(M;F;G)$ and $B\in \Psi^{m',\cF/\mathfrak{t}}_{\QFB}(M;E,F)$, their composition is well-defined with 
$$
  A\circ B\in \Psi^{m+m',\cK/\mathfrak{k}}_{\QFB}(M;E,G),
$$
where, using the convention that $h_0=0$, the  index family $\cK$ is for $i,j\in\{1,\ldots,\ell\}$ given by 
\begin{equation}
\begin{aligned}
\cK_{i0}&= \cE_{i0}\overline{\cup}(\cE_{\ff_{i}}+\cF_{i0}) \overline{\cup} \overline{\bigcup_{k\ge1}} (\cE_{ik}+\cF_{k0}-h_k), \\
\cK_{0i}&= \cF_{0i}\overline{\cup} (\cE_{0i}+ \cF_{\ff_i})\overline{\cup} \overline{\bigcup_{k\ge 1}} (\cE_{0k}+\cF_{ki}-h_k), \\
\cK_{ij}&= (\cE_{\ff_i}+\cF_{ij})\overline{\cup} (\cE_{ij}+ \cF_{\ff_j})\overline{\cup} \overline{\bigcup_{k\ge 0}} (\cE_{ik}+\cF_{kj}-h_k), \\
\cK_{\ff_i}&= (\cE_{\ff_i}+\cF_{\ff_i})\overline{\cup} \overline{\bigcup_{k\ge 0}} (\cE_{ik}+\cF_{ki}-h_k), 
\end{aligned}
\label{co.9a}\end{equation}
and the multiweight $\mathfrak{k}$ is given by 
\begin{equation}
\begin{aligned}
\mathfrak{k}_{i0}&= \min\{\mathfrak{s}_{i0}, (\mathfrak{s}_{\ff_i}\dot{+}\mathfrak{t}_{i0}), \min_{k\ge 1}\{(\mathfrak{s}_{ik}\dot{+}\mathfrak{t}_{k0}-h_k)  \} \},
 \\
\mathfrak{k}_{0i}&= \min\{ \mathfrak{t}_{0i}, (\mathfrak{s}_{0i}\dot{+}\mathfrak{t}_{\ff_i}),  \min_{k\ge 1}\{
  (\mathfrak{s}_{0k}\dot{+}\mathfrak{t}_{ki}-h_k)\}\}, \\
\mathfrak{k}_{ij}&= \min\{(\mathfrak{s}_{\ff_i}\dot{+}\mathfrak{t}_{ij}), (\mathfrak{s}_{ij}\dot{+}\mathfrak{t}_{\ff_j}), \min_{k\ge 0}\{\mathfrak{s}_{ik}\dot{+}\mathfrak{t}_{kj}-h_k\}\},  \\
\mathfrak{k}_{\ff_i}&= \min\{(\mathfrak{s}_{\ff_i}\dot{+}\mathfrak{t}_{\ff_i}),\min_{k\ge 0} \{ (\mathfrak{s}_{ik}\dot{+}\mathfrak{t}_{ki}-h_k) \} \}.
\end{aligned}
\label{co.9b}\end{equation}
If instead $\cE=\cF=\emptyset$ except possibly at $\ff_i$ for $H_i$ a boundary hypersurface, where it could be $\bbN_0$ instead of $\emptyset$, then with  $A\in \Psi^{m,\cE/\mathfrak{s}}_{\QFB,\cn}(M;F,G)$ and $B\in \Psi^{m',\cF/\mathfrak{t}}_{\QFB,\cn}(M;E,F)$ only weakly conormal $\QFB$ pseudodifferential operators, we have that
$$
       A\circ B\in \Psi^{m+m',\cK/\mathfrak{k}}_{\QFB,\cn}(M;E,G)
$$
with indicial family $\cK$ and multiweight $\mathfrak{k}$ still given by\eqref{co.9a} and \eqref{co.9b}.
\label{co.9}\end{theorem}
\begin{proof}
Suppose first that $A\in\Psi^{m,\cE/\mathfrak{s}}_{\QFB}(M;F;G)$ and $B\in \Psi^{m',\cF/\mathfrak{t}}_{\QFB}(M;E,F)$.  For operators of order $-\infty$, it suffices to apply the pushforward theorem of \cite[Theorem~5]{Melrose1992}.  When the operators are of order $m$ and $m'$, we need to combine the pushforward theorem with Lemma~\ref{ts.26} to see that the composed operators have the claimed order, \cf \cite[Proposition~B7.20]{EMM}.  

If instead $A\in \Psi^{m,\cE/\mathfrak{s}}_{\QFB,\cn}(M;F,G)$ and $B\in \Psi^{m',\cF/\mathfrak{t}}_{\QFB,\cn}(M;E,F)$ are only weakly conormal $\QFB$ pseudodifferential operators with the above restrictions on $\cE$ and $\cF$, then we cannot use the pushforward theorem of \cite[Theorem~5]{Melrose1992} since $b$-vector fields do not naturally act on the Schwartz kernel of such operators.  Still, we can adapt the proof of \cite[Theorem~4]{Melrose1992} to this weakly conormal setting.  Indeed, by Remark~\ref{st.2}, Lemma~\ref{ds.2} and the fact that the lift from the left or from the right of $\QFB$ vector fields gives $b$-vector fields on $M^2_{\QFB}$, notice that the stability of the composition under the action of $\QFB$ vector fields from the left or from the right is automatic.  Thus, if $m=m'=-\infty$, it suffices as in the proof of \cite[Theorem~4]{Melrose1992} to apply Fubini's theorem locally to obtain the result.  The expansion at $\ff_i$ can be recovered inductively using the fact that $\ff_{i,T}$ is a triple space for $\ff_i$.  Notice in particular that with our assumptions on the index families, the setting is simpler in that only $\ff_{i,T}$ in the triple space contributes to the expansion at $\ff_i$.   Otherwise, from the definition of weakly conormal $\QFB$ pseudodifferential operators and the conormal case, we can assume at least that either $m=-\infty$, or else $m'=-\infty$, in which case by Lemma~\ref{ts.26} the conormal singularity along the diagonal is simply integrated out so that the  result follows as before.  
\end{proof}

\begin{corollary}
If $E$, $F$ and $G$ are vector bundles over $M$, then
$$
    \Psi^{m}_{\QFB}(M;F,G)\circ \Psi^{m'}_{\QFB}(M;E,F) \subset \Psi^{m+m'}_{\QFB}(M;E,G).
$$
\label{co.10}\end{corollary}
\begin{proof}
It suffices to apply Theorem~\ref{co.9} with $\mathfrak{s}=\mathfrak{t}=\infty$ and $\cE=\cF$ the index family which is $\bbN_0$ at the front face $\ff_i$ for each $i$ and the empty set at all other boundary hypersurfaces.  
\end{proof}

We can apply the same strategy to analyze the composition of $\Qb$ pseudodifferential operators.  The same notation can be used to describe the boundary faces of $M^2_{\Qb}$ and $M^3_{\Qb}$, with the difference though that  when $H_i$ is maximal, $M^2_{\Qb}$ does not have the boundary hypersurface $\ff_i$ and $M^3_{\Qb}$ does not have the boundary hypersurfaces $\ff_{i,T}$ and  $\ff^o_{ij}$ for $o\in\{L,C,R\}$ and $j\in\{0,\ldots,\ell\}$.  For $H_i$ maximal, we should also take $h_i:=0$ instead of $1+\dim S_i$.  With this in mind, we have the following composition result.

\begin{theorem}
Let $E,F$ and $G$ be vector bundles over $M$.  Suppose that on $M^2_{\Qb}$, $\cE$ and $\cF$ are index families and $\mathfrak{s}$ and $\mathfrak{t}$ are multiweights  such that for each $i$, 
\begin{equation}
  \min\{ \mathfrak{s}_{0i}, \min\Re\cE_{0i}\}+\min\{\mathfrak{t}_{i0},\min\Re \cF_{i0}\}> h_i   \quad \mbox{with} \quad h_i:=\left\{\begin{array}{ll}  0, & H_i\; \mbox{ maximal}, \\
     1+\dim S_i, & \mbox{otherwise}. \end{array} \right.
\label{co.13c}\end{equation}
Then given $A\in\Psi^{m,\cE/\mathfrak{s}}_{\Qb}(M;F;G)$ and $B\in \Psi^{m',\cF/\mathfrak{t}}_{\Qb}(M;E,F)$, their composition is well-defined with 
$$
  A\circ B\in \Psi^{m+m',\cK/\mathfrak{k}}_{\Qb}(M;E,G),
$$
where the index family $\cK$, with the convention that $\cE_{\ff_i}=\cF_{\ff_i}=\emptyset$ when $H_i$ is a  maximal hypersurface, is  given by \begin{equation}
\begin{aligned}
\cK_{i0}&= \cE_{i0}\overline{\cup}(\cE_{\ff_{i}}+\cF_{i0}) \overline{\cup} \overline{\bigcup_{k\ge1}} (\cE_{ik}+\cF_{k0}-h_k), \\
\cK_{0i}&= \cF_{0i}\overline{\cup} (\cE_{0i}+ \cF_{\ff_i})\overline{\cup} \overline{\bigcup_{k\ge 1}} (\cE_{0k}+\cF_{ki}-h_k), \\
\cK_{ij}&= (\cE_{\ff_i}+\cF_{ij})\overline{\cup} (\cE_{ij}+ \cF_{\ff_j})\overline{\cup} \overline{\bigcup_{k\ge 0}} (\cE_{ik}+\cF_{kj}-h_k), \\
\cK_{\ff_i}&= (\cE_{\ff_i}+\cF_{\ff_i})\overline{\cup} \overline{\bigcup_{k\ge 0}} (\cE_{ik}+\cF_{ki}-h_k) \quad (\mbox{for $H_i$ not maximal}),
\end{aligned}
\label{co.13a}\end{equation}
and where the multiweight $\mathfrak{k}$, with the convention that $\mathfrak{s}_{\ff_i}=\mathfrak{t}_{\ff_i}=\infty$ when $H_i$ is maximal, is given by 
\begin{equation}
\begin{aligned}
\mathfrak{k}_{i0}&= \min\{\mathfrak{s}_{i0}, (\mathfrak{s}_{\ff_i}\dot{+}\mathfrak{t}_{i0}), \min_{k\ge 1}\{(\mathfrak{s}_{ik}\dot{+}\mathfrak{t}_{k0}-h_k)  \} \},
 \\
\mathfrak{k}_{0i}&= \min\{ \mathfrak{t}_{0i}, (\mathfrak{s}_{0i}\dot{+}\mathfrak{t}_{\ff_i}),  \min_{k\ge 1}\{
  (\mathfrak{s}_{0k}\dot{+}\mathfrak{t}_{ki}-h_k)\}\}, \\
\mathfrak{k}_{ij}&= \min\{(\mathfrak{s}_{\ff_i}\dot{+}\mathfrak{t}_{ij}), (\mathfrak{s}_{ij}\dot{+}\mathfrak{t}_{\ff_j}), \min_{k\ge 0}\{\mathfrak{s}_{ik}\dot{+}\mathfrak{t}_{kj}-h_k\}\},  \\
\mathfrak{k}_{\ff_i}&= \min\{(\mathfrak{s}_{\ff_i}\dot{+}\mathfrak{t}_{\ff_i}),\min_{k\ge 0} \{ (\mathfrak{s}_{ik}\dot{+}\mathfrak{t}_{ki}-h_k) \} \}  \quad (\mbox{for $H_i$ not maximal}).
\end{aligned}
\label{co.13b}\end{equation}
If instead $\cE=\cF=\emptyset$ except possibly at boundary hypersurfaces intersecting the lifted diagonal, where they could be given by $\bbN_0$ instead of $\emptyset$, then  for $A\in\Psi^{m,\cE/\mathfrak{s}}_{\Qb,\cn}(M;F;G)$ and $B\in \Psi^{m',\cF/\mathfrak{t}}_{\Qb,\cn}(M;E,F)$ only weakly conormal $\Qb$ pseudodifferential operators, we have that
$$
  A\circ B\in \Psi^{m+m',\cK/\mathfrak{k}}_{\Qb,\cn}(M;E,G)
$$
with $\cK$ and $\mathfrak{k}$ still given by \eqref{co.13a} and \eqref{co.13b}. 
\label{co.13}\end{theorem}  
\begin{proof}
As for Theorem~\ref{co.9}, this follows from the pushforward theorem and Lemma~\ref{ts.26}, but with the following three differences.  Firstly, for $\Qb$ operators, the preimages \eqref{co.1}, \eqref{co.2} and \eqref{co.3} need to be changed when $H_i$ is maximal to reflect the fact that the boundary hypersurfaces $\ff_i$ on $M^2_{\Qb}$  and $\ff_{i,T}, \ff^o_{ij}$ on $M^3_{\Qb}$ do not exist.  Secondly, the analog of \eqref{co.4} only involves a product over $i$ for $H_i$ not maximal.  Finally, the operators are defined in terms of $\Qb$ densities, and for the analog of \eqref{co.7} for $\Qb$ densities, the product \eqref{co.8} must be replaced by a product on $i$ for $H_i$ not maximal.  
\end{proof}

For the small $\Qb$ calculus, this yields the following composition result.
\begin{corollary}
If $E,F$ and $G$ are vector bundles on $M$,then 
$$
    \Psi^{m}_{\Qb}(M;F,G)\circ \Psi^{m'}_{\Qb}(M;E,F)\subset \Psi^{m+m'}_{\Qb}(M;E,G).
$$
\label{co.14}\end{corollary}

For $\cV=\QFB$ or $\Qb$, notice that $\cV$ $\mathfrak{r}$ residual operators have the semi-ideal property with respect to bounded operators.  This can be stated as follows.
\begin{proposition}
Suppose that $\cV=\QFB$ or $\Qb$.  For $q\in\{1,2\}$, Let $\mathfrak{r}_q$ be a multiweight on $M^2$ such that $\beta_{\cV}^{\#}\mathfrak{r}_q$ is $\cV$ positive.  If $K_q\in \Psi^{-\infty,\mathfrak{r}_q}_{\cV,\res}(M;E_q,F_q)$ for $q\in\{1,2\}$   and $P: L^2_b(M;F_1)\to L^2_b(M;E_2)$ is a bounded operator, then 
$$
     K_2\circ P\circ K_1\in \Psi^{-\infty,\mathfrak{r}}_{\cV,\res}(M;E_1,F_2),
$$
where $\mathfrak{r}$ is the multiweight given by
$$
    \mathfrak{r}(H_i\times M)=\mathfrak{r}_2(H_i\times M) \quad \mbox{and} \quad \mathfrak{r}(M\times H_i)=\mathfrak{r}_1(M\times H_i)
$$
for $H_i$ a boundary hypersurface of $M$.
\label{com.8}\end{proposition}
\begin{proof}
First, using a partition of unity, we can restrict the discussion to coordinate charts where $E_1$, $E_2$, $F_1$, and $F_2$ are trivial.  Hence we can assume that all these vector bundles are in fact trivial line bundles.  In this case, the Schwartz kernel of the operator $P\circ K_1$ can be seen as an element of 
$\sA_{\cV,-}^{\mathfrak{s}}(M; L^2_b(M))\cdot \nu_{\cV}$ with $\nu_{\cV}$ a non-vanishing $\cV$ density on $M$ and $\mathfrak{s}$ is the multiweight on $M$ given by $\mathfrak{s}(H)=\mathfrak{r}_1(M\times H)$ for $H$ a boundary hypersurface of $M$.  Since 
\begin{equation}
v^{\delta}L^2_b(M)\subset L^1_b(M) \quad \forall \delta>0
\label{com.8b}\end{equation} 
 by Hölder inequality, this means that $K_2 P K_1$ is obtained from an element  $f \in \sA^{\mathfrak{r}}_{\cV,-}(M^2; L^1_b(M;\Omega_b(M))\cdot \pr_R^*\nu_{\cV}$ by 
$$
      (K_2 P K_1)(m,m')= \left(\int_M f(m,m') \right).
$$
Thus, we see clearly that $K_2 P K_1\in \Psi^{-\infty,\mathfrak{r}}_{\cV,\res}(M)$ as claimed.  
\end{proof}

Of course, we can deduce composition results for $\cV$ $\mathfrak{r}$ residual pseudodifferential operators from Theorem~\ref{co.9} and Theorem~\ref{co.13}.  However, working directly with the triple space $M^3$, one can obtain the following simpler composition result.
\begin{proposition}
Suppose that $\cV=\QFB$ or $\Qb$ and let $\mathfrak{r}$ and $\mathfrak{s}$ be multiweight on $M^2$ such that $\beta_{\cV}^{\#}\mathfrak{r}$ and $\beta_{\cV}^{\#}(\mathfrak{s})$ are $\cV$ positive.   If $A\in\Psi^{-\infty,\mathfrak{r}}_{\cV,\res}(M;F,G)$ and $B\in\Psi^{-\infty,\mathfrak{s}}_{\cV,\res}(M;E,F)$, then
$$
  A\circ B\in \Psi^{-\infty,\mathfrak{k}}_{\cV,\res}(M;E,G)
$$
with $\mathfrak{k}$ the multiweight on $M^2$ given by
$$
     \mathfrak{k}(H\times M)=\mathfrak{r}(H\times M) \quad \mbox{and} \quad \mathfrak{k}(M\times H)=\mathfrak{s}(M\times H)
$$
for $H$ a boundary hypersurface of $M$.
\label{com.9}\end{proposition}
\begin{proof}
As in the proof of Proposition~\ref{com.8}, we can assume that $E,F$ and $G$ are trivial line bundles.  Then the Schwartz kernel of $A\circ B$ is of the form 
$$
 (K_{A\circ B})(m,m')= \int_M f(m,m')
$$
for some $f\in \cA^{\mathfrak{k}}_{\cV,-}(M^2;L^1_b(M;\Omega_b(M))\cdot \pr_R^*\nu_{\cV}$, from which the result follows.
\end{proof}

\section{Symbol maps of $\QFB$ operators} \label{sm.0}

In this section, we will introduce important symbol maps used in the construction of parametrices.  First, to define the principal symbol of $\QFB$ operators, we can use the principal symbol for conormal distributions in the sense of Hörmander \cite[Theorem~18.2.11]{Hormander3} for $Z$ a submanifold of a manifold $Y$, namely
$$
  \sigma_m: I^m(Y;Z,\Omega^{\frac12}_Y)\to S^{[M]}(N^*Z;\Omega^{\frac12}(N^*Z))
$$
with $M=m-\frac{1}2\dim Y+ \frac12 \dim Z$ and 
$$
   S^{[M]}(N^*Z)= S^{M}(N^*Z)/S^{M-1}(N^*Z),
$$
where $S^M(N^*Z)$ is the space of functions $\psi\in \CI(N^*Z)$ such that
$$
     \sup_{u,\xi} \frac{|D^{\alpha}_uD_{\xi}^{\beta}\psi|}{(1+|\xi|^2)^{\frac{M-|\beta|}2}}<\infty \quad \forall \; \alpha\in \bbN_0^{\dim Z},\beta\in \bbN_0^{\dim Y-\dim Z}
$$
in any local trivialization $N^*Z |_{\cU}\cong \cU\times \bbR_{\xi}^{\dim Y-\dim Z}$ of the conormal bundle of $Z$.  More specifically, for $\QFB$ operators, we need to take 
$Y= M^2_{\QFB}$ and  $Z=\diag_{\QFB}$.  By Corollary~\ref{ds.9b},  $N^*\diag_{\QFB}\cong {}^{\QFB}T^*M$.  Notice that the natural identification of ${}^{\QFB}T^*M$ and $T^*M$ on the interior of $M$ shows that ${}^{\QFB}T^*M$ admits a singular canonical  $1$-form with differential corresponding to a singular symplectic form.  Since ${}^{\QFB}\Omega_R = \pi_R^\ast \OmegaQFB$ is naturally isomorphic to 
${}^{\QFB}\Omega^{\frac12}_L\otimes {}^{\QFB}\Omega^{\frac12}_R$ on $\diag_{\QFB}$ and since this singular symplectic form of ${}^{\QFB}T^*M$ provides a natural trivialization of $\Omega({}^{\QFB}T^*M)$, this gives a map
\begin{equation}
\kridx{\sigma_m}{sigma}{principal symbol}: \Psi^m_{\QFB}(M;E,F)\to S^{[m]}({}^{\QFB}T^*M;\pi^*\hom(E,F)),
\label{sm.1}\end{equation}
where $\pi: {}^{\QFB}T^*M\to M$ is the canonical projection.  By construction, this map induces a short exact sequence 
\begin{equation}
\xymatrix{
0 \ar[r] & \Psi^{m-1}_{\QFB}(M;E,F)\ar[r] & \Psi^{m}_{\QFB}(M;E,F) \ar[r]^-{\sigma_m} & S^{[m]}({}^{\QFB}T^*M; \pi^*\hom(E,F))\ar[r] & 0
}
\label{sm.2}\end{equation}
such that 
\begin{equation}
  \sigma_{m+m'}(A\circ B)= \sigma_m(A) \sigma_{m'}(B) 
\label{sm.3}\end{equation} 
for $A\in \Psi^m_{\QFB}(M;F,G)$ and $B\in\Psi^{m'}_{\QFB}(M;E,F)$.  

\begin{definition}
An operator $P\in \Psi^m_{\QFB}(M;E,F)$ is \textbf{elliptic} if its principal symbol $\sigma_m(P)$ is invertible.   
\label{sm.4}\end{definition}

To capture the asymptotic behavior of $\QFB$ operators near each boundary hypersurface $H_i$ of $M$, we need as in other pseudodifferential calculi  to introduce normal operators.  For a boundary hypersurface $H_i$ of $M$,  consider the space
\begin{equation}
\Psi^m_{\ff_i}(H_i;E,F):=\{ \kappa\in I^m(\ff_i; \diag_{\ff_i}, (\beta^*_{\QFB}\Hom(E,F)\otimes \pi^*_R {}^{\QFB}\Omega)|_{\ff_i}) \; | \; \kappa \equiv 0 \; \mbox{at}  \; \pa\ff_i\setminus \ff_{\QFB,i}\},
\label{sm.5}\end{equation}
where $\ff_{\QFB,i}=  \ff_i\cap (\overline{\ff_{\QFB}\setminus\ff_i})$ and $\diag_{\ff_i}=\diag_{\QFB}\cap \ff_i$.  
Clearly, the restriction to $\ff_i$ gives a map 
\begin{equation}
    \kridx{N_i}{Normi}{normal operator at $\ff_i$}: \Psi^{m}_{\QFB}(M;E,F)\to \Psi^m_{\ff_i}(H_i;E,F)
\label{sm.6}\end{equation}
inducing the short exact sequence
\begin{equation}
\xymatrix{
   0\ar[r] & x_i\Psi^m_{\QFB}(M;E,F) \ar[r] & \Psi^{m}_{\QFB}(M;E,F) \ar[r]^-{N_i} & \Psi^m_{\ff_i}(H_i;E,F)\ar[r] & 0.
}
\label{sm.7}\end{equation}
By Theorem~\ref{co.9}, notice that there is a natural operation of composition
\begin{equation}
  \Psi^m_{\ff_i}(H_i;F,G)\circ \Psi^{m'}_{\ff_i}(H_i;E,F)\subset \Psi^{m+m'}_{\ff_i}(H_i;E,G)
\label{sm.8}\end{equation}
such that 
\begin{equation}
    N_i(A\circ B)= N_i(A)\circ N_i(B)
\label{sm.9}\end{equation}
for $A\in \Psi^m_{\QFB}(M;F,G)$ and $B\in \Psi^{m'}_{\QFB}(M;E,F)$.    To be able to use some argument by induction on the depth of $M$, it will be useful   to describe in a way similar to   \cite[\S~4]{Mazzeo-MelrosePhi} and \cite[\S~6]{DLR} the composition \eqref{sm.8} in terms of the notion of suspended operators originally introduced by Melrose in \cite{Melrose1995}. 

\begin{definition} 
Let $V$ be a finite dimensional real vector space and let $\bV$ be its radial compactification. Given $E$ and $F$ two vector bundles on $M$, the space of \textbf{$V$ suspended $\QFB$ operators of order $m$} acting from sections of $E$ to sections of $F$ is given by 
\begin{equation}
\kridx{\Psi^m_{\QFB,V}}{PsiQFBV}{$V$ suspended QFB pseudodifferential operators}(M;E,F)= \{ \kappa\in I^m(\bV\times M^2_{\QFB};\{0\}\times \diag_{\QFB},\cV)\; | \; \kappa\equiv 0 \; \mbox{at} \; (\bV\times (\pa M^2_{\QFB}\setminus \ff_{\QFB}))\cup (\pa \bV\times M^2_{\QFB})\},
\label{sm.9b}\end{equation} 
where $\cV=\pr_2^*(\beta^*_{\QFB} \Hom(E,F)\otimes \pi_R^* {}^{\QFB}\Omega)\otimes \pr_1^*{}^{\sc}\Omega$ with ${}^{\sc}\Omega$ the natural scattering density bundle on $\bV$ and with
$\pr_1:\bV\times M^2_{\QFB}\to \bV$ and $\pr_2:\bV\times M^2_{\QFB}\to M^2_{\QFB}$ the projections on the first and second factors.
\label{sm.9}\end{definition}
Let 
$$
      \begin{array}{llcl}
      a: & V^2 & \to & V \\
         & (v,v') & \mapsto & v-v'
      \end{array}
$$
be the projection on the anti-diagonal of $V^2$.  Let also  $\pr_L$ and $\pr_R$ be the projections $V^2\to V$ on the left and right factors respectively.  Then the action of an operator
$P\in \Psi^m_{\QFB,V}(M;E,F)$ onto a section $\sigma\in \dot{\cC}^{\infty}(\bV\times M;E)$ is naturally defined by
\begin{equation}
    P\sigma:= (\pr_L\times \pi_L)_*(a^*\kappa_P\cdot (\pr_R\times \pi_R)^*\sigma),
\label{sm.10}\end{equation}  
where $P$ is seen as an operator with Schwartz kernel $\kappa_P$.   For $A\in \Psi^{-\infty}_{\QFB,V}(M;F,G)$ and $B\in \Psi^{-\infty}_{\QFB,V}(M;E,F)$, this means that their composition is given by convolution in the $V$ factor and by the usual composition of $\QFB$ operators in the $M$ factor,
\begin{equation}
  A\circ B(v)= \int_V A(v-v')\circ B(v'),
\label{sm.11}\end{equation}
where the integration is with respect to $v'$ using the density of $B$.  In fact, more generally, given $P\in \Psi^m_{\QFB,V}(M;E,F)$, one can, in the sense of distributions, define its Fourier transform
\begin{equation}
   \widehat{P}(\Upsilon):= \int_V e^{-i\Upsilon\cdot v} P(v),  \quad \Upsilon\in V^*,
\label{sm.12}\end{equation}
again using the density of $P$ to `integrate' in $v$.  For each $\Upsilon\in V^*$, this gives an element $\widehat{P}(\Upsilon)\in \Psi^m_{\QFB}(M;E,F)$.  As for the usual Fourier transform, it is straightforward to check that 
\begin{equation}
    \widehat{A\circ B}(\Upsilon)= \widehat{A}(\Upsilon)\circ \widehat{B}(\Upsilon) \quad \forall \; \Upsilon\in V^*.
\label{sm.13}\end{equation}
Similarly, given $\sigma\in \dot{\cC}^{\infty}(\bV\times M;E)$, we define its Fourier transform by
$$
     \widehat{\sigma}(\Upsilon)= \int_V e^{-i\Upsilon\cdot v} \sigma(v) \nu
$$
with $\nu$ the translation invariant density on $V$ associated to some choice of inner product. In this case, we have that
\begin{equation}
         \widehat{B\sigma}(\Upsilon)= \widehat{B}(\Upsilon) \widehat{\sigma}(\Upsilon)  \quad \forall \; \Upsilon\in V^*.
\label{sm.14}\end{equation}
Since the operator $P$ can be recovered from $\widehat{P}$ by taking the inverse Fourier transform, we see that the family $\Upsilon\mapsto \widehat{P}(\Upsilon)$ completely determines $P$.  Notice however that the Fourier transform of $V$ suspended $\QFB$ operators does not give all smooth families of $\QFB$ operators on $V$.  For instance, by taking also the Fourier transform in directions conormal to $\diag_{\QFB}$, one easily sees that 
\begin{equation}
  D^{\alpha}_{\Upsilon}\widehat{P}(\Upsilon)\in \Psi^{m-|\alpha|}_{\QFB}(M;E,F), \quad \forall \; \alpha\in \bbN_0^{\dim V}, \; \forall\; \Upsilon\in V^*.
\label{sm.15}\end{equation}
For operators of order $-\infty$, the image of the Fourier transform can be completely characterized, namely a smooth family
$$
     V^*\ni \Upsilon\mapsto \widehat{P}(\Upsilon) \in \Psi^{-\infty}_{\QFB}(M;E,F)
$$
is the Fourier transform of an element in $\Psi^{-\infty}_{\QFB,V}(M;E,F)$ if and only if for any Fréchet semi-norm $\|\cdot \|$ of the space $\Psi^{-\infty}_{\QFB}(M;E,F)$, 
\begin{equation}
      \sup_{\Upsilon} \| \Upsilon^{\alpha}D^{\beta}_{\Upsilon} \widehat{P}\| <\infty \quad \forall \; \alpha,\beta\in \bbN_0^{\dim V}.
\label{sm.15b}\end{equation}
More generally, for operators of order $m\in \bbR$, the Fourier transform $\widehat{P}$ of a $V$ suspended $\QFB$ operator  $P\in \Psi^m_{\QFB,V}(M;E,F)$ must satisfy
\begin{equation}
  \sup_{\Upsilon}  \|  (1+|\Upsilon|^2)^{\frac{|\alpha|-m}2} D^{\alpha}_{\Upsilon}\widehat{P}\|<\infty, \quad \forall \; \alpha\in \bbN_0^{\dim V}
\label{sm.16}\end{equation}
for any Fréchet semi-norm of $\Psi^m_{\QFB}(M;E,F)$, though in this case those conditions are not enough to fully characterize the image of $\Psi^m_{\QFB,V}(M;E,F)$ under the Fourier transform.

As in \cite{Mazzeo-MelrosePhi} and \cite{DLR}, the notion of suspended $\QFB$ operators extends naturally to family.  More precisely, let 
\begin{equation}
\xymatrix{
     Z \ar[r] & H \ar[d]^{\nu} \\
          & S
}
\label{sm.17}\end{equation}
be a fiber bundle with $H$ and $S$ compact manifolds with corners.  Suppose that the fibers are manifolds with fibered corners and that consistent choices of compatible boundary defining functions are made, so that the space $\Psi^m_{\QFB}(H/S;E,F)$ of fiberwise $\QFB$ operators of order $m$ acting from sections of $E\to H$ to sections of $F\to H$ makes sense.  The prototypical example is of course the boundary fiber bundle  $\phi_i: H_i\to S_i$ of a manifold with fibered corners $M$.   In this setting, given $V\to S$ a real vector bundle, we can form the space 
$\Psi^m_{\QFB,V}(H/ S;E,F)$ of fiberwise  $V$ suspended $\QFB$ operators of order $m$ acting from sections of $E$ to sections of $F$.  They can be thought of as smooth families of suspended $\QFB$ operators giving for each $s\in S$ a $V_s$ suspended $\QFB$ operator in $\Psi^m_{\QFB,V_s}(\nu^{-1}(s);E,F)$, where $V_s$ is the fiber of $V$ above $s$.   

For the fiber bundle $\phi_i: H_i\to S_i$ coming from the fibered corners structure of $M$, the natural vector bundle to consider on $S_i$ is the vector bundle ${}^{\phi}NS_i$ of \eqref{vf.2}.  As spaces of conormal distributions, there is in fact a canonical identification 
\begin{equation}
   \Psi^m_{\ff_i}(H_i;E,F) \cong \Psi^m_{\QFB,{}^{\phi}NS_i}(H_i/S_i;E,F).  
\label{sm.18}\end{equation}
Indeed, on $\ff_i$, there is a natural fiber bundle 
\begin{equation}
\phi_{\ff_i}:\ff_i\to S_i
\label{sm.18a}\end{equation}
given by
$$
     \phi_{\ff_i}=\phi_i\circ \pr_R\circ\beta_{\QFB}|_{\ff_i}= \phi_i\circ\pr_L\circ \beta_{\QFB}|_{\ff_i}.
$$
For $s\in S_i$, the fiber $\phi_{\ff_i}^{-1}(s)$ is not quite 
$$
     \overline{{}^{\phi}N_sS_i}\times (\phi_i^{-1}(s))^2_{\QFB},
$$
where ${}^{\phi}N_sS_i$ is the fiber of ${}^{\phi}NS_i$ above $s\in S_i$, but a slightly more complicated compactification of 
\begin{equation}
      {}^{\phi}N_sS_i\times (\phi_i^{-1}(s))^2_{\QFB}.
\label{sm.18b}\end{equation}      
                  
More precisely, as expected, the boundary hypersurface $\ff_i\cap H_{ii}$ of $\ff_i$ corresponds  to one face added to infinity to compactify \eqref{sm.18b}, but the blow-up of $\Phi_j$ for $H_j>H_i$ creates another one, namely $\ff_i\cap H_{jj}$.  Notice that this is  consistent with the fact that the isomorphism \eqref{vf.2} involves multiplication by $x_j$ (through multiplication by $v_i$).  Since we impose rapid decay at $\ff_i\cap H_{jj}$ for $H_j\ge H_i$, the identification \eqref{sm.18} can therefore be seen as being induced by the isomorphism \eqref{vf.2}.  
\begin{remark}
The identification \eqref{sm.18} should be compared with the identification \cite[(7.12)]{DLR} which holds essentially for the same reasons.  In fact, the justification in \cite{DLR} for this identification is not the correct one since for the same reason as for the $\QFB$ calculus, the front face considered in \cite{DLR} is not quite the radial compactification of a vector bundle as erroneously claimed in \cite[(7.11)]{DLR}.   
\label{DLR.1}\end{remark}

However, it is not completely clear a priori that the natural composition of \eqref{sm.8} is the same as the composition for suspended $\QFB$ operators.  
\begin{lemma}
Under the identification \eqref{sm.18}, the composition \eqref{sm.8} corresponds to the composition for families of suspended $\QFB$ operators.  
\label{sm.19}\end{lemma}
\begin{proof}
We will follow the approach of \cite[\S~4]{Mazzeo-MelrosePhi}. Fix a boundary hypersurface $H_j$ of $M$.  Given $p\in H_j$, let 
\begin{equation}
\{(x_i,y_i), i\in \{1,\ldots,k\}, z\}
\label{sm.19b}\end{equation}
 be coordinates near $p$ as in \eqref{coor.1}.  Here, we assume without loss of generality that the boundary hypersurfaces have been labelled so that $1\le j\le k$.  On $M^2$, consider the corresponding coordinates \eqref{ds.2c}.  On $M^2_{\QFB}$, instead of \eqref{ds.8}, let us use the coordinates
\begin{equation}
   S_i'= \frac{\sigma_i'-1}{v_i}, x_i, Y_i'= \frac{y_i'-y_i}{v_i}, y_i, z,z',
\label{sm.20}\end{equation}
where $v_i=\prod_{q=i}^k x_q$ and $\sigma_i'=\prod_{q=i}^k s_q'$ with $s_q'= \frac{x_q'}{x_q}$.  Let $e\in \CI(M;E)$ be a smooth section with support in the coordinate chart \eqref{sm.19b}.  Then for $P\in \Psi^m_{\QFB}(M;E,F)$, we see, using the convention that $S_{k+1}=0$ and $v_{k+1}=0$, that in these coordinates
\begin{equation}
(Pe)(x_i,y_i,z)= \int \kappa_P(x_i,y_i,z, S_i',Y_i',z') e\left( \frac{x_i(1+v_iS_i')}{1+v_{i+1}S_{i+1}'}, y_i+v_iY_i,z) \right) \prod_i(dS'_i dY'_i) dz',
\label{sm.21}\end{equation}
where $\kappa_P$, which denotes the Schwartz kernel of $P$, is conormal to $\{S_i'=0, Y_i'=0, 1\le i\le k, z=z'\}$.   If we restrict to $H_j$, that is, if we set $x_j=0$, this gives

\begin{multline}
Pe(x_1,\ldots, x_{j-1},0,x_{j+1},\ldots, x_k, y_1,\ldots, y_k,z)  =  \\
\int \kappa_P(x_1,\ldots, x_{j-1},0,x_{j+1},\ldots, x_k, y_1,\ldots, y_k,z,S_1',\ldots, S_k',Y_1',\ldots, Y_k', z') \\
\cdot e(x_1,\ldots, x_{j-1}, 0, \frac{x_{j+1}(1+v_{j+1}S_{j+1}')}{1+v_{j+2}S_{j+2}'}, \ldots, \frac{x_{k}(1+v_{k}S_{k}')}{1+v_{k+1}S_{k+1}'}, y_1,\ldots, y_{j}, y_{j+1}+v_{j+1}Y_{j+1}',\ldots, y_k+v_kY_k',z') \\
\prod_i(dS'_i dY'_i) dz'.
\label{sm.22}\end{multline}
Thus, as an operator $P:\CI(M;E)\to \CI(M;F)$, the restriction $P|_{H_i}$ in this coordinate chart has Schwartz kernel
\begin{equation}
   \left( \int \kappa_P \prod_{i=1}^j dS_i'dY_i' \right)\otimes \prod_{i=j+1}^k(dS_i'dY_i')dz'.
\label{sm.23}\end{equation}
That is, globally on $H_i$, using the identifications \eqref{vf.2} and \eqref{sm.18}, it is given by
\begin{equation}
        \int_{{}^{\phi}NS_i/S_i} N_i(P)= \widehat{N_i(P)}(0),
\label{sm.23}\end{equation}
the Fourier transform of $N_i(P)$ evaluated at the zero section.  Of course, if $P$ and $Q$ are two $\QFB$ operators, then
\begin{equation}
    (P\circ Q)|_{H_i}= P|_{H_i}\circ Q|_{H_i}\quad \Longrightarrow \quad \widehat{N_i(P\circ Q)}(0)= \widehat{N_i(P)}(0)\circ \widehat{N_i(Q)}(0).  
\label{sm.24}\end{equation}
To obtain the Fourier transform of $N_i(P)$ at other values, we can instead conjugate $P$ by the `oscillatory test function' $e^{if}$ with
\begin{equation}
 f(x_1,\ldots, x_k, y_1,\ldots, y_k,z)= \frac{f_1(y_1)}{v_1}+\cdots + \frac{f_j(y_j)}{v_j}
\label{sm.25}\end{equation}
for $f_1,\ldots, f_j$ smooth functions.  
In fact, one computes that 
\begin{equation}
\frac{f_i(y_i')}{v_i'}-\frac{f_i(y_i)}{v_i}= \frac{f_i(y_i+v_iY_i')}{v_i(1+v_iS_i')}- \frac{f_i(y_i)}{v_i}=  -f(y_i)S_i' +\sum_q \frac{\pa f_i}{\pa y_i^q}(y_i) (Y_i')^q + \mathcal{O}(v_i)
\label{sm.26}\end{equation}
is smooth all the way down to $v_i=0$ in the coordinates \eqref{sm.20}.  Thus, let us extend $\phi_j$ to a fiber bundle in a tubular neighborhood of $H_j$ in $M$. If $\chi\in \CI(M)$ is equal to $1$ near $\phi_j^{-1}(\phi_j(p))$ and is supported in $\phi_j^{-1}(\cU)$ for some open set $\cU$ on $S_j$ where the coordinates $(x_1,\ldots,x_j,y_1,\ldots, y_j)$ on $S_j\times [0,\epsilon)_{x_j}$ are valid, we see by \eqref{sm.26}  that 
\begin{multline}
(\chi e^{-if})\circ P \circ (\chi e^{if})|_{\phi_j^{-1}(\phi_j(p))}= \left(\int_{{}^{\phi}NS_j|_{\phi_j(p)}}  e^{i(\sum_{i\le j}(-f_i(y_i)S_i'+ \sum_q \frac{\pa f_i}{\pa y_i^q}(y_i)(Y_i')^q))}
\kappa_P|_{{}^{\phi}NH_i|_{\phi_j^{-1}(\phi_j(p))}} \prod_{i\le j} dS_i' dY_i'\right) \\
\otimes \left(\prod_{i>j} dS_i' dY_i'\right) dz'.
\label{sm.27}\end{multline}
Moreover, we clearly have that 
\begin{equation}
  (\chi e^{-if})\circ PQ \circ (\chi e^{if})|_{\phi_j^{-1}(\phi_j(p))}= \left((\chi e^{-if})\circ P \circ (\chi e^{if})|_{\phi_j^{-1}(\phi_j(p))}\right) \circ\left( (\chi e^{-if})\circ Q \circ (\chi e^{if})|_{\phi_j^{-1}(\phi_j(p))}\right).
  \label{sm.28}\end{equation}
Hence, since $p\in H_j$ was arbitrary, a careful choice of the smooth functions $f_1,\ldots, f_j$ together with \eqref{sm.27} and \eqref{sm.28} shows that for all $\Upsilon\in {}^{\phi}N_jS_j$,
\begin{equation}
 \widehat{N_j(P\circ Q)}(\Upsilon)= \widehat{N_j(P)}(\Upsilon)\circ\widehat{N_j(Q)}(\Upsilon).
\label{sm.29}\end{equation}
The result therefore follows by taking the inverse Fourier transform.  
\end{proof}

In particular, seen as a suspended operator, the normal operator $N_i(P)$ has a natural action
\begin{equation}
   N_i(P):  \dot{\cC}^{\infty}(\overline{{}^{\phi}NH_i};E) \to \dot{\cC}^{\infty}(\overline{{}^{\phi}NH_i};F)
 \label{sm.30}\end{equation}
induced by \eqref{sm.10}, where $\overline{{}^{\phi}NH_i}$ is the fiberwise radial compactification of the vector bundle ${}^{\phi}NH_i$.   This allows us to formulate the following definition.  
\begin{definition}
An operator $P\in \Psi^m_{\QFB}(M;E,F)$ is \textbf{fully elliptic} if it is elliptic and if 
$$
    N_i(P):  \dot{\cC}^{\infty}(\overline{{}^{\phi}NH_i};E) \to \dot{\cC}^{\infty}(\overline{{}^{\phi}NH_i};F)
  $$
  is invertible for each boundary hypersurface $H_i$ of $M$.  
\label{sm.31}\end{definition}

Suspended $\QFB$ operators have themselves a natural symbol map.  More precisely, the principal symbol  of conormal  distributions induces a principal symbol for $V$ suspended $\QFB$ operators taking the form
\begin{equation}
      \sigma_m: \Psi^m_{\QFB,V}(M;E,F)\to S^{[m]}(V^*\times {}^{\QFB}T^*M; \phi_V^*\hom(E,F)),
\label{sm.32}\end{equation}
where $\phi_V: V^*\times {}^{\QFB}T^*M\to M$ is the composition of the projection on the second factor with the bundle map ${}^{\QFB}T^*M\to M$.  Similarly, we can define the normal operator $N_i(P)$ of $P\in \Psi^m_{\QFB,V}(M;E,F)$ to be the restriction of the Schwartz kernel of $P$ at $V\times \ff_i$ in $V\times M^2_{\QFB}$.    As can be seen from the discussion above, the normal operator can be seen again as a family of suspended $\QFB$ operator associated to the fiber bundle $\phi_i:H_i\to S_i$ and the vector bundle $V\times {}^{\phi}NS_i\to S_i$, so that the notion of normal operator induces a map
\begin{equation}
    N_i: \Psi^m_{\QFB,V}(M;E,F)\to \Psi^m_{\QFB,V\times {}^{\phi}NS_i}(H_i/S_i;E,F).
\label{sm.33}\end{equation} 
In particular, the normal operator $N_i(P)$ has a natural action
\begin{equation}
    N_i(P): \dot{\cC}^{\infty}(\overline{V\times {}^{\phi}NH_i};E)\to \dot{\cC}^{\infty}(\overline{V\times {}^{\phi}NH_i};F)
\label{sm.34}\end{equation}
\begin{definition}
A $V$ suspended $\QFB$ operator $P\in \Psi^m_{\QFB,V}(M;E,F)$ is \textbf{elliptic} if its principal symbol $\sigma_m(P)$ is invertible in $S^{[m]}(V^*\times {}^{\QFB}T^*M; \phi_V^*\hom(E,F))$.  Furthermore, it is said to be \textbf{fully elliptic} if it is elliptic and if the map
$$
N_i(P): \dot{\cC}^{\infty}(\overline{V\times {}^{\phi}NH_i};E)\to \dot{\cC}^{\infty}(\overline{V\times {}^{\phi}NH_i};F)
$$
is invertible for each boundary hypersurface $H_i$ of $M$.  
\label{sm.35}\end{definition}

Fully elliptic operators admit nice parametrices.  To describe those, let us consider the space 
\begin{equation}
\dot{\Psi}^{-\infty}(M;E,F):= \{  A\in \dot{\Psi}^{-\infty}_{\QFB}(M;E,F)\; |  \; \kappa_A \; \mbox{is rapidly vanishing at} \; \ff_{\QFB}\}.
\label{mp.1}\end{equation}
As the notation suggests, the definition of this space only depends on the manifold with corners $M$, not on the structure of manifold with fibered corners or the choice of a Lie algebra of $\QFB$ vector fields.  Similarly, for a vector space $V$, we can consider the space
\begin{equation}
\dot{\Psi}^{-\infty}_V(M;E,F):= \{ A\in \Psi^{-\infty}_{\QFB,V}(M;E,F) \; | \; \widehat{A}(\Upsilon)\in \dot{\Psi}^{-\infty}(M;E,F) \; \forall \; \Upsilon\in V\}.
\label{mp.2}\end{equation}

\begin{proposition}
If $P\in \Psi^m_{\QFB}(M;E,F)$ is fully elliptic, then there exists $Q\in \Psi^{-m}_{\QFB}(M;F,E)$ such that 
$$
     QP-\Id \in \dot{\Psi}^{-\infty}(M;E),  \quad PQ-\Id\in \dot{\Psi}^{-\infty}(M;F).
$$
Moreover, there are natural inclusions $\ker P\subset \dot{\cC}^{\infty}(M;E)$ and $\ker P^*\subset \dot{\cC}^{\infty}(M;F)$, where $P^*$ is the formal adjoint of $P$ with respect to a $\QFB$ density and  choices of bundle metrics for $E$ and $F$.  In fact, if $P$ is a differential operators, then elements of $\ker P$ and $\ker P^*$ decay exponentially fast at infinity. Similarly, if $V$ is a finite dimensional real vector space and $P\in \Psi^m_{\QFB,V}(M;E,F)$ is fully elliptic, then there exists $Q\in \Psi^{-m}_{\QFB,V}(M;F,E)$ such that 
$$
   QP-\Id\in \dot{\Psi}^{-\infty}_V(M;E), \quad PQ-\Id\in \dot{\Psi}^{-\infty}_V(M;F).  
$$ 
\label{mp.3}\end{proposition}
\begin{remark}
The exponential decay of Proposition~\ref{mp.3} can be seen as a generalization of the exponential decay in many body type scattering obtained by Froese-Herbst \cite{Froese-Herbst}  and Vasy \cite{Vasy2004} with very different methods.  Indeed,  suitably interpreted, the conditions in those papers correspond to requiring that a certain $\QAC$ differential operator be fully elliptic. 
\label{nbs.1}\end{remark}
\begin{proof}
If $M$ is a closed manifold, this is a well-known result.  Hence, we can proceed by induction on the depth of $M$.  Let us warn the reader though that the inductive step will only be completed with the proof of Corollary~\ref{mp.4}.  In fact, by this corollary, we can assume that $N_i(P)^{-1}\in\Psi^{-m}_{\ff_i}(M;F,E)$.  This means that we can choose $Q_0\in \Psi^{-m}_{\QFB}(M;E,F)$ such that 
$\sigma_{-m}(Q_0)= \sigma_m(P)^{-1}$ and $N_i(Q_0)=N_i(P)^{-1}$ for all boundary hypersurface $H_i$, which implies that 
$$
     Q_0P-\Id\in v\Psi^{-1}_{\QFB}(M;E), \quad PQ_0-\Id\in v\Psi^{-1}_{\QFB}(M;F),
$$
where $v=\prod_i x_i$ is a total boundary defining function of $M$.  Proceeding by induction on $\ell\in \bbN_0$ as in \cite[Proposition~9.1]{DLR}, we can more generally find 
$Q_{\ell}\in v^{\ell}\Psi^{-m-\ell}_{\QFB}(M;F,E)$ such that $\widetilde{Q}_{\ell}= Q_0+Q_1+\cdots + Q_{\ell}$ satisfies 
$$
      \widetilde{Q}_{\ell}P-\Id\in v^{\ell+1}\psi^{-1-\ell}_{\QFB}(M;E), \quad P\widetilde{Q}_{\ell}-\Id\in v^{\ell+1}\Psi^{-1-\ell}_{\QFB}(M;F).
$$
The parametrix $Q$ can therefore be taken to be an asymptotic sum of the $Q_{\ell}$.  Now, if $e\in \ker P$, then 
$$
      Pe=0 \; \Longrightarrow\; QPe=0 \; \Longrightarrow \; e= (\Id- QP)e\in \dot{\cC}^{\infty}(M;E)
$$  
since $\Id-QP\in \dot{\Psi}^{-\infty}(M;E)$.  There is a similar argument for $\ker P^*$.  If $P$ is a differential operator, then by the property (QFB2) of $\QFB$ vector fields,  the conjugated operator $P_{\lambda}:=e^{-\frac{\lambda}v}Pe^{\frac{\lambda}v}$ for $\lambda\in\bbR$ is also a $\QFB$ operator.  Since the property of being fully elliptic is an open condition, this means that $P_{\lambda}$ is fully elliptic for $\lambda>0$ small enough, implying that
$$
         e^{\frac{\lambda}{v}}\ker P= \ker P_{\lambda}\subset  \dot{\cC}^{\infty}(M;E)  \quad \Longrightarrow \quad \ker P\subset e^{-\frac{\lambda}v} \dot{\cC}^{\infty}(M;E).
$$   
Again, there is a similar argument for $P^*$.
For the parametrix of a suspended $\QFB$ differential operators, it suffices to `suspend' the previous argument, which essentially only involves notational changes.  To complete the induction on the depth of $M$, it remains to prove the corollary below. 
\end{proof}

\begin{corollary}
If $P\in \Psi^m_{\QFB,V}(M;E,F)$ is  fully elliptic and invertible as a map 
$$
        P: \dot{\cC}^{\infty}(\overline{V}\times M;E) \to \dot{\cC}^{\infty}(\overline{V}\times M;F),
$$
then its inverse is in $\Psi^{-m}_{\QFB,V}(M;F,E)$.   
\label{mp.4}\end{corollary}
\begin{proof}
Let $Q$ be the parametrix of the previous proposition.  Taking the Fourier transform in $V$, this gives
$$
           \widehat{P}(\Upsilon)\widehat{Q}(\Upsilon)= \Id+ \widehat{R}_1(\Upsilon) \quad \forall\; \Upsilon\in V^*,   \quad \widehat{Q}(\Upsilon)\widehat{P}(\Upsilon)= \Id+ \widehat{R}_2(\Upsilon) \quad \forall\; \Upsilon\in V^*,
$$
where $R_1\in \dot{\Psi}^{-\infty}_V(M;F)$ and $R_2\in \dot{\Psi}^{-\infty}_V(M;E)$.  By \eqref{sm.15b}, we see that $\widehat{R}_i(\Upsilon)$ decays rapidly when $|\Upsilon|\to \infty$, hence there is $K>0$ such that $\Id+\widehat{R}_i(\Upsilon)$ is invertible for $|\Upsilon|>K$ with inverse $\Id+\widehat{S_i}(\Upsilon)$ for
$$
        \widehat{S}_i(\Upsilon)= \sum_{k=1}^{\infty} (-1)^k \widehat{R}_i(\Upsilon)^k\in \dot{\Psi}^{-\infty}(M;F)
$$
satisfying \eqref{sm.15b}.  Therefore, for $|\Upsilon|>K$, we have that 
$$
       \widehat{P}(\Upsilon)^{-1}= \widehat{Q}(\Upsilon)(\Id+\widehat{S}_1(\Upsilon)).
$$
On the other hand, the invertibility of $P$ clearly implies the invertibility of $\widehat{P}(\Upsilon)$ for all $\Upsilon\in V^*$.  Moreover, we have that
\begin{equation}
\begin{aligned}
\widehat{P}(\Upsilon)^{-1}&= \widehat{P}(\Upsilon)^{-1}(\widehat{P}(\Upsilon)\widehat{Q}(\Upsilon)-\widehat{R}_1(\Upsilon))= \widehat{Q}(\Upsilon)- \widehat{P}(\Upsilon)^{-1}\widehat{R}_1(\Upsilon) \\
& = \widehat{Q}(\Upsilon)- (\widehat{Q}(\Upsilon)\widehat{P}(\Upsilon)-\widehat{R}_2(\Upsilon))\widehat{P}(\Upsilon)^{-1}\widehat{R}_1(\Upsilon) \\
&= \widehat{Q}(\Upsilon)- \widehat{Q}(\Upsilon)\widehat{R}_1(\Upsilon) + \widehat{R}_2(\Upsilon) \widehat{P}(\Upsilon)^{-1}\widehat{R}_1(\Upsilon).
\end{aligned}
\label{mp.5}\end{equation}
Clearly, the last term is an operator of the form 
$$
 \widehat{R}_2(\Upsilon) \widehat{P}(\Upsilon)^{-1}\widehat{R}_1(\Upsilon): \cC^{-\infty}(M;F)\to \dot{\cC}^{\infty}(M;E),
 $$
 hence it is an element of $\dot{\Psi}^{-\infty}(M;F,E)$, see for instance \cite[Proposition~5.4]{DLR}.  Consequently, we have that
 $$
         \widehat{P}(\Upsilon)^{-1}= \widehat{Q}(\Upsilon)+ \widehat{W}(\Upsilon)
 $$
 with $W\in \Psi^{-\infty}_{\QFB,V}(M;F,E)$ such that $\widehat{W}(\Upsilon)= \widehat{Q}(\Upsilon)\widehat{S}(\Upsilon)$ for $|\Upsilon|>K$.  Taking the inverse Fourier transform then gives the result. 
\end{proof}

\section{Mapping properties of $\QFB$ operators}  \label{mp.0}

In this section, we will describe how $\QFB$ operators and $\Qb$ operators act on adapted Sobolev spaces.  As a first step, we can use Schur's test to obtain a criterion for $L^2$ boundedness of operators of order $-\infty$.  

\begin{proposition}
Let $\mathfrak{r}$ and $\mathfrak{r}'$ be the multiweights of \eqref{pdo.14} and \eqref{mwp.1}.   Then an operator 
$$
P\in \Psi^{-\infty,\cE/\mathfrak{s}}_{\QFB,\cn}(M;E,F)
$$ 
induces a bounded operator 
\begin{equation}
P: x^{\mathfrak{t}}L^2_b(M;E)\to x^{\mathfrak{t}'}L^2_b(M;F)
\label{lt.1a}\end{equation}
provided $\Re(\mathfrak{s}-\mathfrak{r}+\pi_R^{\#}\mathfrak{t}-\pi_L^{\#}\mathfrak{t}')> 0$ and $\Re(\cE-\mathfrak{r}+\pi_R^{\#}\mathfrak{t}-\pi_L^{\#}\mathfrak{t}')\ge 0$ with strict inequality except possibly at $\ff_i$ for $H_i$ a boundary hypersurface of $M$. Similarly,
an operator  $P\in \Psi^{-\infty,\cE/\mathfrak{s}}_{\Qb,\cn}(M;E,F)$ induces a bounded operator 
\begin{equation}
P: x^{\mathfrak{t}}L^2_b(M;E)\to x^{\mathfrak{t}'}L^2_b(M;F)
\label{lt.1b}\end{equation}
provided $\Re(\mathfrak{s}-\mathfrak{r}'+(\pr_R\circ\beta_{\Qb})^{\#}\mathfrak{t}-(\pr_L\circ\beta_{\Qb})^{\#}\mathfrak{t}')> 0$ and $\Re(\cE-\mathfrak{r}'+(\pr_R\circ\beta_{\Qb})^{\#}\mathfrak{t}-(\pr_L\circ\beta_{\Qb})^{\#}\mathfrak{t}')\ge 0$ with strict inequality except possibly at $H_{ii}$ for $H_i$ a maximal boundary hypersurface of $M$ and at $\ff_i$ for $H_i$ a non-maximal boundary hypersurface.  
\label{lt.1}\end{proposition}
\begin{proof}
We will give the proof for $\QFB$ operators, the proof for $\Qb$ operators being similar.  Moreover, using partitions of unity, we can reduce the proof to the case where $E$ and $F$ are trivial line bundles. First, since \eqref{lt.1a} is bounded if and only if 
$$
   x^{-\mathfrak{t}'}\circ P \circ x^{\mathfrak{t}}: L^2_b(M)\to L^2_b(M)
$$
is, replacing $\cE$ by $\cE+\pi_R^{\#}\mathfrak{t}-\pi_L^{\#}\mathfrak{t}'$ and $\mathfrak{s}$ by $\mathfrak{s}+\pi_R^{\#}\mathfrak{t}-\pi_L^{\#}\mathfrak{t}'$ , we may suppose that $\mathfrak{t}=\mathfrak{t}'=0$.    By \eqref{pdo.13}, if $\cK_P$ is the Schwartz kernel of $P$, then
$$
  \pi_L^*\nu_b\cdot \cK_P\in \cA^{(\cE-\mathfrak{r})/(\mathfrak{s}-\mathfrak{r})}_{\QFB,\phg}(M^2_{\QFB}; {}^{b}\Omega(M^2_{\QFB})),
$$
where $\nu_b$ is a choice of non-vanishing $b$ density on $M$.  To obtain the result, it suffices then to show that 
\begin{equation}
       (\pi_L)_*(\pi_L^*\nu_b\cdot \cK_P)\quad \mbox{and} \quad (\pi_R)_*(\pi_L^*\nu_b\cdot \cK_P)
\label{lt.2}\end{equation}
are bounded sections of ${}^{b}\Omega(M)$, for then the result is a direct consequence of Schur's test.  But to show that the sections of \eqref{lt.2} are bounded, it suffices to apply Fubini's theorem locally on $M^2_{\QFB}$ together with the Hölder inequality as in \eqref{com.8b} locally in the fibers of $\pi_L$ and $\pi_R$.  
\end{proof} 

\begin{corollary}
An operator $P\in \Psi^{-\infty}_{\QFB}(M;E,F)$ induces a bounded operator 
\begin{equation}
P: x^{\mathfrak{t}}L^2_b(M;E)\to x^{\mathfrak{t}}L^2_b(M;F)
\label{lt.4}\end{equation}
for all multiweights $\mathfrak{t}$.  Similarly, $P\in \Psi^{-\infty}_{\Qb}(M;E,F)$ induces a bounded operator \eqref{lt.4} for all multiweight $\mathfrak{t}$.
\label{lt.3}\end{corollary}

To upgrade these results to operators of order $0$, we can use Hörmander's trick, which relies on the construction of an approximate square root.  

\begin{lemma}
Suppose that $\cV=\QFB$ or $\cV=\Qb$.  Then given $B\in \Psi^{0}_{\cV}(M)$ which is formally self-adjoint with respect to a choice of $b$ density $\nu_b$ on $M$, there exists $C>0$ such that 
$$
C+B= A^*A+R
$$
for some $A\in \Psi^0_{\cV}(M)$ and $R\in \Psi^{\infty}_{\cV}(M)$. 
\label{lt.5}\end{lemma}
\begin{proof}
Take $C>0$ sufficiently large so that $C+\sigma_0(B)$ has a unique positive square root.  Choose $A_0\in \Psi^0_{\cV}(M)$ so that 
$$
    \sigma_0(A_0)=(C+\sigma_0(B))^{\frac12}.
$$ 
Replacing $A_0$ by $\frac12(A_0+A_0^*)$ if necessary, we can assume that $A_0$ is formally self-adjoint with respect to the $b$ density $\nu_b$.  In this case,
$$
      C+P-A_0^2\in \Psi^{-1}_{\cV}(M),
$$
so that $A$ is a square root up to a term in $\Psi^{-1}_{\cV}(M)$.  To improve this error term, we can proceed by induction and assume that we have found a formally self-adjoint operator $A_{\ell}\in \Psi^0_{\cV}(M)$ such that 
$$
      R_{\ell}:=C+P-A_{\ell}^2\in \Psi^{-\ell-1}_{\cV}(M).
$$
We can then try to find a better square root approximation $A_{\ell+1}=A_{\ell}+Q_{\ell}$ with $Q_{\ell}\in\Psi^{-\ell-1}_{\cV}(M)$ a formally self-adjoint operator to be found.   Since we have 
$$
    C+P-A_{\ell+1}^{2}= R_{\ell} -Q_{\ell}A_{\ell}-A_{\ell}Q_{\ell} \quad \mod \; \Psi^{-\ell-2}_{\cV}(M),
$$ 
this suggests to take $Q_{\ell}$ such that 
$$
  \sigma_{-\ell-1}(R_{\ell})= 2\sigma_0(A_{\ell})\sigma_{-\ell-1}(Q_{\ell}),
$$
which can be done with $Q_{\ell}$ formally self-adjoint.  We can then define $A\in\Psi^0_{\cV}(M)$ to be a formally self-adjoint operator given by an asymptotic sum specified by the $A_{\ell}$, so that 
$$
   C+B-A^2\in \Psi^{-\infty}_{\cV}(M).
$$
\end{proof}

\begin{theorem}
Let $E_1$ and $E_2$ be vector bundles on $M$ and set $\cH_i= x^{\mathfrak{t}_i}L^2_b(M;E_i)$  Then an operator $P$ in $\Psi^{0,\cE/\mathfrak{s}}_{\QFB,\cn}(M;E_1,E_2)$ induces a bounded operator 
$$
P: x^{\mathfrak{t}_1}L^2_b(M;E_1)\to x^{\mathfrak{t}_2}L^2_b(M;E_2)
$$
provided  $\Re(\mathfrak{s}-\mathfrak{r}+\pi_R^{\#}\mathfrak{t}_1-\pi_L^{\#}\mathfrak{t}_2)> 0$ and $\Re(\cE-\mathfrak{r}+\pi_R^{\#}\mathfrak{t}_1-\pi_L^{\#}\mathfrak{t}_2)\ge 0$ with strict inequality except possibly at $\ff_i$ for $H_i$ a boundary hypersurface of $M$.
  Moreover,
the induced linear map
\begin{equation}
   \Psi^{0,\cE/\mathfrak{s}}_{\QFB,\cn}(M;E_1,E_2)\to \cL(\cH_1,\cH_2)
\label{lt.6a}\end{equation}
is continuous, where $ \cL(\cH_1,\cH_2)$ is the space of continuous linear maps $\cH_1\to \cH_2$ with the topology induced by the operator norm.  
Similarly, an element $P\in \Psi^{0,\cE/\mathfrak{s}}_{\Qb,\cn}(M;E_1,E_2)$ induces a bounded operator 
$$
P: x^{\mathfrak{t}_1}L^2_b(M;E_1)\to x^{\mathfrak{t}_2}L^2_b(M;E_2)
$$
provided $
\Re(\mathfrak{s}-\mathfrak{r}'+\pi_R^{\#}\mathfrak{t}_1-\pi_L^{\#}\mathfrak{t}_2)> 0$  and $\Re(\cE-\mathfrak{r}'+\pi_R^{\#}\mathfrak{t}_1-\pi_L^{\#}\mathfrak{t}_2)\ge 0$ with strict inequality except possibly at $H_{ii}^{\Qb}$ for $H_i$ maximal and at $\ff_i^{\Qb}$ for $H_i$ non-maximal.
 Again, the induced map
\begin{equation}
   \Psi^{0,\cE/\mathfrak{s}}_{\Qb,\cn}(M;E_1,E_2)\to \cL(\cH_1,\cH_2)
\label{lt.6b}\end{equation}
is continuous.
\label{lt.6}\end{theorem}
\begin{proof}
We will present the proof for $\QFB$ operators, the proof for $\Qb$ operators being similar.  First, using  partitions of unity, we can assume that $E$ and $F$ are trivial line bundles.  Replacing $P$ by $x^{-\mathfrak{t}_2}\circ P\circ x^{\mathfrak{t}_1}$, we can assume as well that $\mathfrak{t}_1=\mathfrak{t}_2=0$.  By Proposition~\ref{lt.1}, we can further assume $P\in \Psi^{0}_{\QFB}(M)$.   Fix a non-vanishing $b$ density $\nu_b$ and let $P^*$ be the formal adjoint of $P$ with respect to this choice.  By Lemma~\ref{lt.5} applied to $B:=-P^*P$, there exists $C>0$ and $A\in\Psi^0_{\QFB}(M;E,F)$ such that 
$$
     C-P^*P= A^*A+R
$$ 
for some $R\in \Psi^{-\infty}_{\QFB}(M)$.  If we denote the $L^2$ norm of the $b$ density $\nu_b$ by $\|\cdot\|_b$ and the inner product by $\langle\cdot,\cdot\rangle_b$, then given $u\in \CI_c(M\setminus \pa M)$, we compute that 
\begin{equation}
\begin{aligned}
\|Pu\|_b^2 &= C\|u\|_b^2- \|Au\|_b^2-\langle u, Ru\rangle_b \\
 &\le C\|u\|_b^2+|\langle u, Ru\rangle_b|\le C\|u\|_b^2+  \|u\|_b\|Ru\|_b  \le  C'\|u\|_b^2
 \end{aligned}
\label{lt.7}\end{equation}
for some constant $C'>C$, where we have used Corollary~\ref{lt.3} in the last step to control the term $\|Ru\|_b$.  Since $\CI_c(M\setminus \pa M)$ is dense in $L^2_b(M)$, $L^2$ boundedness follows.  To prove the continuity of the linear map \eqref{lt.6a}, notice first that the map
$$
 \Psi^{0,\cE/\mathfrak{s}}_{\QFB,\cn}(M;E_1,F_1)\ni P \mapsto \langle u_2, Pu_1\rangle_{\cH_2}
 $$    
 is continuous for all $u_1\in \dot{\cC}^{\infty}(M;E_1)$ and $u_2\in\dot\cC^{\infty}(M;E_2)$.  Hence, the graph of the linear map \eqref{lt.6a} is closed with respect to the topology induced by the semi-norms 
 $A\mapsto |\langle u_2,Au_1\rangle_{\cH_2}|$ on $\cL(\cH_1,\cH_2)$.  Since this topology is weaker than the operator norm topology, this means that the graph of this map is also closed when we use the operator norm topology on $\cL(\cH_1,\cH_2)$.  The map \eqref{lt.6a} is therefore continuous by the closed graph theorem.
\end{proof}

For some applications in our companion paper \cite{KR2}, we need an extension of this result on the weighted space  of $L^2_b$ $\QFB$ conormal sections
$$
    x^{\mathfrak{t}}\cA_{\QFB,2}(M;E):=\{ \sigma\in x^{\mathfrak{t}}L^2_b(M;E) \; | \; \forall k\in \bbN, \; \forall X_1,\ldots, X_k\in \cV_{\QFB}(M), \quad \nabla_{X_1}\cdots\nabla_{X_k}\sigma\in x^{\mathfrak{t}}L^2_b(M;E) \}
$$
for $\nabla$ a choice of connection for the vector bundle $E$.  This is a Fréchet space with semi-norms induced by the norms of $x^{\mathfrak{t}}L^2_b(M;E)$ and the $x^{\mathfrak{t}}L^2_b$ norms of the $\QFB$ derivatives.

\begin{corollary}
Let $E_1$ and $E_2$ be vector bundles on $M$.  Then an operator $P$ in $\Psi^{0,\cE/\mathfrak{s}}_{\QFB,\cn}(M;E_1,E_2)$ induces a continuous linear map 
$$
P: x^{\mathfrak{t}_1}\cA_{\QFB,2}(M;E_1)\to x^{\mathfrak{t}_2}\cA_{\QFB,2}(M;E_2)
$$
provided  $\Re(\mathfrak{s}-\mathfrak{r}+\pi_R^{\#}\mathfrak{t}_1-\pi_L^{\#}\mathfrak{t}_2)> 0$ and $\Re(\cE-\mathfrak{r}+\pi_R^{\#}\mathfrak{t}_1-\pi_L^{\#}\mathfrak{t}_2)\ge 0$ with strict inequality except possibly at $\ff_i$ for $H_i$ a boundary hypersurface of $M$.  In particular, an operator $P\in \Psi^{0}_{\QFB}(M,E_1,E_2)$ induces a continuous linear map 
$$
P: x^{\mathfrak{t}}\cA_{\QFB,2}(M;E_1)\to x^{\mathfrak{t}}\cA_{\QFB,2}(M;E_2)
$$
for all multiweight $\mathfrak{t}$. 
\label{con.2}\end{corollary}
\begin{proof}
Setting $E=E_1\oplus E_2$, we can regard $P$ as an element of $\Psi^{0,\cE/\mathfrak{s}}_{\QFB,\cn}(M;E)$, so that it suffices to show that $P$ induces a continuous linear map
$$
    P: x^{\mathfrak{t}_1}\cA_{\QFB,2}(M;E)\to x^{\mathfrak{t}_2}\cA_{\QFB,2}(M;E). 
$$
Replacing $P$ by $x^{-\mathfrak{t}_2}Px^{\mathfrak{t}_1}$, we can further assume that $\mathfrak{t}_1=\mathfrak{t}_2=0$.
Let $\nabla$ be a choice of connection for $E$ and let $X\in \cV_{\QFB}(M)$ be a $\QFB$ vector field.  Then the principal symbol of $\nabla_X\in \Psi^1_{\QFB}(M;E)$ commutes with the one of $P$, so that $[\nabla_X, P] \in \Psi^{0,\cE/\mathfrak{s}}_{\QFB,\cn}(M;E)$ is of order $0$.  Since $\nabla_X$ induces by definition a continuous linear operator 
$$
   \nabla_X: \cA_{\QFB,2}(M;E)\to \cA_{\QFB,2}(M;E),
$$
the result therefore follows by using Theorem~\ref{lt.6} and 
$$
      \nabla_XP= [\nabla_X,P]+ P \nabla_X
$$
to show recursively that for all $k\in\bbN$ and $X_1,\ldots, X_k\in \cV_{\QFB}(M)$, 
$$
        \nabla_{X_1}\cdots\nabla_{X_{k}}P: \cA_{\QFB,2}(M;E)\to L^2_b(M;E)
$$
is a continuous linear map.  

\end{proof}

We can use the continuity of the map \eqref{lt.6a} to extract a simple compactness criterion for operators of negative order.

\begin{corollary}
Consider a multiweight $\mathfrak{s}$ and an index family $\cE$ such that $\Re(\mathfrak{s}-\mathfrak{r}+\pi_R^{\#}\mathfrak{t}_1-\pi_L^{\#}\mathfrak{t}_2)> 0$ and 
$
\Re(\cE-\mathfrak{r}+\pi_R^{\#}\mathfrak{t}_1-\pi_L^{\#}\mathfrak{t}_2)\ge 0,
$
with strict inequality except possibly at $\ff_i$ for $H_i$ a boundary hypersurface of $M$.  Then for $\delta>0$, 
an operator $A\in \Psi^{-\delta,\cE/\mathfrak{s}}_{\QFB,\cn}(M;E_1,E_2)$ is compact when acting from $\cH_1$ to $\cH_2$  for $\cH_i= x^{\mathfrak{t}_i}L^2_{b}(M;E_i)$ provided $N_j(A)=0$ for all boundary hypersurface $H_j$ of $M$.  Similarly, consider instead a multiweight $\mathfrak{s}$  such that  $\Re(\mathfrak{s}-\mathfrak{r}'+\pi_R^{\#}\mathfrak{t}_1-\pi_L^{\#}\mathfrak{t}_2)> 0$, and an index family $\cE$  such that  
$$
\Re(\cE-\mathfrak{r}'+\pi_R^{\#}\mathfrak{t}_1-\pi_L^{\#}\mathfrak{t}_2)\ge 0,
$$
with strict inequality except possibly at $H_{ii}^{\Qb}$ for $H_i$ maximal and at $\ff_{i}^{\Qb}$ for $H_i$ not maximal.  Then for $\delta>0$, an operator $A\in \Psi^{-\delta,\cE/\mathfrak{s}}_{\Qb,\cn}(M;E_1,E_2)$ is compact in $\cL(\cH_1,\cH_2)$ provided its restriction to $H_{ii}^{\Qb}$ for $H_i$ maximal and to $\ff^{\Qb}_i$ for $H_i$ non-maximal vanishes.
 \label{mp.15}\end{corollary}
\begin{proof}
We will give the proof for $\QFB$ operators, the proof for $\Qb$ operators being similar.  Replacing $P$ by $x^{-\mathfrak{t}_2}Px^{\mathfrak{t}_1}$ if necessary, we can assume that $\mathfrak{t}_1=\mathfrak{t}_2=0$.  Recall first that the space
$\cK(\cH_1,\cH_2)$ of compact operators acting from  $\cH_1$ to $\cH_2$ is the closure of the operators of finite rank in $\cL(\cH_1,\cH_2)$.  Since $\dot{\cC}^{\infty}(M;E_i)$ is dense in $L^2_{b}(M;E_i)$, the space $\cK(\cH_1,\cH_2)$ can also be seen as the closure of finite rank operators in $\dot{\Psi}^{-\infty}(M;E_1,E_2)$, where 
\begin{equation}
\dot{\Psi}^{-\infty}(M;E_1,E_2):= \{  A\in \Psi^{-\infty}_{\QFB}(M;E_1,E_2)\; |  \; \kappa_A \; \mbox{is rapidly vanishing at} \; \ff_{\QFB}\}.
\label{mp.1}\end{equation}
The finite rank operators are clearly dense in $\dot{\Psi}^{-\infty}(M;E_1,E_2)$, so $\cK(\cH_1,\cH_2)$ is also the closure of \eqref{mp.1} in $\cL(\cH_1,\cH_2)$.  Now, for $\mathfrak{w}$ a $\QFB$ positive multiweight and $\cF$ an index family given by $\bbN_0$ at $\ff_i$ for $H_i\in \cM_1(M)$ and the empty set elsewhere, the map $\Psi^{0,\cF/\mathfrak{w}}_{\QFB,\cn}(M;E_1,E_2)\to \cL(\cH_1,\cH_2)$ is continuous by Theorem~\ref{lt.6}.   On the other hand, using the topology induced by the one of  $\Psi^{0,\cF/\mathfrak{w}}_{\QFB,\cn}(M;E_1,E_2)$ for $\mathfrak{w}$ $\QFB$ positive and sufficiently small, the closure of $\dot{\Psi}^{-\infty}(M;E_1,E_2)$ in $\Psi^{-\delta,\cE/\mathfrak{s}}_{\QFB}(M;E_1,E_2)$ is precisely given by those operators   $A\in \Psi^{-\delta,\cE/\mathfrak{s}}_{\QFB,\cn}(M;E_1,E_2)$ for which $N_i(A)=0$ for all boundary hypersurfaces $H_i$ of $M$.  It follows that those operators are compact.  
\end{proof}

There are more generally natural Sobolev spaces associated to $\QFB$ operators.  For $m>0$, the associated $\QFB$ Sobolev space of order $m$ is 
\begin{equation}
\kridx{H^m_{\QFB}}{HQFB}{QFB Sobolev space}(M;E):=\{  f\in L^2_{\QFB}(M;E) \; |\; Pf\in L^2_{\QFB}(M;E)  \; \forall P\in \Psi^m_{\QFB}(M;E)\}.
\label{mp.21}\end{equation}
For $m<0$, we define instead the $\QFB$ Sobolev space of order $m$ by
\begin{equation}
  H^m_{\QFB}(M;E):= \{  f\in \cC^{-\infty}(M;E) \; |\; f=\sum_{j=1}^N P_jf_j \; \mbox{for some} \; N\in\bbN, \; f_j\in L^2_{\QFB}(M;E),  \; P_j\in \Psi^{-m}_{\QFB}(M;E) \}.
\label{mp.22}\end{equation}
Similarly, we can define the $V$ suspended version by 
\begin{equation}
\kridx{H^m_{\QFB,V}}{HQFBV}{$V$ suspended QFB Sobolev space}(M;E):=\{  f\in L^2_{\QFB}(\bV\times M;E) \; |\; Pf\in L^2_{\QFB}(\bV\times M;E)  \; \forall P\in \Psi^m_{\QFB,V}(M;E)\}
\label{mp.23}\end{equation}
for $m>0$ and 
\begin{multline}
  H^m_{\QFB,V}(M;E):= \{  f\in \cC^{-\infty}(\bV\times M;E) \; |\; f=\sum_{j=1}^N P_jf_j \;\mbox{for some} \; N\in\bbN, \\ 
   f_j\in L^2_{\QFB}(\bV\times M;E),  \; P_j\in \Psi^{-m}_{\QFB,V}(M;E) \}
\label{mp.24}\end{multline}
for $m<0$, where $L^2_{\QFB}(\bV\times M;E)$ is the space of square integrable sections induced by a choice of Hermitian metric on $E$ and by the density associated to the Cartesian product of an Euclidean metric on $V$ and a $\QFB$ metric on $M$.  These spaces can be given a Hilbert space structure using fully elliptic $\QFB$ operators.  To see this, for $m>0$, let $A_m\in \Psi_{\QFB,V}^{\frac{m}2}(M;E)$ be a choice of elliptic $V$ suspended $\QFB$ operator and consider the formally self-adjoint operator $D_m\in \Psi^m_{\QFB,V}(M;E)$ defined by
\begin{equation}
 D_m: A_m^*A_m+ \Id_E.
\label{mp.25}\end{equation}
\begin{lemma}
For $m>0$, the operator $D_m$ is fully elliptic and invertible.
\label{mp.26}\end{lemma}
\begin{proof}
Notice that the normal operators of $D_m$ are of the same form, so proceeding by induction on the depth of $M$, we see it suffices to show that $D_m$ is invertible assuming it is fully elliptic.   Taking the Fourier transform in $V$, set for fixed $\Upsilon\in V^*$
$$
      P:= \widehat{D}_m(\Upsilon)= B^*B+\Id_E \quad \mbox{with} \; B:=\widehat{A}_m(\Upsilon).
$$
By Proposition~\ref{mp.3}, if $Pu=0$, then $u\in \dot{\cC}^{\infty}(M;E)$.  Hence, we can integrate by parts, yielding
\begin{equation}
\begin{aligned}
Pu=0 \; &\Longrightarrow \; \langle u, B^* Bu +u \rangle =0 \\
             &\Longrightarrow \| Bu\|^2_{L^2}+ \|u\|^2_{L^2}=0 \\
             &\Longrightarrow \; \|u\|^2_{L^2}=0 \; \Longrightarrow \; u\equiv 0.
\end{aligned}
\label{mp.27}\end{equation}
Since $P$ is formally self-adjoint, this means $\ker P= \ker P^*= \{0\}$.  This shows that 
$$P: \dot{\cC}^{\infty}(M;E)\to \dot{\cC}^{\infty}(M;E)
$$ 
is injective.  To see that it is also surjective, let $f\in \dot{\cC}^{\infty}(M;E)$ be given.  Since $\ker P^*=\{0\}$, we know that there is a sequence $u_i\in \dot{\cC}^{\infty}(M;E)$ such that 
$Pu_i\to f$ in $L^2$.  Notice that this  sequence must be bounded in $L^2$, for otherwise, scaling a subsequence, we  can obtain  a  sequence $v_i\in \dot{\cC}^{\infty}(M;E)$ with $\|v_i\|_{L^2}=1$ such that $Pv_i\to 0$ in $L^2$, a contradiction since
$$
       \|Pv_i\|^2_{L^2}= \|Bv_i\|^2_{L^2}+ \|v_i\|^2_{L^2}\ge \|v_i\|^2_{L^2}=1.
$$
 Using a parametrix as in Proposition~\ref{mp.3}, namely
$$
     QP=\Id+ R,  \quad \mbox{with} \; Q\in \Psi^{-m}_{\QFB}(M;E), \quad R\in \dot{\Psi}^{-\infty}(M;E),
$$
we have that 
$$
        u_i= QPu_i- R u_i.  
$$
By Theorem~\ref{mp.15}, this means that  we can extract a subsequence such that $Ru_i$, and hence $u_i$, converges in $L^2$.  If $u_i\to u$ in $L^2$, then 
$$
      u= Qf-Ru,
$$
so $u\in \dot{\cC}^{\infty}(M;E)$ and $\displaystyle Pu=\lim_{i\to \infty}Pu_i=f$, showing that $P$ is indeed surjective.  Thus, $P$ is invertible.  Since $\Upsilon\in V^*$ was arbitrary, this means that $D_m$ is invertible, concluding the proof. 
\end{proof}

Using the operator $D_m$ for $m>0$, $D_m:=(D_{-m})^{-1}$ for $m<0$ and $D_0=\Id_E$ for $m=0$, we can define a norm on $H^m_{\QFB}(M;E)$ by 
\begin{equation}
  \|u\|_{H^m_{\QFB}}:= \| D_m u\|_{L^2}.
\label{mp.28}\end{equation}
Using the density of $\dot{\cC}^{\infty}(M;E)$ in $L^2_{\QFB}(M;E)$ and Theorem~\ref{lt.6} when $m<0$, it is straightforward to check that $H^m_{\QFB}(M;E)$ is the closure of $\dot{\cC}^{\infty}(M;E)$ with respect to the norm \eqref{mp.28}.  Notice in particular that $D_m$ induces an isomorphism 
\begin{equation}
D_m: H^m_{\QFB}(M;E)\to L^2_{\QFB}(M;E).
\label{mp.28b}\end{equation}  
\begin{proposition}
Let $\w$ be the multiweight given by $\w(H_i)= \frac{h_i}2=\frac{\dim S_i+1}2$, so that 
$$
L^2_{\QFB}(M;E)= x^{\w}L^2_b(M;E).
$$  Then a $\QFB$ pseudodifferential operator $P\in\Psi^{m,\cE/\mathfrak{s}}_{\QFB,\cn}(M;E,F)$ induces a bounded linear map 
$$
       P: x^{\mathfrak{t}_1}H^{p}_{\QFB}(M;E)\to x^{\mathfrak{t}_2}H^{p-m}_{\QFB}(M;F)
$$
for any $p\in \bbR$ provided $\Re(\mathfrak{s}-\mathfrak{r}+\pi_R^{\#}(\mathfrak{t}_1+\w)-\pi_L^{\#}(\mathfrak{t}_2+\w))> 0$ and $\Re(\cE-\mathfrak{r}+\pi_R^{\#}(\mathfrak{t}_1+\w)-\pi_L^{\#}(\mathfrak{t}_2+\w))\ge 0$ with strict inequality except possibly at $\ff_i$ for $H_i$ a boundary hypersurface of $M$.
\label{mp.29}\end{proposition}
\begin{proof}
Replacing the operator $P$ by $x^{-\mathfrak{t}_2}Px^{\mathfrak{t}_1}$ if necessary, we can assume that $\mathfrak{t}_1=\mathfrak{t}_2=0$.  Clearly, we can also assume that $\cE$ is the empty set except possibly at $\ff_i$ for $H_I$ a boundary hypersurface of $M$.  Since \eqref{mp.28b} is an isomorphism, the result then follows from Theorem~\ref{lt.6} by noticing by Theorem~\ref{co.9} that 
$$
       P= D_{m-p}\widetilde{P}D_p \quad \mbox{with} \quad \widetilde{P}= D_{p-m}PD_{-p}\in \Psi^{0,\cE/\mathfrak{s}}_{\QFB,\cn}(M;E).
$$
\end{proof}

The previous proposition shows in particular that $D_m$ induces an isomorphism
\begin{equation}
     D_m: x^kH^p_{\QFB}(M;E)\to x^kH^{p-m}_{\QFB}(M;E)
\label{mp.30}\end{equation}
for all $p\in \bbR$ and $k\in \bbR^{\ell}$.  

\begin{proposition}
There is a continuous inclusion $x^kH^m_{\QFB}(M;E) \subset x^{k'}H^{m'}_{\QFB}(M;E)$ if and only if $m\ge m'$ and $k\ge k'$.  The inclusion is compact  if $m>m'$ and $k>k'$.
\label{mp.31}\end{proposition}
\begin{proof}
The fact that these inclusions are continuous follows from the isomorphism \eqref{mp.30} and Proposition~\ref{mp.29}.  The compactness follows using the isomorphism \eqref{mp.30} and the fact that for $\epsilon>0$, 
$v^{\epsilon} D_{-\epsilon}\in v^{\epsilon}\Psi^{-\epsilon}_{\QFB}(M;E)$ is a compact operator, where $v$ is a total boundary defining function for $M$.  
\end{proof}

Finally, we obtain the following Fredholm criterion.  

\begin{theorem}
A fully elliptic operator $P\in \Psi^m_{\QFB}(M;E,F)$ induces a Fredholm operator
$$
     P: x^{\mathfrak{t}}H^{p+m}_{\QFB}(M;E)\to x^{\mathfrak{t}}H^p_{\QFB}(M;F)
$$
for all $p\in \bbR$ and $\mathfrak{t}\in\bbR^{\ell}$.  
\label{mp.32}\end{theorem}
\begin{proof}
The parametrix construction of Proposition~\ref{mp.3} together with Proposition~\ref{mp.29} and Proposition~\ref{mp.31} show that $P$ is invertible modulo a compact operator, from which the result follows.  
\end{proof}

Here is a natural application of Proposition~\ref{mp.3}, Corollary~\ref{mp.4} and Theorem~\ref{mp.32}.

\begin{corollary}
Let $g$ be an exact $\QFB$ metric in the sense of Definition~\ref{pt.6} and $\Delta_g$ be the corresponding scalar Laplacian (with positive spectrum).  Then for $\lambda\in \bbC\setminus [0,\infty)$, the operator $\Delta_g-\lambda$ is fully elliptic and invertible with inverse in $\Psi^{-2}_{\QFB}(M)$.  
\label{mp.33}\end{corollary}
\begin{proof}
We can proceed by induction on the depth of $M$.  When $M$ is closed, the result is standard.  Hence, assume the result holds for all manifolds with fibered corners of lower depth.  Now the normal operators of $\Delta-\lambda$  are of similar form, but are suspended and are defined on a manifold with fibered corners of lower depth, namely the fibers of the various fiber bundles $\phi_i:H_i\to S_i$.  For a point $s\in S_i$, the corresponding suspended operator is of the form
$$
     \Delta_{\phi_i^{-1}(s)}+ \Delta_{{}^{\phi}N_sS_i}-\lambda,
$$ 
where $\Delta_{\phi_i^{-1}(s)}$ is the induced Laplacian $\phi_i^{-1}(s)$ and $\Delta_{{}^{\phi}N_sS_i}$ is the Laplacian for a certain norm $\|\cdot\|_{{}^{\phi}N_sS_i}$ on the vector space ${}^{\phi}N_sS_i$.   
 the Fourier transform gives an operator of the form
$$
\Delta_{\phi_i^{-1}(s)}+ \|\Upsilon\|^2_{{}^{\phi}N_sS_i}-\lambda,  \quad \Upsilon\in {}^{\phi}N^*_sS_i.
$$
By our induction hypothesis, this is fully elliptic and invertible for all $\Upsilon$, hence invertible as a suspended operator.  We thus see that $\Delta_g-\lambda$ is fully elliptic.  By standard arguments using integration by parts, $\Delta_g-\lambda$ and its formal adjoint $\Delta_g-\overline{\lambda}$ have no kernel.  Hence, 
$$
\Delta_g-\lambda: H^2_{\QFB}(M)\to H^0_{\QFB}(M)
$$
is invertible.  By Corollary~\ref{mp.4}, the inverse is in $\Psi^{-2}_{\QFB}(M)$, which completes the proof.  
\end{proof}

\section{Symbol maps for $\Qb$ operators and the edge calculus} \label{smb.0}

In this section, we will introduce symbol maps for $\Qb$ operators.  Some of these symbol maps will be in terms of edge operators.  For completeness, we will therefore give a detailed description of the edge calculus of \cite{MazzeoEdge, AG} and derive some specific results that we will need. 
  
Now, to define the principal symbol of pseudodifferential $\Qb$ operators, we can proceed in the same way as for pseudodifferential $\QFB$ operators.  Indeed, we can use Corollary~\ref{ds.15b} and the principal symbol for conormal distributions of Hörmander to define a principal symbol map
\begin{equation}
\sigma_m: \Psi^m_{\Qb}(M;E,F)\to S^{[m]}({}^{\Qb}T^*M; \pi^*\hom(E,F)),
\label{smb.1}\end{equation}
where $\pi: {}^{\Qb}T^*M\to M$ is the canonical projection.  This map is surjective and induces a short exact sequence
\begin{equation}
\xymatrix{
0 \ar[r] & \Psi^{m-1}_{\Qb}(M;E,F) \ar[r] & \Psi^{m}_{\Qb}(M;E,F)\ar[r]^-{\sigma_m} & S^{[m]}({}^{\Qb}T^*M;\pi^*\hom(E,F))\ar[r] &0
}
\label{smb.2}\end{equation}
such that 
\begin{equation}
\sigma_{m+m'}(A\circ B)= \sigma_m(A)\sigma_{m'}(B)
\label{smb.3}\end{equation}
for $A\in \Psi^{m}_{\Qb}(M;F,G)$ and $B\in \Psi^{m'}_{\Qb}(M;E,F)$.  
\begin{definition}
A $\Qb$ operator $P\in \Psi^m_{\Qb}(M;E,F)$ is \textbf{elliptic} if its principal symbol $\sigma_m(P)$ is invertible.  
\label{smb.4}\end{definition}

As for $\QFB$ operators, to capture the asymptotic behavior of $\Qb$ operators near each boundary hypersurface $H_i$ of $M$, we need to introduce normal operators.  First, for $H_i$ a maximal boundary hypersurface of $M$, we consider the spaces 
\begin{equation}
\begin{gathered}
\Psi^m_{\Qb,ii}(H_i;E,F):=  \{A\in I^m(H^{\Qb}_{ii};\Delta_{ii};\beta^*_{\Qb}(\Hom(E,F)\otimes \pr_R^*\Omega_{\Qb})) \; | \;  A\equiv 0 \;\mbox{at} \;\pa H^{\Qb}_{ii}\cap \overline{(\ff^{\Qb}\setminus H^{\Qb}_{ii})} \},\\
\Psi^{-\infty,\cH/\mathfrak{h}}_{\Qb,ii}(H_i;E,F):=\sA^{\cH/\mathfrak{h}}_{\phg}(H^{\Qb}_{ii};\beta^*_{\Qb}(\Hom(E,F)\otimes \pr_R^*\Omega_{\Qb})), \\
\Psi^{m,\cH/\mathfrak{h}}_{\Qb,ii}(H_i;E,F):=\Psi^{m}_{\Qb,ii}(H_i;E,F)+ \Psi^{-\infty,\cH/\mathfrak{h}}_{\Qb,ii}(H_i;E,F), \quad m\in \bbR,
\end{gathered}
\label{smb.5}\end{equation}
for $\cH$ and $\mathfrak{h}$ an index family and multiweight associated to the manifold with corners $H_{ii}^{\Qb}$, where  
$$\diag_{ii}:=H_{ii}^{\Qb}\cap \diag_{\Qb}
$$ 
is the restriction of the lifted diagonal to $H_{ii}^{\Qb}$.  
Clearly, restriction to $H_{ii}^{\Qb}$ induces a map 
\begin{equation}
 N_{ii}: \Psi^{m,\cE/\mathfrak{s}}_{\Qb}(M;E,F)\to \Psi^{m,\cG/\mathfrak{g}}_{\Qb,ii}(H_i;E,F)
\label{smb.6}\end{equation}
when $\inf\Re \cE(H_{ii})\ge 0$ with $(i\bbR,0)\cap \cE(H_{ii})\subset (0,0)$, in which case $\cG$ and $\mathfrak{g}$ are the restriction of $\cE$ and $\mathfrak{s}$ to $H^{\Qb}_{ii}$.  
 This gives rise to a short exact sequence
\begin{equation}
\xymatrix{
0\ar[r] & x_i\Psi^m_{\Qb}(M;E,F) \ar[r] & \Psi^m_{\Qb}(M;E,F) \ar[r]^-{N_{ii}} & \Psi^m_{\Qb,ii}(H_i;E,F)\ar[r] & 0.
}
\label{smb.7}\end{equation}
The composition rules of Theorem~\ref{co.13} induce a natural operation of composition
\begin{equation}
\Psi^{m,\cG/\mathfrak{g}}_{\Qb,ii}(H_i;F,G)\circ \Psi^{m',\cH/\mathfrak{h}}_{\Qb,ii}(H_i;E,F)\subset \Psi^{m+m',\cL/\mathfrak{l}}_{\Qb,ii}(H_i;E,G)
\label{smb.8}\end{equation}
such that 
\begin{equation}
 N_{ii}(A\circ B)= N_{ii}(A)\circ N_{ii}(B).
\label{smb.9}\end{equation} 
 For $H_i$ not maximal, we consider instead the spaces 
 \begin{equation}
 \begin{gathered}
\Psi^m_{\ff^{\Qb}_i}(H_i;E,F):=  \{A\in I^m(\ff^{\Qb}_{i};\diag_{\ff^{\Qb}_i};\beta^*_{\Qb}(\Hom(E,F)\otimes \pr_R^*\Omega_{\Qb})) \; | \;  A\equiv 0 \;\mbox{at} \;\pa\ff^{\Qb}_{i}\cap \overline{(\ff^{\Qb}\setminus \ff^{\Qb}_i)} \},\\
\Psi^{-\infty,\cH/\mathfrak{h}}_{\ff^{\Qb}_i}(H_i;E,F):=\sA^{\cH/\mathfrak{h}}_{\phg}(\ff^{\Qb}_{i};\beta^*_{\Qb}(\Hom(E,F)\otimes \pr_R^*\Omega_{\Qb})), \\
\Psi^{m,\cH/\mathfrak{h}}_{\ff^{\Qb}_i}(H_i;E,F):=\Psi^{m}_{\ff^{\Qb}_i}(H_i;E,F)+ \Psi^{-\infty,\cH/\mathfrak{h}}_{\ff^{\Qb}_i}(H_i;E,F), \quad m\in\bbR,
\end{gathered}
 \label{smb.11}\end{equation}
 for $\cH$ and $\mathfrak{h}$ an index family and a multiweight on $\ff_i^{\Qb}$, where $\diag_{\ff^{\Qb}_i}:=\diag_{\Qb}\cap \ff^{\Qb}_i$ is the restriction to $\ff^{\Qb}_i$ of the lifted diagonal.   
Again, provided $\inf\Re \cE(\ff^{\Qb}_i)\ge 0$ with $(i\bbR,0)\cap \cE(\ff^{\Qb}_i)\subset (0,0)$, restricting to $\ff^{\Qb}_{i}$ induces a map
\begin{equation}
  N_i: \Psi^{m,\cE/\mathfrak{s}}_{\Qb}(M;E,F)\to \Psi^{m,\cH/\mathfrak{h}}_{\ff^{\Qb}_i}(H_i;E,F)
\label{smb.12}\end{equation}
with $\cH$ and $\mathfrak{h}$ the restrictions of $\cE$ and $\mathfrak{s}$ to $\ff^{\Qb}_i$.  
This induces a short exact sequence
\begin{equation}
\xymatrix{
0\ar[r] & x_i\Psi^m_{\Qb}(M;E,F) \ar[r] & \Psi^m_{\Qb}(M;E,F) \ar[r]^-{N_{i}} & \Psi^m_{\ff^{\Qb}_i}(H_i;E,F)\ar[r] & 0.
}
\label{smb.13}\end{equation}
By Theorem~\ref{co.13}, there are  natural operations of composition
\begin{equation}
\Psi^{m,\cG/\mathfrak{g}}_{\ff^{\Qb}_i}(H_i;F,G)\circ \Psi^{m',\cH/\mathfrak{h}}_{\ff^{\Qb}_i}(H_i;E,F)\subset \Psi^{m+m',\cL/\mathfrak{l}}_{\ff^{\Qb}_i}(H_i;E,G)
 \label{smb.14}\end{equation}
such that 
\begin{equation}
    N_i(A\circ B)= N_i(A)\circ N_i(B).
\label{smb.15}\end{equation}

One important difference between $\Qb$ operators and $\QFB$ operators is that the normal operators of $\Qb$ operators do not admit a simple iterative description in terms of suspended $\Qb$ operators.  In fact, the composition rules in \eqref{smb.8} and \eqref{smb.14} are more intricate and `non-commutative' in nature.  For this reason, to invert those normal operators, it is often easier to work with $\QAC$ operators.  When $H_i$ is maximal, one ultimately has to deal with edge operators to describe the inverse of the normal operators.  

To see how edge operators arise, let us first review their definition from \cite{MazzeoEdge,AG}.   Thus, let $(M,\phi)$ be a manifold with fibered corners with boundary hypersurfaces $H_1,\ldots, H_{\ell}$ and boundary fiber bundles  $\phi_i:~H_i\to~S_i$.  Let $x_1,\ldots, x_{\ell}$ be corresponding boundary defining functions.  Then the Lie algebra of \textbf{edge vector fields} $\kridx{\cV_e}{Ve}{edge vector fields}(M)$ is given by 
\begin{equation}
\cV_e(M)= \{ \xi\in \cV_b(M)\; | \; (\phi_i)_*(\xi|_{H_i})=0 \; \forall\; i\} 
\label{smb.17}\end{equation}
and corresponds to $b$ vector fields tangent to the fibers of $\phi_i:H_i\to S_i$ for each boundary hypersurface $H_i$.  In terms of the coordinates \eqref{coor.1}, a local basis of vector fields is given by 
$$
      v_i\frac{\pa}{\pa x_i}, v_i\frac{\pa}{\pa y_i^{j}}, \frac{\pa}{\pa z^k}.
$$
By the Serre-Swan theorem, there is a corresponding Lie algebroid $\kridx{{}^{e}T}{TMe}{edge tangent bundle} M\to M$ called the \textbf{edge tangent bundle} with anchor map $a_e: {}^{e}TM\to TM$ inducing a canonical isomorphism of Lie algebras
$$
  \CI(M;{}^{e}TM)=\cV_e(M).
$$
Notice that away from $\pa M$, the anchor map induces a canonical isomorphism of vector bundles
$$
       {}^{e}TM|_{M\setminus \pa M} \cong TM|_{M\setminus \pa M}.
$$
However, on $\pa M$, the anchor map is neither injective nor surjective.  
\begin{definition}
An \textbf{edge metric} smooth up to the boundary on $M\setminus \pa M$ is a complete Riemannian metric $g_e$ on $M\setminus \pa M$ which is of the form 
$$
        g_e= (a_e^{-1})^*(h_e|_{M\setminus \pa M})
$$ 
for some smooth bundle metric $h_e\in \CI(M; S^2({}^{e}T^*M))$ on ${}^{e}TM$.  A \textbf{wedge metric} smooth up to the boundary on $M\setminus \pa M$ is instead a Riemannian metric $g_w$ on $M\setminus \pa M$ of the form 
$$
     g_w= v^2g_e
$$
for some edge metric $g_e$ smooth up to the boundary, where $v=\prod_{i=1}^{\ell} x_i$ is a total boundary defining function.  
\label{smb.18}\end{definition}
\begin{remark}
A wedge metric is not geodesically complete.  It is also called \textbf{incomplete iterated edge metric} in \cite{ALMP2012}.
\label{smb.19}\end{remark}
In the coordinates \eqref{coor.1}, an example of edge metric is given by 
\begin{equation}
\sum_{i=1}^l \left(\frac{dx_i^2}{v_i^2}+ \sum_{j=1}^{k_i} \frac{(dy_i^j)^2}{v_i^2}\right) + \sum_{m=1}^q dz_m^2.
\label{smb.20}\end{equation}
As for $\QFB$ metrics, there are important subclasses of edge and wedge metrics.
\begin{definition}
An edge metric $g_e$ is of \textbf{product-type} near $H_i$ if in a tubular neighborhood of $H_i$ as in \eqref{pt.1}, it is of the form 
\begin{equation}
   \frac{dx_i^2}{v_i^2}+ \frac{\phi_i^*g_{S_i}}{v_i^2}+ \kappa_i,
\label{e.1a}\end{equation}
where $g_{S_i}$ is an edge metric on $S_i$, $\kappa_i$ is a $2$-tensor inducing an edge metric on the fibers of $\phi_i: H_i\to S_i$ such that $\phi_i^*g_{S_i}+ \kappa_i$ is a Riemannian metric turning $\phi_i$ into a Riemannian submersion onto $(S_i,g_{S_i})$.  An edge metric is \textbf{exact} if it is product-type near $H_i$ up to a term in $x_i\CI(M;S^2({}^eT^*M))$ for each boundary hypersurface $H_i$.  A wedge metric $g_w$ is \textbf{product-type} near $H_i$ (respectively \textbf{exact}) if the corresponding edge metric $g_e= \frac{g_w}{v^2}$ is product-type near $H_i$ (respectively exact).  Alternatively, in terms of \eqref{e.1a}, a wedge metric $g_w$ is product-type near $H_i$ if in a tubular neighborhood of $H_i$ as in \eqref{pt.1},
\begin{equation}
  g_w= \rho_i^2 (dx_i^2+ \phi_i^*g_{S_i}+ x_i^2 \kappa_{w,i}), \quad \mbox{with} \; \rho_i:=\prod_{H_j<H_i} x_j,
\label{e.1b}\end{equation}
where $\kappa_{w,i}= \lrp{\frac{v_i^2}{x_i^2}}\kappa_i$ is a 2-tensor inducing a wedge metric on the fibers of $\phi_i: H_i\to S_i$ in such a way that $\phi_i^*g_{S_i}+ \kappa_{w,i}$ is a Riemannian metric turning $\phi_i$ into a Riemannian submersion onto $(S_i,g_{S_i})$.
\label{e.1}\end{definition}

The edge double space of \cite{MazzeoEdge,AG} can be described as follows.  On the codimension two face $H_i\times H_i$ in $M\times M$, consider the fiber diagonal
\begin{equation}
  \Phi_i^e= (\phi_i\times \phi_i)^{-1}(\diag_{S_i}),
\label{smb.21}\end{equation}
where $\diag_{S_i}$ is the diagonal in $S_i\times S_i$.  Unless $H_i$ is a minimal boundary hypersurface, notice that $\diag_{S_i}$ is not a $p$-submanifold of $S_i^2$, and correspondingly $\Phi^e_i$ is not a $p$-submanifold of $M^2$.  Choosing the labelling of the boundary hypersurfaces of $M$ to be compatible with the partial order of the boundary hypersurfaces in the sense that
$$
     H_i<H_j\; \Longrightarrow \; i<j,
$$ 
the \textbf{edge double space} can nevertheless be defined by 
\begin{equation}
  \kridx{M^2_e}{M2e}{edge double space}=[M^2;\Phi_1^{e},\ldots, \Phi_{\ell}^{e}] \quad \mbox{with blow-down map} \quad \beta_e: M^2_e\to M^2.  
\label{smb.22}\end{equation}
Indeed, each blow-up is well-defined, since for $H_i$ not minimal, the lift of $\Phi_i^{e}$ becomes a $p$-submanifold once the first $i-1$ blow-ups are performed.  Let us denote by 
$\ff_i^{e}$ the boundary hypersurface created by the blow-up of $\Phi_i^{e}$.  Let also $\ff_e$ denote the union of all these new boundary hypersurfaces,
$$
     \ff_e:= \bigcup_{i=1}^{\ell} \ff_i^{e}.
$$
Denote by $H_{i0}^{e}$ and $H_{0i}^{e}$ the lifts to $M^2_e$ of the boundary hypersurfaces $H_i\times M$ and $M\times H_i$ in $M^2$.  Finally, let $\diag_e\subset M^2_e$ denote the lift of the diagonal in $M^2$ and let  
$$
     \Omega_e:= \left|  \Lambda^{\operatorname{top}}({}^{e}T^*M)\right|
$$
denote the edge density bundle on $M$.  Given $E$ and $F$ vector bundles on $M$, the \textbf{small calculus of edge pseudodifferential operators} acting from sections of $E$ to sections of $F$ is then the union over $m\in \bbR$ of the spaces
\begin{equation}
\kridx{\Psi^m_e}{Psie}{edge pseudodifferential operators (small calculus)}(M;E,F):= \{ \kappa\in I^m(M_e^2;\diag_e; \beta_e^*(\Hom(E,F)\otimes \pr_R^*\Omega_e))\; | \; \kappa\equiv 0\; \mbox{at} \; \pa M^2_e\setminus \ff_e\}.
\label{smb.23}\end{equation}
More generally, given an index family $\cE$ for $M^2_e$ and a multiweight $\mathfrak{s}$, there are  corresponding spaces of edge pseudodifferential operators 
\begin{equation}
\begin{gathered}
\Psi^{-\infty,\cE/\mathfrak{s}}_e(M;E,F):= \sA^{\cE/\mathfrak{s}}_{\phg}(M^2_e;\diag_e;\beta^*_e(\Hom(E,F)\otimes \pr_R\Omega_e)), \\
\kridx{\Psi^{m,\cE/\mathfrak{s}}_e}{PsieEs}{edge pseudodifferential operators (large calculus)}(M;E,F):=\Psi^{m}_e(M;E,F)+ \Psi^{-\infty,\cE/\mathfrak{s}}_e(M;E,F), \quad m\in \bbR.
\end{gathered}
\label{smb.24}\end{equation}
\begin{remark}
To be consistent with our notation for $\QFB$ operators and $\Qb$ operators, we use in \eqref{smb.24} a right density slightly different from the one of \cite[(3.11)]{AG}.
\label{smb.25}\end{remark}
As in Definition~\ref{ext.8} we can define the spaces of weakly conormal edge pseudodifferential operators by
\begin{equation}
\begin{gathered}
\Psi^{-\infty,\cE/\mathfrak{s}}_{e,\cn}(M;E,F):= \sA^{\cE/\mathfrak{s}}_{e,\phg}(M^2_e;\diag_e;\beta^*_e(\Hom(E,F)\otimes \pr_R\Omega_e)), \\
\kridx{\Psi^{m,\cE/\mathfrak{s}}_{e,\cn}}{PsieEscn}{weakly conormal edge pseudodifferential operators}(M;E,F):=\Psi^{m}_e(M;E,F)+ \Psi^{-\infty,\cE/\mathfrak{s}}_{e,\cn}(M;E,F), \quad m\in \bbR.
\end{gathered}
\label{smb.24cn}\end{equation}

For $\tau>0$, we can also consider the spaces
\begin{equation}
\kridx{\Psi^{-\infty,\tau}_{e,\res}}{Psievr}{very residual edge pseudodifferential operators}(M;E,F):= \Psi_{e,\cn}^{-\infty,\cE/\mathfrak{s}_{\tau}}(M;E,F) \quad \mbox{and} \quad \Psi^{m,\tau}_{e,\cn}(M;E,F):= \Psi_e^{m}(M;E,F)+ \Psi^{-\infty,\cF/\mathfrak{s}_{\tau}}_{e,\res}(M;E,F)
\label{smb.26}\end{equation}
obtained by taking $\mathfrak{s}_{\tau}$ to be the multiweight given by $\tau$ at $H_{i0}^{e}$ and $\ff_i^{e}$ and by $\tau+\dim S_i$ at $H_{0i}^{e}$, $\cE$ to be the index family given by the empty set at every boundary hypersurfaces and $\cF$ to be the index family given by $\bbN_0$ at $\ff_i^{e}$ for $H_i$ a boundary hypersurface and by the empty set elsewhere.  Taking into account Remark~\ref{smb.25}, we can now state the composition theorem of \cite{AG}  and its generalization to weakly conormal edge pseudodifferential operators as follows.

\begin{theorem}[\cite{AG}]
Suppose that $\cE,\cF$ are index families and $\mathfrak{s},\mathfrak{t}$ are multiweights such that for all $i$,
$$
  \min\{\mathfrak{s}(H^{e}_{0i}),\min \Re \cE(H_{0i}^{e})\}+ \min\{\mathfrak{t}(H^{e}_{i0}),\min\Re \cF(H^{e}_{i0})\}>\dim S_i.
$$
Then given $A\in \Psi^{m,\cE/\mathfrak{s}}_e(M;F,G)$ and $B\in \Psi^{m',\cF/\mathfrak{t}}_e(M;E,F)$, we have that 
$$
       A\circ B\in \Psi^{m+m',\cK/\mathfrak{k}}_e(M;E,G),
$$
where the index family $\cK$ is given by 
\begin{equation}
\begin{aligned}
\cK(H^{e}_{i0})&= \cE(H^{e}_{i0})\overline{\cup} \left( \cE(\ff^{e}_i)+ \cF(H^{e}_{i0}) \right), \\
\cK(H_{0i}^{e})&= \cF(H_{0i}^{e})\overline{\cup} \left( \cE(H_{0i}^{e})+ \cF(\ff^{e}_i) \right), \\
\cK(\ff_i^{e}) &= \left(\cE(\ff^{e}_i)+ \cF(\ff_i^{e}) \right) \overline{\cup} \left(\cE(H^{e}_{i0})+\cF(H^{e}_{0i}) \right),
\end{aligned}
\label{smb.27b}\end{equation}
while the multiweight $\mathfrak{k}$ is given by
\begin{equation}
\begin{aligned}
\mathfrak{k}(H^{e}_{i0})&= \min\left\{\mathfrak{s}(H^{e}_{i0}), \mathfrak{s}(\ff^{e}_i)\dot{+} \mathfrak{t}(H^{e}_{i0}) \right\}, \\
\mathfrak{k}(H_{0i}^{e})&= \min\left\{\mathfrak{t}(H_{0i}^{e}),  \mathfrak{s}(H_{0i}^{e})\dot{+} \mathfrak{t}(\ff^{e}_i) \right\}, \\
\mathfrak{k}(\ff_i^{e}) &= \min\left\{\mathfrak{s}(\ff^{e}_i)\dot{+} \mathfrak{t}(\ff_i^{e}), \mathfrak{s}(H^{e}_{i0})\dot{+}\mathfrak{t}(H^{e}_{0i}) \right\}.
\end{aligned}
\label{smb.27c}
\end{equation}
If instead $\cE=\cF=\emptyset$ except possibly at $\ff_i^e$ for $H_i$ a boundary hypersurface, where it could be $\bbN_0$ instead of $\emptyset$, then with $A\in \Psi^{m,\cE/\mathfrak{s}}_{e,\cn}(M;F,G)$ and $B\in \Psi^{m',\cF/\mathfrak{t}}_{e,\cn}(M;E,F)$ only weakly conormal edge pseudodifferential operators, we have that
$$
       A\circ B\in \Psi^{m+m',\cK/\mathfrak{k}}_{e,\cn}(M;E,G),
$$
\label{smb.27}\end{theorem}
\begin{proof}
By \cite{AG}, we only need to provide a proof for weakly conornal edge pseudodifferential operators, in which case it is a straightforward adaptation of the corresponding argument for the composition of weakly conormal $\QFB$ pseudodifferential operators in the proof of Theorem~\ref{co.9}.
\end{proof}

Applying this result to \eqref{smb.23} and \eqref{smb.26} yields the following.
\begin{corollary}
For $m,m'\in \bbR$, and $E,F$ and $G$ vector bundles on $M$,
$$
    \Psi^{m}_e(M;F,G)\circ \Psi^{m'}_e(M;E,F)\subset \Psi_e^{m+m'}(M;E,G).
$$
Similarly, for $A\in \Psi^{m,\tau}_e(M;F,G)$ and $B\in \Psi^{m',\tau}_{e}(M;E,F)$, we have that
$$
     A\circ B\in \Psi^{m+m',\tau}_e(M;E,G).
$$ 
Moreover, $A\circ B$ is in fact in  $\Psi^{-\infty,\tau}_{e,\res}(M;E,G)$ as soon as $A\in \Psi^{-\infty,\tau}_{e,\res}(M;F,G)$ or $B\in  \Psi^{-\infty,\tau}_{e,\res}(M;E,F)$.\label{smb.28}\end{corollary}

Coming back to $\Qb$ pseudodifferential operators, let us explain now how the edge double space $(S_i)^2_e$ naturally arises within $H_{ii}^{\Qb}$ for $H_i$ maximal.  
\begin{figure}[h]
\begin{tikzpicture}
\draw(-1,0) arc [radius=1, start angle=180, end angle=150];
\draw(0,1) arc[radius=1, start angle=90, end angle= 120];
\draw(-0.866025,0.5) arc[radius=0.366025404, start angle=180, end angle=90];
\draw(-1,0)--(5,0);
\draw(0,1)--(4,1);
\draw(5,0) arc[radius=1, start angle=0, end angle=30];
\draw(4,1) arc[radius=1, start angle=90, end angle=60];
\draw(4.866025,0.5) arc[radius=0.366025404, start angle=0, end angle=90];
\draw(-0.866025,0.5)--(4.866025,0.5);
\draw(-0.5,0.866025)--(4.5,0.866025);
\draw(0,1) arc[radius=1, start angle=0, end angle=90];
\draw(5,0) arc[radius=2, start angle=240, end angle=270];
\draw(-1,2)--(-1,4);
\draw(-1,2)--(-4,-1);
\draw(-1,0)--(-3,-2);
\draw(-3,-2) arc[radius=1, start angle=0, end angle=90];
\draw(6,-0.267949)--(6.732050808,-2);
\draw(-3,-2)--(6.732050808,-2);
\draw(6,-0.267949)--(6,4);
\draw(4,1)--(4,4);
\draw(-4,-1)--(-4,2);
\draw(6.73205,-2)--(6.73205,2);
\node at (2,-1) {$H^{\Qb}_{10}$};
\node at (-2.5,0) {$H^{\Qb}_{12}$};
\node at (-2.5,2) {$H^{\Qb}_{02}$};
\node at (2,2) {$H^{\Qb}_{01}$};
\node at (2,0.7) {$\ff^{\Qb}_{1}$};
\node at (3,0.25) {$H^{\Qb}_{11}$};
\node at (6.35,1) {$H^{\Qb}_{20}$};
\node at (5.5,2) {$H^{\Qb}_{21}$};

\draw[dashed] (-0.866025,0.7)--(4.866025,0.7);
\draw[dashed] (-0.866025,0.7)--(-3.29289,-1.29289);
\draw[dashed] (4.866025,0.7)--(4.866025,4);

\end{tikzpicture}
\caption{Picture of $H^{\Qb}_{22}$ with $H_1<H_2$ when $M$ is of depth $2$ with the dashed lines representing the intersections of $\diag_2^{\Qb}$ with the corresponding boundary hypersurfaces. Beware that $\diag_2^{\Qb}$ intersects $H^{\Qb}_{11}$, but due to the lack of dimensions, this is not represented in this picture.}
\label{fig.1}\end{figure}
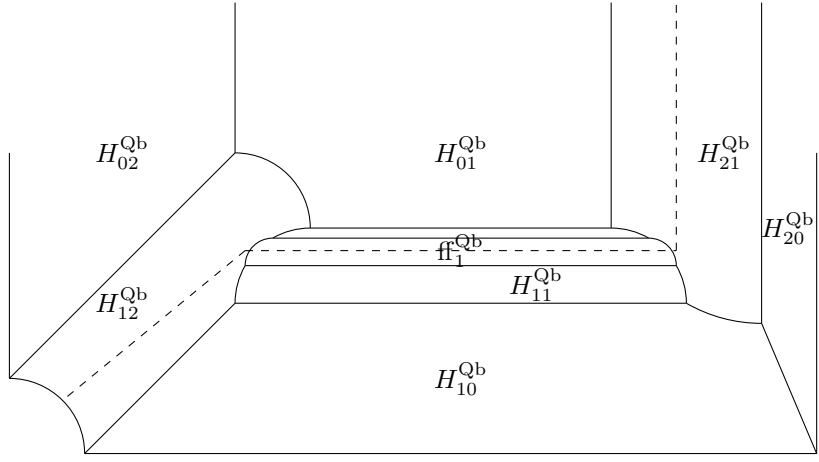
Let $\diag_i^{\Qb}$ be the lift of the $p$-submanifold $\diag_i$ of Lemma~\ref{ps.0} to the $\Qb$ double space.  Since $\phi_i: H_i\to S_i$ is the identity map, we see from \eqref{ps.6} that there is a natural diffeomorphism 
\begin{equation}
  \diag_i\cong (S_i)^2_{\op}.
\label{smb.29}\end{equation}
 Now, for $H_j<H_i$, there is a corresponding boundary hypersurface $\pa_jS_{i}$ of $S_i$ such that  $\phi_j: H_j\to S_j$ restricts to a fiber bundle 
$$
    \phi_j: \pa_jS_{i}\to S_j.
$$
Similarly, by \cite[Corollary~4.15]{KR3}, the fiber bundle  $\phi_j\rttimes \phi_j: \diag_j\to (S_j)^2_{\op}$ restricts to a fiber bundle
\begin{equation}
  \phi_j\rttimes \phi_j: \pa_{jj}(S_i)^2_{\op}\to (S_j)^2_{\op},
\label{smb.30}\end{equation}
where $\pa_{jj}(S_i)^2_{\op}$ is the boundary hypersurface of $(S_i)^2_{\op}$ corresponding to the lift of $\pa_j S_{i}\times \pa_jS_{i}\subset S_i^2$.

From the definition of the $\Qb$ double space and the identification \eqref{smb.29}, it clearly follows that there is a natural diffeomorphism
\begin{equation}
\diag^{\Qb}_i\cong [(S_i)^2_{\op}; (\phi_{j_1}\rttimes \phi_{j_1})^{-1}(\diag^2_{S_{j_1,\op}}),\ldots,(\phi_{j_k}\rttimes \phi_{j_k})^{-1}(\diag^2_{S_{j_k,\op}})],  
\label{smb.30}\end{equation}
where $H_{j_1},\ldots, H_{j_k}$ is an exhaustive list of the boundary hypersurfaces not equal to $H_i$ but intersecting it such that 
$$
         H_{j_m}<H_{j_{m'}}\quad \Longrightarrow \quad m<m'.
$$ 
\begin{lemma}
There is a natural diffeomorphism 
$$
     \diag_i^{\Qb}\cong [(S_i)^2_e; \cM_1(S_i)\times \cM_1(S_i)],
$$
where the blow-ups of the lifts of the elements of $\cM_1(S_i)\times \cM_1(S_i)$ are done in an order compatible with the partial order on $\cM_1(S_i)\times \cM_1(S_i)$.  
\label{smb.31}\end{lemma}
\begin{proof}
Since $(\phi_{j_s}\rttimes \phi_{j_s})^{-1}(\diag^2_{S_{j_s,\op}})$ is just the lift of $(\phi_{j_s}\times \phi_{j_s})^{-1}(\diag^2_{S_{j_s}})\subset S_i\times S_i$ to $(S_i)^2_{\op}$, we see by the definition of the ordered product that \eqref{smb.30} can be rewritten 
\begin{equation}
\diag_i^{\Qb}\cong [(S_i)^2; \cM_1(S_i)\times \cM_1(S_i), (\phi_{j_1}\times \phi_{j_1})^{-1}(\diag^2_{S_{j_1}}),\ldots,(\phi_{j_k}\times \phi_{j_k})^{-1}(\diag^2_{S_{j_k}})]  
\label{smb.32}\end{equation}
with the blow-ups of the elements of $\cM_1(S_i)\times \cM(S_i)$ done in an order compatible with the partial order of $\cM_1(S_i)\times \cM(S_i)$.   
Hence, the results will follow provided we can permute the blow-ups so that the last $k$ blow-ups $(\phi_{j_1}\times \phi_{j_1})^{-1}(\diag^2_{S_{j_1}}),\ldots,(\phi_{j_k}\times \phi_{j_k})^{-1}(\diag^2_{S_{j_k}})$ are performed first instead of at the end.  

Starting with $(\phi_{j_1}\times \phi_{j_1})^{-1}(\diag^2_{S_{j_1}})$, notice that on $[(S_i)^2;\pa_{j_1}S_{i}\times \pa_{j_1}S_{i}]$, the blow-up of its lift commutes with the blow-up of each element of $(\cM_1(S_i)\times \cM_1(S_i))\setminus\{(\pa_{j_1}S_{i},\pa_{j_1}S_{i})\}$ by transversality.  The blow-up of $(\phi_{j_1}\times \phi_{j_1})^{-1}(\diag^2_{S_{j_1}})$ then commutes with the one of $S_{ij_1}\times S_{ij_1}$  by the commutativity of nested blow-ups.  

We can proceed in a similar way for $(\phi_{j_2}\times \phi_{j_2})^{-1}(\diag^2_{S_{j_2}}),\ldots,(\phi_{j_{k-1}}\times \phi_{j_{k-1}})^{-1}(\diag^2_{S_{j_{k-1}}})$ and $(\phi_{j_k}\times \phi_{j_k})^{-1}(\diag^2_{S_{j_k}})$ with  the commutativity of blow-ups coming from transversality, except between $S_{j_s}\times S_{j_s}$ and $(\phi_{j_s}\times \phi_{j_s})^{-1}(\diag^2_{S_{j_s}})$ where it is a consequence of the commutativity of nested blow-ups.  One new feature for $1\le p<m$ is that the commutativity of the blow-ups of $\pa_{j_p}S_{i}\times \pa_{j_{p}}S_{i}$ and  $(\phi_{j_m}\times \phi_{j_m})^{-1}(\diag^2_{S_{j_m}})$ follows from the fact that those $p$-submanifolds are disjoints after the blow-up of $(\phi_{j_p}\times \phi_{j_p})^{-1}(\diag^2_{S_{j_p}})$. 

\end{proof}

The last lemma indicates how edge operators will arise in the construction of a parametrix for $\QFB$ operators.  To invert the model edge operators arising in such a construction, we can rely on \cite{MazzeoEdge, ALMP2012, AG}.  Unfortunately, the specific results that we will need are not quite stated in these papers.  For the convenience of the reader, we will therefore describe the edge operators that we will encounter in the construction of a parametrix for $\QFB$ operators and explain how they can be inverted within the edge calculus.  

Thus, let $g_w$ be an exact wedge metric.  It induces a Clifford bundle associated to the wedge tangent bundle $\kridx{{}^{w}T}{Tw}{wedge tangent bundle} M:= (v^{-1}){}^{e}TM$.  Let $E\to M$ be a choice of Clifford module and let $\nabla^E$ be a choice of Clifford connection which is at the same time a ${}^{w}T^*M$ valued connection in the sens of \cite[Definition~2.20]{ALN04}.  Let $\eth_w\in \frac{1}{v}\Diff^1_e(M;E)$ be the corresponding wedge Dirac operator.  It acts formally on $L^2_w(M;E)$, the space of square integrable sections of $E$ with respect to the wedge metric $g_w$.  It will be convenient to consider the conjugated operator 
$$
       D_w:= \lrp{\prod_i x_i^{\frac{f_i+1}2}}\eth_w \lrp{\prod_i x_i^{-\frac{f_i+1}2}},
$$
where $f_i$ is the dimension of the fibers of $\phi_i: H_i\to S_i$.  In this way, the operator $\eth_w$ acting formally on $L^2_w(M;E)$ will correspond to $D_w$ acting formally on 
$$
L^2_b(M;E):= \lrp{\prod_i x_i^{\frac{f_i+1}2}}L^2_w(M;E),
$$
the space of square integrable sections of $E$ with respect to the same fiberwise norm, but using a $b$-density instead of the volume form of $g_w$ on $M$.   The operator $v^{\frac12}D_w v^{\frac12}$ is an elliptic edge differential operator.  As such, for each boundary hypersurface $H_i$, it has a corresponding normal operator
$$
     N_{i,e}(v^{\frac12}D_w v^{\frac12}):= v^{\frac12}D_w v^{\frac12}|_{\ff^e_i}.
$$
We will be able to construct a nice parametrix for $(v^{\frac12}D_w v^{\frac12})$ provided we can invert those normal operators.  Relying on \cite{ALMP2012, AG}, this can be determined by considering the indicial families
$$
      I_{H_i}(vD_w,\lambda)u:= (x_i^{-\lambda}(vD_w) x_i^{\lambda}\widetilde{u})|_{H_i},  \quad \lambda\in \bbC,
$$
for $u\in\CI(H_i;E)$ and $\widetilde{u}\in \CI(M;E)$ such that $\widetilde{u}|_{H_i}= u$.  To describe this indicial family,  let $\eth_{w,i}$ be the Dirac operator associated to the metric
\begin{equation}
dx_i^2+ \phi_i^*g_{S_i}+ x_i^2 \kappa_{w,i},
\label{e.2}\end{equation}
namely the metric \eqref{e.1b} without the conformal factor of $\rho_i^2$.  By \eqref{e.1b}, the conformal invariance of Dirac operator \cite[Theorem~5.24]{Lawson} and the local description of Clifford modules \cite[Proposition~3.40]{BGV}, the Dirac operator $\eth_{w,i}$  differs from $\rho_i \eth_w$ by a term of order $0$ which is bounded in terms of the norm of the metric \eqref{e.1b}.  This means that the difference $x_i\rho_i\eth_{w}-x_i\eth_{w,i}$ is of order zero and vanishes on $H_i$.  This implies that $\frac{v}{\rho_i}\eth_{w,i}$ yields the same indicial family as $v\eth_w$ at $H_i$.  Now, for each $s\in S_i$, the metric \eqref{e.2} induces a cone metric 
$$
  dx_i^2+ x^2_i\kappa_{w,i,s} \quad \mbox{on} \quad \kridx{\cC_i}{Ci}{model fiber cone for $H_i$ in wedge setting} := [0,\infty)_{x_i}\times F_{i,s},
$$   
where $\kappa_{w,i,s}$ is the wedge metric obtained by  restriction of $ \kappa_{w,i}$ to the fiber $F_{i,s}:=\phi_i^{-1}(s)$.  There is an associated Dirac operator, which by \cite[Proposition~2.5]{Chou} and \cite[(2.21)]{KR0} takes the form
$$
     \eth_{\cC_i,s} = c_i\frac{\pa}{\pa x_i}+ \frac{1}{x_i} \lrp{\eth_{F_{i,s}}+ \frac{c_i f_i}2}
$$
where $c_i= \cl(dx_i)$ is Clifford multiplication by $dx_i$ while  $\eth_{F_{i,s}}= \hat{\eth}_{F_{i,s}} + \cN_{i,s}$ with $\hat{\eth}_{F_{i,s}}$ the Dirac operator on the fiber $F_{i,s}$ and $\cN_{i,s}\in\CI(F_{i,s};\End(E))$  a self-adjoint operator anti-commuting with $c_i$.  To work with an operator acting formally on sections square integrable with respect to a $b$-density, it is convenient to consider the conjugated operator
\begin{equation}
     D_{\cC_i,s}:= \lrp{\prod_{H_j\ge H_i} x_j^{\frac{f_j+1}2}} \eth_{\cC_i,s}  \lrp{\prod_{H_j\ge H_i} x_j^{-\frac{f_j+1}2}}.
\label{e.3}\end{equation}
The indicial family of $D_w$ at $H_i$ restricted to the fiber $F_{i,s}$ then corresponds to the indicial family of $D_{\cC_i,s}$, namely
$$
      I_{H_i}(vD_w,\lambda)|_{F_{i,s}}= \frac{v}{x_i \rho_i} I(x_iD_{\cC_i,s},\lambda)
$$
with
\begin{equation}
    I(x_iD_{\cC_i,s},\lambda)= c_i\lambda + D_{F_{i,s}}-\frac{c_i}2,
\label{e.3b}\end{equation}
where 
$$
 D_{F_{i,s}}:= \lrp{\prod_{H_j> H_i} x_j^{\frac{f_j+1}2}} \eth_{F_{i,s}}  \lrp{\prod_{H_j> H_i} x_j^{-\frac{f_j+1}2}}.
$$

When $\eth_w$ is a Hodge-deRham operator acting on forms taking values in a flat unitary vector bundle $F$, we can give a more explicit description of the self-adjoint operator $\cN_{i,s}$.  If $\overline{\eta}$ is a $F$ valued $k$-form on the cone $[0,\infty)_{x_i}\times F_{i,s}$ obtained by parallel transport of its restriction $\eta$ to $\{1\}\times F_{i,s}$ along geodesics emanating from the tip of the cone, then there is a decomposition
$$
   \overline{\eta}= \overline{\alpha}+ dx_i\wedge \overline{\beta}, \quad \eta=\alpha+dx_i\wedge \beta, \quad \overline{\alpha}=x_i^k\alpha, \; \overline{\beta}= x^{k-1}\beta
$$
for some $F$ valued forms $\alpha,\beta$ on $F_{i,s}$.  In terms of this decomposition, Clifford multiplication by $dx_i$ is given by 
$$
        c_i= \lrp{\begin{matrix} 0 & -1 \\ 1 & 0 \end{matrix}}.
$$
By \cite[(2-9)]{ARS3}, the operator $\hat{\eth}_{F_{i,s}}$ is then given by 
$$
       \hat{\eth}_{F_{i,s}}= \lrp{\begin{matrix} \mathfrak{d}_{F_{i,s}}& 0 \\
          0 & -\mathfrak{d}_{F_{i,s}}  \end{matrix} },
$$
where $\mathfrak{d}_{F_{i,s}}$ is the Hodge-deRham operator acting on forms on $F_{i,s}$ taking values in $F$, while
$$
    \cN_{i,s}= \lrp{\begin{matrix}  0 & N_{F_{i,s}}-\frac{f_{i,s}}2 \\ N_{F_{i,s}}-\frac{f_{i,s}}2 & 0    \end{matrix}   }
$$
with $N_{F_{i,s}}$ the degree operator multiplying a form (of pure degree) by its degree.  In this case, the indicial family \eqref{e.3b} is given by 
$$
   I(x_iD_{\cC_i,s},\lambda)= \lrp{\begin{matrix} \mathfrak{D}_{F_i,s} & -\lambda+N_{F_{i,s}}-\frac{f_{i,s}-1}2 \\
             \lambda+N_{F_{i,s}}-\frac{f_{i,s}+1}2 & -\mathfrak{D}_{F_{i,s}}  \end{matrix} }
$$ 
with
$$
     \mathfrak{D}_{F_{i,s}}= \lrp{\prod_{H_j> H_i} x_j^{\frac{f_j+1}2}} \mathfrak{d}_{F_{i,s}} \lrp{\prod_{H_j> H_i} x_j^{-\frac{f_j+1}2}}. 
$$

The assumption that we will use to ensure that the various normal operators of $v^{\frac{1}2}D_w v^{\frac12}$ can be inverted is the following.
\begin{assumption}
For each boundary hypersurface $H_i$ and for each $s\in S_i$, we suppose that 
$$
c_i I(x_iD_{\cC_i,s},0)=cD_{F_{i,s}}+\frac12
$$
 is essentially self-adjoint as a wedge operator acting formally on $L^2_b(F_{i,s};E)$.  We also suppose  that
\begin{equation}
    c_iI(x_iD_{\cC_i,s},\lambda): w_i^{\frac12}L^2_b(F_{i,s};E)\to w_i^{-\frac12}L^2_b(F_{i,s};E) \quad \mbox{with} \quad w_i:=\prod_{H_j>H_i}x_j,
\label{e.3a}\end{equation}
is invertible for $\lambda\in [0,1]$.  \label{e.3}\end{assumption}
It is sometimes also useful to make the following assumption to obtain polyhomogeneity results.
\begin{assumption}
For each boundary hypersurface $H_i$, we suppose that the spectrum of $c_i I(x_iD_{\cC_i,s},0)$ does not depend on $s\in S_i$.
\label{e.3cn}\end{assumption}

The values of $\lambda$ for which \eqref{e.3a} is not invertible are the \textbf{indicial roots} of $I(x_iD_{\cC_i,s},\lambda)$.  When $\eth_w$ is a Hodge-deRham operator, $c_iI(x_iD_{\cC_i,s},0)$ will be essentially self-adjoint if and only if $\mathfrak{d}_{F_{i,s}}$ is.  In this case, the indicial roots can be described in terms of the spectrum of $\mathfrak{d}_{F_{i,s}}$.  If $\mathfrak{d}_{F_{i,s}}= d^{F_{i,s}}+ \delta^{F_{i,s}}$ with $d^{F_{i,s}}$ the deRham differential on $F_{i,s}$ and $\delta^{F_{i,s}}$ its formal adjoint, then the Friedrich extension gives a self-adjoint extension for $\delta^{F_{i,s}}d^{F_{i,s}}$ and $d^{F_{i,s}}\delta^{F_{i,s}}$.  Denote by $(\delta^{F_{i,s}}d^{F_{i,s}})_q$ and $(d^{F_{i,s}}\delta^{F_{i,s}})_q$ the parts of $\delta^{F_{i,s}}d^{F_{i,s}}$ and $d^{F_{i,s}}\delta^{F_{i,s}}$ acting on forms of degree $q$.  

\begin{lemma}(\cite[Proposition~2.3]{ARS3}). When $\eth_w$ is a Hodge-deRham operator, the indicial roots of $I(x_iD_{\cC_i,s},\lambda)$ are given by 
\begin{equation}
\begin{gathered}
  \{ q-\frac{f_{i.s}-1}2, -\left(q-\frac{f_{i,s}+1}2\right), \; | \;  \ker_{L^2_w}\mathfrak{d}_{F_{i,s}} \quad \mbox{is not trivial in degree} \; q\} \\
  \bigcup \left\{ \ell\pm \sqrt{\zeta+(q-\frac{f_{i,s}-1}2)^2} \quad | \quad \ell \in \{0,1\}, \zeta\in \Spec(\delta^{F_{i,s}}d^{F_{i,s}})_q\setminus \{0\} \right\}  \\
  \bigcup \left\{ \ell\pm \sqrt{\zeta+(q-\frac{f_{i,s}+1}2)^2} \quad | \quad \ell \in \{0,1\}, \zeta\in \Spec(d^{F_{i,s}}\delta^{F_{i,s}})_q\setminus \{0\} \right\}.
\end{gathered}
\label{e.3g}\end{equation}
\label{e.3f}\end{lemma}
\begin{remark}
By Lemma~\ref{e.3f}, when $\eth_w$ is a Hodge-deRham operator, the condition of invertibility of the indicial family in Assumption~\ref{e.3} will be satisfied provided
\begin{equation}
 \left| q-\frac{f_{i,s}}2 \right|\le \frac12\quad \Longrightarrow \quad \ker_{L^2_w}\mathfrak{d}_{F_i,s} \; \mbox{is trivial in degree $q$}
\label{e.3i}\end{equation}
and the positive eigenvalues of $\delta^{F_{i,s}}d^{F_{i,s}}$ $d^{F_{i,s}}\delta^{F_{i,s}}$ are sufficiently large.  The latter condition can always be arranged by changing the metric $g_{w}$ via a suitable scaling of the family of wedge metrics $\kappa_{w,i,s}$, namely by requiring that $\kappa_{w,i,s}$ is sufficiently small with respect to some fixed family of wedge metrics.  When \eqref{e.3i} will hold, we will therefore usually assume that the family of wedge metrics $\kappa_{w,i,s}$ is \textbf{sufficiently small} for Assumption~\ref{e.3} to hold.
\label{e.3h}\end{remark}

Using Assumption~\ref{e.3}, it is possible to construct the following parametrix for $D_w$.
\begin{theorem}
Let $\eth_w$ be a wedge Dirac operator associated to an exact wedge metric and satisfying Assumption~\ref{e.3}.  Then for $\lambda\in\bbC$ and $a\in [-\frac12,\frac12]$, there exist $Q_{\lambda,a}\in \Psi^{-1,\cQ/\mathfrak{q}}_{e,\cn}(M;E)$ and $R_{\lambda,a}\in \Psi^{-\infty,\cR/\mathfrak{r}}_{e,\cn}(M;E)$ such that
$$
       v^{\frac12-a}(D_w-\lambda)v^{\frac12+a}Q_{\lambda,a}=\Id-R_{\lambda,a},
$$
where $\cQ$ and $\cR$ are index families depending on $a$, but not on $\lambda$,  such that for each boundary hypersurface $H_i$,  
$$
\begin{aligned}
\inf \Re(\cQ|_{\ff_i^e})\ge 0, \quad \inf \Re(\cQ|_{H_{i0}^e})>\frac12-a, \quad \inf \Re(\cQ|_{H_{0i}^e})>\dim S_i+\frac12+a, \\
\cR|_{\ff^e_i}=\bbN_0+1, \quad \inf\Re(\cR|_{H^e_{i0}})>\frac12-a \quad \mbox{and} \quad  \inf \Re(\cR|_{H_{0i}^e})>\dim S_i+\frac12+a,
\end{aligned}
$$
while $\mathfrak{q}$ and $\mathfrak{r}$ are multiweights such that
$$
\begin{aligned}
\mathfrak{q}(\ff_i^e)> 0, \quad \mathfrak{q}(H_{i0}^e)>\frac12-a, \quad \mathfrak{q}(H_{0i}^e)>\dim S_i+\frac12+a, \\
\mathfrak{r}(\ff^e_i)>0, \quad \mathfrak{r}(H^e_{i0})>\frac12-a \quad \mbox{and} \quad  \mathfrak{r}(H_{0i}^e)>\dim S_i+\frac12+a.
\end{aligned}
$$
Furthermore, if Assumption~\ref{e.3cn} also holds, then in fact  $Q_{\lambda,a}\in \Psi^{-1,\cQ}_{e}(M;E)$ and $R_{\lambda,a}\in \Psi^{-\infty,\cR}_{e}(M;E)$ and we can enforce that  $\cR|_{\ff^e_i}=\cR|_{H^e_{i0}}=\emptyset$.  
\label{e.4}\end{theorem}
\begin{proof}
As stated, this theorem is not quite in \cite{AG}, but the main ideas for its proof are.  Thus, we will mostly focus on the small complementary details needed for a complete proof.  First, the edge operator 
$$
v^{\frac12-a}(D_w-\lambda)v^{\frac12+a}
$$ 
has the same principal symbol and normal operators as $v^{\frac12-a}D_w v^{\frac12+a}$.  By Assumption~\ref{e.3} and \cite{ALMP2012,AG}, its normal operators are invertible, so there exists a parametrix $Q_1\in \Psi^{-1,\cQ_1/\mathfrak{q}_1}_{e,\cn}(M;E)$ such that 
$$
     v^{\frac12-a}(D_w-\lambda)v^{\frac12+a}Q_1=\Id-R_1
$$
for some $R_1\in \Psi^{-\infty,\cR_1/\mathfrak{r}_1}_{e,\cn}(M;E)$, where the index families $\cQ_1$ and $\cR_1$ are such that
$$
\begin{aligned}
      \cQ_1|_{\ff^e_i}=\bbN_0, \quad \inf\Re (\cQ_1|_{H^e_{i0}})>\frac12-a, \quad \inf\Re (\cQ_1|_{H^e_{0i}})>\dim S_i+\frac12+a \\
       \cR_1|_{\ff^e_i}=\bbN_0+1, \quad \inf\Re (\cR_1|_{H^e_{i0}})>\frac12-a, \;\mbox{and} \; \inf\Re (\cR_1|_{H^e_{0i}})>\dim S_i+\frac12+a
\end{aligned}
$$
and the multiweights $\mathfrak{q}_1$ and $\mathfrak{r}_1$ are such that
$$
\begin{aligned}
\mathfrak{q}_1(\ff_i^e)> 0, \quad \mathfrak{q}_1(H_{i0}^e)>\frac12-a, \quad \mathfrak{q}_1(H_{0i}^e)>\dim S_i+\frac12+a, \\
\mathfrak{r}_1(\ff^e_i)>0, \quad \mathfrak{r}_1(H^e_{i0})>\frac12-a \quad \mbox{and} \quad  \mathfrak{r}_1(H_{0i}^e)>\dim S_i+\frac12+a.
\end{aligned}
$$
Here, we have used induction on depth (with inductive step completed with Corollary~\ref{e.7} below) to determine $Q_1$ at $\ff^e_i$ for each boundary hypersurface $H_i$.  If we also assume that Assumption~\ref{e.3cn} holds, then using induction on depth  (with inductive step completed with Corollary~\ref{e.7} below), we can in fact assume that $Q_1\in \Psi^{-1,\cQ_1}_{e}(M;E)$ and $R_1\in \Psi^{-\infty,\cR_1}_{e}(M;E)$.  In this case, using the invertibility of \eqref{e.3a} and proceeding in increasing order with respect to partial order on the boundary hypersurfaces, we can remove the expansion of the error term $R_1$ at $H^e_{i0}$ proceeding essentially as in \cite[Lemma~5.44]{MelroseAPS}.  Indeed, if $H_i<H_j$ then once the expansion has been removed at $H_{i0}^e$, the term that we add to $Q_1$ at $H^e_{j0}$ to eliminate terms in the expansion of $R_1$ at $H^e_{j0}$ vanishes rapidly at $H^e_{i0}$, so do not compromise the decay already achieved at $H^e_{i0}$.  Hence, we find $Q_2\in \Psi^{-1,\cQ_2}_e(M;E)$ and $R_2\in\Psi^{-\infty,\cR_2}_e(M;E)$ such that 
$$
       v^{\frac12-a}(D_w-\lambda)v^{\frac12+a}Q_2=\Id-R_2
$$
with $\cQ_2$ and $\cR_2$ satisfying the same properties as $\cQ_1$ and $\cR_1$, but with $\cR_2|_{H_{i0}^e}=\emptyset$ for each boundary hypersurface $H_i$.  Using \eqref{smb.27b}, one can then construct $S\in \Psi^{-1,\cS}_e(M;E)$ with $\cS$ satisfying the same properties as $\cR_2$ and with $S\sim \sum_{j=1}^{\infty} R_2^j$, so that
$$
     (\Id-R_2)(\Id+S)=\Id- R_{\lambda,a}
$$
with $R_{\lambda,a}$ as in the statement of the theorem.  It suffices then to take 
$$
    Q_{\lambda,a}:= Q_2(\Id+S)
$$
to get the result.  In this construction, the index families of $Q_{\lambda,a}$ and $R_{\lambda,a}$ are determined by the normal operators of $v^{\frac12-a}(D_w-\lambda)v^{\frac12+a}$.  Since these are the same as those of $v^{\frac12-a}D_wv^{\frac12+a}$, this means that these indicial families do not depend on $\lambda$ as claimed.
\end{proof}
Denote by $H^k_e(M;E)$ the $L^2$ Sobolev space of order $k$ associated to the edge metric $g_e= \frac{g_w}{v^2}$ and set
$$
     \kridx{H^k_w}{Hw}{wedge Sobolev space}(M;E):=  v^{-\frac{n}2}H^k_e(M;E) \quad \mbox{where} \quad n=\dim M.
$$ 
\begin{corollary}
The operator $\eth_w$ of Theorem~\ref{e.4} induces a Fredholm operator
$$
        (\eth_w-\lambda): v^{a+\frac12}H^1_w(M;E)\to v^{a-\frac12}L^2_w(M;E)  
$$
of index zero for each $\lambda\in \bbC$ and $a\in [-\frac12,\frac12]$.
\label{e.5}\end{corollary}
\begin{proof}
This corresponds to inverting the conjugated operator $v^{\frac12-a}(D_w-\lambda) v^{\frac12+a}$ modulo compact operators.  By \cite[Corollary~3.8]{AG}, the parametrix $Q_{\lambda,a}$ of Theorem~\ref{e.4} provides a right inverse.  On the other hand, the parametrix $Q_{\overline{\lambda},-a}$ is such that
$$
      (v^{\frac12+a}(D_w-\overline{\lambda}) v^{\frac12-a})Q_{\overline{\lambda},-a}=\Id -R_{\overline{\lambda},-a},
$$
so taking its adjoints yields a left inverse for $v^{\frac12-a}(D_w-\lambda) v^{\frac12+a}$ modulo compact operators
$$
    Q_{\overline{\lambda},-a}^*(v^{\frac12-a}(D_w-\lambda) v^{\frac12+a})= \Id-R_{\overline{\lambda},-a}^*.
$$
Indeed, by \cite[Corollary~3.8]{AG}, the operator $R_{\overline{\lambda},-a}^*$ is compact when acting on 
$$
 \lrp{\prod_i x_i^{\frac{f_i+1}2}}H^1_w(M;E)= \lrp{\prod_i x_i^{\frac{f_i+1}2}}v^{-\frac{n}2}H^1_e(M;E)\subset L^2_b(M;E).
$$
Clearly, the cokernel of $v^{\frac12-a}(D_w-\lambda) v^{\frac12+a}$ is identified with the kernel of its formal adjoint $v^{\frac12+a}(D_w-\overline{\lambda}) v^{\frac12-a}$.  In particular, for $a=\lambda=0$, the index of $v^{\frac12-a}(D_w-\lambda) v^{\frac12+a}$ is zero.  By continuity of the index in the space of Fredholm operators, the index is therefore zero for all $a\in [-\frac12,\frac12]$ and $\lambda\in \bbC$.
\end{proof}

\begin{corollary}
if Assumption~\ref{e.3} holds, the eigensections of $D_w$ in $L^2_b(M;E)$ are contained in $v^{1+\epsilon}L^2_b(M;E)$ for some $\epsilon>0$.  Moreover, they are polyhomogeneous if Assumption~\ref{e.3cn} holds.
\label{e.6}\end{corollary}
\begin{proof}
Let $\psi\in L^2_b(M;E)$ be such that $D_w\psi=\lambda\psi$ for some $\lambda\in \bbC$.  The adjoint of the parametrix $Q_{\overline{\lambda},\frac12}$ of Theorem~\ref{e.4} is such that
$$
  Q_{\overline{\lambda},\frac12}^*v(D_w-\lambda) = \Id-R_{\overline{\lambda},\frac12}^*.
$$  
Hence, 
$$
\begin{aligned}
(D_w-\lambda)\psi=0  \quad & \Longrightarrow \quad Q_{\overline{\lambda},\frac12}^*v(D_w-\lambda)\psi=0  \quad \Longrightarrow \quad (\Id-R_{\overline{\lambda},\frac12}^*)\psi=0 \\
  & \Longrightarrow \quad \psi= R_{\overline{\lambda},\frac12}^*\psi.
\end{aligned}
$$
Noticing that $\inf\Re (\cR|_{H^e_{0i}})-\dim S_i> \frac12+\frac12=1$ and  $\mathfrak{r}(H^e_{0i})-\dim S_i> \frac12+\frac12=1$ by Theorem~\ref{e.4}, we deduce by bootstrapping that $\psi\in v^{1+\epsilon}L^2_b(M;E)$ for some $\epsilon>0$ as claimed.  If Assumption~\ref{e.3cn} also holds, then the equation   $\psi= R_{\overline{\lambda},\frac12}^*\psi$ combined with Theorem~\ref{e.4} shows that  $\psi$ is polyhomogeneous with index set at $H_i$ corresponding to $(\cR|_{H^e_{0i}}-\dim S_i)$ with $\cR$ the index family of $R_{\overline{\lambda},\frac12}$.  
\end{proof}

Let $P_{\lambda,a}: L^2_b(M;E)\to \ker_{L^2_b}(v^{\frac12-a}(D_w-\lambda)v^{\frac12+a})$ be the orthogonal projection onto the kernel of $v^{\frac12-a}(D_w-\lambda)v^{\frac12+a}$ for $a\in [-\frac12,\frac12]$.  By Corollary~\ref{e.5}, this kernel is finite dimensional and the cokernel of  $v^{\frac12-a}(D_w-\lambda)v^{\frac12+a}$ is identified with the kernel $\ker_{L^2_b}(v^{\frac12+a}(D_w-\overline{\lambda})v^{\frac12-a})$ of its formal adjoint. By Corollary~\ref{e.6}, the projection $P_{\lambda,a}$ is in $\Psi^{-\infty,\tau}_{e,\res}(M;E)$ with $\tau> 1-(\frac12+a)=\frac12-a$.  If Assumption~\ref{e.3cn} holds, it is also very residual.  
\begin{corollary}
If Assumption~\ref{e.3} holds, there exists $G_{\lambda,a}\in\Psi^{-1,\cG/\mathfrak{g}}_{e,\cn}(M;E)$ such that 
$$
\begin{aligned}
  v^{\frac12-a}(D_w-\lambda)v^{\frac12+a}G_{\lambda,a}& = \Id- P_{\overline{\lambda},-a},  \\
  G_{\lambda,a}v^{\frac12-a}(D_w-\lambda)v^{\frac12+a} &= \Id- P_{\lambda,a},
\end{aligned}
$$
with index family $\cG$ and multiweight $\mathfrak{g}$ depending on $a$, but not on $\lambda$, such that 
\begin{equation}
     \inf\Re(\cG|_{\ff^e_i})\ge 0, \quad \inf \Re(\cG|_{H_{i0}^e})>\frac12-a \quad \mbox{and} \quad \inf\Re(\cG|_{H^e_{0i}})>\dim S_i+\frac12+a,
\label{e.7b}\end{equation}
and
\begin{equation}
     \mathfrak{g}(\ff^e_i)> 0, \quad \mathfrak{g}(H_{i0}^e)>\frac12-a \quad \mbox{and} \quad \mathfrak{g}(H^e_{0i})>\dim S_i+\frac12+a.
\label{e.7bcn}\end{equation}
Furthermore, if Assumption~\ref{e.3cn} holds, then in fact $G_{\lambda,a}\in\Psi^{-1,\cG}_{e}(M;E)$.
\label{e.7}\end{corollary}
\begin{proof}
This can be proved using the approach of \cite[Theorem~4.20]{MazzeoEdge}.
By Corollary~\ref{e.5}, the operator $G_{\lambda,a}$ exists at least as a bounded operator.  If $Q_{\lambda,a}$ is the parametrix of Theorem~\ref{e.4}, then
\begin{equation}
\begin{aligned}
  G_{\lambda,a}&= G_{\lambda,a}\Id = G_{\lambda,a}((v^{-a+\frac12}(D_w-\lambda)v^{a+\frac12})(Q_{\lambda,a}) + R_{\lambda,a}) \\
    &= (\Id- P_{\lambda,a})(Q_{\lambda,a})+ G_{\lambda,a} (R_{\lambda,a}).
\end{aligned}  
\label{e.8}\end{equation}
Similarly, the adjoint of the parametrix $Q_{\overline{\lambda},-a}$ of Theorem~\ref{e.4} is such that
$$
   Q_{\overline{\lambda},-a}^* v^{\frac12-a}(D_w-\lambda)v^{\frac12+a}= \Id-R^*_{\overline{\lambda},-a},
$$
so
\begin{equation}
\begin{aligned}
 G_{\lambda,a} &= \Id G_{\lambda,a}= [Q^*_{\overline{\lambda},-a}(v^{\frac12-a}(D_w-\lambda)v^{\frac12+a})+ R^*_{\overline{\lambda},-a}] G_{\lambda,a} \\
   &=  Q^*_{\overline{\lambda},-a}(\Id-P_{\overline{\lambda},-a})+ R^*_{\overline{\lambda},-a}G_{\lambda,a}.
\end{aligned}
\label{e.9}\end{equation}
Plugging \eqref{e.9} in \eqref{e.8} therefore yields
\begin{equation}
G_{\lambda,a}= (\Id- P_{\lambda,a})(Q_{\lambda,a})+ Q^*_{\overline{\lambda},-a}(\Id-P_{\overline{\lambda},-a})R_{\lambda,a}+ R^*_{\overline{\lambda},-a}G_{\lambda,a}R_{\lambda,a}.
\label{e.10}\end{equation}
Since $G_{\lambda,a}$ is a bounded operator on $L^2_b(M;E)$ while $R_{\lambda,a}$ and $R^*_{\overline{\lambda},-a}$ are in $\Psi^{-\infty,\tau}_{e,\res}(M;E)$ for some $\tau>0$, we conclude from \eqref{e.10} and the edge version of Proposition~\ref{com.8} that the operator $G_{\lambda,a}$ is a weakly conormal edge operator as claimed.  If Assumption~\ref{e.3cn} holds, then $R_{\lambda,a}$ and $R^*_{\overline{\lambda},-a}$ are very residual and we can conclude from \eqref{e.10} that in fact $G_{\lambda,a}\in\Psi^{-1,\cG}_{e}(M;E)$ as claimed.
\end{proof}

\begin{corollary}
The wedge operator $\eth_w$ is essentially self-adjoint with unique self-adjoint extension given by $vH^1_w(M;E)$.
\label{e.11}\end{corollary}
\begin{proof}
We can proceed as in \cite[Lemma~1.4 and 1.5]{ARS3} to show that $D_w$ is essentially self-adjoint.  Indeed, let $\cD_{\min}(\eth_w)$ and $\cD_{\max}(\eth_w)$ denote the minimal and maximal extensions of $\eth_w$.  Since $\CI_c(M;E)$ is dense in $vH^1_w(M;E)$,  
$$
        vH^1_w(M;E)\subset \cD_{\min}(\eth_{w})\subset \cD_{\max}(\eth_{w}),
$$
so the result will follow by showing that $\cD_{\max}(\eth_w)\subset vH^1_w(M;E)$.  Thus, let $u\in D_{\max}(\eth_w)$ be given.  By definition, $u\in L^2_w(M;E)$ and $f:=\eth_wu\in L^2_w(M;E)$.  In terms of the conjugated operator, this means that 
$$
\widetilde{u}:= \lrp{\prod_i x_i^{\frac{f_i+1}2}}u\in L^2_b(M;E)
$$  
is such that $\widetilde{f}:= D_w\widetilde{u}\in L^2_b(M;E)$.  
Let $G_{0,-\frac12}$ be the operator of Corollary~\ref{e.7}, so that
$$
    G_{0,-\frac12}vD_w= \Id- P_{0,-\frac12}.
$$
In this case,
$$
   \quad (G_{0,\frac12}v)\widetilde{f}=(\Id-P_{0,-\frac12})\widetilde{u} \quad
     \Longrightarrow \quad \widetilde{u}= (G_{0,\frac12}v)\widetilde{f} +P_{0,-\frac12}\widetilde{u}.
$$
By \eqref{e.7b}, \eqref{e.7bcn} and Corollary~\ref{e.6}, $G_{0,\frac12}v$ and $P_{0,-\frac12}$ map $L^2_b(M;E)$ into 
$$
   v\lrp{\prod_i x_i^{\frac{f_i+1}2}}H^1_w(M;E). 
$$
 Hence,
$$
 u=   \lrp{\prod_i x_i^{-\frac{f_i+1}2}}\widetilde{u}=    \lrp{\prod_i x_i^{-\frac{f_i+1}2}}\left(  (G_{0,-\frac12}v)\widetilde{f} +P_{0,-\frac12}\widetilde{u}\right)\in vH^1_w(M;E)
$$
as claimed.
\end{proof}

\begin{corollary}
The unique self-adjoint extension of $\eth_w$ has discrete spectrum.  
\label{e.13}\end{corollary}
\begin{proof}
By Corollary~\ref{e.7}, $(D_w+i)^{-1}= v G_{i,\frac12}$ is a compact operator, hence the spectrum of $D_w$ and $\eth_w$ must be discrete.
\end{proof}

\begin{remark}
Corollary~\ref{e.7} completes the inductive step in the proof of Theorem~\ref{e.4}.  Proceeding by induction on the depth, we see from Corollary~\ref{e.11} that essential self-adjointness in Assumption~\ref{e.3} automatically holds.
\label{e.14}\end{remark}

\section{Parametrix construction of Dirac $\QFB$ operators} \label{do.0}

The purpose of this section is  to construct a parametrix for a certain class of $\QFB$ Dirac operators which includes examples of Hodge-deRham operators.  When $M$ is a manifold with fibered boundary, such a construction can be found in \cite[\S~3]{KR0}, itself a generalization of \cite[Proposition~16]{HHM2004}.  We will follow the same strategy for Dirac operators associated to more general $\QFB$ metrics.  This will require substantial changes however, one of the important differences being that the models at infinity that need to be inverted will no longer have a discrete spectrum.  In fact, using results of \cite[\S~9]{KR0}, it is only in the depth two case that we will initially be able to check that all the required assumptions in the construction hold; see Theorem~\ref{do.62} below.  The subsequent sections of the paper will consist in generalizing \cite{KR0} to allow in  \S~\ref{HdR.0} to remove this restriction on the depth.     

Keeping this in mind, let $(M,\phi)$ be a manifold with fibered corners of dimension $n$ and let $g_{\QFB}$ be an associated  $\QFB$ metric.  
\begin{assumption}
For each boundary hypersurface $H_i$ of $M$, we suppose that the $\QFB$ metric is exact and such that
$$
    g_{\QFB}- \left( \frac{dv^2}{v^4}+ \frac{\phi_i^*g_{S_i}}{v^2}+\kappa_i\right)\in x_i^2\CI(M;S^2({}^{\QFB}T^*M))
$$
for a choice of compatible total boundary defining function $v$ possibly depending on $H_i$, 
where $g_{S_i}$ is a wedge metric on $S_i\setminus \pa S_i$ and $\kappa_i\in \CI(H_i,S^2({}^{\QFB}T^*(H_i/S_i))$ is a fiberwise family of $\QFB$ metrics seen as a $2$-tensor on $H_i$ via a choice of connection for the fiber bundle $\phi_i: H_i\to S_i$.  In particular, $\phi_i: H_i\setminus \pa H_i\to S_i\setminus \pa S_i$ is a Riemannian submersion for the metrics $\phi_i^*g_{S_i}+\kappa_i$ and $g_{S_i}$ on $H_i\setminus \pa H_i$ and $S_i\setminus \pa S_i$ respectively.  
\label{do.1}\end{assumption}      

The metric $g_{\QFB}$ induces a Clifford bundle associated to the $\QFB$ tangent bundle ${}^{\QFB}TM$.  Let $E\to M$ be a choice of Clifford module and let $\nabla^E$ be a choice of Clifford connection which is at the same time a ${}^{\QFB}T^*M$-valued connection in the sense of \cite[Definition~2.20]{ALN04}.  Let $\eth_{\QFB}\in\Diff^1_{\QFB}(M;E)$ be the corresponding Dirac operator.  An example to keep in mind is when $E=\Lambda^*({}^{\QFB}T^*M)$ with $\nabla^E$ the connection induced by the Levi-Civita connection of $g_{\QFB}$, in which case $\eth_{\QFB}$ is the Hodge-deRham operator of the $\QFB$ metric $g_{\QFB}$.  We want to consider the operator $\eth_{\QFB}$ as acting formally on the weighted $L^2$ space $v^{\delta}L^2_{\QFB}(M;E)$ for some $\delta\in \bbR$.

\begin{definition}
At $H_i$ a boundary hypersurface of $M$, the \textbf{vertical family} of $\eth_{\QFB}$, denoted by $\eth_{v,i}$, is the fiberwise family of operators $\eth_{v,i}\in \Diff^1_{\QFB}(H_i/S_i;E)$ obtained by restricting the action of $\eth_{\QFB}$ to $H_i$.
\label{su.3}\end{definition}

The vertical family $\eth_{v,i}$ can be seen as the vertical part of the normal operator $N_i(\eth_{\QFB})$.  Indeed, thinking of $N_i(\eth_{\QFB})$ as a ${}^{\phi}NS_i$ suspended family of $\QFB$ operators, we see that the vertical family
$$
          \eth_{v,i}= \widehat{N_i(\eth_{\QFB})}(0)
$$
is the Fourier transform of the normal operator in ${}^{\phi}NS_i$ evaluated at the zero section.  Alternatively, 
\begin{equation}
  N_i(\eth_{\QFB})= \eth_{v,i}+ \eth_{h,i}
\label{su.3b}\end{equation}
where $\eth_{h,i}$ is a family of fiberwise Euclidean Dirac operators in the fibers of ${}^{\phi}NS_i$.  Taking the Fourier transform in ${}^{\phi}NS_i$, we thus have that 
\begin{equation}
   \widehat{N_i(\eth_{\QFB})}(\xi)= \eth_{v,i} + \sqrt{-1}\cl(\xi),  \quad \xi\in {}^{\phi}N^*S_i,
\label{su.3c}\end{equation}
where $\cl(\xi)$ denotes Clifford multiplication by ${}^{\phi}N^*S_i$.  From \eqref{su.3c}, we obtain a simple criterion for the operator $\eth_{\QFB}$ to be fully elliptic.  
\begin{theorem}
If $\eth_{v,i}$ is a family of fully elliptic invertible $\QFB$ operators for each boundary hypersurface $H_i$, then $\eth_{\QFB}$ is a fully elliptic $\QFB$ operator.  Thus, Proposition~\ref{mp.3} and Theorem~\ref{mp.32}  apply. In particular 
$$
       \eth_{\QFB}: x^{\mathfrak{t}}H^{m+1}_{\QFB}(M;E)\to x^{\mathfrak{t}}H^m_{\QFB}(M;E)
$$
is Fredholm for each multiweight $\mathfrak{t}$ and $m\in \bbR$.   
\label{su.3d}\end{theorem}  
\begin{proof}
Using that $\cl(\xi)$ anti-commutes with $\eth_{v,i}$, we have that 
$$
          \widehat{N_i(\eth_{\QFB})}^2(\xi)= \eth_{v,i}^2+ |\xi|_{{}^{\phi}NS_i}^2
$$ 
where the norm in $\xi$ is computed in terms of the metric induced by $g_{\QFB}$ on ${}^{\phi}NS_i$.  By the invertibility of $\eth_{v,i}$ and its formal self-adjointness, which follows from \eqref{su.3b} and the formal self-adjointness of $\eth_{\QFB}$, we thus see that $N_{i}(\eth_{\QFB})^2$, and hence $N_i(\eth_{\QFB})$, is a family of invertible suspended $\QFB$ operators.  In other words, $\eth_{\QFB}$ is a fully elliptic $\QFB$ operator.   
\end{proof}
\begin{example}
Suppose that $M$ is spin, as well as all the fibers of each fiber bundle $\phi_i: H_i\to S_i$.  Suppose that for each fiber of each fiber bundle $\phi_i:H_i\to S_i$, the $\QFB$ metric induced by $g_{\QFB}$ has non-negative scalar curvature which is positive somewhere.  If $\eth_{\QFB}$ denotes the corresponding Dirac operator, then in this case, $\eth_{v,i}$ essentially corresponds to the family of Dirac operators in the fibers of $\phi_i: H_i\to S_i$.  Hence, proceeding by induction on the depth of the fibers of $\phi_i:H_i\to S_i$ and using Lichnerowicz formula as well as Theorem~\ref{su.3d}, we can show that $\eth_{v,i}$ is a family of  fully elliptic invertible $\QFB$ operators.  Hence, Theorem~\ref{su.3d} can be applied to $\eth_{\QFB}$ as well. 
\label{su.3e}\end{example}

For many interesting examples, notably the Hodge-deRham operator, the hypothesis of Theorem~\ref{su.3d} is not satisfied.  To be able to construct a good parametrix in these cases, we will rely on a weaker invertibility assumption on the vertical family $\eth_{v,i}$.   
Let $L^2_{\QFB}(H_i/S_i;E)$ be the Hilbert bundle corresponding to the fiberwise $L^2$ spaces on $\phi_i:H_i\to S_i$ specified by the fiberwise $\QFB$ metrics $\kappa_i$ and the bundle metric of $E$. 
 \begin{assumption}
For $H_i$ a boundary hypersurface, the vertical family $\eth_{v,i}$ is such that its fiberwise $L^2$ kernels in $L^2_{\QFB}(H_i/S_i;E)$ form a (finite rank) vector bundle $\ker \eth_{v,i}$ over $S_i$.  When $H_i$ is not maximal, we will also assume that 
\begin{equation}
\begin{aligned}
\ker\eth_{v,i} &\subset \left(\prod_{H_j>H_i}x_j^{\frac{\dim S_j+1-\dim S_i}2}\right)\cA_{\QFB,2}(H_i\setminus S_i;E) \\
 &\subset \left(\prod_{H_j>H_i}x_j^{\frac{\dim S_j+1-\dim S_i}2}\right)L^2_{b}(H_i\setminus S_i;E)=\left(\prod_{H_j>H_i}x_j^{\frac12}\right)L^2_{\QFB}(H_i\setminus S_i;E).
\end{aligned}
\label{su.7b}\end{equation}
\label{su.7}\end{assumption}
\begin{remark}
If $\eth_{\QFB}$ is an Hodge-deRham operator, then for each fiber of $\phi_i:H_i\to S_i$, the operator $\eth_{v,i}$ corresponds to a direct sum of Hodge-deRham operators.  In particular, if $H_i$ is maximal, then the dimension of its $L^2$ kernel is finite dimensional by Hodge theory.  Similarly, if $H_i$ is \textbf{submaximal}, that is, if the fibers of $\phi_i$ are manifolds with boundary, then the dimension of its $L^2$ kernel is finite dimensional and topological by \cite{HHM2004}.  Hence, in both cases, the assumption that the fiberwise $L^2$ kernel of $\ker\eth_{v,i}$ form a finite rank vector bundle $\ker \eth_{v,i}\to S_i$ is automatic. 
\label{su.7e}\end{remark}

Using the fiberwise $\QFB$ metric on fibers of $\phi_i: H_i\to S_i$ induced by $g_{\QFB}$ and the bundle metric $E$, we can define the family of fiberwise orthogonal $L^2$ projections 
\begin{equation}
   \Pi_{h,i}: \CI(S_i;L^2(H_i/S_i;E))\to \CI(S_i;\ker\eth_{v,i})
\label{su.8}\end{equation}
onto $\ker \eth_{v,i}$.  This can be used to define a natural indicial family.  
\begin{definition}
For $H_i$ a boundary hypersurface, the \textbf{indicial family} $\bbC\ni \lambda\mapsto I(\eth_{b,i},\lambda)\in \Diff^1_{w}(S_i;\ker\eth_{v,i})$ associated to $\eth_{\QFB}$ is defined for $u\in \CI(S_i;\ker\eth_{v,i})$ by
$$
   I(\eth_{b,i},\lambda)u:= \Pi_{h,i}\left( v^{-\lambda}(v^{-1}\eth_{\QFB})v^{\lambda}\widetilde{u} \right)|_{H_i}, 
$$
where $\widetilde{u}\in\CI(M,E)$ for $H_i$ maximal and $\widetilde{u}\in \left(\prod_{H_j>H_i}x_{j}^{\frac{\dim S_j+1-\dim S_i}2 }\right)\cA_{\QFB,2}(M;E)$, when $H_i$ is not maximal, is a smooth extension off $H_i$ of $u$ such that $\widetilde{u}|_{H_i}=u$.  
\label{su.8b}\end{definition}
\begin{lemma}
The indicial family $I(\eth_{b,i},\lambda)$ is well-defined, namely $I(\eth_{b,i},\lambda)$ does not depend on the choice of extension $\widetilde{u}$.  
\label{su.9}\end{lemma}
\begin{proof}
When $H_i$ is maximal, the proof is as in \cite[Lemma~3.5]{KR0} or \cite[Lemma~4.3]{ARS1} and relies on the fact that 
$$
           [\eth_{\QFB},x_{\max}]\in x_{\max}^2\CI(M;\End(E)).
$$       
For $H_i$ submaximal, the general strategy is also to proceed as in \cite[Lemma~3.5]{KR0}, but the fact that the fibers of $\phi_i: H_i\to S_i$ are not closed manifolds yields to some subtleties.   First, we see from condition ($\QFB$ 2) in the definition of $\QFB$ vector fields that 
\begin{equation}
    [\eth_{\QFB},v]\in v^2\CI(M;\End(E)).
\label{su.10}\end{equation}        
Now, if $\widetilde{u}_1$ and $\widetilde{u_2}$ are two choices of smooth extensions of $u$, then for some  
$$w\in \left( \prod_{H_j>H_i} x_{i}^{\frac{\dim S_j+1-\dim S_i}2}\right)\cA_{\QFB,2}(M;E)
$$ 
smooth in $x_i$, 
$$
     \widetilde{u}_1-\widetilde{u}_2= x_i w= v(\widetilde{v}_{i}^{-1}\rho_i^{-1}w), \quad \mbox{where} \; \rho_i= \prod_{H_j<H_i} x_j\quad \widetilde{v}_i= \prod_{H_j>H_i}x_j.
$$
Hence, we compute that 
\begin{equation}
\begin{aligned}
\left(v^{-\lambda}(v^{-1}\eth_{\QFB})v^{\lambda}(\widetilde{u}_1-\widetilde{u}_2)\right)|_{H_i} &= 
     \left( v^{-\lambda-1}\eth_{\QFB}v^{\lambda+1} \widetilde{v}_{i}^{-1}\rho_i^{-1}w \right)|_{H_i} \\
     &= \left(\eth_{\QFB}(\widetilde{v}_{i}^{-1}\rho_i^{-1}w)\right)|_{H_i}+ \left(v^{-\lambda-1}[\eth_{\QFB},v^{\lambda+1}](\widetilde{v}_{i}^{-1}\rho_i^{-1}w)\right)|_{H_i} \\
     &=\left(\eth_{\QFB}(\widetilde{v}_{i}^{-1}\rho_i^{-1}w)\right)|_{H_i}+ (\lambda+1)\left( v^{-1}[\eth_{\QFB},v](\widetilde{v}_{i}^{-1}\rho_i^{-1}w) \right)|_{H_i} \\
     &= \eth_{v,i}\left(\widetilde{v}_{i}^{-1}\rho_i^{-1}w|_{H_i} \right).
     \end{aligned}
\label{su.11}\end{equation}
Now, by \eqref{su.3b}, the formal self-adjointness of $\eth_{\QFB}$ on $L^2_{\QFB}(M;E)$ implies the formal self-adjointness of $\eth_{v,i}$ on $L^2_{\QFB}(H_i/S_i;E)$.  Hence, given $\psi\in \rho_i\CI(S_i;\ker\eth_{v,i})$, thanks to \eqref{su.7b}, we can integrate by parts to obtain
\begin{equation}
\langle \psi, \eth_{v,i}(\widetilde{v}_{i}^{-1}\rho_i^{-1}w|_{H_i})\rangle_{L^2_{\QFB}(H_i/S_i;E)}=\langle \eth_{v,i}\psi, \widetilde{v}_{i}^{-1}\rho_i^{-1}w|_{H_i}\rangle_{L^2_{\QFB}(H_i/S_i;E)}=0.  
\label{su.12}\end{equation}
This implies that 
$$
\Pi_{h,i}\left(v^{-\lambda}(v^{-1}\eth_{\QFB})v^{\lambda}(\widetilde{u}_1-\widetilde{u}_2)\right)|_{H_i}=\Pi_{h,i}\eth_{v,i}\left(\widetilde{v}_{i}^{-1}\rho_i^{-1}w|_{H_i} \right)=0,
$$
which shows that the definition of the indicial family does not depend on the choice of smooth extension $\widetilde{u}$ as claimed.  
\end{proof}

To have a more precise description of the indicial family, we will suppose, as for the $\QFB$ metric, that the Clifford module $E$ is asymptotically modelled near $H_i$ on a Clifford module (also denoted $E$) for the model metric
\begin{equation}
  g_{\cC_{\phi_i}}:= \frac{dv^2}{v^4}+ \frac{\phi_i^*g_{S_i}}{v^2}+ \kappa_i \quad \mbox{on} \quad \kridx{\cC_{\phi_i}}{Ciphi}{model cone for $H_i$ in QFB setting}:= (0,\infty)_v\times (H_i\setminus\pa H_i).
\label{su.1}\end{equation}
Denote by $\eth_{\cC_{\phi_i}}$ the corresponding Dirac operator.  
\begin{assumption}
Near $H_i$ for $H_i$ a boundary hypersurface, we have that
$$
     \eth_{\QFB}-\eth_{\cC_{\phi_i}}\in x_i^2\Diff^1_{\QFB}(M;E).
$$
\label{su.2}\end{assumption}
 By Assumption~\ref{su.2}, the model operator $\eth_{\cC_{\phi_i}}$ has the same vertical family as $\eth_{\QFB}$.  Now, on $\cC_{\phi_i}$, notice that the fiber bundle 
$$
    \Id\times \phi_i: (0,\infty)_v\times (H_i\setminus \pa H_i) \to (0,\infty)_v\times (S_i\setminus \pa S_i)
$$
induces a Riemannian submersion from $(\cC_{\phi_i},g_{\cC_{\phi_i}})$ onto the cone $\cC_i= (0,\infty)_v\times (S_i\setminus\pa S_i)$ with cone metric
\begin{equation}
   g_{\cC_i}:= \frac{dv^2}{v^4}+ \frac{g_{S_i}}{v^2}.
\label{su.4}\end{equation}
By \cite[Proposition~10.12, Lemma~10.13]{BGV}, the Dirac operator $\eth_{\cC_{\phi_i}}$ takes the form
\begin{equation}
\eth_{\cC_{\phi_i}}= \eth_{v,i}+ \widetilde{\eth}_{\cC_i},
\label{su.5}\end{equation}
where $\eth_{v,i}$ is the vertical family acting on the fibers of $\Id\times \phi_i$ and $\widetilde{\eth}_{\cC_i}$ is a horizontal Dirac operator induced by the connection of $\Id\times \phi_i$ and the Clifford connection 
\begin{equation}
      \nabla^E+\frac{c(\omega)}{2},
\label{su.6}\end{equation}  
where $\omega$ is the (pull-back to $\cC_{\phi_i}$ of the) $\Lambda^2(T^*(H_i\setminus\pa H_i))$-valued form on $H_i$ of \cite[Definition~10.5]{BGV} associated to the Riemannian submersion $\phi_i: H_i\setminus \pa H_i\to S_i\setminus \pa S_i$ with respect the metrics $\phi_i^*g_{S_i}+ \kappa_i$ on $H_i\setminus \pa H_i$ and $g_{S_i}$ on $S_{i}\setminus\pa S_i$, while $c(\omega)$ is defined in terms of $\omega$ in \cite[Proposition~10.12 (2)]{BGV}.  

The indicial family is of course related with the horizontal operator $\widetilde{\eth}_{\cC_i}$.  Using the projection $\Pi_{h,i}$, we can first form the Dirac operator
\begin{equation}
  \eth_{\cC_i}= \Pi_{h,i}\widetilde{\eth}_{\cC_i}\Pi_{h,i}
\label{do.10}\end{equation}
with Clifford connection $\Pi_{h,i} (\nabla^E+ \frac{c(\omega)}{2})\Pi_{h,i}$.  The term $c(\omega)$ is defined in terms of the second fundamental form and the curvature of the fiber bundle $\phi_i: H_i\to S_i$.  Those depend only on the fiberwise metric, so really are pull-back of forms on $H_i$ via the projection $(0,\infty)\times H_i\to H_i$.  However, when measured with respect to $g_{\cC_{\phi_i}}$, that is, in terms of the $\QFB$ tangent bundle, the part involving the curvature is $\mathcal{O}(v_i^2)$ with $v_i=\prod_{H_j\ge H_i}x_j$.  In particular, this part will not contribute to the indicial family.  However, the part coming from the fundamental form is $\mathcal{O}(v_i)$, so does contribute to the indicial family. 

In terms of the decomposition 
\begin{equation}
\CI(\cC_i;\ker\eth_{v,i})\cong \CI((0,\infty))\widehat{\otimes}\CI(S_i; \ker\eth_{v,i}|_{\{1\}\times \cC_i})
\label{do.11}\end{equation}
induced by parallel transport along geodesics emanating from the tip of the cone, we can show, using \cite[Proposition~2.5]{Chou} as in \cite[(3.21)]{KR0}, that 
\begin{equation}
\eth_{\cC_i}= cv^2\frac{\pa}{\pa v} + v\left(  \eth_{S_i}-\frac{c\bd_i}2\right)+ v_i^2\cV_{\Omega_i},
\label{do.12}\end{equation}
where $\bd_i=\dim S_i$,  $c=\cl\left(\frac{dv}{v^2}\right)$ is Clifford multiplication by $\frac{dv}{v^2}$, $v_i^2\cV_{\Omega_i}$is the part of $\Pi_h \frac{c(\omega)}{2}\Pi_h$ coming from the curvature of $\phi_i:H_i\to S_i$, and $\eth_{S_i}=\widehat{\eth}_{S_i}+\cN_i$ with $\widehat{\eth}_{S_i}\in\Diff^1_{w}(S_i;\ker \eth_v)$ the Dirac operator induced by the connection 
\begin{equation}
\Pi_{h,i}\left(  \nabla^E+ \frac{\widehat{c(\omega)}}2\right)\Pi_{h,i}
\label{do.13}\end{equation}
with $\widehat{c(\omega)}$ the part of $c(\omega)$ involving the second fundamental form of the fiber bundle  $\phi_i: H_i\to S_i$, and 
$
\cN_i\in \CI(S_i;\End(\ker\eth_{v,i}))
$ 
is a self-adjoint operator anti-commuting with $c$.

Let us give more details about the contribution of  $\widehat{c(\omega)}$.  As mentioned above, the contribution from the second fundamental form is $\cO(v_i)$, but because of the factor of $v$ in front of $\eth_{S_i}$ in \eqref{do.12}, the contribution coming from $\widehat{c(\omega)}$ is in
\begin{equation}
         \frac{1}{\rho_i}\CI(S_i;\End(\ker\eth_{v,i}))\subset \frac{1}{\rho_i}\Diff^1_e(S_i;\ker\eth_{v,i})=: \Diff^1_{w}(S_i;\ker\eth_{v,i}),
\label{do.14}\end{equation}           
which is consistent with the fact $\widehat{\eth}_{S_i}$ is a wedge operator.  

Plugging \eqref{do.12} into the definition of indicial family of $\eth_{\QFB}$ at $H_i$, we thus find that 
\begin{equation}
   I(\eth_{b,i},\lambda)= c\lambda+ \left( \eth_{S_i}-\frac{c\bd_i}2\right).
\label{do.15}\end{equation}
Notice from this description that the indicial family $I(\eth_{b,i},\lambda)$ can be seen as the Mellin transform of the operator 
\begin{equation}
  \eth_{b,i}:=cv\frac{\pa}{\pa v}+ \eth_{S_i}-\frac{c\bd_i}2.
\label{do.15b}\end{equation} 
As in \cite[\S~3]{KR0}, if $\eth_{\QFB}$ is the Hodge-deRham operator of $g_{\QFB}$, then we can give a more explicit description of the self-adjoint term $\cN_i$ using \cite[\S~5.3]{HHM2004}.  First, in this case, 
\begin{equation}
\ker\eth_{v,i}= \Lambda^*({}^{\phi}N^*S_i)\otimes \cH^*_{L^2}(H_i/S_i),
\label{do.15c}\end{equation} 
where ${}^{\phi}N^*S_i$ is the dual of the vector bundle ${}^{\phi}NS_i$ in \eqref{vf.2} and $\cH^*_{L^2}(H_i/S_i)$ is the bundle of fiberwise $L^2$ harmonic forms, which by \cite[Proposition~15]{HHM2004} is a flat vector bundle with respect to the connection \eqref{do.13}.  Now, if $\overline{\eta}$ is a $\cH^*_{L^2}(H_i/S_i)$ valued $k$-form on $\cC$ obtained by parallel transport of its restriction $\eta$ to $\{1\}\times S_i$ along geodesics emanating from the tip of the cone $\cC_i$, then there is a decomposition
\begin{equation}
    \overline{\eta}= \overline{\alpha}+ \frac{dv}{v^2}\wedge\overline{\beta}, \quad \eta=\alpha+ dv\wedge \beta,  
    \quad \overline{\alpha}=\frac{\alpha}{v^k}, \quad \overline{\beta}=\frac{\beta}{v^{k-1}},
\label{do.16}\end{equation}
for some $\cH^*_{L^2}(H_i/S_i)$ valued forms $\alpha,\beta\in \CI(S_i;\Lambda^*({}^{w}T^*S_i)\otimes\cH^*_{L^2}(H_i/S_i))$.  In terms of such a decomposition,  Clifford multiplication by $\frac{dv}{v^2}$ is given by
$$
   c= \left( \begin{array}{cc} 0 & -1 \\ 1 & 0  \end{array} \right).
$$
More importantly, we infer from \cite[Proposition~15]{HHM2004} that the operator $\eth_{S_i}=\widehat{\eth}_{S_i}+\cN_i$ is such that 
\begin{equation}
  \widehat{\eth}_{S_i}= \left( \begin{array}{cc}  \mathfrak{d}_{S_i}  & 0 \\
                                                                   0 & -\mathfrak{d}_{S_i}  \end{array} \right)
\label{do.17}\end{equation}
 with $\mathfrak{d}_{S_i}$ the Hodge-deRham operator acting on $\CI(S_i;\Lambda^*({}^{w}T^*S_i)\otimes \cH^*_{L^2}(H_i/S_i)) $ and 
\begin{equation}
   \cN_i= \left(  \begin{array}{cc} 0 & \frac{\bd_i}2- \cN_{S_i} \\ \frac{\bd_i}2-\cN_{S_i} & 0\end{array}\right)
\label{do.18}\end{equation}
with $\cN_{S_i}$ the degree operator multiplying a form on $S_i$ (of pure degree) by its degree.

 Let $\w$ be the multiweight given by $\w(H_i)= \frac{h_i}2=\frac{\bd_i+1}2=\frac{\dim S_i+1}2$, so that $L^2_{\QFB}(M;E)= x^{\w}L^2_b(M;E)$. To work with an operator acting formally on $L^2_b(M;E)$, we will therefore consider the conjugated operator
\begin{equation}
      D_{\QFB}:= x^{-\w}\eth_{\QFB} x^{\w}.
\label{do.2}\end{equation}  
\begin{definition}
At $H_i$ a boundary hypersurface of $M$, the \textbf{vertical family} of $D_{\QFB}$, denoted by $D_{v,i}$, is the fiberwise family of operators $D_{v,i}\in \Diff^1_{\QFB}(H_i/S_i;E)$ obtained by restricting the action of $D_{\QFB}$ to $H_i$.
\label{do.3}\end{definition}

When $H_i$ is maximal, notice that the restriction of $\eth_{\QFB}$ yields the same vertical family as $D_{\QFB}$.     However, if $H_i$ is not maximal, this is no longer the case.  Indeed, if $H_j$ is a boundary hypersurface such that $H_i<H_j$ with no boundary hypersurface $H$ such that $H_i<H<H_j$, then notice that in the coordinates \eqref{coor.1}, the $\QFB$ vector field 
\begin{equation}
         \xi=v_j\left(x_j\frac{\pa}{\pa x_j}- x_i\frac{\pa}{\pa x_i}\right),
\label{do.4}\end{equation}
which restricts to $v_jx_j\frac{\pa}{\pa x_j}$ on $H_i$, commutes with $v_i$, but not with $x_i$.  More precisely, 
$$
(x_i^{-1}\circ\xi\circ x_i)-\xi=x_i^{-1}[\xi,x_i]=-v_j
$$ 
does not restrict to $0$ on $H_i$.  In fact, more generally, from \eqref{ds.4}, we see that  
$$
        v_q^{-1}\circ\xi'\circ v_{q}- \xi'
$$          
restricts to zero on $H_i$ for all $\xi'\in \cV_{\QFB}(M)$ if and only if $H_q\le H_i$.  Now, in terms of the coordinates \eqref{coor.1} and the convention that $\bd_0=0$, we have that
\begin{equation}
x^{\w}= v_1^{\frac12}\prod_q v_q^{\frac{\bd_q-\bd_{q-1}}2} = v_1^{\frac12}\left( \prod_{q\le i} v_q^{\frac{\bd_q-\bd_{q-1}}2} \right)\left(\prod_{q>i} v_q^{\frac{\bd_q-\bd_{q-1}}{2}}\right)=v_1^{\frac12}\left( \prod_{q\le i} v_q^{\frac{\bd_q-\bd_{q-1}}2} \right)
\left( \prod_{q>i} x_q^{\nu_q} \right)
\label{do.5}\end{equation} 
with 
\begin{equation}
    \nu_q:= \sum_{i<p\le q} \frac{\bd_p-\bd_{p-1}}2= \frac{\bd_q-\bd_{i}}2.
\label{do.6}\end{equation}
Hence, from this discussion, we see that if $\eth_{v,i}$ denotes the restriction of $\eth_{\QFB}$ to $H_i$, then in the coordinates \eqref{coor.1},
\begin{equation}
                    D_{v,i}= \left( \prod_{q>i} x_q^{\nu_q} \right)^{-1} \eth_{v,i}\left( \prod_{q>i} x_q^{\nu_q} \right). 
\label{do.7}\end{equation}
Then notice that the weights $\nu_q$ are precisely such that 
\begin{equation}
     L^2_{\QFB}(H_i/S_i;E)= \left( \prod_{q>i} x_q^{\nu_q} \right)L_b^2(H_i/S_i;E)
\label{do.8}\end{equation}
with $L^2_b(H_i/S_i;E)$ the corresponding Hilbert bundle of fiberwise $L^2$ spaces with respect to the family of fiberwise $b$-metrics conformally related to $\kappa_i$.  Indeed, it can be checked  readily that $\nu_q= \frac{\bd_{iq}+1}2$ with $\bd_{iq}$ the dimension of the base of the fiber bundle $\phi_q: H_q\to S_q$ restricted to $H_q\cap \phi_i^{-1}(s)$ for some $s\in S_i$.  Hence, from \eqref{do.7}, we see that the vertical family $\eth_{v,i}$   
 acting formally on $L^2_{\QFB}(H_i/S_i;E)$ is unitarily equivalent to the vertical family $D_{v,i}$ acting formally on $L^2_b(H_i/S_i;E)$.  In particular, if $H_i$ is submaximal, notice that the family of vertical metrics on the fibers of $\phi_i$ induced from $g_{\QFB}$ is just a family of fiber boundary metrics in the sense of \cite{Mazzeo-MelrosePhi,HHM2004}.    
 
 In terms of the vertical family $D_{v,i}$, the normal operator $D_{\QFB}$ is given by
 \begin{equation}
   N_i(D_{\QFB})=D_{v,i}+ \eth_{h,i}.
 \label{no.1}\end{equation}
 This is a family of suspended $\QFB$ operators.  We do not assume that it is fully elliptic, so we cannot invert it in the sense of Corollary~\ref{mp.4}.  In fact, one of the difficulty is to invert the normal operator on the range of the conjugated projection
\begin{equation}
   \widetilde{\Pi}_{h,i}:= \left( \prod_{q>i} x_q^{\nu_q}\right)^{-1} \Pi_{h,i} \left( \prod_{q>i} x_q^{\nu_q}\right)
\label{do.19}\end{equation}
corresponding to the fiberwise $L^2$ projection onto the kernel of $D_{v,i}$ in $L^2_b(H_i/S_i;E)$.  Of course, for $H_i$ maximal, we have that $L^2_{\QFB}(H_i/S_i;E)=L^2_b(H_i/S_i;E)= L^2(H_i/S_i;E)$ since the fibers of $\phi_i$ are closed manifolds, so we can simply take $\widetilde{\Pi}_{h,i}=\Pi_{h,i}$ in this case.    We will assume that the normal operator in \eqref{no.1} is invertible in the following sense.
\begin{assumption}
For each boundary hypersurface $H_i$, the normal operator \eqref{no.1} has a well-defined inverse
\begin{equation}
    G_{\ff_i}\in \Psi^{-1,\cE/\mathfrak{e}}_{\ff_i}(H_i;E),
\label{do.46b}\end{equation}
where $\cE$ is an index family which is the empty set except at $H_{ii}$, $H_{jj}$ and $\ff_j$ for $H_j$ such that $H_i<H_j$, where  
\begin{equation}
     \cE|_{H_{ii}}= \bd_i+\bbN_0,  \quad \cE|_{H_{jj}}=\bd_j+\bbN_0 \quad \mbox{and}\quad \cE|_{\ff_j}=\bbN_0,
\label{do.46c}\end{equation}
and $\mathfrak{e}$ a $\QFB$ positive multiweight except at $H_{jj}$ for $H_j>H_i$ where
$$
    \mathfrak{e}(H_{jj})>\bd_j.
$$  
  Furthermore , for $H_j\ge H_i$, we suppose that the term $A_j$ of order $\bd_j$ at $H_{jj}$
 is such that 
 \begin{equation}
 \widetilde{\Pi}_{h,j}A_j= A_j\widetilde{\Pi}_{h,j}=A_j
 \label{do.46d}\end{equation}
 with $A_i$ more precisely given by
 \begin{equation}
   A_i =\sqrt{-1}\cl(\xi)\widetilde{\Pi}_{h,i} \quad \mbox{for}  \; \xi\in S({}^{\phi}NS_i),
 \label{do.46dd}\end{equation}
 where $S({}^{\phi}NS_i)$, the unit vector bundle of ${}^{\phi}NS_i$, can be identified with the image of $H_{ii}\cap \ff_i$ under the fiber bundle induced by \eqref{ps.6}.
\label{do.46}\end{assumption}
\begin{remark}
Implicit in Assumption~\ref{do.46} is the invertibility of $\widetilde{v}_i^{-\frac12}D_{v,i}\widetilde{v}_i^{-\frac12}$ off its kernel and cokernel in the sense of Corollary~\ref{hd.4} (with $\delta=-\frac12$) below.  Indeed, the way we will ultimately establish Assumption~\ref{do.46} is by recurrence on the depth through Theorems~\ref{qfble.38} and \ref{ift.15} below and the proofs of these theorems rely on Corollary~\ref{hd.4}.  
\label{do.46e}\end{remark}

Even if we can invert the normal operator \eqref{no.1} in the sense of Assumption~\ref{do.46}, the  decay of the inverse at $H_{ii}$ is not enough to yield a compact error term.  To improve the situation at $H_{ii}$, we need to consider 
the indicial family of $D_{\QFB}$  at $H_i$ defined in terms of the conjugated projection \eqref{do.19} as follows.
\begin{definition}
The \textbf{indicial family} $\bbC\ni \lambda\mapsto I(D_{b,i},\lambda)\in \Diff^1_w(S_i;\ker D_{v,i})$ of $D_{\QFB}$ at $H_i$  is given by 
$$
      I(D_{b,i},\lambda)u:= \widetilde{\Pi}_{h,i}\left( v^{-\lambda}(v^{-1}D_{\QFB}v^{\lambda}\widetilde{u}) \right)|_{H_i},  \quad u\in\CI_c(S_i\setminus\pa S_i;\ker D_{v,i}),
$$
where $\widetilde{u}\in \widetilde{v}_{i}^{\frac12}\cA_{\QFB}(M;E)$ is a smooth extension off $H_i$ such that $\widetilde{u}|_{H_i}=u$.   
\label{do.20}\end{definition}
As for $I(\eth_{b,i},\lambda)$, the proof of Lemma~\ref{su.9} can be adapted to show that $I(D_{b,i},\lambda)$ does not depend on the choice of smooth extension $\widetilde{u}$.  Alternatively, from the definition of $D_{\QFB}$ and $\widetilde{\Pi}_{h,i}$, we have that
\begin{equation}
    I(D_{b,i},\lambda)= x^{-\w}I(\eth_{b,i},\lambda)x^{\w}.
\label{do.21}\end{equation}
Now, with $\rho_q=\prod_{H_p<H_q}x_p$, we see from \eqref{do.5} that
\begin{equation}
\begin{aligned}
x^{\w} &= v_1^{\frac12}\left( \prod_{q\le i} \left(\frac{v_1}{\rho_q}\right)^{\frac{\bd_q-\bd_{q-1}}2} \right) 
\left( \prod_{q>i} x_q^{\nu_q}\right)=v_1^{\frac{\bd_i+1}2}\left( \prod_{q\le i} \rho_q^{-\frac{\bd_q-\bd_{q-1}}2} \right) 
\left( \prod_{q>i} x_q^{\nu_q}\right) \\
&=v_1^{\frac{\bd_i+1}2}\left( \prod_{q\le i} x_q^{\frac{\bd_q-\bd_{i}}2} \right) 
\left( \prod_{q>i} x_q^{\nu_q}\right)= v_1^{\frac{\bd_i+1}2}\left( \prod_{q< i} x_q^{\frac{\bd_q-\bd_{i}}2} \right) 
\left( \prod_{q>i} x_q^{\nu_q}\right).
\end{aligned}
\label{do.22}\end{equation}
But  $\bd_i-\bd_q-1$ is precisely the dimension of the fibers of $\phi_{iq}:\pa_qS_i\to S_q$ of \eqref{iterated.1}.  In particular, we have that 
\begin{equation}
  L^2_w(S_i;\ker D_{v,i})= \left( \prod_{q<i}x^{\frac{\bd_q-\bd_i}2} \right)L^2_b(S_i;\ker D_{v,i}).
\label{do.23}\end{equation}
This allows us to conclude from \eqref{do.21} that $I(\eth_{b,i},\lambda+\frac{\bd_i+1}2)$ acting formally on $L^2_w(S_i;\ker \eth_{v,i})$ corresponds to $I(D_{b,i},\lambda)$ acting formally on $L^2_b(S_i;\ker D_{v,i})$.  Thus, in terms of \eqref{do.15}, we have that
\begin{equation}
    I(D_{b,i},\lambda)=c\lambda + \left( D_{S_i}+ \frac{c}2 \right),
\label{do.24}\end{equation}
where 
\begin{equation}
   D_{S_i}= \left( \prod_{H_q<H_i}x_q^{\frac{\bd_q-\bd_i}2} \right)^{-1}\left( \prod_{q>i}x_{q}^{\nu_q} \right)^{-1}\eth_{S_i}\left( \prod_{q>i}x_{q}^{\nu_q} \right)\left( \prod_{H_q<H_i}x_q^{\frac{\bd_q-\bd_i}2} \right). 
\label{do.25}\end{equation}
In particular, $\eth_{S_i}$ acting formally on $L^2_w(S_i;\ker\eth_{v,i})$ is unitarily equivalent to $D_{S_i}$ acting formally on $L^2_{b}(S_i;\ker D_{v,i})$.   As for \eqref{do.15}, the indicial family \eqref{do.24} can be seen as the Mellin transform of 
\begin{equation}
 D_{b,i}:= cv\frac{\pa}{\pa v} + \left(D_{S_i}+ \frac{c}2\right).
\label{do.25b}\end{equation}

To formulate our main assumption on the indicial family of $D_{\QFB}$, notice finally that $D_{S_i}$ anti-commutes with $c$, so that $cD_{S_i}$ is formally self-adjoint as a wedge operator acting formally on $L^2_b(S_i;E)$. 
\begin{assumption}
For $H_i$ a boundary hypersurface, we suppose that the wedge operator 
$$
cI(D_{b,i},0)=cD_{S_i}-\frac12
$$  
is essentially self-adjoint as a wedge operator acting formally on $L^2_b(S_i;\ker D_v)$.  In other words, we suppose that $c\eth_{S_i}$ is essentially self-adjoint as a wedge operator acting formally on $L^2_{w}(S_i;\ker\eth_{v,i})$.  Furthermore, we suppose that there is $\delta\in\bbR$ such that $cI(D_{b,i},\delta-\frac12+ \lambda)$ seen as a formal operator
\begin{equation}
    cI(D_{b,i},\delta-\frac12+\lambda): \rho_i^{\frac12}L^2_b(S_i;\ker D_{v,i})\to \rho_{i}^{-\frac12}L^2_b(S_i;\ker D_{v,i}),   
\label{do.26a}\end{equation}
is invertible for $\Re \lambda \in (-\mu_R,\mu_L)$ for some $\mu_L>\frac12$ and $\mu_R>\frac12$  
with inverse in $\rho_i^{\frac12}\Psi^{-1,\tau}_{e,\cn}(S_i;\ker D_{v,i})\rho_i^{\frac12}$ for some fixed $\tau>0$ not depending on $\lambda$.  
\label{do.26}\end{assumption}
The assumption that $cI(D_{b,i},0)$ is essentially self-adjoint allows us to define the indicial roots of $I(D_{b,i},\lambda)$ unambiguously.  
\begin{definition}
Assuming $cI(D_{b,i},0)$ is essentially self-adjoint, an \textbf{indicial root} of $I(D_{b,i},\lambda)$ is a complex number $\mu$ such that 
$cI(D_{b,i},\mu)$ is not invertible when acting on the unique self-adjoint extension of $cI(D_{b,i},0)$.  A \textbf{critical weight} of the indicial family $I(D_{b,i},\lambda)$ is a real number $\delta$ such that $\delta+i\mu$ is an indicial root for some $\mu\in\bbR$.    
\label{do.27}\end{definition}

\begin{remark}
Since $v^{-\frac12}D_{\QFB}v^{-\frac12}= v^{\frac12}(v^{-1}D_{\QFB})v^{-\frac12}$ is formally self-adjoint, notice that the set of indicial roots of $I(D_{b,i},\lambda)$ is symmetric with respect to the line $\Re\lambda=-\frac12$.  Hence, if $\zeta$ is an indicial root of $I(D_{b,i},\lambda)$, then so is $-1-\overline{\zeta}$.  
\label{do.26b}\end{remark}
When $\eth_{\QFB}$ is the Hodge-deRham operator, then by \eqref{do.15}, \eqref{do.17} and \eqref{do.18}, the indicial family $I(D_{b,i},\lambda)$ is given by 
\begin{equation}
I(D_{b,i},\lambda)=\left(  \begin{array}{cc} \mathfrak{D}_{S_i}  & -\lambda+\frac{\bd_i-1}2-\cN_{S_i} \\
 \lambda+\frac{\bd_i+1}2-\cN_{S_i} & -\mathfrak{D}_{S_i}   \end{array}\right)
\label{do.27b}\end{equation}
with 
\begin{equation}
      \mathfrak{D}_{S_i}= \left( \prod_{H_q<H_i}x_q^{\frac{\bd_q-\bd_i}2} \right)^{-1}\left( \prod_{q>i}x_{q}^{\nu_q} \right)^{-1}\mathfrak{d}_{S_i}\left( \prod_{q>i}x_{q}^{\nu_q} \right)\left( \prod_{H_q<H_i}x_q^{\frac{\bd_q-\bd_i}2} \right).
\label{do.27c}\end{equation}
In particular, the operator $cI(D_{b,i},0)$ will be essentially self-adjoint if and only if $\mathfrak{d}_{S_i}$ will be essentially self-adjoint on $L^2_w(S_i;\ker\eth_{v,i})$.  By Remark~\ref{e.3h}, this will be the case provided \eqref{e.3i} holds for $\eth_{S_i}$ and the wedge metric $g_{S_i}$ is suitably chosen, that is, the metrics on the links of each strata of the stratified space associated to $S_i$ are all sufficiently small.  In this case, the $L^2$ kernel $\ker_{L^2_w}\mathfrak{d}_{S_i}$ decomposes in terms of degree.  Let $\ker^q_{L^2_w}\mathfrak{d}_{S_i}$ be the part corresponding to forms of degree $q$.  Writing $\mathfrak{d}_{S_i}= d^{S_i}+ \delta^{S_i}$ with $d^{S_i}$ the deRham differential and $\delta^{S_i}$ its formal adjoint, notice that we can also use the Friedrich extension to have self-adjoint extensions of $\delta^{S_i}d^{S_i}$ and $d^{S_i}\delta^{S_i}$. In particular, we can define the spectrum of these operators using these self-adjoint extensions.  Denote by $(\delta^{S_i}d^{S_i})_q$ and $(d^{S_i}\delta^{S_i})_q$ the parts of $(\delta^{S_i}d^{S_i})$ and $(d^{S_i}\delta^{S_i})$ acting on forms of degree $q$.     
\begin{lemma}
When $\eth_{\QFB}$ is a Hodge-deRham operator, the indicial roots of $I(D_{b,i},\lambda)$ are given by
\begin{equation}
\begin{gathered}
  \{ q-\frac{\bd_i+1}2, -\left(q-\frac{\bd_i-1}2\right), \; | \;  \ker_{L^2_w}\mathfrak{d}_{S_i} \quad \mbox{is not trivial in degree} \; q\} \\
  \bigcup \left\{ \ell\pm \sqrt{\zeta+(q-\frac{\bd_i-1}2)^2} \quad | \quad \ell \in \{-1,0\}, \zeta\in \Spec(\delta^{S_i}d^{S_i})_q\setminus \{0\} \right\}  \\
  \bigcup \left\{ \ell\pm \sqrt{\zeta+(q-\frac{\bd_i+1}2)^2} \quad | \quad \ell \in \{-1,0\}, \zeta\in \Spec(d^{S_i}\delta^{S_i})_q\setminus \{0\} \right\}.
\end{gathered}
\label{do.27e}\end{equation}
\label{do.27d}\end{lemma}
\begin{proof}
It suffices to follow the argument of \cite[Proposition~2.3]{ARS3}.  However, since \cite[Proposition~2.3]{ARS3}
 consider the indicial family at the tip of the cone instead of at the infinite end, notice that the signs of our indicial roots are flipped with respect to \cite[Proposition~2.3]{ARS3}.
 \end{proof}

For $\delta$ as in Assumption~\ref{do.26}, we want to construct a parametrix for the operator $\eth_{\QFB}$ acting formally on $v^{\delta}L^2_{\QFB}(M;E)$, which by the discussion above is unitarily equivalent to the operator 
\begin{equation}
  D_{\QFB,\delta}:=v^{-\delta}D_{\QFB}v^{\delta}
\label{do.27}\end{equation}
acting formally on $L^2_{b}(M;E)$.

\begin{theorem}
Suppose that Assumptions \ref{do.1}, \ref{su.7}, \ref{su.2}, \ref{do.46} and \ref{do.26} hold.  Then there exist $Q\in\Psi^{-1,\cQ/\mathfrak{q}}_{\QFB,\cn}(M;E)$ and $R\in\Psi^{-\infty,\mathfrak{r}}_{\QFB,\cn}(M;E)$ such that 
\begin{equation}
     D_{\QFB,\delta}Q=\Id+R,
\label{do.28b}\end{equation}
where $\cQ$ is a $\QFB$ nonnegative index family, except at $H_{ii}$ for $H_i$ a boundary hypersurface, where we have $\cQ|_{H_{ii}}=\bd_i+\bbN_0$, while the multiweight $\mathfrak{r}$ is  $\QFB$ positive and such that
$$
  \mathfrak{r}(H_{i0})=\mu_L+\frac12>1, \quad \mathfrak{r}(H_{0i})=\bd_i+\frac12+\mu_R>\bd_i+1
$$ 
and
\begin{equation}
     \mathfrak{r}(\ff_i)=\mu_L+\mu_R>1, \quad  \mathfrak{r}(H_{ij})=\mu_L+\mu_R> \bd_j+2,
\label{ip.1a}\end{equation}
and the multiweight $\mathfrak{q}$ is such that
$$
\mathfrak{q}(H_{i0})=\mu_L-\frac12, \quad \mathfrak{q}(H_{0i})=\bd_i+\frac12+\mu_R>\bd_i+1, \quad \mathfrak{q}(\ff_i)>0, \quad \mathfrak{q}(H_{ij})> \left\{ \begin{array}{ll} b_i, & i=j, \\
                              b_j+1, & i\ne j\ne 0.   \end{array} \right.
$$
Moreover, the term $A_i$ of order $\bd_i$ at $H_{ii}$ of $Q$ is such that $\widetilde{\Pi}_{h,i}A_i= A_i\widetilde{\Pi}_{h,i}=A_i$.
\label{do.28}\end{theorem}
  \begin{remark}
  In \eqref{do.28b}, the fact that the error term decays faster than $\bd_i+1$ at $H_{ii}$ implies that $\widetilde{\Pi}_{h,i}Q$ has the same asymptotic behavior as $Q$ at $H_{ii}$ up to a term of order $\bd_i+1$ and a weakly conormal section vanishing faster than $x_i^{\bd_i+1+\nu}$ for some $\nu>0$.  The same remark applies to various other parametrix constructions in the paper, for instance the one in Assumption~\ref{do.46}.
  \label{do.28c}\end{remark}
\begin{remark}
Alternatively, the error term of Theorem~\ref{do.28} can be seen as a $\QFB$ $\mathfrak{r}'$ residual pseudodifferential operator with $\mathfrak{r}'(H\times M)=\mu_L+\frac12$ and $\mathfrak{r}'(M\times H)=\mu_R-\frac12$ for $H\in\cM_1(M)$,
$$
    R\in \Psi^{-\infty,\mathfrak{r}'}_{\QFB,\res}(M;E).
$$
\label{do.28d}\end{remark}

The proof of the theorem will follow the strategy of the proof of \cite[Theorem~3.9]{KR0} for fibered boundary operators, really a strategy that goes back to the work of Vaillant \cite{Vaillant} for the related class of fibered cusp operators.  However,  one major difference is that we were not able to obtain an analogue of Step 5 of the proof of \cite[Theorem~3.9]{KR0}, which would correspond in our setting to obtain rapid decay of the error term at $H_{i0}$ for all $i$.  It seems possible to ensure such decay at one of the boundary hypersurface $H_{i0}$, but we could not find a systematic way to proceed to enforce decay at all other such boundary hypersurfaces $H_{j0}$ without compromising the decay initially obtained at $H_{i0}$.  Notice that such an issue does not arise for the corresponding parametrix construction for edge operators in Theorem~\ref{e.4} when Assumption~\ref{e.3cn} hold, but in terms of the conditions imposed on the edge Dirac operator, the analogue in the $\QFB$ setting would be to ask that the Dirac operator  be fully elliptic, in which case Proposition~\ref{mp.3}  can be applied.  As already explain however, full ellipticity is too strict a condition for many interesting situations.  To palliate the fact that we were not able to obtain rapid decay of the error term at $H_{i0}$ for all $i$, milder conditions on the Dirac operator were imposed to ensure that the decay at $H_{i0}$ is at least `sufficiently good' to proceed with the remaining part of the parametrix construction and obtain ultimately a compact error term.  Because of this lack of decay, compared to \cite[Corollary~3.15]{KR0}, this is the main reason why we cannot obtained polyhomogeneity results.  More importantly, this is the main reason behind  the introduction of the larger class of weakly conormal $\QFB$ pseudodifferential operators, namely this lack of the decay of the error term will only allow us subsequently in Corollary~\ref{hd.4} to show that the inverse of the Dirac operator, when it exists, is within this larger class of pseudodifferential operators.  

 As for \cite[Theorem~3.9]{KR0}, the proof of Theorem~\ref{do.28} will be decomposed in few steps. 

\subsection*{Step 0: Symbolic inversion}

\begin{proposition}
There exists $Q_0\in\Psi^{-1}_{\QFB}(M;E)$ and $R_0\in\Psi^{-\infty}_{\QFB}(M;E)$ such that 
$$
        D_{\QFB,\delta}Q_0=\Id+R_0.
$$
\label{do.29}\end{proposition}
\begin{proof}
Since $D_{\QFB,\delta}$ is an elliptic $\QFB$ operator, we can take $Q_0'\in \Psi^{-1}_{\QFB}(M;E)$ with $\sigma_{-1}(Q_0')= \sigma_1(D_{\QFB,\delta})^{-1}$, so that 
$$
     D_{\QFB,\delta}Q_0'= \Id + R_0' \quad \mbox{with} \quad R_0'\in \Psi^{-1}_{\QFB}(M;E).
$$
Proceeding inductively, we can find more generally $Q_0^{(k)}= -Q_0'R_0^{(k-1)}\in \Psi^{-k}_{\QFB}(M;E)$ and $R_0^{(k)}\in \Psi^{-k}_{\QFB}(M;E)$ such that
$$
  D_{\QFB,\delta} \left( \sum_{j=1}^k Q_0^{(j)}\right)= \Id + R_0^{(k)}.
$$
Taking an asymptotic sum over the $Q_0^{(k)}$ gives the desired operators $Q_0$ and $R_0$.  
\end{proof}

\subsection*{Step 1: Inversion at $\ff_i$ for $H_i$ maximal}

In this step, we improve the parametrix by removing the error term at $\ff_i$ for $H_i$ maximal.  

\begin{proposition}
There exists $Q_1\in \Psi^{-1\cQ_1}_{\QFB}(M;E)$ and $R_1\in \Psi^{-\infty,\cR_1}_{\QFB}(M;E)$ such that 
$$
         D_{\QFB,\delta}Q_1=\Id+ R_1,
$$
where the index family $\cQ_1$ and $\cR_1$ are given by the empty set  except at $\ff_i$ for $H_i$ non-maximal, where 
$$
    \cQ_1|_{\ff_i}=\bbN_0, \quad \cR_1|_{\ff_i}=\bbN_0,
$$
and at $\ff_i$ and $H_{ii}$ for $H_i$ maximal, where
$$
   \cQ_1|_{\ff_i}=\bbN_0, \quad \cR_1|_{\ff_i}=\bbN_0+1, \quad \cQ_1|_{H_{ii}}= \bd_i+\bbN_0, \quad \cR_1|_{H_{ii}}= \bd_i+1+\bbN_0.
$$
Moreover, the leading terms $A$ and $B$ of $Q_1$ and $R_1$ at $H_{ii}$ for $H_i$ maximal are such that 
\begin{equation}
A=\widetilde{\Pi}_{h,i} A\widetilde{\Pi}_{h,i}, \quad \widetilde{\Pi}_{h,i}B= \widetilde{\Pi}_{h,i}B\widetilde{\Pi}_{h,i}.
\label{do.30b}\end{equation}
\label{do.30}\end{proposition}
\begin{proof}
We need to find the inverse of 
$$
    N_i(D_{\QFB,\delta})= N_i(\eth_{\QFB})= \eth_{v,i}+ \eth_{h,i}.
$$
To do this, we can decompose this operator using the fiberwise $L^2$ projection $\widetilde{\Pi}_{h,i}=\Pi_{h,i}$ onto the bundle $\ker\eth_{v,i}$.  Indeed, by \cite[Lemma~3.2]{KR0}, the operators $\eth_{v,i}$ and $\eth_{h,i}$ anti-commute, which implies that they both commute with $\widetilde{\Pi}_{h,i}$.  Now, on the $L^2$ orthogonal complement of $\ker\eth_{v,i}$, the operator $N_i(\eth_{\QFB})$ can be inverted as a family of ${}^{\phi}NS_i$ suspended pseudodifferential operators.  Denote the inverse by $q_1^{\perp}$.  On $\ker\eth_{v,i}$, the operator $N_i(\eth_{\QFB})$ corresponds to the Euclidean Dirac operator $\eth_{h,i}$ acting on sections of $\ker\eth_{v,i}$ seen as a bundle over ${}^{\phi}NS_i$.  The square of the this operator,
$$
    \Delta_{h,i}:= \eth_{h,i}^2,
$$ 
is just a family of Euclidean Laplacian acting in each fiber of ${}^{\phi}NS_i$ on a trivial vector bundle, namely the pull-back of $\ker \eth_{v,i}$ to that fiber.  The inverse is well-known to be the Green function
\begin{equation}
     G_{h,i}=\left\{  \begin{array}{ll} \frac{\Id_{\ker\eth_{v,i}}}{|u|^{\bd_i+1-2}_{{}^{\phi}NS_i}}= \frac{\Id_{\ker\eth_{v,i}}}{|u|^{\bd_i-1}_{{}^{\phi}NS_i}}, & \bd_i\ne 1, \\
     \Id_{\ker\eth_{v,i}}\log |u|_{{}^{\phi}NS_i}, & \bd_i=1,
     \end{array} \right.
\label{do.31}\end{equation}
where $|\cdot|_{{}^{\phi}NS_i}$ is the Euclidean norm induced by the $\QFB$ metric $g_{\QFB}$.  Correspondingly, the inverse of $\eth_{h,i}$ acting on sections of $\ker\eth_{v,i}$ is the composed operator,
\begin{equation}
      (\eth_{h,i})^{-1}= \eth_{h,i}G_{h,i}.
\label{do.32}\end{equation}
If $\eth_{h,i}=\sum_j \gamma^j\pa_{u_j}$ in local Euclidean coordinates in the fibers of ${}^{\phi}NS_i$, then in fact
\begin{equation}
  (\eth_{h,i})^{-1}= \sum_j \frac{\gamma^ju_j}{|u|_{{}^{\phi}NS_i}^{\bd_i+1}}= \frac{\Id_{\ker\eth_{v,i}}}{|u|_{{}^{\phi}NS_i}^{\bd_i+1}}\cl(u)
\label{do.32d}\end{equation}
with $\cl(u)$ denoting Clifford multiplication by $u$.  Denote this inverse by $q_1^o$.  Then, the combination $q_1=q_1^{\perp}+ q_1^o$ will be an inverse for $N_i(D_{\QFB,\delta})$.  Since $N_i(Q_0)$ is an inverse of $N_i(D_{\QFB,\delta})$ modulo operators of order $-\infty$, we see that $q_1$ and $N_i(Q_0)$ agree up to operators of order $-\infty$.  One important difference however is that $N_i(Q_0)$ decays rapidly in the fibers of ${}^{\phi}NS_i$, which is not the case for $q_1$.  In fact, because of the term $q_1^o$, the term $q_1$ will only decay like $|u|^{-\bd_i}_{{}^{\phi}NS_i}$ at infinity in the fibers of ${}^{\phi}NS_i$.  Hence, seen as a Schwartz kernel on $\ff_i$, this means that $q_1$ does not decay rapidly at $H_{ii}$, but has a term $A$ of order $\bd_i$.  Since it comes from $q_1^o$, this term is such that $\widetilde{\Pi}_{h,i}A\widetilde{\Pi}_{h,i}=A$.

Thus, it suffices to take $Q_1\in\Psi^{-1,\cQ_1}_{\QFB}(M;E)$ such that $N_{i}(Q_1)=q_1$ with difference $Q_1-Q_0$ an operator of order $-\infty$.  We can also choose $Q_1$ such that its term $A$ of order $\bd_i$ at $H_{ii}$ is such that 
\begin{equation}
\widetilde{\Pi}_{h,i}A\widetilde{\Pi}_{h,i}=A.
\label{do.32b}\end{equation}
Let $R_1$ be the error term such that
 \begin{equation}
  D_{\QFB,\delta}Q_1=\Id+ R_1.
\label{do.33}\end{equation}  
Then the claimed properties for $R_1$ follows from the composition result of Theorem~\ref{co.9} together with the fact that $N_i(Q_1)$ is the inverse of $N_{i}(D_{\QFB,\delta})$.  Thanks to \eqref{do.32b} and Assumptions \ref{do.1} and \ref{su.2}, we deduce that the leading term $B$ of $R_1$  at $H_{ii}$ is of order $\bd_i+1$ and such that \eqref{do.30b} holds.  

\end{proof}

\subsection*{Step 2: Inversion at $H_{ii}$ for $H_i$ maximal}

The parametrix $Q_1$ has an error term of order $\bd_i+1$ at $H_{ii}$ for $H_i$ maximal.  To obtain a better behavior at that face, we should choose more carefully the term $A$ of order $\bd_i$ of $Q_1$ at $H_{ii}$.  Since this term is such that $\widetilde{\Pi}_{h,i}A\widetilde{\Pi}_{h,i}=A$, we know by the proof of Lemma~\ref{su.9} that the way the operator $D_{\QFB,\delta}$ acts on it is essentially through its indicial family.  For $D_{\QFB,0}=D_{\QFB}$, the model operator acting on $A$ is $vD_{b,i}$ with $D_{b,i}$ given in \eqref{do.25b}. The corresponding model for $D_{\QFB,\delta}$ is the conjugated operator 
\begin{equation}
   vD_{b,i,\delta}=  v(v^{-\delta}D_{b,i}v^{\delta})= v^{\frac12}(v^{-(\delta-\frac12)}D_{b,i}v^{\delta-\frac12})v^{\frac12}=      v^{\frac12}\left( cv\frac{\pa}{\pa v}+ D_{S_i}+ c\delta\right)v^{\frac12}.
\label{do.34}\end{equation}  
This model operator can be seen as a $\QAC$ operator, so the corresponding $\Qb$ operator is 
\begin{equation}
\begin{aligned}
 D_{\Qb,i,\delta} &:= x_{\max}^{-\frac12}(vD_{b,i,\delta})x_{\max}^{-\frac12}= \left(\frac{v}{\rho_i}\right)^{-\frac12}(v^{\frac12}D_{b,i,\delta-\frac12}v^{\frac12})\left(\frac{v}{\rho_i}\right)^{-\frac12} \\
 &= \rho_{i}^{\frac12}D_{b,i,\delta-\frac12} \rho_i^{\frac12}
 = \rho_{i}^{\frac12}\left( cv\frac{\pa}{\pa v}+ D_{S_i}+ c\delta \right)\rho_i^{\frac12}.
 \end{aligned}
\label{do.35}\end{equation}
Taking the Mellin transform in $v$ yields the indicial family
\begin{equation}
 I(D_{\Qb,i,\delta},\lambda)= \rho_{i}^{\frac12}I(D_{b,i},\delta-\frac12+\lambda)\rho_i^{\frac12}.
\label{do.36}\end{equation}
By Assumption~\ref{do.26}, this is invertible for $|\Re \lambda|<\mu$ with inverse 
\begin{equation}
     I(D_{b,i},\delta-\frac12+\lambda)^{-1}\in \rho_i^{\frac12}\Psi^{-1,\tau}_{e,\cn}(S_i;\ker D_{v,i})\rho_i^{\frac12}
\label{do.37}\end{equation}
for some fixed $\tau>0$ not depending on $\lambda$.  Hence, taking the inverse Mellin transform for $\Re\lambda=0$ yields the following candidate for the inverse $D_{\QFB,\delta}$ at $H_{ii}$,
\begin{equation}
(D_{\Qb,i,\delta})^{-1}= \rho_i^{-\frac12}\left( \frac1{2\pi} \int_{-\infty}^{\infty} e^{i\xi\cdot \log s} I(D_{b,i},\delta-\frac12+i\xi)^{-1}d\xi\right)\rho_i^{-\frac12},
\label{do.38}\end{equation}
where $s=\frac{v}{v'}$ is the function of \eqref{ps.1}.  Since the indicial family \eqref{do.37} is invertible for all $\Re\lambda\in (-\mu_R,\mu_L)$, we see that not only \eqref{do.38} decays rapidly as $\log s\to \pm \infty$, but also 
\begin{equation}
\begin{gathered}
e^{-\nu_L\log s}(D_{\Qb,i,\delta})^{-1}\quad \mbox{decays rapidly as}\quad \log s\to -\infty, \\
e^{\nu_R\log s}(D_{\Qb,i,\delta})^{-1}\quad \mbox{decays rapidly as}\quad \log s\to +\infty,\end{gathered}
\label{do.38b}\end{equation}
for all $\nu_L<\mu_L$ and $\nu_R<\mu_R$.  
\begin{lemma}
The operator $(D_{\Qb,i,\delta})^{-1}$ is an element of $\Psi^{-1,\mathfrak{t}}_{\Qb,ii}(S_i;\ker D_{v,i})$, where $\mathfrak{t}$ is a $\Qb$ positive  multiweight such that $\mathfrak{t}(H_{i0}\cap H_{ii})=\mu_L$,  
$\mathfrak{t}(H_{0i}\cap H_{ii})=\mu_R$ (in terms of $\Qb$ density) and $\mathfrak{t}(H_{j0}\cap H_{ii})=\mu_L$,  
$\mathfrak{t}(H_{0j}\cap H_{ii})=\bd_j+1+\mu_R$ for $H_j< H_i$.  
\label{do.39}\end{lemma} 
\begin{proof}
Let $\diag^{\Qb}_i$ be the $p$-submanifold of \eqref{smb.30} seen as a $p$-submanifold of the $\Qb$ double space of the manifold with fibered corners $M_{S_i}:=S_i\times [0,1)_{x_i}$.  Let $\diag_{\Qb}$ denote the lift of the diagonal to the $\Qb$ double space $(M_{S_i})^2_{\Qb}$.  Let $\ff_{\Qb}$ be the union of the boundary hypersurfaces of $(M_{S_i})^2_{\Qb}$ intersecting $\diag_{\Qb}$ and let $H^{\Qb}_{ii}$ be the front face corresponding to the maximal boundary hypersurface $S_i\times\{0\}$.  

By \eqref{do.37} and Lemma~\ref{smb.31}, for each fixed $\Re \lambda\in(-\mu_R,\mu_L)$ the inverse in \eqref{do.37} is naturally a  conormal distribution on $\diag^{\Qb}_i$, namely $\rho_i^{-\frac12} \left(I(D_{b,i},\delta-\frac12+\lambda)  \right)^{-1}\rho_i^{-\frac12}$ is an element of
\begin{multline}
 \sA_{e,-}^{\mathfrak{t}}(\diag_i^{\Qb};\beta_{\Qb}^*(\Hom(E))\otimes \beta_{\Qb-e}^*\pr_R^*\Omega_e) \\+ \{u\in I^{-1}(\diag^{\Qb}_i;\diag^{\Qb}_i\cap\diag_{\Qb};\beta_{\Qb}^*(\Hom(E))\otimes \beta_{\Qb-e}^*\pr_R^*\Omega_e)\;| \; u\equiv 0 \;\mbox{on} \; \pa \diag_i^{\Qb}\setminus (\overline{\ff_{\Qb}\setminus H_{ii}^{\Qb}})\},
\label{pc.14e}\end{multline}
where $\beta_{\Qb-e}: \diag^{\Qb}_i\to (S_i)^2_e$ is the blow-down map induced by Lemma~\ref{smb.31} and $\mathfrak{t}=\beta_{\Qb-e}^{\#}(\mathfrak{s}_{\tau})$ is the natural pull-back to $\diag_i^{\Qb}$ of the multiweight $\mathfrak{s}_{\tau}$ of \eqref{smb.26} when we take $M=S_i$.  

Thus, the only question remaining is what happens when we integrate in $\xi$ in \eqref{do.38}.  First, by  \cite[Theorem~7.1]{KR3}, notice that there is a natural diffeomorphism

\begin{multline}
    H^{\Qb}_{ii}\cong [S_i^2\times \overline{\bbR^{+}_s}; S_i\times \cM_1(S_i)\times \{\infty\}, \cM_1(S_i)\times S_i\times \{0\}, \cM_1(S_i)\times \cM_1(S_i)\times \overline{\bbR^{+}_s},  \\
    (\phi_{j_1}\times\phi_{j_1})^{-1}(\cD^2_{S_{j_1}})\times \{0\},\ldots,(\phi_{j_k}\times\phi_{j_k})^{-1}(\cD^2_{S_{j_k}})\times \{0\}]
\label{diff.1}\end{multline}
with the blow-ups performed in any order compatible with the partial order and with $\pa_{j_1}S_i,\ldots,\pa_{j_k}S_i$ an exhaustive list of the boundary hypersurfaces of $S_i$ listed in a way compatible with the partial order.  Under this diffeomorphism, the $p$-submanifold $\cD_i^{\Qb}$ just corresponds to the lift of $S_i^2\times \{0\}$, so the inverse Fourier transform in \eqref{do.38} naturally yields a Schwartz kernel on $H^{\Qb}_{ii}$.  Let us emphasize that the diffeomorphism \eqref{diff.1} does not follow from an argument relying on the commutativity of blow-ups of transversal or nested $p$-submanifolds.  In fact, in the interior, the diffeomorphism does not even correspond to the identity map, making it harder in principle to track how  boundary hypersurfaces are identified under this diffeomorphism.  This is nevertheless possible and a precise description is also provided in \cite[\S~7]{KR3}.  Since it will be particularly important for our discussion, let us recall briefly how this goes.  First, in terms of this diffeomorphism,  the boundary hypersurface $H^{\Qb}_{j0}\cap H^{\Qb}_{ii}$ of $H^{\Qb}_{ii}$ corresponds to the lift of $\pa_j S_i\times S_i\times \{0\}$ in the right hand side of  \eqref{diff.1}, while $H^{\Qb}_{ji}\cap H^{\Qb}_{ii}$ corresponds to the lift of $\pa_j S_i\times S_i\times \overline{\bbR^{+}_s}$.  Similarly, $H^{\Qb}_{0j}\cap H^{\Qb}_{ii}$ corresponds to the lift of $S_i\times \pa_j S_i\times \{\infty\}$ in the right hand side of \eqref{diff.1}, while $H^{\Qb}_{ij}\cap H^{\Qb}_{ii}$ corresponds to the lift of $S_i\times \pa_j S_i\times \overline{\bbR^+_s}$.  These identifications are consistent with the facts that $s=0$ at $H_{j0}^{\Qb}$, but not at $H_{ji}^{\Qb}$, and that $s=\infty$ at $H^{\Qb}_{0j}$, but not at $H^{\Qb}_{ij}$.  On the other hand, for $(j, j')\notin\{(0,0),(i,i)\}$, the boundary hypersurface $H^{\Qb}_{jj'}\cap H^{\Qb}_{ii}$ corresponds to the lift of $\pa_jS_i \times \pa_{j'}S_i\times \overline{\bbR^+_s}$, while $\ff^{\Qb}_j\cap H^{\Qb}_{ii}$ corresponds to the lift of $(\phi_{j}\times\phi_{j})^{-1}(\cD^2_{S_j})\times \{0\}$.

By the discussion above, at faces where $s=0$  in $H_{ii}^{\Qb}$,  that is, at the faces $H^{\Qb}_{j0}\cap H^{\Qb}_{ii}$ for $H_{j}\cap H_i\ne \emptyset$, we obtain for all $\nu_L<\mu_L$ decay like $\rho^{\nu_L}$ for $\rho$ the corresponding boundary defining function.    while at faces where $\frac{1}{s}=0$, that is, at the faces $H^{\Qb}_{0j}\cap H^{\Qb}_{ii}$ for $H_{j}\cap H_i\ne \emptyset$, we obtain for all $\nu_R<\mu_R$ decay like $\rho^{\nu_R}$ .  At the faces $H^{\Qb}_{jj'}\cap H^{\Qb}_{ii}$ for $(j, j')\notin\{(0,0),(i,i)\}$, taking into account the change from edge to $\Qb$ density, we already get a decay of order $\tau+\widetilde{h}_{j'} $ by \eqref{do.37}, where $\widetilde{h}_j$ is defined in \eqref{max.1}

 For the remaining faces, notice first that symbolically, we know already that $D_{\Qb,i,\delta}$ is elliptic as an operator in $\Psi^1_{\Qb,ii}(S_i;\ker D_{v,i})$.  In fact, when we take a Fourier transform of $D_{\Qb,i,\delta}$ in the direction conormal to  $\diag_{i}^{\Qb}$, we should do it with respect to $\frac{\log s}{\rho_{ii}}$, the natural variable taking into account the blow up of $\Phi_j$ for $H_j<H_i$,  not $\log s$, where 
$$
  \rho_{ii}= \prod_{H_j<H_i} \rho_{\ff^{\Qb}_{ij}}
$$
with $\rho_{\ff^{\Qb}_{ij}}$ a boundary defining function for the lift $\ff^{\Qb}_{ij}$ of $(\phi_j\times \phi_j)^{-1}(\diag^2_{S_j})$ in $(M_{S_i})^2_{\Qb}$.  In terms of the dual variable, this means that the variable $\widetilde{\xi}= \rho_{ii} \xi$, not $\xi$, must be used to measure ellipticity, in agreement with \eqref{do.35}.  Because of this, notice that \eqref{pc.14e} is not quite the Fourier transform of a suspended edge operator of order $-2$.

In fact, this suggests to make the change of variable $\widetilde{\xi}=\rho_{ii}\xi$ in \eqref{do.38} to get
\begin{equation}
(D_{\Qb,i,\delta})^{-1}= \rho_i^{-\frac12}\left(  \frac{1}{2\pi}\int_{-\infty}^{\infty} e^{i\widetilde{\xi}\cdot\frac{\log s}{\rho_{ii}}} \left(I(D_{b,i},\delta-\frac12+ \frac{i\widetilde{\xi}}{\rho_{ii}})  \right)^{-1}\frac{d\widetilde{\xi}}{\rho_{ii}} \right) \rho_i^{-\frac12}.
\label{pc.14f}\end{equation}
The integral in $\widetilde{\xi}$ then gives rapid decay in $\frac{\log s}{\rho_{ii}}$.  Hence, at $\ff_{j}^{\Qb}\cap H^{\Qb}_{ii}$, we obtain the expected term of order zero plus decay of order $\tau$ from \eqref{do.37} and \eqref{pc.14f}.  Moreover, the rapid decay in $\frac{\log s}{\rho_{ii}}$ ensures that on 
$$
     [(S_i)^2;\cM_1(S_i)\times \cM_1(S_i)]\times \bbR_{\log s},
$$  
we only need to blow up the lift of $(\phi_j\times\phi_j)^{-1}(\diag^2_{S_j})\times \{0\}$ for $H_j<H_i$ to obtain a conormal distribution conormal (or polyhomogeneous) at all boundary hypersurfaces, that is, there is no need to blow up as well the lift of
$(\phi_j\times\phi_j)^{-1}(\diag^2_{S_j})\times\bbR_{\log s}$ for $H_j<H_i$, a place where $\rho_{ii}=0$.  Correspondingly, this means that no further blow-up is required on $H_{ii}^{\Qb}$ to give a complete conormal description of $(D_{\Qb,i,\delta})^{-1}$.  
\end{proof}

\begin{proposition}
There exists  $Q_2\in\Psi^{-1,\cQ_2/\mathfrak{q}_2}_{\QFB,\cn}(M;E)$ and $R_2\in\Psi^{-\infty,\cR_2/\mathfrak{r}_2}_{\QFB,\cn}(M;E)$ such that 
$$
   D_{\QFB,\delta}Q_2=\Id+R_2, \quad N_i(R_2)=0 \quad \mbox{for} \quad H_i \; \mbox{maximal},
$$
where $\mathfrak{q}_2$ and $\mathfrak{r}_2$ are $\QFB$ positive multiweights and $\cR_2$ and $\cQ_2$ are $\QFB$ nonnegative index families   except at $H_{ii}$ for $H_i$ maximal where 
$$
 \cQ_2|_{H_{ii}}= \bd_i+ \bbN_0.
$$
Moreover, for $H_i$ maximal and for a boundary hypersurface $H_j<H_i$, the index families $\cQ_2$ and $\cR_2$ are given by the empty set at $H_{j0}, H_{0j}$ and $H_{i0}, H_{0i}$, while the multiweights are given by
\begin{equation}
\begin{gathered}
\mathfrak{q}_2(H_{i0})=\mu_L-\frac12, \quad \mathfrak{r}_2(H_{i0})=\mu_L+\frac12, \quad \mathfrak{q}_2(H_{0i})=\mathfrak{r}_2(H_{0i})=\bd_i+\frac12+\mu_R, \\
   \mathfrak{q}_2(H_{j0})=\mu_L, \quad \mathfrak{r}_2(H_{j0})=\mu_L+1 \quad \mbox{and} \quad 
  \mathfrak{q}_2(H_{0j})=\mathfrak{r}_2(H_{0j})=\bd_j+1+\mu_R.
\end{gathered}  
\label{do.40a}\end{equation}
Finally, for $H_i$ maximal, the term $A_i$ of order $\bd_i$ at $H_{ii}$ of $Q_2$ is such that $\widetilde{\Pi}_{h,i}A_i=A_i\widetilde{\Pi}_{h,_i}=A_i$.
\label{do.40}\end{proposition}
\begin{proof}
The idea is to first take take $\widetilde{Q}_2$ as $Q_1$, except that at $H_{ii}$ for $H_i$ maximal, we choose explicitly $\widetilde{Q}_2$ to be at leading order 
\begin{equation}
       q_2:= x_{\max}^{-\frac12}(D_{\Qb,i,\delta})^{-1}x_{\max}^{-\frac12}.
\label{do.40b}\end{equation}
By Lemma~\ref{do.39} and taking into account the change from $\Qb$ densities to $\QFB$ densities, we see that $\widetilde{Q}_2$ is in $\Psi^{-1,\cQ_2/\mathfrak{q}_2}_{\QFB,\cn}(M;E)$ for $\cQ_2$ and $\mathfrak{q}_2$ as claimed.   To see it agrees with $Q_1$ at $\ff_i\cap H_{ii}$, notice that the leading order behavior of $x_{\max}^{-\frac12}(D_{\Qb,i,\delta})^{-1}x_{\max}^{-\frac12}$ at $\ff_i$ is the inverse Fourier transform of its principal symbol, that is, the inverse Fourier transform of the inverse of the principal symbol of $vD_{b,i,\delta}$, which has the same principle symbol as $vD_{b,i}$.  Thus, this is precisely $\frac{\cl(u)}{|u|_{{}^{\phi}NS_i}}$, the term of order $\bd_i$ of $Q_1|_{\ff_i}$ at $H_{ii}$.   

With this choice, we have that
$$
          D_{\QFB,\delta}\widetilde{Q}_2=\Id +\widetilde{R}_2
$$
with error term $\widetilde{R}_2$ as claimed except for the fact that it still has possibly a term $\widetilde{B}$ of order $\bd_i+1$ at $H_{ii}$ and the extra decay of the error term at $H_{i0}$ and $H_{j0}$.  For the latter, notice that $q_2$ is such that $\widetilde{\Pi}_{h,i}q_2=0$, so if we extend it off $H_{ii}$ by $\widetilde{q}_2$ so that we still have $\widetilde{\Pi}_{h,i}\widetilde{q}_2=0$ near $H_{i0}$ and $H_{j0}$, we see from Lemma~\ref{su.9} that $D_{\QFB,\delta}\widetilde{q}_2$ will vanish at order $\mu_L+\frac12$ at $H_{i0}$ and at order $\mu_L+1$ at $H_{j0}$ as claimed.  For the term $\widetilde{B}$ of order $\bd_i+1$ at $H_{ii}$, by the proof of Lemma~\ref{su.9}, this term $\widetilde{B}$ is such that
$\widetilde{\Pi}_{h,i}\widetilde{B}=0$.  Hence, subtracting the term $D_{v,i}^{-1}\widetilde{B}$ at order $\bd_i+1$ at $H_{ii}$ to $\widetilde{Q}_2$,  we can get a parametrix $Q_2$ with error term as claimed.  
\end{proof}

Assume now that $H_1,\ldots,H_{\ell}$ is a complete list of the boundary hypersurfaces of $M$ such that
$$
   H_i<H_j\; \Longrightarrow \; i<j.
$$
The goal of the next two steps will be to prove the following proposition by induction on $i$.
\begin{proposition}
There exist $Q_{3,i}\in \Psi^{-1,\cQ_{3,i}/\mathfrak{q}_{3,i}}_{\QFB,\cn}(M;E)$ and $R\in \Psi^{-\infty,\cR_{3,i}/\mathfrak{r}_{3,i}}_{\QFB,\cn}(M;E)$ such that 
$$
    D_{\QFB,\delta}Q_{3,i}=\Id+ R_{3,i} \quad \mbox{with} \quad N_j(R_{3,i})=0 \; \mbox{for} \; j>i,
$$
where $\cR_{3,i}$ is $\bbN_0$ at $\ff_j$ for all $j$ and is the empty set elsewhere, while $\cQ_{3,i}$ is $\QFB$ nonnegative index families except at $H_{jj}$ for $j>i$ where
$$
      \left.  \cQ_{3,i}\right|_{H_{jj}}=\bd_j+\bbN_0.
$$
Here, $\mathfrak{r}_{3,i}$ is a $\QFB$ positive multiweight with
$$
     \mathfrak{r}_{3,i}(H_{j0})=\mu_L+\frac12 \quad \mbox{and} \quad \mathfrak{r}_{3,i}(H_{0j})=\bd_j+\frac12+\mu_R \quad \mbox{for} \;  j>i,
$$
while $\mathfrak{q}$ is a multiweight which is $\QFB$ positive, except at $H_{jj}$ for $j>i$, where 
$$
       \mathfrak{q}_{3,i}(H_{jj})>\bd_j,
$$
and such that
$$
     \mathfrak{q}_{3,i}(H_{j0})=\mu_L-\frac12 \quad \mbox{and} \quad \mathfrak{q}_{3,i}(H_{0j})=\bd_j+\frac12+\mu_R \quad \mbox{for} \;  j>i.
$$
Moreover, for $j>i$, the term $A_j$ of order $\bd_j$ of $Q_{3,i}$ at $H_{jj}$ is such that $\widetilde{\Pi}_{h,j}A_j=A_j\widetilde{\Pi}_{h,j}=A_j$ and $\widetilde{\Pi}_{h,j} Q_{3,i}$ has the same asymptotic behavior as $Q_{3,i}$ at $H_{jj}$ up to a term of order $\bd_j+1$ and a weakly conormal section vanishing faster than $x_j^{\bd_j+1+\nu}$ for some $\nu>0$.
\label{hd.1}\end{proposition}

By \textbf{Step 2}, Proposition~\ref{hd.1} holds for $i$ whenever $H_i$ is maximal.  Thus, assuming it holds for some $i$, we will show in the next two steps that it holds for $i-1$.

\subsection*{Step 3: Inversion at $\ff_i$ for $H_i$ not maximal}

In this step, we will assume that $H_i$ is not maximal and that we have already a parametrix $Q_{3,i}$ with remainder $R_{3,i}$ as in Proposition~\ref{hd.1}.  We need to improve the parametrix so that its error term decays at $\ff_i$. By \eqref{su.3c} and \eqref{do.7}, the corresponding normal operator at this face is 
\begin{equation}
N_i(D_{\QFB,\delta})= D_{v,i}+ \eth_{h,i}.
\label{do.41}\end{equation}
Thanks to Assumption~\ref{do.46}, this normal operator is invertible, so that we can improve the parametrix as follows.
\begin{proposition}
Assuming there is a parametrix $Q_{3,i}$ as in Proposition~\ref{hd.1}, there exists 
$$Q_{4,i}\in \Psi^{-1,\cQ_{4,i}/\mathfrak{q}_{4,i}}_{\QFB,\cn}(M;E)\quad  \mbox{and} \quad R_{4,i}\in\Psi^{-\infty,\cR_{4,i}/\mathfrak{r}_{4,i}}_{\QFB,\cn}(M;E)
$$
 such that 
$$
    D_{\QFB,\delta}Q_{4,i}=\Id+R_{4,i} \quad \mbox{with} \; N_j(R_{4,i})=0 \; \mbox{for} \; j\ge i,
$$
where $\cQ_{4,i},\cR_{4,i}, \mathfrak{q}_{4,i}$ and $\mathfrak{r}_{4,i}$ are as $\cQ_{3,i}, \cR_{3,i}, \mathfrak{q}_{3,i}$ and $\mathfrak{r}_{3,i}$, except that 
$$
  \left. \cQ_{4,i}\right|_{H_{ii}}= \bd_i+\bbN_0 \quad \mbox{and}\quad \left.  \cR_{4,i}\right|_{H_{ii}}=\bd_i+1+\bbN_0.
$$
Moreover, for $j\ge i$, the term $A_j$ of order $\bd_j$ of $\cQ_{4,i}$ at $H_{jj}$ is such that $\widetilde{\Pi}_{h,j}A_j\widetilde{\Pi}_{h,j}=A_j$, while the term $B_i$ of order $\bd_i+1$ of $R_{4,i}$ at $H_{ii}$ is such that $\widetilde{\Pi}_{h,i}B= \widetilde{\Pi}_{h,i}B\widetilde{\Pi}_{h,i}$.  Also, for $j\ge i$, $\widetilde{\Pi}_{h,j}Q_{4,i}$ has the same asymptotic behavior as $Q_{4,i}$ at $H_{jj}$ up to a term of order $\bd_j+1$ and a weakly conormal section vanishing faster than $x_j^{\bd_j+1+\nu}$ for some $\nu>0$.
\label{hd.2}\end{proposition} 
\begin{proof}
 The idea then is to take $Q_{4,i}$ to be $Q_{3,i}$, except at $\ff_i$ where we take $Q_{4,i}|_{\ff_i}=G_{\ff_i}$.  Compared to $Q_{3,i}$, this forces $Q_{4,i}$ to have a term $A$ of order $\bd_i$ at $H_{ii}$ which we can assume is such that $\widetilde{\Pi}_{h,i}A\widetilde{\Pi}_{h,i}=A$.  We can apply the composition formula of Theorem~\ref{co.9} to deduce the corresponding properties for $R_{4,i}$ at $H_{ii}$.  Of course, to be able to do this substitution, we need to know that $G_{\ff_i}$ agrees  with $Q_{3,i}$ at $\ff_{i}\cap\ff_j$ and $\ff_i\cap H_{jj}$ for $H_j>H_i$, which follows from 
$$
    G_{\ff_i}= G_{\ff_i}\Id= G_{\ff_i}\circ N_i(D_{\QFB,\delta}Q_{3,i}-R_{3,i})= N_{i}(Q_{3,i})- G_{\ff_i}N_i(R_{3,i}).
$$
Indeed, setting $\cE_{\ff_i}=\cF_{\ff_i}=0$ in Theorem~\ref{co.9} yields a composition result for operators on $\ff_i$ which implies in particular that $G_{\ff_i}N_i(R_{3,i})$ has only terms of order strictly higher than $0$ and $\bd_j$ at $\ff_j\cap \ff_i$ and $H_{jj}\cap\ff_i$ respectively.  Notice that it may in particular introduce terms of order between $\bd_j$ and $\bd_j+1$ at $\ff_i\cap H_{jj}$, but by Remark~\ref{do.28c} applied to Assumption~\ref{do.46}, we can extend $G_{\ff_i}$ off $\ff_i$ in such a way that these do not create terms of order less or equal $\bd_j+1$ at $H_{jj}$ for the error term $R_{4,i}$.
  
\end{proof}

\subsection*{Step 4: Inversion at $H_{ii}$ for $H_i$ not maximal}

\begin{proof}[Proof of Proposition~\ref{hd.1} for each $i$]
We can now complete the inductive step of Proposition~\ref{hd.1}.  Indeed,
it suffices to improve the parametrix of Proposition~\ref{hd.2} to remove the term of order $\bd_i+1$ of the error term at $H_{ii}$.  In order to do this, we can proceed as in Step~2, using \eqref{do.46dd} to see that the term of order $\bd_i$ we need to add at $H_{ii}$ matches what we have so far.  However, in the last part of Step~2, namely at the end of the proof of Proposition~\ref{do.40}, we need to explain how to construct the right inverse $D^{-1}_{v,i}$ to $D_{v,i}$ used to get rid of the term $\widetilde{B}$.  By Remark~\ref{do.46e}, we can assume that $(\widetilde{v}^{-\frac12}_i D_{v,i}\widetilde{v}^{-\frac12})$ is invertible in the sense of Corollary~\ref{hd.4} below (with $\delta=-\frac12$).  We can now follow the approach of Guillarmou-Hassell \cite[\S~3.2]{GH2} as in \cite[(8.19)]{KR0} to obtain the desired right inverse for $D_{v,i}$ to complete Step~2.  Since such an argument is more naturally discussed in the study of the low energy limit of the resolvent, we refer to \eqref{qfble.12} below for a detailed discussion of  \cite[(8.19)]{KR0} adapted to the $\QFB$ setting.      This yields Proposition~\ref{hd.1} for $i-1$.  Proceeding by induction on $i$ and using Proposition~\ref{do.40}, we thus see that Proposition~\ref{hd.1} holds for all $i$ including $i=0$.
\end{proof}

\subsection*{Step 5: Final improvement of the error term}

\begin{proof}[Proof of Theorem~\ref{do.28}]
The parametrix of Proposition~\ref{hd.1} for $i=0$ is almost the one claimed in Theorem~\ref{do.28}.  The only part missing is a better decay rate of the error term at $H_{ij}$ and $\ff_{i}$ for $i,j\in\{1,\ldots,\ell\}$.  If $R_{3,0}$ is the remainder term given by Proposition~\ref{hd.1} for $i=0$, then by Theorem~\ref{co.9}, there exists $q\in\bbN$ such that 
$R^q_{3,0}\in \Psi^{-\infty,\cR_{3,0}/\mathfrak{r}_q}_{\QFB}(M;E)$ with $\mathfrak{r}_q$ satisfying \eqref{ip.1a} and such that
$$
          \mathfrak{r}_q(H_{i0})=\frac12+\mu_L>1, \quad \mathfrak{r}_q(H_{0i})=\bd_i+\frac12+\mu_R >\bd_i+1. 
$$ 
Hence, it suffices to replace by $Q_{3,0}$ by $\displaystyle Q_{3,0}\left( \sum_{k=0}^{q-1}(-R_{3,0})^k \right)$, since then
$$
  D_{\QFB,\delta}Q_{3,0}\left( \sum_{k=0}^{q-1}(-R_{3,0})^k \right)= (\Id +R_{3,0})\left( \sum_{k=0}^{q-1}(-R_{3,0})^k \right)= \Id +(-1)^{q-1}R^q_{3,0}.
$$
\end{proof}

The parametrix of Theorem~\ref{do.28} can be used  to show that $D_{\QFB,\delta}$ is Fredholm when acting on a suitable space.  To describe this space, consider the first order operator
\begin{equation}
     D_{\QFC,\delta}:= v^{-1}D_{\QFB,\delta}
\label{do.51}\end{equation}
associated to the metric 
$$
    g_{\QFC}= v^2 g_{\QFB}.
$$
By analogy with the relation between fibered boundary metrics and fibered cusp metrics, we say that $g_{\QFC}$ is a quasi-fibered cusp metric ($\QFC$ metric).  We refer to \cite{KR2} for a more detailed discussion on $\QFC$ metrics.
Let 
\begin{equation}
    M_{\QFC,\delta}:= \{\sigma\in L^2_b(M;E)\; | \; D_{\QFC,\delta}u\in L^2_b(M;E)\}
\label{do.52}\end{equation}
be the maximal closed extension of $D_{\QFC,\delta}$ acting on $L^2_b(M;E)$.  Since $g_{\QFC}$ is a complete metric, notice that $M_{\QFC}$ is also equal to the minimal closed extension, so $D_{\QFC}$ has only one closed extension over $L^2_b(M;E)$.  
\begin{corollary} 
If  Assumptions~\ref{do.1}, \ref{su.7}, \ref{su.2} and \ref{do.26} hold, then $D_{\QFB}$ induces a Fredholm operator
\begin{equation}
  D_{\QFB}: v^{\delta}M_{\QFC,\delta}\to v^{\delta+1}L^2_b(M;E).
\label{do.53b}\end{equation}
\label{do.53}\end{corollary}  
\begin{proof}
By Theorem~\ref{do.28}, there exist $Q_1\in\Psi^{-1,\cQ/\mathfrak{q}}_{\QFB,\cn}(M;E)$ and $R_1\in \Psi^{-\infty,\cR/\mathfrak{r}}_{\QFB,\cn}(M;E)$ such that 
\begin{equation}
     D_{\QFB,\delta}Q_1= \Id+R_1.
\label{do.54}\end{equation}
Conjugating by $v$ then gives
\begin{equation}
     D_{\QFC,\delta}Q_1v= \Id+ v^{-1}R_1v.
\label{do.55}\end{equation}
By Remark~\ref{do.26b}, we know that Assumption~\ref{do.26} also holds with $\delta$ replaced with $-\delta$, so by Theorem~\ref{do.28}, there exists $Q_2\in\Psi^{-1,\cQ_2/\mathfrak{q}_2}_{\QFB,\cn}(M;E)$ and $R_2\in \Psi^{-\infty,\mathfrak{r}_2}_{\QFB,\cn}(M;E)$ such that 
\begin{equation}
  (v^{\delta}D_{\QFB}v^{-\delta})Q_2=\Id+R_2.
\label{do.56}\end{equation}
Taking the adjoint then gives
\begin{equation}
   Q_2^*v^{-\delta}D_{\QFB}v^{\delta}=\Id+R_2^*,
\label{do.57}\end{equation}
that is,
\begin{equation}
(Q_2^*v)D_{\QFC,\delta}= \Id+ R_2^*.
\label{do.58}\end{equation}
By Corollary~\ref{mp.15}, we see that $v^{-1}R_1v$  and $R_2^*$ induce compact operators
\begin{equation}
    v^{-1}R_1v: L^2_b(M;E)\to L^2_b(M;E)
\label{do.59}\end{equation}
and
\begin{equation}
  R_2^*: L^2_b(M;E)\to L^2_b(M;E).
\label{do.60}\end{equation}
In fact, thanks to the extra decay provided by \eqref{ip.1a}, we see that
\begin{equation}
  D_{\QFC,\delta} R_2^*: L^2_b(M;E)\to L^2_b(M;E)
\label{do.60}\end{equation}
is also a compact operator, which means that the operator $R_2^*$ also induces a compact operator
$$
R_2^*: M_{\QFC,\delta}\to M_{\QFC,\delta}.
$$
This means that the operator 
$$
 D_{\QFC,\delta}: M_{\QFC,\delta}\to L^2_b(M;E)
$$
is invertible on both sides modulo compact operators, that is, it is Fredholm.  In terms of the operator $D_{\QFB}$, this means that \eqref{do.53b} is Fredholm as claimed.
\end{proof}

This has the following consequence on the $L^2$ kernel of the operator $D_{\QFB,\delta}$. 
\begin{corollary}
If  Assumptions~\ref{do.1}, \ref{su.7}, \ref{su.2}, \ref{do.46} and \ref{do.26} hold, then the space 
\begin{equation}
   \{  \sigma\in v^{\delta}L^2_{\QFB}(M;E)\; | \; \eth_{\QFB}\sigma=0\} 
\label{do.48a}\end{equation}
is finite dimensional and included in 
$$
     v^{\nu+\delta}x^{\mathfrak{w}}(\cA_{\QFB,2}(M;E)\cap\cA_{\QFB}(M;E))\subset v^{\nu+\delta}x^{\mathfrak{w}}L^2_b(M;E)=v^{\nu+\delta}L^2_{\QFB}(M;E)
$$
for all $\nu<\mu_L-\frac12$.  In particular, the $L^2$ orthogonal projection $P_1: L^2_b(M;E)\to \ker_{L^2_b}D_{\QFB,\delta}$ is an element of $\Psi^{-\infty,\mathfrak{m}}_{\QFB,\res}(M;E)$ with $\mathfrak{m}$ the multiweight given by $\mu_L-\frac12$ and $\mu_L-\frac12 +\bd_i+1$ at the boundary hypersurfaces $H_i\times M$ and $M\times H_i$ on $M^2$. 
\label{do.48}\end{corollary}
\begin{proof}
By Corollary~\ref{do.53}, the space \eqref{do.48a} is finite dimensional.  
By Remark~\ref{do.26b}, we can apply Theorem~\ref{do.28} to $D_{\QFB,-\delta}$, but with $\mu_L$ and $\mu_R$ interchanged, so  there exist $Q\in \Psi^{-1,\cQ/\mathfrak{q}}_{\QFB,\cn}(M;E)$ and $R\in\Psi^{-\infty,\mathfrak{r}}_{\QFB,\cn}(M;E)$ such that 
$$
       D_{\QFB,-\delta}Q=\Id+R.
$$
Taking the adjoint gives
$$
Q^*v^{-\delta}D_{\QFB}v^{\delta}=\Id +R^*,
$$
 so that 
\begin{equation}
      Q^*D_{\QFB,\delta}=\Id +R^*.
\label{do.49}\end{equation}
Hence, given $u\in L^2_b(M;E)$ which is in the kernel of $D_{\QFB,\delta}$, we can apply both sides of \eqref{do.49} to it, which yields
\begin{equation}
       u=-R^*u.
\label{do.50}\end{equation}
By Proposition~\ref{lt.1}, we thus deduce from \eqref{do.50} and the properties of $\mathfrak{r}$ that 
$$
   u\in v^{\nu}L^2_{b}(M;E)
$$
for all $\nu<\mu_L-\frac12$.  By Remark~\ref{do.28d}, we can also deduce from \eqref{do.50} using Hölder inequality as in \eqref{com.8b} that 
$$
   u\in v^{\nu}L^{\infty}(M;E)
$$
for all $\nu<\mu_L-\frac12$.  
Since $R^*$ stays of the same form when composed with $\QFB$ differential operators, we have in fact that
$$
u\in v^{\nu}(\cA_{\QFB,2}(M;E)\cap \cA_{\QFB}(M;E)) \quad \forall \nu<\mu_L-\frac12.
$$
The result then follows from the unitary equivalence between $D_{\QFB,\delta}$ acting formally on $L^2_b(M;E)$ and $\eth_{\QFB}$ acting formally on $v^{\delta}L^2_{\QFB}(M;E)$.  
\end{proof}

When $\eth_{\QFB}$ is the Hodge-deRham operator, the decay obtained in Corollary~\ref{do.48} can sometime be improved when we know that the kernel of $\eth_{\QFB}$ in $v^{\delta}L^2_{\QFB}(M;E)$ only occurs in certain degrees.  This improved decay turns out to be essential to apply Theorem~\ref{do.28} to important examples, for instance in \cite{KR2}.  For this reason, we will give a precise formulation. 
\begin{corollary}
Suppose that the operator $\eth_{\QFB}$ of Corollary~\ref{do.48} is a Hodge-deRham operator and its kernel $\ker_{v^\delta L^2_{\QFB}}\eth_{\QFB}$ in $v^{\delta}L^2_{\QFB}(M;E)$ is only non-trivial in certain degrees.  If $\Pi$ is the projection on the degrees where this kernel is non-trivial and if Assumption~\ref{do.26} holds for some $\widetilde{\mu}_L\ge \mu_{L}$ when $D_{b,i}$ is replaced by $D_{b,i}\Pi$ for all $i$, then in fact 
$$
 \ker_{v^\delta L^2_{\QFB}}\eth_{\QFB}\subset v^{\nu+\delta}x^{\mathfrak{w}}(\cA_{\QFB,2}(M;E)\cap \cA_{\QFB}(M;E))
$$
for all $\nu<\widetilde{\mu}_L-\frac12$.
\label{id.1}\end{corollary}
\begin{proof}
We can follow the same strategy as in the proof of Corollary~\ref{do.48}, except that we can replace \eqref{do.49} by
\begin{equation}
  \Pi Q^*D_{\QFB,\delta}\Pi= \Pi+ \Pi R^*\Pi
\label{id.2}\end{equation}
with $\Pi R^*\Pi$ having multiweight $\widetilde{\mathfrak{r}}$  similar to $\mathfrak{r}$ in Theorem~\ref{do.28}, except that
$$
   \widetilde{\mathfrak{r}}(H_{i0})=\widetilde{\mu}_L-\frac12 \quad \mbox{and}  \quad \widetilde{\mathfrak{r}}(H_{0i})= \bd_i +\frac12+ \mu_R
$$
for each $i$.  Proceeding as in Step~5 of the proof of Theorem~\ref{do.28}, we can also assume that
$$
    \widetilde{\mathfrak{r}}(\ff_i)= \widetilde{\mu}_L+\mu_R \quad \mbox{and} \quad \widetilde{\mathfrak{r}}(H_{ij})= \bd_j+1+ \widetilde{\mu}_L+\mu_R.
$$

  In particular, applying both sides of \eqref{id.2} to $u\in L^2_b(M;E)$ in the kernel of $D_{\QFB,\delta}$ yields
$$
    u= -\Pi R^* \Pi u.
$$ 
Hence by Proposition~\ref{lt.1} and Remark~\ref{do.28d}, $u\in v^{\nu}(L^2_b(M;E)\cap L^{\infty}(M;E))$ for all $\nu<\widetilde{\mu}_L-\frac12$.  Since $\Pi R^*\Pi$ stays of the same form when we apply $\QFB$ differential operators, we see that 
$$
u\in v^{\nu}(\cA_{\QFB,2}(M;E)\cap\cA_{\QFB}(M;E))
$$ 
for all $\nu<\widetilde{\mu}_L-\frac12$, from which the result follows.
\end{proof}

We can also give a pseudodifferential characterization of the inverse of the Fredholm operator 
\begin{equation}
       v^{-1}D_{\QFB,\delta}: M_{\QFC,\delta}\to L^2_b(M;E).
\label{hd.3}\end{equation}
By Corollary~\ref{do.48}, we know already that the $L^2$ orthogonal projection $P_1$ onto the kernel of this operator is a $\QFB$ residual operator,
$$
     P_1\in \Psi^{-\infty,\mathfrak{m}}_{\QFB,\res}(M;E).
$$
  By the formal self-adjointness of $D_{\QFB}$, one can check that the orthogonal complement of the range of \eqref{hd.3} is precisely
$$
     \ker_{L^2_b}D_{\QFB,-\delta-1}.
$$
Using a slightly different argument, we can get the following analog of Corollary~\ref{do.48} for the cokernel of \eqref{hd.3}.
\begin{corollary}
If  Assumptions~\ref{do.1}, \ref{su.7}, \ref{su.2}, \ref{do.46} and \ref{do.26} hold, then the space $\ker_{L^2_b}D_{\QFB,-\delta-1}$
is finite dimensional and included in 
$
     v^{\nu}(\cA_{\QFB,2}(M;E)\cap\cA_{\QFB}(M;E))
$
for all $\nu<\mu_R+\frac12$.  In particular, the $L^2$ orthogonal projection $P_2: L^2_b(M;E)\to \ker_{L^2_b}D_{\QFB,-\delta-1}$ is an element of $\Psi^{-\infty,\mathfrak{m}'}_{\QFB,\res}(M;E)$ with $\mathfrak{m}'$ the multiweight given by $\mu_R+\frac12$ and $\mu_R+\frac12 +\bd_i+1$ at the boundary hypersurfaces $H_i\times M$ and $M\times H_i$ on $M^2$. 
\label{do.48b}\end{corollary}
\begin{proof}
By Corollary~\ref{do.53} and the discussion above, $\ker_{L^2_b}D_{\QFB,-\delta-1}$ is finite dimensional.  By Theorem~\ref{do.28} applied to $D_{\QFB,\delta}$, there exist $Q\in \Psi^{-1,\cQ/\mathfrak{q}}_{\QFB,\cn}(M;E)$ and $R\in\Psi^{-\infty,\mathfrak{r}}_{\QFB,\cn}(M;E)$ such that 
$$
       D_{\QFB,\delta}Q=\Id+R.
$$
Now, if $w\in \ker_{L^2_b}D_{\QFB,-\delta-1}$ and $u\in L^2_b(M;E)$, then 
$$
\begin{aligned}
0&= \langle v^{-1}D_{\QFB,\delta}Qu, w\rangle_{L^2_b}= \langle v^{-1}(\Id+R)u,w\rangle_{L^2_b} \\
  &= \langle u, (\Id+R^*)v^{-1}w\rangle_{L^2_b}.
\end{aligned}
$$
Since $u\in L^2_b(M;E)$ is arbitrary, this means that $v^{-1}w= -R^*v^{-1}w$, that is,
\begin{equation}
 w= v(R^*v^{-1})w.
\label{do.48c}\end{equation}
By the properties of $\mathfrak{r}$, we deduce from \eqref{do.48c}, Proposition~\ref{lt.1} and Remark~\ref{do.28d} that 
$$
    w\in v^{\nu}(L^2_b(M;E)\cap L^{\infty}(M;E))
$$
for all $\nu<\mu_R+\frac12$.  Since $vR^*v^{-1}$ stays of the same form when composed with $\QFB$ differential operators, we have in fact that $w\in v^{\nu}(\cA_{\QFB,2}(M;E)\cap\cA_{\QFB}(M;E))$ for all $\nu<\mu_R+\frac12$.  
\end{proof}
\begin{remark}
In particular, since $\mu_R+\frac12>1$, Corollary~\ref{do.48b} shows that
$$
         \ker_{L^2_b} D_{\QFB,-\delta-1}= v\ker_{L^2_b}D_{\QFB,-\delta}.
$$
\label{do.48d}\end{remark}

By Corollaries~\ref{do.48}, \ref{do.48b} and the Fredholmness of the operator \eqref{hd.3}, there is a bounded operator $G_{\delta}: L^2_b(M;E)\to M_{\QFC,\delta}\subset L^2_b(M;E)$ such that
\begin{gather}
\label{hd.3a} G_{\delta} v^{-1}D_{\QFB,\delta}= \Id-P_1, \\
\label{hd.3b} v^{-1}D_{\QFB,\delta}G_{\delta}=\Id-P_2.
\end{gather} 

\begin{corollary}
The inverse $G_{\delta}$ of \eqref{hd.3} is an element of $\Psi^{-1,\cG/\mathfrak{g}}_{\QFB,\cn}(M;E)$, where $\cG$ is an index family given by
$$
     \left.\cG\right|_{H_{ii}}=\bd_i+1+\bbN_0, \quad \Re\lrp{\left.\cG\right|_{\ff_i}}=\bbN_0+1 \quad \forall i
$$
and elsewhere given by the empty set, while  $\mathfrak{g}$ is a multiweight such that for $i,j\in\{1,\ldots,\ell\}$,
$$
   \mathfrak{g}(H_{i0})=\mu_L-\frac12, \quad \mathfrak{g}(H_{0i})=\bd_i+\frac32+\mu_R, \quad \mathfrak{g}(\ff_i)>1\quad  
   \mathfrak{g}(H_{ii})>\bd_i+1\quad \mbox{and} \quad \mathfrak{g}(H_{ij})>\bd_j+2, \; 0\ne i\ne j\ne0. 
$$
Moreover, for each $i$, the term $A_i$ of order $\bd_i+1$ at $H_{ii}$ of $G_{\delta}$ is such that $\widetilde{\Pi}_{h,i}A_i=A_i\widetilde{\Pi}_{h,i}=A_i$. 
\label{hd.4}\end{corollary}
\begin{proof}
By Corollaries~\ref{do.48}, \ref{do.48b}, we can use the argument of \cite[Theorem~4.20]{MazzeoEdge} to obtain a pseudodifferential characterization of the inverse $G_{\delta}$.  More precisely, notice first that in terms of the parametrix $Q_{\delta}$ with remainder $R_{\delta}$ of $D_{\QFB,\delta}$ provided by Theorem~\ref{do.28}, 
\begin{equation}
\begin{aligned}
G_{\delta}&= G_{\delta}\Id= G_{\delta}(v^{-1}\Id v)= G_{\delta}(v^{-1}(D_{\QFB,\delta}Q_{\delta}-R_{\delta}) v) = G_{\delta}(v^{-1}D_{\QFB,\delta})Q_{\delta}v- G_{\delta}(v^{-1}R_{\delta}v) \\
   &=  (\Id-P_1)Q_{\delta}v-G_{\delta}(v^{-1}R_{\delta}v).
\end{aligned}
\label{hd.6}\end{equation}
By Remark~\ref{do.26b}, Theorem~\ref{do.28} can also be applied with $\delta$ replaced by $-\delta$, so that there is a parametrix $Q_{-\delta}$ with remainder $R_{-\delta}$ such that
$$
       D_{\QFB,-\delta}Q_{-\delta}=\Id +R_{-\delta}.
$$
Taking the adjoint on both sides and using the formal self-adjointness of $D_{\QFB}$ yields
$$
        Q_{-\delta}^*D_{\QFB,\delta}=\Id+R_{-\delta}^*,
$$
so that
\begin{equation}
\begin{aligned}
G_{\delta}&= \Id G_{\delta}= (Q_{-\delta}^*D_{\QFB,\delta}-R^*_{-\delta})G_{\delta}= Q_{-\delta}^*v(v^{-1}D_{\QFB,\delta}G_{\delta})- R^*_{-\delta}G_{\delta} \\
  &= Q^*_{-\delta}v(\Id-P_2)-R^*_{-\delta}G_{\delta}.
\end{aligned}
\label{hd.7}\end{equation}
Inserting \eqref{hd.7} in \eqref{hd.6} then yields
\begin{equation}
\begin{aligned}
G_{\delta} &= Q_{\delta}v- P_1Q_{\delta}v- \left[ Q_{-\delta}^*v(\Id-P_2)-R^*_{-\delta}G_{\delta} \right]v^{-1}R_{\delta}v \\
    &=\left[ Q_{\delta}-P_1Q_{\delta}- Q^*_{-\delta}v(\Id-P_2)v^{-1}R_{\delta} + R^*_{-\delta}G_{\delta}(v^{-1}R_{\delta}) \right] v.
\end{aligned}
\label{hd.8}\end{equation}     
Since $R^*_{-\delta}$ and $v^{-1}R_{\delta}$ are $\QFB$ residual operators, we see by the semi-ideal property of Proposition~\ref{com.8} that $R^*_{-\delta}G_{\delta}(v^{-1}R_{\delta})$ is a $\QFB$ residual operator.  Hence the result follows from \eqref{hd.8} and Theorem~\ref{co.9}, the better decay at some faces being a consequence of the multiplication on the right by $v$.  
\end{proof}

Let us conclude this section by providing a simple criterion to apply Theorem~\ref{do.28} in the depth $2$ case when $\eth_{\QFB}$ is a Hodge-deRham operator.  Lemma~\ref{do.27d} can then be used to determine what $\delta$ we can pick in Assumption~\ref{do.26}.  Similarly, Remark~\ref{su.7e} ensures that Assumption~\ref{su.7} is satisfied.  On the other hand, for $H_i$ submaximal, let $Z_i$ be a fiber of $\phi_i: H_i\to S_i$.  For  $H_j>H_i$, the fiber bundle $\phi_j$ on $H_j$ induces a fiber bundle $\phi_j:\pa _jZ_i\to Y_{ij}$ for some closed manifold $Y_{ij}$, where $\pa_jZ_i= Z_i\cap H_j$.  The operator $\eth_{v,i}$ induces a corresponding vertical family $\eth_{v,ij}$ on the fibers of 
$\phi_{j}:\pa_j Z_i\to Y_{ij}$.  Again, this operator corresponds to a family of (direct sums of) Hodge-deRham operators, so its fiber kernels form a flat vector bundle $\ker \eth_{v,ij}\to Y_{ij}$.  By \cite[Example~8.3]{KR0}, we will have that 
\cite[Assumption~8.1]{KR0} will hold with $\epsilon_1>1$ provided 
\begin{equation}
     H^q(Y_{ij};\ker\eth_{v,ij})=\{0\} \quad \mbox{for} \quad q=\frac{\dim Y_{ij}\pm \ell}2, \quad \ell\in \{0,1,2,3\}.  
\label{do.61}\end{equation} 
or
\begin{equation}
\begin{gathered}
H^q(Y_{ij};\ker\eth_{v,ij})=\{0\} \quad \mbox{for} \quad q=\frac{\dim Y_{ij}\pm \ell}2, \quad \ell\in \{0,1\} \\
\quad \mbox{and} \quad \\
\ker\eth_{v,i} \quad \mbox{is trivial}, 
\end{gathered}
\label{do.61b}\end{equation}
in fact only trivial in degrees $\frac{\dim Y_{ij}+1\pm \ell}2$ for $\ell\in\{1,2,3,4\}$ when $\pa_j Z_i=Y_{ij}$ and $\phi_{j}=\Id$.
Requiring that
\begin{equation}
    \dim Y_i>1
\label{do.61c}\end{equation}
will also ensure that \cite[Theorem~9.1]{KR0} can be applied so that Assumption~\ref{do.46} will hold.
These observations can be summarized as follows.
\begin{theorem}
   Suppose that $M$ is of depth 2, that \eqref{do.61c} holds for $H_i$ submaximal  and that  \eqref{do.61} or \eqref{do.61b} holds for each fiber of $\phi_i:H_i\to S_i$.  Then there exists a $\QFB$ metric $g_{\QFB}$ as in Assumption~\ref{do.1} such that Theorem~\ref{do.28} applies to the associated Hodge-deRham operator for $\delta$ not in \eqref{do.27e} and at distance $\mu>\frac12$ from this set.   Moreover, the metric $g_{\QFB}$ can be chosen so that Theorem~\ref{do.28} applies to the associated Hodge-deRham operator for $\delta=-\frac12$ provided $\mathfrak{d}_{S_i}$ has no $L^2$ kernel in degrees $\frac{\bd_i\pm q}2$ for $q\in\{0,1,2\}$ when $H_{i}$ is maximal or submaximal. 
\label{do.62}\end{theorem}
\begin{proof}
By \eqref{do.61c} as well as \eqref{do.61} or \eqref{do.61b} and \cite[Example~8.3]{KR0}, for $H_i$ submaximal, taking the metrics $g_{Y_{ij}}$ smaller if needed, we know that Assumptions~\ref{su.7} and \ref{do.46} will hold.  
From \eqref{do.61} or \eqref{do.61b}, we also deduce that for $H_i$ maximal, $\mathfrak{d}_{S_i}$ will be essentially self-adjoint by requiring that the induced metrics on the fibers of $\phi_{ij}: \pa_jS_i\to S_j$ be sufficiently small in the sense of Remark~\ref{e.3h} for all $H_j<H_i$. This means that Assumption~\ref{do.26} holds for $\delta$ at distance at least $\mu>\frac12$ from any point of \eqref{do.27e} for $H_i$ maximal or submaximal.  Changing the $\QFB$ metric by scaling the metrics $g_{S_i}$  to make them smaller if needed, we can also assume that Assumption~\ref{do.26} will hold for $\delta=-\frac12$ provided $\mathfrak{d}_{S_i}$ has no $L^2$ kernel in degrees $\frac{\bd_i\pm q}2$ for $q\in\{0,1,2\}$ for $H_{i}$ maximal or submaximal.  
\end{proof}
\begin{remark}
In Theorem~\ref{ift.15} below, which is a generalization to higher depth of \cite[Theorem~9.1]{KR0}, we were able to remove assumption \eqref{do.61c}.
\label{do.63}\end{remark}

\section{$\k,\QFB$ operators}\label{kqfb.0}

For the results of the previous section to hold more generally, we need to generalize the results of \cite{KR0} and give a pseudodifferential characterization of the low energy limit of the resolvent of a Hodge-deRham operator associated to a $\QFB$ metric.   The first step is to introduce the relevant class of pseudodifferential operators.  

Thus, as in \S~\ref{vf.0}, we suppose that $(M,\phi)$ is a manifold with fibered corners which, together with a choice of compatible boundary defining functions, comes with a Lie algebra of $\QFB$ vector fields.   On the manifold $M\times [0,\infty)_{\k}$, we can look at the lift of $\QFB$ vector fields, 
\begin{equation}
\kridx{\cV_{\k,\QFB}}{VQFBk}{$\k$ QFB vector fields}(M\times I)= \left\{ \xi\in \cV(M\times [0,\infty))\; | \; (\pr_2)_*\xi=0,\quad \xi |_{M\times \{\k\}}\in \cV_{\QFB}(M) \; \forall \; \k\in [0,\infty) \right\},
\label{kqfb.1}\end{equation}
where $\pr_2: M\times [0,\infty)_{\k}\to [0,\infty)_{\k}$ is the projection on the second factor. 

If $H_1,\ldots, H_{\ell}$ is an exhaustive list of the boundary hypersurfaces of $M$ compatible with the partial order in the sense that 
$$
H_i<H_j \; \Longrightarrow \; i<j,
$$
we can consider the blown up space
\begin{equation}
M_{\k,\QFB}:= [ M\times [0,\infty)_{\k}; H_{\ell}\times \{0\}, \ldots, H_1\times \{0\}],
\label{kqfb.2}\end{equation}
with blow-down map $\beta^1_{\k,\QFB}: M_{\k,\QFB}\to M\times [0,\infty)$, obtained by blowing up $H_{\ell}\times \{0\},\ldots, H_{1}\times \{0\}$ in that order.  In other words, in the terminology of \S~\ref{ds.0}, 
$$
M_{\k,\QFB}:= M\rttimes [0,\infty)_{\k}.
$$  
We denote by $H_{i,0}$ the boundary hypersurface created by the blow-up of $H_i\times\{0\}$ and by $H_{i,+}$ the lift of $H_i\times [0,\infty)_{\k}$ to $M_{\k,\QFB}$.  We also denote by $H_{0,0}$ the boundary hypersurface corresponding to the lift of $M\times \{0\}$.  
\begin{definition}
The \textbf{Lie algebra of $\k,\QFB$ vector fields} $\cV_{\k,\QFB}(M_{\k,\QFB})$ is the Lie algebra of vector fields on $M_{\k,\QFB}$ generated by $\CI(M_{\k,\QFB})$ and the lifts of vector fields in 
$\cV_{\k,\QFB}(M\times [0,\infty)_{\k}))$ to $M_{\k,\QFB}$.  The space of \textbf{differential $\k,\QFB$ operators} is the universal enveloping algebra over \\ $\CI(M_{\k,\QFB})$ of $\cV_{\k,\QFB}(M_{\k,\QFB})$.  In other words, the space
$\kridx{\Diff^m_{\k,\QFB}}{DiffQFBk}{$\k$ QFB differential operators}(M_{\k,\QFB})$ of differential $\k,\QFB$ operators of order $m$ is generated by multiplication by elements of $\CI(M_{\k,\QFB})$ and up to $m$ vector fields in $\cV_{\k,\QFB}(M_{\k,\QFB})$.  
\label{kqfb.2b}\end{definition}
If $E$ and $F$ are vector bundles on $M_{\k,\QFB}$, one can more generally consider the space 
\begin{equation}
   \Diff^m_{\k,\QFB}(M_{\k,\QFB};E,F):= \Diff^m_{\k,\QFB}(M_{\k,\QFB})\otimes_{\CI(M_{\k,\QFB})}\CI(M_{\k,\QFB}; E^*\otimes F).
\label{kqfb.2c}\end{equation}

By the Serre-Swan theorem, there is a vector bundle $\kridx{{}^{\k,\QFB}T}{TQFBk}{$\k,\QFB$ tangent bundle}M_{\k,\QFB}\to M_{\k,\QFB}$, the $\k,\QFB$ tangent bundle, together with a natural identification
\begin{equation}
  \cV_{\k,\QFB}(M_{\k,\QFB})=\CI(M_{\k,\QFB};{}^{\k,\QFB}TM_{\k,\QFB}).
\label{la.1}\end{equation}
This identification is induced by an anchor map $a: {}^{\k,\QFB}TM_{\k,\QFB}\to TM_{\k,\QFB}$ which confers a Lie algebroid structure to ${}^{\k,\QFB}TM_{\k,\QFB}$.  Clearly, the restriction of ${}^{\k,\QFB}TM_{\k,\QFB}$ to $H_{0,0}=M$ is just the $\QFB$ tangent bundle ${}^{\QFB}TM$.  On the boundary hypersurface $H_{i,+}$ instead, the fiber bundle $\phi_i:H_i\to S_i$ induces a fiber bundle $\phi_{i,+}: H_{i,+}\to S_i$ with fibers 
$$
             (\phi_{i,+})^{-1}(s)= (\phi_i^{-1}(s))_{\k,\QFB}:= \phi_i^{-1}(s)\rttimes [0,\infty),\quad s\in S_i.
$$
Consequently, there is a corresponding vertical $\k,\QFB$ tangent bundle $\kridx{{}^{\k,\QFB}T(H_{i,+}/S_i)}{TQFBkHS}{$\k,\QFB$ vertical tangent bundle}$ and the anchor map on the interior of  $H_{i,+}$ extends to give a short exact sequence of vector bundles
\begin{equation}
\xymatrix{
0 \ar[r] & \kridx{{}^{\k,\QFB}N_{i,+}}{NQFBk}{$\k,\QFB$ normal bundle} \ar[r] & \left.  {}^{\k,\QFB}TM_{\k,\QFB} \right|_{H_{i,+}} \ar[r]^-{a} & {}^{\k,\QFB}T(H_{i,+}/S_i)\ar[r] & 0.
}
\label{la.2}\end{equation}
If $\beta_{i,+}: H_{i,+}\to H_i\times [0,\infty)$ is the natural blow-down map, then there is a natural identification
$$
         {}^{\k,\QFB}N_{i,+}\cong \beta_{i,+}^*\pr_1^* {}^{\phi}NH_i,
$$
 where $\pr_1: H_i\times [0,\infty)\to H_i$ is the projection on the first factor.  Thanks to Lemma~\ref{nb.2}, there is therefore a natural isomorphism 
 \begin{equation}
   {}^{\k,\QFB}N_{i,+}\cong \phi_{i,+}^* {}^{\phi}NS_i.  
 \label{la.3}\end{equation}
This induces a short exact sequence of vector bundles
\begin{equation}
\xymatrix{
0 \ar[r] & {}^{\k,\QFB}T(H_{i,+}/S_i)\ar[r] & \left.{}^{\k,\QFB}TM_{\k,\QFB}\right|_{H_{i,+}} \ar[r]^-{(\phi_{i,+})_*} & \phi^*_{i,+}({}^{\phi}NS_i) \ar[r] & 0.
}
\label{la.4}\end{equation}
In particular, the short exact sequences \eqref{la.2} and \eqref{la.4} induce the splitting
\begin{equation}
    \left.{}^{\k,\QFB}TM_{\k,\QFB}\right|_{H_{i,+}}= {}^{\k,\QFB}N_{i,+}\oplus {}^{\k,\QFB}T(H_{i,+}/S_i).
\label{la.5}\end{equation}

On $H_{i,0}$ for $i\ne 0$, the fiber bundle $\phi_i: H_i\to S_i$ also induces a fiber bundle, but instead of being over $S_i$, it is over
\begin{equation}
    S_{i}\rttimes [0,\frac{\pi}2]:= [S_i\times [0,\frac{\pi}2]; \pa_{\ell_i}S_i\times \{0\}, \ldots, \pa_1S_i\times \{0\}],
\label{la.6}\end{equation}
where $\pa_1S_i,\ldots,\pa_{\ell_i}S_i$ is an exhaustive list of the boundary hypersurfaces of $S_i$ compatible with the partial order and on $[0,\frac{\pi}2]$, we only consider the boundary $\{0\}$ to define the reverse ordered product \eqref{la.6}.  More precisely, $\phi_i$ induces a fiber bundle $\phi_{i,0}: H_{i,0}\to S_i\rttimes [0,\frac{\pi}2]$ with fibers
$$
     \phi_{i,0}^{-1}(s)= \phi_i^{-1}(\pr_1\circ\beta_{i,0}(s)), \quad s\in S_i\rttimes [0,\frac{\pi}2],
$$
where $\beta_{i,0}: S_i\rttimes [0,\frac{\pi}2]\to S_i\times [0,\frac{\pi}2]$ is the natural blow-down map and $\pr_1: S_i\times [0,\frac{\pi}
2]\to S_i$ is the projection on the first factor.  The fibers are thus equipped with a natural Lie algebra of $\QFB$ vector fields, so that there is a natural vertical $\QFB$ tangent bundle ${}^{\QFB}T(H_{i,0}/(S_i\rttimes [0,\frac{\pi}2]))$.  Again, the anchor map induces a short exact sequence of vector bundles 
\begin{equation}
\xymatrix{
0\ar[r] & {}^{\k,\QFB}N_{i,0} \ar[r] & \left.{}^{\k,\QFB}TM_{\k,\QFB}\right|_{H_{i,0}} \ar[r]^-{a} & {}^{\QFB}(H_{i,0}/(S_i\rttimes [0,\frac{\pi}2])) \ar[r] &0.}
\label{la.7}\end{equation}
Using Lemma~\ref{nb.2}, one can see that there is a natural identification
\begin{equation}
  {}^{\k,\QFB}N_{i,0}= \phi_{i,0}^*( {}^{\phi}N(S_i\rttimes [0,\frac{\pi}2])) \quad \mbox{with}\quad  {}^{\phi}N(S_i\rttimes [0,\frac{\pi}2]):= \beta_{i,0}^*\pr_1^*{}^{\phi}NS_i.
\label{la.8}\end{equation}
This induces a short exact sequence of vector bundles
\begin{equation}
\xymatrix{
0 \ar[r] & {}^{\QFB}T(H_{i,0}/(S_i\rttimes[0,\frac{\pi}2])) \ar[r] & \left. {}^{\k,\QFB}TM_{\k,\QFB}\right|_{H_{i,0}} \ar[r]^-{(\phi_{i,0})_*} & {}^{\k,\QFB}N_{i,0} \ar[r] & 0.
}
\label{la.9}\end{equation}
The two short exact sequences \eqref{la.7} and \eqref{la.9} induces the splitting 
\begin{equation}
\left.{}^{\k,\QFB}TM_{\k,\QFB}\right|_{H_{i,0}}= {}^{\k,\QFB}N_{i,0}\oplus {}^{\QFB}T(H_{i,0}/(S_i\rttimes[0,\frac{\pi}2])).\label{la.10}\end{equation}

If the fiber bundles on the maximal boundary hypersurfaces of $M$ are all induced by the identity map $\Id:H_i\to H_i$, then $\cV_{\QFB}(M)$ corresponds to the Lie algebra of $\QAC$ vector fields $\cV_{\QAC}(M)$.  In this case, we denote $M_{\k,\QFB}$ by $M_{\k,\QAC}$, $\cV_{\k,\QFB}(M_{\k,\QFB})$ by $\cV_{\k,\QAC}(M_{\k,\QAC})$ and ${}^{\k,\QFB}TM_{\k,\QFB}$ by ${}^{\k,\QAC}TM_{\k,\QAC}$.  If $H_i$ is maximal, then ${}^{\QAC}T(H_{i,0}/(S_i\rttimes[0,\frac{\pi}2]))$ is the trivial vector bundle of rank $0$, which indicates that the vector fields of $\cV_{\k,\QAC}(M_{\k,\QAC})$ vanish to order one at the boundary hypersurface $H_{i,0}$. Hence, one can define a new Lie algebra of vector fields as follows.
\begin{definition}
The \textbf{Lie algebra of $\QAC-\Qb$ vector fields} on $M_{\k,\QAC}$ is given by
\begin{equation}
\kridx{\cV_{\QAC-\Qb}}{VQACQb}{$\QAC-\Qb$ vector fields}(M_{\k,\QAC}):= \frac{1}{x_{\max,0}}\cV_{\k,\QAC}(M_{\k,\QAC}),
\label{kqfb.3}\end{equation}
where $x_{\max,0}$ is a product of the boundary defining functions for $H_{j,0}$ with $H_j$ a maximal boundary hypersurface.  Similarly, on $\bM_{\k,\QAC}$, we set
$$
\cV_{\QAC-\Qb}(\bM_{\k,\QAC}):= \frac{1}{x_{\max,0}}\cV_{\k,\QAC}(\bM_{\k,\QAC}).
$$
 The space of \textbf{differential  $\QAC-\Qb$ operators} is the universal enveloping algebra over $\CI(M_{\k,\QAC})$ of the Lie algebra of $\QAC-\Qb$ vector fields.  Thus, the space $\Diff^m_{\QAC-\Qb}(M_{\k,\QAC})$ of differential $\QAC-\Qb$ operators of order $m$ is generated by multiplication of elements in $\CI(M_{\k,\QAC})$ and up to $m$ $\QAC-\Qb$ vector fields.  For $E$ and $F$ vector bundles on $M_{\k,\QAC}$, we define more generally the space of differential $\QAC-\Qb$ operators of order $m$ acting from sections of $E$ to sections of $F$ by
$$
\Diff^m_{\QAC-\Qb}(M_{\k,\QAC};E,F):= \Diff^m_{\QAC-\Qb}(M_{\k,\QAC})\otimes_{\CI(M_{\k,\QAC})}\CI(M_{\k,\QAC}; E^*\otimes F).
$$  
\label{kqfb.3b}\end{definition} 

\begin{remark}
In the asymptotically conical setting, the Lie algebra $\cV_{\QAC-\Qb}(M_{\k,\QAC})$ corresponds to the Lie algebra of vector fields of the transition calculus of \cite[\S~3.1]{Kottke}.
\label{kqfb.4}\end{remark}

\begin{definition}
The $\k,\QFB$ double space associated to $(M,\phi)$ and a choice of compatible boundary defining functions is the manifold with corners given by
$$
  \kridx{M^2_{\k,\QFB}}{M2QFBk}{$\k,\QFB$ double space}=[M^2_{\rp}\rttimes [0,\infty)_{\k}; \Phi_{1,+}, \Phi_{1,0}, \ldots, \Phi_{\ell,+}, \Phi_{\ell,0}]
$$
with blow-down map
$$
   \beta_{\k,\QFB}: M^2_{\k,\QFB}\to M^2\times [0,\infty)_{\k},
$$
where $\Phi_{i,+}$ and $\Phi_{i,0}$ denote the lifts of $\Phi_i\times [0,\infty)_{\k}$ and $\Phi_i\times \{0\}$ to $M^2_{rp}\rttimes [0,\infty)_{\k}$.  When the fiber bundles of the maximal boundary hypersurfaces are given by the identity map, we denote the $\k,\QFB$ double space by $M^2_{\k,\QAC}$ and call it the $\k,\QAC$ double space.  
\label{kqfb.5}\end{definition}  
To describe the various boundary hypersurfaces of $M^2_{\k,\QFB}$, we will use the following notation.  With the convention that $H_0:=M$, for $i,j\in \{0,1,\ldots,\ell\}$ not both equal to 0, denote by $H_{ij,0}$ the lift of $H_i\times H_j\times \{0\}$ in $M^2\times [0,\infty)$ to $M^2_{\k,\QFB}$, and by $H_{ij,+}$ the lift of $H_i\times H_j\times [0,\infty)$.  Similarly, let $\ff_{i,0}$ and $\ff_{i,+}$ be the lift of $\ff_i\times \{0\}$ and $\ff_i\times [0,\infty)$ in $M^2_{\QFB}\times [0,\infty)$ to $M^2_{\k,\QFB}$.  Finally, let $H_{00,0}$ be the lift of $M^2\times \{0\}$ to 
$M^2_{\k,\QFB}$.  
\begin{remark}
Since $\Phi_{i,+}$ is blown-up before $\Phi_{i,0}$, notice that $\ff_{i,+}\cap H_{jk,+}=\emptyset$ for all $j,k\in\{0,\ldots,\ell\}$ with $(j,k)\ne (0,0)$.  Moreover, by the proof of Lemma~\ref{new.2}, we have that $\ff_{i,0}\cap H_{j0,0}=\emptyset$ and $\ff_{i,0}\cap H_{0j,0}=\emptyset$ for all $j\in\{1,\ldots,\ell\}$. 
\label{new.3}\end{remark}

In the $\QAC$ setting, we see from \cite{KR0} that the analogue of the Guillarmou-Hassell double space should be the following.

\begin{definition}
If the fiber bundles of the maximal boundary hypersurfaces of $(M,\phi)$ are given by the identity map, the \textbf{$\QAC-\Qb$ double space} of $(M,\phi)$ with its choice of compatible boundary defining functions is the manifold with corners
$$
  \kridx{M^2_{\QAC-\Qb}}{M2QACQb}{$\QAC-\Qb$ double space}:=  [M^2_{\rp}\rttimes [0,\infty)_{\k}; \Phi_{1,+},\Phi_{1,0}, \ldots,\Phi_{\ell'-1,+}, \Phi_{\ell'-1,0}, \Phi_{\ell',+}, \ldots, \Phi_{\ell,+}]
$$
with blow-down map 
$$
       \beta_{\QAC-\Qb}: M^2_{\QAC-\Qb}\to M^2\times [0,\infty)_{\k},
$$
where $H_{\ell'},\ldots, H_{\ell}$ is assumed to be an exhaustive list of maximal boundary hypersurfaces.  
\label{kqfb.10}\end{definition}
We can describe the boundary hypersurfaces of $M^2_{\QAC-\Qb}$ using the following notation.  For $i,j\in \{0,1,\ldots,\ell\}$ not both equal to $0$, denote by $H_{ij,0}^{\Qb}$ and $H_{ij,+}^{\Qb}$ the lifts of $H_i\times H_j\times \{0\}$ and $H_i\times H_j\times [0,\infty)$ in $M^2 \times [0,\infty)$ to $M^2_{\QFB}$.  When $H_i$ is not maximal, let $\ff_{i,0}^{\Qb}$ and $\ff_{i,+}^{\Qb}$ denote the lifts of $\ff_i^{\Qb}\times \{0\}$ and $\ff_i^{\Qb}\times [0,\infty)$ in $M^2_{\Qb}\times [0,\infty)$ to $M^2_{\QAC-\Qb}$.  If instead $H_i$ is maximal, let $\ff_{i,+}^{\Qb}$ denote the lift of $\Phi_i\times [0,\infty)$ in $M^2_{\rp}\times [0,\infty)$ to $M^2_{\QAC-\Qb}$.  Finally, let $H_{00,0}^{\Qb}$ be the lift of $M^2\times \{0\}$ to $M^2_{\QAC-\Qb}$.

The $\k,\QFB$ double space admits natural maps on $M_{\k,\QFB}$.

\begin{lemma}
The projections $\pr_L\times \Id_{[0,\infty)_{\k}}: M^2\times [0,\infty)_{\k}\to M\times [0,\infty_{\k})$ and  $\pr_R\times \Id_{[0,\infty)_{\k}}: M^2\times [0,\infty)_{\k}\to M\times [0,\infty)_{\k}$
lift to $b$-fibrations  $\pi_{\k,L}: M^2_{\k,\QFB}\to M_{\k,\QFB}$ and $\pi_{\k,R}: M^2_{\k,\QFB}\to M_{\k,\QFB}$.  \label{kqfb.17}\end{lemma}
\begin{proof}
The proof is the same for both maps, so we will prove the result for $\pr_L\times \Id_{[0,\infty)_{\k}}$.  By Lemma~\ref{bf.1},  
notice first that the map 
$$
(\pr_L\circ\beta_{\rp})\times \Id_{[0,\infty)_{\k}}: M^2_{\rp}\times [0,\infty)_{\k}\to M\times [0,\infty)_{\k} 
$$
is a $b$-fibration.  By Lemma~\ref{bf.3}, this further lifts to a b-fibration
$$
   M^2_{\rp}\rttimes [0,\infty)_{\k} \to M_{\k,\QFB}= M\rttimes [0,\infty)_{\k}.
$$
Finally, applying Lemma~\ref{bf.1} once more, this lifts to a $b$-fibration
$$
    \pi_{\k,L}: M^2_{\k,\QFB}\to M_{\k,\QFB}.
$$
\end{proof}

There is a similar result for the $\QAC-\Qb$ -double space.
\begin{lemma}
The projections $\pr_L\times \Id_{[0,\infty)_{\k}}: M^2\times [0,\infty)_{\k}\to M\times [0,\infty_{\k})$ and  $\pr_R\times \Id_{[0,\infty)_{\k}}: M^2\times [0,\infty)_{\k}\to M\times [0,\infty)_{\k}$
lift to $b$-fibrations  $\pi^{\Qb}_{\k,L}: M^2_{\QAC-\Qb}\to M_{\k,\QAC}$ and $\pi^{\Qb}_{\k,R}: M^2_{\QAC-\Qb}\to M_{\k,\QAC}$.  
\label{kqfb.18}\end{lemma}
\begin{proof}
The approach is the same as in the proof of Lemma~\ref{kqfb.17}.  
\end{proof}

There is also a natural map from $M^2_{\k,\QFB}$ to $M^2_{\QFB}\times [0,\infty)_{\k}$.
\begin{lemma}
There is a natural surjective $b$-submersion 
\begin{equation}
     M^2_{\k,\QFB}\to M^2_{\QFB}\times [0,\infty)_{\k}.
\label{kqfb.18c}\end{equation}
\label{kqfb.18b}\end{lemma}
\begin{proof}
Since we only blow up corners, notice that the blow-down map
$$
      M^2_{\rp}\rttimes [0,\infty)_{\k}\to M^2_{\rp}\times [0,\infty)_{\k}
$$
is a surjective $b$-submersion.  Applying Lemma~\ref{bf.3}, it lifts to a natural $b$-submersion \eqref{kqfb.18c} as claimed.
\end{proof}
\begin{remark}
By \cite[Lemma~4.4]{KR0}, the surjective $b$-submersion \eqref{kqfb.18c} is actually a blow-down map when $M$ is a manifold with fibered boundary.
\label{kqfb.18d}\end{remark}

Let us denote by $\diag_{\k,\QFB}$ the lift of $\diag_M\times [0,\infty)_{\k}\subset M^2\times [0,\infty)_{\k}$ to $M^2_{\k,\QFB}$, where $\diag_M\subset M^2$ is the diagonal.  Clearly, the lifted diagonal $\diag_{\k,\QFB}$ is a $p$-submanifold.  Similarly, let us denote by $\diag_{\QAC-\Qb}$ the lift of $\diag_M\times [0,\infty)_{\k}$ to $M^2_{\QAC-\Qb}$.  For these lifted diagonals, we have the following transversality results.

\begin{lemma}
The lift of $\cV_{\k,\QFB}(M_{\k,\QFB})$ via the maps $\pi_{\k,L}$ and $\pi_{\k,R}$ are transversal to $\diag_{\k,\QFB}$.  
\label{kqfb.9}\end{lemma}
\begin{proof}
The result is trivial away from the boundary, so let $p\in \diag_{\k,\QFB}\cap \pa M^2_{\k,\QFB}$ be given.  Without loss of generality, we can assume that $\beta_{\k,\QFB}(p)\in (\pa M)^2\times\{0\}$, since the statement is trivial away from $(\pa M)^2\times [0,\infty)_{\k}$ and it follows from Lemma~\ref{ds.2} for $\k>0$.  We need to show that the lemma holds in a neighborhood of $p$ in $M^2_{\k,\QFB}$.    By Remark~\ref{ps.0a}, we only need to consider the result for $\pi_{\k,L}$.  After relabeling the boundary hypersurfaces of $M$ if necessary, we can assume that $H_1\times[0,\infty)_{\k},\ldots,H_k\times[0,\infty)_{\k}$ and $M^2\times\{0\}$ are the boundary hypersurfaces of $M^2\times [0,\infty)_{\k}$ containing $\beta_{\k,\QFB}(p)$ and that
\begin{equation}
  H_1<\cdots < H_k.
\label{lepso.1}\end{equation}   
Then the coordinates \eqref{ds.2c} together with $\k$ give coordinates
\begin{equation}
     x_i, x_i', y_i,y_i', z,z', \k,
\label{lepso.2}\end{equation}
 near $\beta_{\k,\QFB}(p)$ in $M^2\times [0,\infty)_{\k}$.  As indicated in the proof of Lemma~\ref{ds.2}, when we want to lift these coordinates to $M^2_{\rp}\times [0,\infty)_{\k}$, we only need to consider the blow-ups of $H_i\times H_i\times [0,\infty)_{\k}$ for $i\in \{1,\ldots,k\}$, so that  
 the coordinates \eqref{ds.6} together with $\k$ give coordinates on $M^2_{\rp}\times [0,\infty)_{\k}$.  In terms of these coordinates, the lift of $\cV_{K,\QFB}(M\times [0,\infty)_{\k})$ from the left is spanned by the basis of vector fields \eqref{ds.7}.  
 
 Now, in terms of \eqref{lepso.2}, when we blow up $M^2_{\rp}\times [0,\infty)_{\k}$ to obtain $M^2_{\rp}\rttimes [0,\infty)_{\k}$, only the blow-ups of $H_{kk}^{\rp}\times \{0\}, \ldots, H_{11}^{\rp}\times \{0\}$ occur and are near the lifted diagonal.  These blow-ups are implemented by the new coordinate system
\begin{equation}
  \sigma_i, y_i,y_i', z,z', \;i\in\{1,\ldots,k\}, \quad x_k',\ldots, x_{j+1}', \frac{v_j'}{\k}, x_{j-1}',\ldots, x_1', \frac{\k}{v_{j+1}'},
\label{lepso.3}\end{equation}    
near the intersection of the lifts of $H_{jj}^{\rp}\times [0,\infty)_{\k}$ and $H^{\rp}_{jj}\times\{0\}$ when $j\ne 0$ with the convention that $v_{k+1}'=1$, and by the coordinate system 
\begin{equation}
  \sigma_i, x_i',y_i,y_i', z,z', \;i\in\{1,\ldots,k\}, \quad \frac{\k}{v_1'},
\label{lepso.4}\end{equation}
near the intersection of the lifts of $H_{11}^{\rp}\times \{0\}$ and $M^2_{\rp}\times\{0\}$.  In terms of \eqref{lepso.3}, the vector fields of \eqref{ds.7} lift to
\begin{equation}
\begin{aligned}
& v_k'\sigma_k^2\frac{\pa}{\pa \sigma_k}, v_k'\sigma_k\frac{\pa}{\pa y_k^{n_k}}, \ldots, v_{j+1}'\sigma_{j+1}^2\frac{\pa}{\pa \sigma_{j+1}}, 
 v_{j+1}'\sigma_{j+1}\frac{\pa}{\pa y_{j+1}^{n_{j+1}}},  \\
& v_{j+1}'\lrp{\frac{\k}{v_{j+1}'}}\lrp{\frac{v_j'}{\k}}\sigma_j^2\frac{\pa}{\pa \sigma_j}, 
 v_{j+1}'\lrp{\frac{\k}{v_{j+1}'}}\lrp{\frac{v_j'}{\k}}\sigma_j\frac{\pa }{\pa y_j^{n_j}}, \\
 & v_{j+1}'\lrp{\frac{\k}{v_{j+1}'}}\lrp{\frac{v_j'}{\k}} x_{j-1}'\sigma^2_{j-1}\frac{\pa}{\pa \sigma_{j-1}}, v_{j+1}'\lrp{\frac{\k}{v_{j+1}'}}\lrp{\frac{v_j'}{\k}} x_{j-1}'\sigma_{j-1}\frac{\pa}{\pa y_{j-1}^{n_{j-1}}}, \ldots, \\\
  & \hspace{3cm} v_{j+1}'\lrp{\frac{\k}{v_{j+1}'}}\lrp{\frac{v_j'}{\k}} \lrp{\prod_{i=1}^{j-1}x_i'}\sigma^2_{1}\frac{\pa}{\pa \sigma_{1}}, v_{j+1}'\lrp{\frac{\k}{v_{j+1}'}}\lrp{\frac{v_j'}{\k}} \lrp{\prod_{i=1}^{j-1}x_i'}\sigma_{1}\frac{\pa}{\pa y_{1}^{n_{1}}}, \frac{\pa}{\pa z^q},
 \end{aligned}
\label{lepso.5}\end{equation}
while the vector fields keep the same form when we lift with respect to the coordinates \eqref{lepso.4}.  Now, in terms of the coordinates \eqref{lepso.4}, we only need to blow up the lifts of $\Phi_i\times\{0\}$ for $i\in\{1,\ldots,k\}$ to obtain $M^2_{\k,\QFB}$, so we can proceed as in the proof of Lemma~\ref{ds.2} to see that the vector fields in \eqref{ds.7} lift to be transversal to the lifted diagonal.  In terms of the coordinates \eqref{lepso.3}, the blow-ups needed to obtain $M^2_{\k,\QFB}$ are those of the lifts of 
\begin{equation}
\Phi_1\times[0,\infty)_{\k}, \ldots, \Phi_{j-1}\times [0,\infty)_{\k}, \Phi_j\times [0,\infty)_{\k}, \Phi_j\times \{0\}, , \Phi_{j+1}\times \{0\},\ldots, \Phi_k\times\{0\}.
\label{lepso.6a}\end{equation}
Blowing up those submanifolds corresponds to replacing the coordinates \eqref{lepso.3} by
\begin{equation}
 S_i=\frac{\sigma_i-1}{v_i'}, Y_i=\frac{y_i-y_i'}{v_i'}, y_i', i\in\{1,\ldots, k\}, \quad z,z',  \quad x_k',\ldots, x_{j+1}', \frac{v_j'}{\k},x_{j-1}',\ldots, x_1', \frac{\k}{v_{j+1}'},
\label{lepso.6}\end{equation}
so that the vector fields of \eqref{lepso.5} lift to 
\begin{equation}
  (1+v_i'S_i)\frac{\pa}{\pa S_i}, (1+v_i'S_i)\frac{\pa}{\pa Y_i^{n_i}}, \frac{\pa}{\pa z_q}, \quad i\in \{1,\ldots,k\}.
\label{lepso.7}\end{equation}
Since the lifted diagonal is given by the equations
\begin{equation}
    S_i=0, Y_i=0, z=z', \quad i\in\{1,\ldots,k\},
\label{lepso.8}\end{equation}
this local basis of vector fields is clearly transversal to the lifted diagonal $\diag_{\k,\QFB}$, from which the result follows.
\end{proof}

\begin{lemma}
The lift of $\cV_{\QAC-\Qb}(M_{\k,\QAC})$ via the maps $\pi^{\Qb}_{\k,L}$ and $\pi^{\Qb}_{\k,R}$ are transversal to $\diag_{\k,\Qb}$.  
\label{kqfb.11}\end{lemma}
\begin{proof}
As in the proof of Lemma~\ref{kqfb.9}, it suffices to prove the result for $\pi^{\Qb}_{\k,L}$ locally near a point $p\in M^2_{\QAC-\Qb}$ with $\beta_{\QAC-\Qb}(p)\in (\pa M)^2\times \{0\}$ starting with the coordinates \eqref{lepso.2}.  If $H_k$ is not maximal, the argument is exactly the same as in the proof of Lemma~\ref{kqfb.9}.  If instead $H_k$ is maximal, then we can still use the coordinates \eqref{lepso.3} and \eqref{lepso.4} on $M^2_{\rp}\rttimes [0,\infty)_{\k}$, but we must lift the vector fields \eqref{ds.5} and divide them by the boundary defining function of the boundary hypersurface created by the blow-up of $H^{\rp}_{kk}\times\{0\}$.  In terms of the coordinates \eqref{lepso.4}, this yields the vector fields \eqref{ds.13} and we can proceed as in the proof of Lemma~\ref{ds.11} to check that they further lift to be transversal to the lifted diagonal in $M^2_{\QAC-\Qb}$.  In terms of the coordinates \eqref{lepso.3}, this gives instead the vector fields \eqref{lepso.5} multiplied by $(x_k')^{-1}$ if $j<k$ and multiplied  by $\k^{-1}$ if $j=k$.  Compared to \eqref{lepso.6a}, the blow-ups needed in terms of the coordinates \eqref{lepso.3} to obtain $M^2_{\QAC-\Qb}$ are those of the lifts of 
\begin{equation}
\Phi_1\times[0,\infty)_{\k}, \ldots, \Phi_{j-1}\times [0,\infty)_{\k}, \Phi_j\times [0,\infty)_{\k}, \Phi_j\times \{0\}, , \Phi_{j+1}\times \{0\},\ldots, \Phi_{k-1}\times\{0\},
\label{lepso.9}\end{equation}
that is, compared to \eqref{lepso.6a}, the blow-up of the lift of $\Phi_k\times \{0\}$ is omitted.  Blowing up those submanifolds corresponds to introduce the new coordinates
\begin{equation}
S_i=\frac{\sigma_i-1}{w_i'}, Y_i=\frac{y_i-y_i'}{w_i'}, y_i', i<k, \quad  \quad x_k',\ldots, x_{j+1}', \frac{v_j'}{\k},x_{j-1}',\ldots, x_1', \frac{\k}{v_{j+1}'},
\label{lepso.10}\end{equation}
where $w_i'= \frac{v_i'}{x_k'}$ if $j<k$ and $w_i'=\frac{v_i'}{\k}$ if $j=k$, so that the vector fields lift to
\begin{equation}
(1+w_i'S_i)\frac{\pa}{\pa S_i}, (1+w_i'S_i)\frac{\pa}{\pa Y_i^{n_i}}, \; i<k, \quad \sigma_k\frac{\pa}{\pa \sigma_k}, \frac{\pa}{\pa y_k}.\label{lepso.11}\end{equation}
Since the lifted diagonal is given locally by the equations
$$
    S_i=0, Y_i=0, \; i<k, \quad \sigma_k=1, y_k=y_k', 
$$ 
the vector fields \eqref{lepso.11} are clearly transversal to it, from which the result follows.

\end{proof}

These transversality results allow to give a simple description of the Schwartz kernels of differential $\k,\QFB$ operators and differential $\QAC-\Qb$ operators.  Starting with the former, one computes as in \eqref{pdo.3} that in the coordinates \eqref{lepso.6}, the Schwartz kernel of the identity operator takes the form
\begin{equation}
\kappa_{\Id}= \left( \prod_{i=1}^{k} \delta(-S_i)\delta(-Y_i)\delta(z'-z)\right) \beta_{\k,\QFB}^*(\pr_R\times\Id_{[0,\infty)})^*\pr_1^*\nu_{\QFB},
\label{kqfb.19}\end{equation}
where $\pr_1: M\times [0,\infty)\to M$ is the projection on the first factor and $\nu_{\QFB}$ is some non-vanishing $\QFB$ density.  Thus, 
$$
\kappa_{\Id}\in \cD^0(\diag_{\k,\QFB})\cdot \nu^{R}_{\k,\QFB}
$$
with $\nu^{R}_{\k,\QFB}=\beta_{\k,\QFB}^*(\pr_R\times\Id_{[0,\infty)})^*\pr_1^*\nu_{\QFB}$ a lift from the right of some non-vanishing $\k,\QFB$ density.  More generally, by the transversality of Lemma~\ref{kqfb.9}, the Schwartz kernel of an operator $P\in \Diff^m_{\k,\QFB}(M_{\k,\QFB})$ is of the form
\begin{equation}
\kappa_P= \pi^*_{\k,L}P\cdot \kappa_{\Id}\in \cD^m(\diag_{\k,\QFB})\cdot \nu^{R}_{\k,\QFB},
\label{kqfb.20}\end{equation}
where $\cD^m(\diag_{\k,\QFB})$ is the space of smooth delta distributions of order $m$ supported on $\diag_{\k,\QFB}$,
$$
          \cD^{m}(\diag_{k,\QFB})=\Diff^m(M^2_{\k,\QFB})\cdot \cD^0(\diag_{\QFB}).
$$
In fact, by the transversality of Lemma~\ref{kqfb.9}, the space $\cD^m(\diag_{\k,\QFB})\cdot \nu^{R}_{\k,\QFB}$ exactly corresponds to the space of Schwartz kernels of differential $\k,\QFB$ operators of order $m$. This suggests the following definition for the small calculus of pseudodifferential $\k,\QFB$ operators.

\begin{definition}
The \textbf{small calculus of pseudodifferential $\k,\QFB$ operators} is the union over all $m\in\bbR$ of the spaces
\begin{multline}
\kridx{\Psi^{m}_{\k,\QFB}}{PsiQFBk}{$\k,\QFB$ pseudodifferential operators (small calculus)}(M;E,F):=\{ K\in I^{m-\frac14} (M^2_{\k,\QFB};\diag_{\k,\QFB}; \Hom_{\k,\QFB}(E,F)\otimes {}^{\k,\QFB}\Omega_R ); \\
  \kappa\equiv 0 \; \mbox{at}\; \pa M^2_{\k,\QFB}\setminus \ff_{\k,\QFB} \},
\label{kqfb.21b}\end{multline}
where $\ff_{\k,\QFB}$ is the union of the boundary hypersurfaces of $M^2_{\k,\QFB}$ intersecting the lifted diagonal $\diag_{\k,\QFB}$, 
$$
  \Hom_{\k,\QFB}(E,F):= \pi_{\k,L}^*\pr_1^*F \otimes \pi_{\k,R}^*\pr_1^*E^*
$$
and  
$$
\kridx{{}^{\k,\QFB}\Omega}{OQFBk}{right $\k,\QFB$ density bundle}_R:= \beta_{\k,\QFB}^*(\pr_R\times \Id_{[0,\infty)})^*\pr_1^* {}^{\QFB}\Omega.
$$
We also denote by $\dot{\Psi}^{-\infty}_{\k,\QFB}(M;E,F)$ the subspace of $\Psi^{-\infty}_{\k,\QFB}(M;E;F)$ consisting of those operators with Schwartz kernels vanishing rapidly at all boundary hypersurfaces of $M^2_{\k,\QFB}$.
\label{kqfb.21}\end{definition}

For differential $\QAC-\Qb$ operators, there is a parallel description.  First, using coordinates \eqref{lepso.10}, one can check that the Schwartz kernel of the identity operator takes the form
\begin{equation}
  \kappa_{\Id}\in \cD^0(\diag_{\QAC-\Qb})\cdot \nu_{\QAC-\Qb},
\label{kqfb.22}\end{equation}
where 
$$
   \nu_{\QAC-\Qb}= (\pi^{\Qb}_{\k,R})^*[(x^{\dim M+1}_{\max,0})(\beta^1_{\k,\QAC})^* \pr_1^*\nu_{\QAC}]
$$
with $\nu_{\QAC}$ a non-vanishing $\QAC$ density on $M$.  Hence, by the transversality of Lemma~\ref{kqfb.11}, we see more generally that differential $\QAC-\Qb$ operators of order $m$ corresponds to the space of Schwartz kernels
\begin{equation}
       \cD^m(\diag_{\QAC-\Qb})\cdot \nu_{\QAC-\Qb}.
\label{kqfb.23}\end{equation}
Let 
\begin{equation}
{}^{\QAC-\Qb}\Omega
\label{kqfb.23b}\end{equation}
 be the density bundle on $M_{\k,\QAC}$ whose space of sections is 
$$
        \CI(M_{\k,\QAC})\cdot x_{\max,0}^{\dim M +1} (\beta^1_{\k,\QAC})^*\pr_1^*\nu_{\QAC}.
$$
\begin{definition}
The small calculus of pseudodifferential $\QAC-\Qb$ operators is the union over $m\in\bbR$ of the spaces
\begin{multline}
\kridx{\Psi^m_{\QAC-\Qb}}{PsiQACQb}{$\QAC-\Qb$ pseudodifferential operators (small calculus)}(M;E,F):= \{ \kappa\in I^{m-\frac14}(M^2_{\QAC-\Qb};\diag_{\QAC-\Qb}; \Hom_{\QAC-\Qb}(E,F)\otimes 
(\pi_{\k,R}^{\Qb})^{*}{}^{\QAC-\Qb}\Omega); \\
\kappa\equiv 0 \; \mbox{at} \; \pa M^2_{\QAC-\Qb}\setminus \ff_{\QAC-\Qb}\},
\label{kaqfb.24b}\end{multline}
where $\ff_{\QAC-\Qb}$ is the union of the boundary hypersurfaces of $M^2_{\QAC-\Qb}$ intersecting the lifted diagonal $\diag_{\QAC-\Qb}$ and 
$$
    \Hom_{\QAC-\Qb}(E,F)= (\pi^{\Qb}_{\k,L})^*\pr_1^*F \otimes (\pi^{\Qb}_{\k,R})^*\pr_1^*E^*.
$$
\label{kqfb.24}\end{definition}

More generally, given an index family $\cE$ for $M^2_{\k,\QFB}$ and a multiweight $\mathfrak{s}$, we can consider the spaces of pseudodifferential operators
\begin{equation}
\begin{gathered}
\Psi^{-\infty,\cE/\mathfrak{s}}_{\k,\QFB}(M;E,F):= \sA^{\cE/\mathfrak{s}}_{\phg}(M^2_{\k,\QFB};\Hom_{\k,\QFB}(E,F)\otimes {}^{\k,\QFB}\Omega_R), \\
\Psi^{m,\cE/\mathfrak{s}}_{\k,\QFB}(M;E,F):= \Psi^{m}_{\k,\QFB}(M;E,F)+ \Psi^{-\infty,\cE/\mathfrak{s}}_{\k,\QFB}(M;E,F), \quad m\in\bbR.
\end{gathered}
\label{kqfb.25}\end{equation}

Similarly, for $\cE$ an index family for $M^2_{\QAC-\Qb}$ and a multiweight $\mathfrak{s}$, we can consider the spaces of pseudodifferential operators
\begin{equation}
\begin{gathered}
\Psi^{-\infty,\cE/\mathfrak{s}}_{\QAC-\Qb}(M;E,F):= \sA^{\cE/\mathfrak{s}}_{\phg}(M^2_{\QAC-\Qb};\Hom_{\QAC-\Qb}(E,F)\otimes (\pi^{\Qb}_{\k,R})^{*}{}^{\QAC-\Qb}\Omega_R), \\
\Psi^{m,\cE/\mathfrak{s}}_{\QAC-\Qb}(M;E,F):= \Psi^{m}_{\QAC-\Qb}(M;E,F)+ \Psi^{-\infty,\cE/\mathfrak{s}}_{\QAC-\Qb}(M;E,F).
\end{gathered}
\label{kqfb.26}\end{equation}

As for $\QFB$ and $\Qb$ operators, we will need weakly conormal versions of these calculi.  
\begin{definition}
For $E\to M$ a vector bundle over $M$, the space of $\k,\QFB$ conormal sections is defined by 
\begin{multline}
\cA_{\k,\QFB}(M_{\k,\QFB};E)= \{ \sigma\in L^\infty(M_{\k,\QFB};E)\; | \;  \forall p,q\in\bbN_0, \\ \forall X_1,\ldots,X_p\in \cV_{\k,\QFB}(M_{\k,\QFB}), \;
 \left(\k\frac{\pa}{\pa \k} \right)^q\nabla_{X_1}\cdots\nabla_{X_p}\sigma\in L^\infty(M_{\k,\QFB};E)\}, 
\end{multline} 
Similarly, the space of $\QAC-\Qb$ conormal sections is defined by
\begin{multline}
\cA_{\QAC-\Qb}(M_{\k,\QAC};E)= \{ \sigma\in L^\infty(M_{\k,\QAC};E)\; | \;  \forall p,q\in\bbN_0, \\ \forall X_1,\ldots,X_p\in \cV_{\QAC-\Qb}(M_{\k,\QAC}), \;
 \left(\k\frac{\pa}{\pa \k} \right)^q\nabla_{X_1}\cdots\nabla_{X_p}\sigma\in L^\infty(M_{\k,\QAC};E)\}.
\end{multline} \label{wc.1}\end{definition}
More generally, we can consider the weighted and slightly enlarged spaces of weakly conormal sections
\begin{equation}
\cA^{\mathfrak{s}}_{\k,\QFB,-}(M_{\k,\QFB};E)= \bigcap_{\mathfrak{t}<\mathfrak{s}} \rho^{\mathfrak{t}}\cA_{\k,\QFB}(M_{\k,\QFB};E) 
\label{wc.2}\end{equation}
and
\begin{equation}
\cA^{\mathfrak{s}}_{\QAC-\Qb,-}(M_{\k,\QAC};E)= \bigcap_{\mathfrak{t}<\mathfrak{s}} \rho^{\mathfrak{t}}\cA_{\QAC-\Qb}(M_{\k,\QAC};E) 
\label{wc.3}\end{equation}
for $\mathfrak{s}$ a multiweight for $M_{\k,\QFB}$ or $M_{\k,\QAC}$.  Similarly, on $M^2_{\k,\QFB}$ and $M^2_{\QAC-\Qb}$, we can define the spaces of weakly conormal functions 
\begin{multline}
\cA_{\k,\QFB}(M^2_{\k,\QFB})= \{ \kappa\in L^\infty(M^2_{\k,\QFB})\; | \;  \forall p,q,r\in\bbN_0, \;  \forall X_1,\ldots,X_{p+q}\in \cV_{\k,\QFB}(M_{\k,\QFB}), \\
 \left(\k\frac{\pa}{\pa \k} \right)^r\left(\pi^*_{\k,L}X_{1}\right)\cdots\left(\pi^*_{\k,L}X_{p}\right)\left(\pi^*_{\k,R}X_{p+1}\right)\cdots\left(\pi^*_{\k,R}X_{p+q}\right)\kappa\in L^\infty(M^2_{\k,\QFB})\}
\end{multline} 
and 
\begin{multline}
\cA_{\QAC-\Qb}(M^2_{\QAC-\Qb})= \{ \kappa\in L^\infty(M^2_{\QAC-\Qb})\; | \;  \forall p,q,r\in\bbN_0, \;  \forall X_1,\ldots,X_{p+q}\in \cV_{\QAC-\Qb}(M_{\k,\QAC}), \\
 \left(\k\frac{\pa}{\pa \k} \right)^r\left((\pi^{\Qb}_{\k,L})^*X_1\right)\cdots\left((\pi^{\Qb}_{\k,L})^*X_p\right)\left((\pi^{\Qb}_{\k,R})^*X_{p+1}\right)\cdots\left((\pi^{\Qb}_{\k,R})^*X_{p+q}\right)\kappa\in L^\infty(M^2_{\QAC-\Qb})\}, 
\end{multline} 
as well as the weighted versions 
$$
   \cA^{\mathfrak{s}}_{\k,\QFB}(M^2_{\k,\QFB})= \bigcap_{\mathfrak{t}<\mathfrak{s}} \rho^{\mathfrak{t}} \cA_{\k,\QFB}(M^2_{\k,\QFB})
$$
and 
$$
   \cA^{\mathfrak{s}}_{\QAC-\Qb}(M^2_{\QAC-\Qb})= \bigcap_{\mathfrak{t}<\mathfrak{s}} \rho^{\mathfrak{t}} \cA_{\QAC-\Qb}(M^2_{\QAC-\Qb})
$$
for $\mathfrak{s}$ a multiweight for $M^2_{\k,\QFB}$ or $M^2_{\QAC-\Qb}$.  If $\cB$ is any boundary hypersurface of $M^2_{\cV}$, where $\cV$ stands for either $\k,\QFB$ or $\QAC-\Qb$, then we can consider the space 
\begin{equation}
\cA_{\cV}(\cB)= \{  \kappa\in L^\infty(\cB)\; | \; \widetilde{\kappa}\in \cA_{\cV}(M^2_{\cV}) \},
\label{wc.4}\end{equation}
where $\widetilde{\kappa}$ is a smooth extension off $\cB$, and correspondingly the space
$$
    \cA^{\mathfrak{s}}_{\cV,-}(\cB)=\bigcap_{\mathfrak{t}<\mathfrak{s}} \rho^{\mathfrak{t}}\cA_{\cV}(\cB)
$$
for $\mathfrak{s}$ a multiweight for the manifold with corners $\cB$.  We can then proceed as in \eqref{ext.5} and \eqref{ext.7}, but with $\cV$ corresponding to $\k,\QFB$ or $\QAC-\Qb$, to define the space 
\begin{equation}
\cA^{\cE/\mathfrak{s}}_{\cV,\phg}(M^2_{\cV};E)
\label{wc.5}\end{equation}
of partially polyhomogeneous sections of $E$, where $\cE$ is an index family and $\mathfrak{s}$ is a multiweight both associated to the manifold with corners $M^2_{\cV}$.  We can now introduce the larger class of $\k,\QFB$ operators that we need.

\begin{definition}
Let $E,F$ be vector bundles on $M$.  For $\cE$ an index family for $M^2_{\k,\QFB}$ and $\mathfrak{s}$ a multiweight, we can consider the spaces of weakly conormal $\k,\QFB$ pseudodifferential operators
\begin{equation}
\begin{gathered}
    \Psi^{-\infty,\cE/\mathfrak{s}}_{\k,\QFB,\cn}(M;E,F):= \cA^{\cE/\mathfrak{s}}_{\k,\QFB,\phg}(M^2_{\k,\QFB};\Hom_{\k,\QFB}(E,F)\otimes {}^{\k,\QFB}\Omega_R)), \\
    \kridx{\Psi^{m,\cE/\mathfrak{s}}_{\k,\QFB,\cn}}{PsiQFBkEscn}{weakly conormal $\k,\QFB$ pseudodifferential operators}(M;E,F):=\Psi^{m}_{\k,\QFB}(M;E,F)  + \Psi^{-\infty,\cE/\mathfrak{s}}_{\k,\QFB,\cn}(M;E,F)  \quad m\in \bbR. \end{gathered}
\label{wc.6a}\end{equation}
\label{wc.6}\end{definition}
In particular, for $\mathfrak{s}$ a multiweight, we can consider the space of operator
\begin{equation}
      \Psi^{-\infty,\mathfrak{s}}_{\k,\QFB,\cn}(M;E,F):= \Psi^{-\infty,\emptyset/\mathfrak{s}}_{\k,\QFB,\cn}(M;E,F).
\label{wc.7}\end{equation}
Such a multiweight will be said to be \textbf{$\k,\QFB$ positive} if 
$$
        \mathfrak{s}(H_{ij,0})> h_j, \; \mathfrak{s}(H_{ij,+})=\infty\quad  \forall i,j \quad \mbox{and}  \quad \mathfrak{s}(\ff_{i,0})>0, \; \mathfrak{s}(\ff_{i,+})=\infty \quad \forall i>0,
$$
where we recall that $h_0=0$ and $ h_j=\bd_j+1$ for $j\ne 0$.  Similarly, an index family $\cE$ will be said to be \textbf{$\k,\QFB$ nonnegative} if 
$$
        \inf\Re(\cE(H_{ij,0}))> h_j, \; \cE(H_{ij,+})=\emptyset \quad  \forall i,j \quad \mbox{and}  \quad \inf\Re(\cE(\ff_{i,0}))\ge 0, \; \cE(\ff_{i,+})=\bbN_0 \; \forall i>0.
$$

There is a corresponding definition for $\QAC-\Qb$ operators.
\begin{definition}
Let $E,F$ be vector bundles on $M$.  For $\cE$ an index family for $M^2_{\QAC-\Qb}$ and $\mathfrak{s}$ a multiweight, we can consider the spaces of \textbf{weakly conormal $\QAC-\Qb$ pseudodifferential operators}
\begin{equation}
\begin{gathered}
\Psi^{-\infty,\cE/\mathfrak{s}}_{\QAC-\Qb,\cn}(M;E,F):= \cA^{\cE/\mathfrak{s}}_{\QAC-\Qb,\phg}(M^2_{\QAC-\Qb};\Hom_{\QAC-\Qb}(E,F)\otimes (\pi^{\Qb}_{\k,R})^*{}^{\QAC-\Qb}\Omega)), \\
    \kridx{\Psi^{m,\cE/\mathfrak{s}}_{\QAC-\Qb,\cn}}{PsiQACQbEscn}{weakly conormal $\QAC-\Qb$ pseudodifferential operators}(M;E,F):=\Psi^{m}_{\QAC-\Qb}(M;E,F)  + \Psi^{-\infty,\cE/\mathfrak{s}}_{\QAC-\Qb,\cn}(M;E,F)  \quad m\in \bbR.     \end{gathered}
\label{wc.9b}\end{equation}
 \label{wc.9}\end{definition}   
 Similarly, if $\mathfrak{s}$ is a multiweight, we can consider the space
 $$
    \Psi^{-\infty,\mathfrak{s}}_{\QAC-\Qb,\cn}(M;E,F)= \Psi^{-\infty,\emptyset/\mathfrak{s}}_{\QAC-\Qb,\cn}(M;E,F).
 $$
 We say that such a multiweight $\mathfrak{s}$ is \textbf{$\QAC-\Qb$ positive} if   
 $$
        \mathfrak{s}(H_{ij,0})> \left\{ \begin{array}{ll} 
      0, & H_j \; \mbox{maximal} \\
       h_j, &  \mbox{otherwise},  \end{array} \right. \; \mathfrak{s}(H_{ij,+})=\infty \quad  \forall i,j,  \quad \mathfrak{s}(\ff_{i,0})>0 \; (H_i \mbox{non-maximal}),\; \mathfrak{s}(\ff_{i,+})=\infty \quad \forall i>0. 
$$
 Similarly, we say that an index family $\cE$ is \textbf{$\QAC-\Qb$ nonnegative} if  
 $$
 \begin{aligned}
        \inf \Re(\cE(H_{ij,0})) \left\{ \begin{array}{ll} 
      >0, & H_j \; \mbox{maximal}, i\ne j \\
       \ge 0, & H_j \; \mbox{maximal}, i= j \\
       h_j, &  \mbox{otherwise},  \end{array} \right. \; \cE(H_{ij,+})=\emptyset \; \forall i,j,   \\ 
       \quad \inf\Re(\cE(\ff_{i,0}))\ge 0\; (H_i \; \mbox{non-maximal}), \; \cE(\ff_{i,+})=\bbN_0  \quad \forall i>0.
\end{aligned}       
$$

 \begin{remark}
By Lemma~\ref{kqfb.11} and the fact that $\k,\QFB$ vector fields lift from the left and from the right to $b$-vector fields on $M^2_{\cV}$, we see that the space $\Psi^{-\infty,\cE/\mathfrak{s}}_{\k,\QFB,\cn}(M;E,F)$ is stable under left or right composition with $\nabla^F_\xi$ and $\nabla^E_{\xi}$ respectively, where $\xi \in\cV_{\k,\QFB}(M_{\k,\QFB})$ and $\nabla^E$ and $\nabla^F$ are choices of connections for $E$ and $F$.  
\label{wc.11}\end{remark}

\section{The $\k,\QFB$ triple space} \label{tkqfb.0}

To obtain composition results for $\k,\QFB$ operators, we need first to introduce a corresponding $\k,\QFB$ triple space.  To this end, let us start with the Cartesian product $M^3\times [0,\infty)_{\k}$ and consider the projections 
$$
      \pr_{\k,o}= \pr_L\times \Id \quad \mbox{for}  \; o\in \{L,C,R\}
$$
with $\pr_o$ defined in \eqref{ts.1a}.  By Proposition~\ref{ts.8}, for $o=L,C,R$, the projection $\pr_{\k,o}$ lifts to a $b$-fibration
\begin{equation}
    \pi^{\rp}_o\times \Id: M^3_{\rp}\times [0,\infty)_{\k}\to M^2_{\rp}\times [0,\infty)_{\k}.
\label{tkqfb.1}\end{equation}
\begin{lemma}
For $o=L,C,R$, the $b$-fibration $\pi^{\rp}_{o}\times \Id$ lifts to a $b$-fibration
$$
     \pi^{\rp}_{\k,o}: M^3_{\rp}\rttimes [0,\infty)_{\k}\to M^2_{\rp}\rttimes [0,\infty)_{\k}.
$$
\label{tkqfb.2}\end{lemma}
\begin{proof}
By symmetry, it suffices to prove the result for $o=L$.  By definition of the partial order on $M^3_{\rp}$ and $M^2_{\rp}$, the blow-ups required to obtain $M^3\rttimes [0,\infty)_{\k}$ and $M^2_{\rp}\rttimes [0,\infty)_{\k}$ can be done in lexicographic order with respect to the reverse order, namely, on $M^2_{\rp}\times [0,\infty)_{\k}$, we can blow up 
\begin{equation}
\begin{aligned}
& H^{\rp}_{\ell \ell}\times \{0\}, H^{\rp}_{\ell (\ell-1)}\times \{0\}, \ldots, H^{\rp}_{\ell 0}\times \{0\}, H^{\rp}_{(\ell-1)\ell}\times \{0\}, \ldots, H_{(\ell-1)0}\times \{0\},..., 
 H^{\rp}_{1\ell}\times\{0\}, \ldots, H^{\rp}_{10}\times \{0\}, \\ 
 & H^{\rp}_{0\ell}\times \{0\},\ldots, H^{\rp}_{01}\times \{0\}
\end{aligned}
\label{tkqfb.3}\end{equation}
in this order, while on $M^3_{\rp}\times [0,\infty)_{\k}$, we can blow up 
\begin{equation}
\begin{aligned}
& H^{\rp}_{\ell \ell\ell}\times\{0\}, H^{\rp}_{\ell \ell (\ell-1)}\times \{0\}, \ldots, H^{\rp}_{\ell \ell 0}\times \{0\}, H^{\rp}_{\ell(\ell-1)\ell}\times \{0\},\ldots, H^{\rp}_{\ell(\ell-1)0}\times\{0\}, \ldots, H^{\rp}_{\ell0\ell}\times \{0\},\ldots, H^{\rp}_{\ell 0 0}\times \{0\}, \\
& H^{\rp}_{(\ell-1)\ell \ell}\times \{0\},\ldots, H^{\rp}_{(\ell-1)\ell 0}\times \{0\}, H^{\rp}_{(\ell-1)(\ell-1)\ell}\times \{0\}, \ldots, H^{\rp}_{(\ell-1)(\ell-1)0}\times \{0\},\ldots, H^{\rp}_{10\ell}\times\{0\}, \ldots, H^{\rp}_{100}\times\{0\}, \\
& H^{\rp}_{00\ell}\times\{0\},\ldots, H^{\rp}_{001}\times \{0\}
\end{aligned}
\label{tkqfb.4}\end{equation}
in that order.  Let $\cS$ be the ordered collection of $p$-submanifolds \eqref{tkqfb.4} with the last $\ell$ $p$-submanifolds 
\begin{equation}
    H^{\rp}_{00\ell}\times\{0\},\ldots, H^{\rp}_{001}\times \{0\} 
\label{tkqfb.5}\end{equation}
removed.  Applying Lemma~\ref{bf.3}, we see that $\pi^{\rp}_L\times \Id$ lifts to a $b$-fibration
\begin{equation}
   [M^3_{\rp}\times [0,\infty)_{\k};\cS]\to M^2_{\rp}\rttimes [0,\infty)_{\k}.
\label{tkqfb.6}\end{equation}
Applying Lemma~\ref{bf.1} to the blow-ups of \eqref{tkqfb.5}, we then see that \eqref{tkqfb.6} further lifts to a $b$-fibration
$$
     \pi^{\rp}_{\k,L}: M^3_{\rp}\rttimes [0,\infty)_{\k}\to M^2_{\rp} \rttimes [0,\infty)_{\k}
$$
as claimed.
\end{proof}

Similarly,  applying Lemma~\ref{bf.1}, we see that the projection on the first factor 
$$
   \pr_1: M^3_{\rp}\times [0,\infty)_{\k}\to M^3_{\rp}
$$
lifts to a $b$-fibration
\begin{equation}
  \widetilde{\pr}_1: M^3_{\rp}\rttimes [0,\infty)_{\k}\to M^3_{\rp}.
\label{tkqfb.7}\end{equation}
For $i,j,k$ not all equal to zero, let $H^{\rp}_{ijk,0}$ and $H^{\rp}_{ijk,+}$ be the boundary hypersurfaces of $M^3_{\rp}\rttimes [0,\infty)_{\k}$ corresponding to the lifts of $H^{\rp}_{ijk}\times \{0\}$ and $H^{\rp}_{ijk}\times [0,\infty)_{\k}$.  Let also $H^{\rp}_{000,0}$ be the boundary hypersurface of $M^3_{\rp}\rttimes [0,\infty)_{\k}$ corresponding to the lift of $M^3_{\rp}\times \{0\}$.  Now, recall from Definition~\ref{ts.11} that, to define the $\QFB$ triple space out of $M^3_{\rp}$, we need to blow the $p$-submanifolds 
$$
       \{\Phi_{i,T}, \{\Phi^L_{ij}\}_{j=0}^\ell,\{\Phi^C_{ij}\}_{j=0}^\ell, \{\Phi^R_{ij}\}_{j=0}^\ell\}
$$
for $i\in\{1,\ldots,\ell\}$ with $i$ and $j$ increasing.  Under the $b$-fibration \eqref{tkqfb.7}, each of these $p$-submanifolds lifts to two $p$-submanifolds on $M^3_{\rp}\rttimes [0,\infty)_{\k}$.  More precisely, we have that
$$
       \widetilde{\pr}_1^{-1}(\Phi_{i,T})= \Phi_{i,T,0}\cup \Phi_{i,T,+} 
$$
for $p$-submanifolds  $\Phi_{i,T,0}\subset H^{\rp}_{iii,0}$ and $\Phi_{i,T,+}\subset H^{\rp}_{iii,+}$ while 
$$
     \widetilde{\pr}_1^{-1}( \Phi^{L}_{ij})= \Phi^L_{ij,0}\cup \Phi^L_{ij,+}
$$
for $p$-submanifolds $\Phi^L_{ij,0}\subset H^{\rp}_{iij,0}$ and $\Phi^L_{ij,+}\subset H^{\rp}_{iij,+}$ with similar normal families $\{\Phi^C_{ij,0}, \Phi^C_{ij,+}\}$ and $\{\Phi^R_{ij,0}, \Phi^R_{ij,+}\}$ for the lifts of $\Phi^C_{ij}$ and $\Phi^R_{ij}$.  

\begin{definition}
The $\k, \QFB$ triple space $\kridx{M^3_{\k,\QFB}}{M3QFBk}{$\k,\QFB$ triple space}$ is the manifold with corners obtained from $M^3_{\rp}\rttimes [0,\infty)_{\k}$ by blowing up the sequence of families of $p$-submanifolds 
\begin{equation}
\{\Phi_{i,T,+}, \{\Phi^L_{ij,+}\}_{j=0}^{\ell}, \{\Phi^C_{ij,+}\}_{j=0}^{\ell}, \{\Phi^R_{ij,+}\}_{j=0}^{\ell}, \Phi_{i,T,0}, \{\Phi^L_{ij,0}\}_{j=0}^{\ell},\{\Phi^C_{ij,0}\}_{j=0}^{\ell}, \{\Phi^R_{ij,0}\}_{j=0}^{\ell}\}
\label{tkqfb.8a}\end{equation}
for $i\in\{1,\ldots,\ell\}$ with $i$ and $j$ increasing.  We denote by 
\begin{equation}
  \beta_{\k,\QFB}^3: M^3_{\k,\QFB}\to M^3\times [0,\infty)_{\k}
\label{tkqfb.8b}\end{equation}
the corresponding blow down map.
\label{tkqfb.8}\end{definition}

Similarly, in the $\QAC$ setting, we can define the $\QAC-\Qb$ triple space as follows.  
\begin{definition}
Suppose that the maximal boundary hypersurfaces of $M$ are given by $H_{k+1},\ldots,H_{\ell}$ and are such that 
$S_i=H_i$ with $\phi_i$ the identity map.  In this case, the $\QAC-\Qb$ triple space $\kridx{M^3_{\QAC-\Qb}}{M3QACQb}{$\QAC-\Qb$ triple space}$ is obtained from $M^3_{\rp}\rttimes[0,\infty)_{\k}$ by blowing up the sequence of families of $p$-submanifolds \eqref{tkqfb.8a} for $i\in\{1,\ldots,k\}$ with $i$ and $j$ increasing followed by the sequence of families of $p$-submanifolds
\begin{equation}
\{\Phi_{i,T,+}, \{\Phi^L_{ij,+}\}_{j=0}^{\ell}, \{\Phi^C_{ij,+}\}_{j=0}^{\ell}, \{\Phi^R_{ij,+}\}_{j=0}^{\ell}\}
\label{tkqfb.9a}\end{equation}
for $i\in\{k+1,\ldots,\ell\}$ with $i$ and $j$ increasing.  
\label{tkqfb.9}\end{definition}
\begin{proposition}
The projection $\pr_o\times \Id: M^3\times [0,\infty)_{\k}\to M^2\times [0,\infty)_{\k}$ for $o\in \{L,C,R\}$ lifts to a $b$-fibration
$$
      \pi^3_{\k,\QFB,o}: M^3_{\k,\QFB}\to M^2_{\k,\QFB}.
$$
Similarly, in the $\QAC$ setting, it lifts to a $b$-fibration 
$$
    \pi^3_{\QAC-\Qb,o}: M^3_{\QAC-\Qb}\to M^2_{\QAC-\Qb}.
$$
\label{tkqfb.10}\end{proposition}
\begin{proof}
The proof is similar to the one of Proposition~\ref{ts.23}.  First, by Lemma~\ref{tkqfb.2}, we know that $\pr_o\times \Id$ lifts to a $b$-fibration
$$
    \pi^{\rp}_{\k,o}: M^3_{\rp}\rttimes [0,\infty)_{\k}\to M^2_{\rp}\rttimes[0,\infty)_{\k},
$$
so it suffices to show that $\pi^{\rp}_{\k,o}$ lifts to a $b$-fibration
$$
    \pi^3_{\k,\QFB}: M^3_{\k,\QFB}\to M^2_{\k,\QFB}
$$
and similarly for the $\QAC-\Qb$ triple space.  Now, in the sequence of blow-ups 
\begin{equation}
\{\Phi_{i,T,+}, \{\Phi^L_{ij,+}\}_{j=0}^{\ell}, \{\Phi^C_{ij,+}\}_{j=0}^{\ell}, \{\Phi^R_{ij,+}\}_{j=0}^{\ell}\},
\label{tkqfb.11}\end{equation}
notice that by Proposition~\ref{ts.18}, the first blow-up separates $\Phi^L_{ii,+}$, $\Phi^C_{ii,+}$ and $\Phi^R_{ii,+}$.  Combined with Lemma~\ref{ts.12}, this means that in this sequence, the blow-ups after the one of $\Phi_{i,T}$ commute when those are associated to distinct $o,o'\in\{L,C,R\}$.  Replacing $+$ by $0$, a similar observation applies to the sequence of blow-ups
$$
\{\Phi_{i,T,0}, \{\Phi^L_{ij,0}\}_{j=0}^{\ell}, \{\Phi^C_{ij,0}\}_{j=0}^{\ell}, \{\Phi^R_{ij,0}\}_{j=0}^{\ell}\}.
$$
This means that without loss of generality, we can assume that $o=L$.  Looking first at the sequence of blow-ups \eqref{tkqfb.11} for $i=1$, set
$$
  M^3_{\k,\QFB,1+}:= [M^3_{\rp}\rttimes[0,\infty)_{\k}; \Phi_{1,T,+}, \{\Phi^L_{1j,+}\}_{j=0}^{\ell}, \{\Phi^C_{1j,+}\}_{j=0}^{\ell}, \{\Phi^R_{1j,+}\}_{j=0}^{\ell}].
$$ 
As in the proof of Proposition~\ref{ts.23}, we cannot apply Lemma~\ref{bf.3} and Lemma~\ref{bf.1} directly to show that $\pi^{\rp}_{\k,L}$ lifts to a $b$-fibration
\begin{equation}
   [M^3_{\rp}\rttimes[0,\infty)_{\k}; \Phi_{1,T,+}, \{\Phi^L_{1j,+}\}_{j=0}^{\ell}]   \to [M^2_{\rp}\rttimes [0,\infty)_{\k};\Phi_{1,+}],
\label{tkqfb.12}\end{equation}
the problem being that the blow-ups of $\Phi_{1,T,+}$ and $\Phi^{L}_{1j,+}$ do not necessarily commute when $j>1$.  However, as in the proof of Proposition~\ref{ts.23}, the proof of Lemma~\ref{bf.3} can still be applied even if we blow up first $\Phi_{i,T}$, so that we can still conclude that $\pi^{\rp}_{\k,L}$ lifts to a $b$-fibration \eqref{tkqfb.12}.  Applying Lemma~\ref{bf.1}, this further lifts to a $b$-fibration
\begin{equation}
  M^3_{\k,\QFB,1+}\to  [M^2_{\rp}\rttimes [0,\infty)_{\k};\Phi_{1,+}].
\label{tkqfb.13}\end{equation}
Setting 
$$
  M^3_{\k,\QFB,1}:= [M^3_{\k,\QFB,1+}; \Phi_{1,T,0}, \{\Phi^L_{1j,0}\}_{j=0}^{\ell}, \{\Phi^C_{1j,0}\}_{j=0}^{\ell}, \{\Phi^R_{1j,0}\}_{j=0}^{\ell}],
$$
this argument can be used again to show that the $b$-fibration \eqref{tkqfb.13} lifts to a $b$-fibration
$$
    M^3_{\k,\QFB,1}\to  [M^2_{\rp}\rttimes [0,\infty)_{\k};\Phi_{1,+},\Phi_{1,0}].
$$
Defining recursively 
$$
   M^3_{\k,\QFB,i}:= [M^3_{\k,\QFB,i-1}; \Phi_{i,T,+}, \{\Phi^L_{ij,+}\}_{j=0}^{\ell}, \{\Phi^C_{ij,+}\}_{j=0}^{\ell}, \{\Phi^R_{ij,+}\}_{j=0}^{\ell}, \Phi_{i,T,0}, \{\Phi^L_{ij,0}\}_{j=0}^{\ell}, \{\Phi^C_{ij,0}\}_{j=0}^{\ell}, \{\Phi^R_{ij,0}\}_{j=0}^{\ell}],
$$
we can more generally iterate this argument to show that $\pi^{\rp}_{\k,L}$ lifts to a $b$-fibration
$$
M^3_{\k,\QFB,i}\to  [M^2_{\rp}\rttimes [0,\infty)_{\k};\Phi_{1,+},\Phi_{1,0},\ldots,\Phi_{i,+},\Phi_{i,0}],
$$
so that the result follows for the $\k,\QFB$ triple space by taking $i=\ell$.  For the $\QAC-\Qb$ triple space, the same argument works, except that we should set 
$$
    M^3_{\QAC-\Qb,k}=M^3_{\k,\QFB,k}
$$
and define recursively 
$$
      M^3_{\QAC-\Qb,i}= [M^3_{\QAC-\Qb,i-1}; \Phi_{i,T,+}, \{\Phi^L_{ij,+}\}_{j=0}^{\ell}, \{\Phi^C_{ij,+}\}_{j=0}^{\ell}, \{\Phi^R_{ij,+}\}_{j=0}^{\ell}]
$$
for $i>k$ to conclude that $\pi^{\rp}_{\k,L}$ lifts to a $b$-fibration
$$
     M^3_{\QAC-\Qb,i}\to  [M^2_{\rp}\rttimes [0,\infty)_{\k};\Phi_{1,+},\Phi_{1,0},\ldots,\Phi_{k,+},\Phi_{k,0}, \Phi_{k+1,+},\ldots,\Phi_{i,+}]
 $$
for $i>k$.

\end{proof}

As for the $\QFB$ triple space, the $b$-fibrations $\pi^3_{\k,\QFB,o}$ and $\pi^3_{\QAC-\Qb,o}$ behave well with respect to lifted diagonals.  For $o\in \{L,C,R\}$, let $\diag_{\k,\QFB,o}$ (respectively $\diag_{\QAC-\Qb,o}$) be the $p$-submanifold corresponding to the lift of $(\pr_o\times\Id)^{-1}(\diag_M\times [0,\infty)_\k)$ to $M^3_{\k,\QFB}$ (respectively $M^3_{\QAC-\Qb}$).  For $o\ne o'$, the intersection $\diag_{\k,\QFB,o}\cap \diag_{\k,\QFB,o'}$ (respectively $\diag_{\QAC-\Qb,o}\cap \diag_{\QAC-\Qb,o'}$) is the $p$-submanifold $\diag_{\k,\QFB,T}$ (respectively $\diag_{\QAC-\Qb,T}$) given by the lift of the triple diagonal
$$
     \diag^3_{M}\times [0,\infty)_{\k} \subset M^3\times [0,\infty)_{\k}
$$
to $M^3_{\k,\QFB}$ (respectively $M^3_{\QAC-\Qb}$).

\begin{lemma}
For $o\ne o'$, the $b$-fibration $\pi^3_{\k,\QFB,o}$ is transversal to $\diag_{\k,\QFB,o'}$ and induces a diffeomorphism 
$$
      \diag_{\k,\QFB,o'}\cong M^2_{\k,\QFB}
$$
sending $\diag_{\k,\QFB,T}\subset \diag_{\k,\QFB,o'}$ onto $\diag_{\k,\QFB}\subset M^2_{\k,\QFB}$.  Furthermore, an analogous statement holds for the $\QAC-\Qb$ triple space.
\label{tkqfb.14}\end{lemma}
\begin{proof}
By Lemma~\ref{ts.8b}, the $b$-fibration $\pi^{\rp}_o\times \Id$ of \eqref{tkqfb.1} is transversal to $\diag_{\rp,o'}\times [0,\infty)_{\k}$ for $o\ne o'$ and induces a diffeomorphism 
$$
     \diag_{\rp,o'}\times [0,\infty)_{\k}\cong M^2_{\rp}\times [0,\infty)_{\k}  
$$
sending $\diag_{\rp,T}\times [0,\infty)_{\k}$ onto $\diag_{\rp}\times [0,\infty)$.  Performing the blow-ups leading to $M^3_{\k,\QFB}$, we can check at each step that transversality is preserved, so that $\pi^3_{\k,\QFB,o}$ is transversal to $\diag_{\k,\QFB,o'}$ and induces the claimed diffeomorphism.  There is a similar argument for the $\QAC-\Qb$ triple space.
\end{proof}

\section{Composition of $\k,\QFB$ pseudodifferential operators} \label{ckqfb.0}

The triple space of the previous section can be used to obtain composition results for $\k, \QFB$ pseudodifferential operators.  This consists essentially in adapting the discussion in \S~\ref{com.0} to $\k,\QFB$ operators.  Let $H_{ijk,0}$ and $H_{ijk,+}$ denote the lifts of $H^{\rp}_{ijk,0}$ and $H^{\rp}_{ijk,+}$ to $M^3_{\k,\QFB}$.  Let us also denote by $\ff^o_{i,T,0}$, $\ff^o_{i,T,+}$, $\ff^o_{ij,0}$ and $\ff_{ij,+}^o$ the boundary hypersurfaces of $M^3_{\k,\QFB}$ created by the blow-ups of $\Phi^o_{i,T,0}$, $\Phi^o_{i,T,+}$, $\Phi^o_{ij,0}$ and $\Phi^o_{ij,+}$ for $o\in\{L,C,R\}$, $i\in\{1,\ldots,\ell\}$ and $j\in\{0,1,\ldots,\ell\}$.  With this notation, we can now describe how boundary hypersurfaces are mapped in terms of the $b$-fibration of Proposition~\ref{tkqfb.10}.  First, the $b$-fibration $\pi^3_{\k,\QFB,L}$ sends $H_{00i,+}$ surjectively on $M^2_{\k,\QFB}$ for $i\in\{1,\ldots,\ell\}$, and otherwise is such that for $\nu\in \{0,+\}$ and $i,j\in \{1,\ldots,\ell\}$,  
\begin{equation}
\begin{aligned}
(\pi^{3}_{\k,\QFB,L})^{-1}(H_{i0,\nu})&= \ff^C_{i0,\nu}  \cup \bigcup_{k=0}^{\ell} H_{i0k,\nu}, \quad 
(\pi^{3}_{\k,\QFB,L})^{-1}(H_{0i,\nu})=  \ff^R_{i0,\nu}  \cup  \bigcup_{k=0}^{\ell} H_{0ik,\nu}, \\
(\pi^{3}_{\k,\QFB,L})^{-1}(H_{ij,\nu})&= \ff^C_{ij,\nu} \cup \ff^R_{ji,\nu} \cup \bigcup_{k=0}^{\ell} H_{ijk,\nu}, \quad
(\pi^{3}_{\k,\QFB,L})^{-1}(\ff_{i,\nu})= \ff_{i,T,\nu}\cup \bigcup_{k=0}^{\ell} \ff^{L}_{ik,\nu}, \\
 (\pi^{3}_{\k,\QFB,L})^{-1}(H_{00,0})&= \bigcup_{k=0}^{\infty} H_{00k,0}.
\end{aligned}\label{ckqfb.1}\end{equation}
Similarly, the $b$-fibration $\pi^3_{\k,\QFB,C}$ sends $H_{0i0,+}$ surjectively onto $M^2_{\k,\QFB}$ for $i\in\{1,\ldots,\ell\}$ and otherwise is such that for $\nu\in \{0,+\}$ and $i,j\in \{1,\ldots,\ell\}$, 
\begin{equation}
\begin{aligned}
(\pi^{3}_{\k,\QFB,C})^{-1}(H_{i0,\nu})&= \ff^L_{i0,\nu}\cup \bigcup_{k=0}^{\ell}H_{ik0,\nu}, \quad 
(\pi^{3}_{\k,\QFB,C})^{-1}(H_{0i,\nu})= \ff^R_{i0,\nu} \cup \bigcup_{k=0}^{\ell} H_{0ki,\nu}, \\ 
(\pi^{3}_{\k,\QFB,C})^{-1}(H_{ij,\nu})&= \ff^L_{ij,\nu}\cup \ff^R_{ji,\nu}\cup \bigcup_{k=0}^{\ell}H_{ikj,\nu}, \quad     
(\pi^{3}_{\k,\QFB,C})^{-1}(\ff_{i,\nu})= \ff_{i,T,\nu} \cup \bigcup_{k=0}^{\ell} \ff^C_{ik,\nu}, \\
(\pi^{3}_{\k,\QFB,C})^{-1}(H_{00,0})&= \bigcup_{k=0}^{\ell} H_{0k0,0}.
\end{aligned}
\label{ckqfb.2}\end{equation} 
Finally, the $b$-fibration $\pi^3_{\k,\QFB,R}$ sends $H_{i00,+}$ surjectively onto $M^2_{\k,\QFB}$ for $i\in\{1,\ldots,\ell\}$ and otherwise is such that for $\nu\in \{0,+\}$ and $i,j\in\{1,\ldots,\ell\}$,
\begin{equation}
\begin{aligned}
(\pi^{3}_{\k,\QFB,R})^{-1}(H_{i0,\nu})&=\ff^L_{i0,\nu}\cup \bigcup_{k=0}^{\ell} H_{ki0,\nu}, \quad
(\pi^{3}_{\k,\QFB,R})^{-1}(H_{0i,\nu})= \ff^C_{i0,\nu}\cup \bigcup_{k=0}^{\ell}   H_{k0i,\nu}, \\
(\pi^{3}_{\k,\QFB,R})^{-1}(H_{ij,\nu})&=\ff^L_{ij,\nu}\cup \ff^C_{ji,\nu}\cup \bigcup_{k=0}^{\ell} H_{kij,\nu}, \quad
(\pi^{3}_{\k,\QFB,R})^{-1}(\ff_{i,\nu})= \ff_{i,T,\nu}\cup \bigcup_{k=0}^{\ell} \ff^R_{ik,\nu},  \\
 (\pi^{3}_{\k,\QFB,R})^{-1}(H_{00,0})&= \bigcup_{k=0}^{\ell} H_{k00,0}.
 \end{aligned}\label{ckqfb.3}\end{equation}
 On the other hand, using Lemma~\ref{pdo.12a}, we see that 
 \begin{equation}
 (\beta^3_{\k,\QFB})^*({}^b \Omega(M^3\times [0,\infty))= \left[ \prod_{\nu\in\{0,+\}} \prod_{i=1}^{\ell}\left( \rho^2_{\ff_{i,T,\nu}} \prod_{j=0}^{\ell} \left( \rho_{\ff^L_{ij,\nu}} \rho_{\ff^C_{ij,\nu}} \rho_{\ff^R_{ij,\nu}} \right) \right)^{h_i}  \right]
 {}^{b}\Omega(M^3_{\k,\QFB}),
 \label{ckqfb.4}\end{equation}
 where $\rho_H$ denotes a boundary defining function for the boundary hypersurface $H$.  
 On the $\k,\QFB$ double space $M^2_{\k,\QFB}$, we also have that 
 \begin{equation}
 \pi_{\k,\QFB}^*(x_i)= \rho_{\ff_{i,0}} \rho_{\ff_{i,+}} \prod_{j=0}^{\ell} \left( \rho_{ji,0}\rho_{ji,+} \right),
 \label{ckqfb.5}\end{equation}
where $\rho_{ji,\nu}$ denotes a boundary defining function for the boundary hypersurface $H_{ji,\nu}$.  Pulling this back to the $\k,\QFB$ triple space via $\pi^3_{\k,\QFB,L}$ and $\pi^3_{\k,\QFB,R}$, we obtain from \eqref{ckqfb.1} and \eqref{ckqfb.3} that 
\begin{equation}
(\pi^3_{\k,\QFB,L})^* \pi_{\k,R}^* x_i= a_i \prod_{\nu\in \{0,+\}} \left[ \rho_{\ff_{i,T,\nu}} \prod_{j=0}^{\ell} \left( \rho_{\ff^L_{ij,\nu}}\rho_{\ff^R_{ij,\nu}}\prod_{k=0}^{\ell} \rho_{jik,\nu} \right) \left( \prod_{j=1}^{\ell}\rho_{\ff^C_{ji,\nu}} \right)  \right]
\label{ckqfb.6}\end{equation}
and 
\begin{equation}
(\pi^3_{\k,\QFB,R})^* \pi_{\k,R}^* x_i= b_i \prod_{\nu\in\{0,+\}} \left[ \rho_{\ff_{i,T,\nu}}\prod_{j=0}^{\ell}\left( \rho_{\ff^R_{ij,\nu}}\rho_{\ff^C_{ij,\nu}} \prod_{k=0}^{\ell}\rho_{kji,\nu} \right) \left( \prod_{j=1}^{\ell} \rho_{\ff^L_{ji,\nu}} \right)    \right],
\label{ckqfb.7}\end{equation}
where $a_i$ and $b_i$ are smooth positive functions and $\rho_{ijk,\nu}$ is a boundary defining function for $H_{ijk,\nu}$.  Hence, recalling equation \eqref{pdo.12}, we see that 
\begin{multline}
(\pi^3_{\k,\QFB,L})^*(\beta^{*}_{\k,\QFB}(\pr_L\times \Id_{[0,\infty)})^* ({}^{b}\Omega(M\times [0,\infty))) ) \; \cdot \; (\pi^3_{\k,\QFB,L})^*\beta_{\k,\QFB}^* [(\pr_R\times\Id_{[0,\infty)})^* \pr_1^* {}^{\QFB}\Omega(M)  ]  \\
  \cdot \; (\pi^3_{\k,\QFB,R})^* \beta^*_{\k,\QFB}[ (\pr_R\times \Id_{[0,\infty)})^* \pr_1^* {}^{\QFB}\Omega(M)  ]= (\rho^{\mathfrak{a}}){}^{b}\Omega(M^3_{\k,\QFB})
\label{ckqfb.8}\end{multline}
with multiweight $\mathfrak{a}$ such that 
\begin{equation}
\rho^{\mathfrak{a}}\prod_{\nu\in\{0,+\}}  \left[ \prod_{i=1}^{\ell} \left( \left( \prod_{j=1}^{\ell} \rho_{\ff^C_{ji,\nu}}\rho_{\ff^L_{ji,\nu}} \right)  \left(  \prod_{j=0}^{\ell} \rho_{\ff^R_{ij,\nu}} \prod_{k=0}^{\ell} \rho_{jik,\nu}\rho_{kji,\nu}  \right)   \right)^{-h_i}  \right].
\label{ckqfb.9}\end{equation}
We can now state and prove the following composition result.  
\begin{theorem}
Let $E,F$ and $G$ be vector bundles over $M$.  Suppose that $\cE$ and $\cF$ are index families and $\mathfrak{s}$ and $\mathfrak{t}$ are multiweights  such that for each $i$, 
\begin{equation}
   \min\{\mathfrak{s}_{0i,+},\min\Re\cE_{0i,+}\}+\min\{\mathfrak{t}_{i0,+},\min\Re \cF_{i0,+}\}> h_i=1+\dim S_i.
\label{ckqfb.9c}\end{equation}
Then given $A\in\Psi^{m,\cE/\mathfrak{s}}_{\k,\QFB}(M;F;G)$ and $B\in \Psi^{m',\cF/\mathfrak{t}}_{\k,\QFB}(M;E,F)$, their composition is well-defined with 
$$
  A\circ B\in \Psi^{m+m',\cK/\mathfrak{k}}_{\k,\QFB}(M;E,G),
$$
where, using the convention that $h_0=0$, $\cE_{00,+}=\cF_{00,+}=\bbN_0$ and $\mathfrak{s}_{00,+}=\mathfrak{t}_{00,+}=0$, the  index family $\cK$ is for $i,j\in\{1,\ldots,\ell\}$ and $\nu\in \{0,+\}$  given by 
\begin{equation}
\begin{aligned}
\cK_{i0,\nu}&= (\cE_{i0,\nu}+\cF_{00,\nu})\overline{\cup}(\cE_{\ff_{i,\nu}}+\cF_{i0,\nu}) \overline{\cup} \overline{\bigcup_{k\ge1}} (\cE_{ik,\nu}+\cF_{k0,\nu}-h_k), \\
\cK_{0i,\nu}&= (\cE_{00,\nu}+\cF_{0i,\nu})\overline{\cup} (\cE_{0i,\nu}+ \cF_{\ff_{i,\nu}})\overline{\cup} \overline{\bigcup_{k\ge 1}} (\cE_{0k,\nu}+\cF_{ki,\nu}-h_k), \\
\cK_{ij,\nu}&= (\cE_{\ff_{i,\nu}}+\cF_{ij,\nu})\overline{\cup} (\cE_{ij,\nu}+ \cF_{\ff_{j,\nu}})\overline{\cup} \overline{\bigcup_{k\ge 0}} (\cE_{ik,\nu}+\cF_{kj,\nu}-h_k), \\
\cK_{\ff_{i,\nu}}&= (\cE_{\ff_{i,\nu}}+\cF_{\ff_{i,\nu}})\overline{\cup} \overline{\bigcup_{k\ge 0}} (\cE_{ik,\nu}+\cF_{ki,\nu}-h_k), \\
\cK_{00,0} &= \overline{\bigcup_{k\ge 0}} (\cE_{0k,0}+ \cF_{k0,0}-h_k),
\end{aligned}
\label{ckqfb.9a}\end{equation}
and the multiweight $\mathfrak{k}$ is given by 
\begin{equation}
\begin{aligned}
\mathfrak{k}_{i0,\nu}&= \min\{\mathfrak{s}_{i0,\nu}+\mathfrak{t}_{00,\nu}, (\mathfrak{s}_{\ff_{i,\nu}}\dot{+}\mathfrak{t}_{i0,\nu}), \min_{k\ge 1}\{(\mathfrak{s}_{ik,\nu}\dot{+}\mathfrak{t}_{k0,\nu}-h_k)  \} \},
 \\
\mathfrak{k}_{0i,\nu}&= \min\{\mathfrak{s}_{00,\nu}+ \mathfrak{t}_{0i,\nu}, (\mathfrak{s}_{0i,\nu}\dot{+}\mathfrak{t}_{\ff_{i,\nu}}),  \min_{k\ge 1}\{
  (\mathfrak{s}_{0k,\nu}\dot{+}\mathfrak{t}_{ki,\nu}-h_k)\}\}, \\
\mathfrak{k}_{ij,\nu}&= \min\{(\mathfrak{s}_{\ff_{i,\nu}}\dot{+}\mathfrak{t}_{ij,\nu}), (\mathfrak{s}_{ij,\nu}\dot{+}\mathfrak{t}_{\ff_{j,\nu}}), \min_{k\ge 0}\{\mathfrak{s}_{ik,\nu}\dot{+}\mathfrak{t}_{kj,\nu}-h_k\}\},  \\
\mathfrak{k}_{\ff_{i,\nu}}&= \min\{(\mathfrak{s}_{\ff_{i,\nu}}\dot{+}\mathfrak{t}_{\ff_{i,\nu}}),\min_{k\ge 0} \{ (\mathfrak{s}_{ik,\nu}\dot{+}\mathfrak{t}_{ki,\nu}-h_k) \} \}, \\
\mathfrak{k}_{00,0}&= \min_{k\ge 0}\{ \mathfrak{s}_{0k,0}\dot{+} \mathfrak{t}_{k0,0}-h_k  \}.
\end{aligned}
\label{ckqfb.9b}\end{equation}
If instead $\cE=\cF=\emptyset$, except possibly at $H_{00,0}$ and at $\ff_{i,\nu}$ for $\nu\in\{0,+\}$ and $H_i$ a boundary hypersurface, where it could possibly be $\bbN_0$, then for $A\in\Psi^{m,\cE/\mathfrak{s}}_{\k,\QFB,\cn}(M;F,G)$ and 
$B\in\Psi^{m,\cF/\mathfrak{t}}_{\k,\QFB,\cn}(M;E,F)$ only weakly conormal $\k,\QFB$ pseudodifferential operators,
$$
       A\circ B\in \Psi^{m+m',\cK/\mathfrak{k}}_{\k,\QFB,\cn}(M;E,G)
$$
with indicial family $\cK$ and multiweight $\mathfrak{k}$  still given by \eqref{ckqfb.9a} and \eqref{ckqfb.9b}.
\label{ckqfb.10}\end{theorem}
\begin{proof}
Suppose first that $A\in \Psi^{m,\cE/\mathfrak{s}}_{\k,\QFB}(M;F,G)$ and $B\in\Psi^{m',\cF/\mathfrak{t}}_{\k,\QFB}(M;E,F)$.  For operators of order $-\infty$, it suffices to apply the pushforward theorem of \cite[Theorem~5]{Melrose1992}.  When the operator are of order $m$ and $m'$, we need to combine the pushforward theorem with Lemma~\ref{tkqfb.14} to see that the operator has the claimed order.  

If instead $A\in \Psi^{m,\cE/\mathfrak{s}}_{\k,\QFB,\cn}(M;F,G)$ and $B\in\Psi^{m',\cF/\mathfrak{t}}_{\k,\QFB,\cn}(M;E,F)$
are only weakly conormal $\k,\QFB$ operators, but with the above restriction on $\cE$ and $\cF$, then, as in the proof of Theorem~\ref{co.9}, we cannot use the pushforward theorem of \cite[Theorem~5]{Melrose1992}, but we can adapt the proof of \cite[Theorem~4]{Melrose1992} to this weakly conormal setting.  Indeed, by Lemma~\ref{kqfb.9} and the fact that the lift of $\k\frac{\pa}{\pa \k}$ and the lifts from the left and from the right of vector fields in $\cV_{\k,\QFB}(M_{\k,\QFB})$ are $b$-vector fields on $M^2_{\k,\QFB}$, notice that the stability of the composition under the action of these lifts is automatic.  Hence, if $m=m'=-\infty$ and $\cE=\cF=\emptyset$, it suffices to apply Fubini's theorem locally to obtain the result.  If $\cE$ and $\cF$ are not necessarily the empty set everywhere, then the expansions at $H_{00,0}$ and $\ff_{i,\nu}$ for $\nu\in\{0,+\}$ and $H_i$ a boundary hypersurface can be recovered inductively using the fact that $H_{000,0}$ and $\ff_{i,T,\nu}$ are natural triple spaces for those spaces. Otherwise, from the definition of weakly conormal $\k,\QFB$ pseudodifferential operators, we can assume that either $m=-\infty$ or else $m'=-\infty$, in which case by Lemma~\ref{tkqfb.14}, the conormal singularity along the diagonal is simply integrated out, so that the result follows as before.
\end{proof}

The same strategy can be applied to obtain composition results for  $\QAC-\Qb$ operators.  The same notation can be used to describe the boundary hypersurfaces of $M^2_{\QAC-\Qb}$ and $M^3_{\QAC-\Qb}$, except that when $H_i$ is maximal, $M^2_{\QAC-\Qb}$ does not have the boundary hypersurfaces $\ff_{i,0}$ and $M^3_{\QAC-\Qb}$ does not have the boundary hypersurfaces $\ff_{i,T,0}$ and $\ff^o_{ij,0}$ for $o\in\{L,C,R\}$ and $j\in\{0,\ldots,\ell\}$.  With this understood, and using the notations
\begin{equation}
   h_{i,+}:=h_i \quad \mbox{and} \quad h_{i,0}:= \left\{ \begin{array}{ll} 0, & H_i \; \mbox{is maximal},  \\ h_i, & \mbox{otherwise,} \end{array}   \right.
\label{not.1}\end{equation}
we have the following composition result.  
\begin{theorem}
Let $E,F$ and $G$ be vector bundles over $M$.  Suppose that on $M^2_{\QAC-\Qb}$, $\cE$ and $\cF$ are index families and $\mathfrak{s}$ and $\mathfrak{t}$ are multiweights  such that for each $i$, 
\begin{equation}
  \min\{ \mathfrak{s}_{0i,+}, \min\Re\cE_{0i,+}\}+\min\{\mathfrak{t}_{i0,+},\min\Re \cF_{i0,+}\}> h_{i,+}.
\label{ckqfb.12c}\end{equation}
Then given $A\in\Psi^{m,\cE/\mathfrak{s}}_{\QAC-\Qb}(M;F;G)$ and $B\in \Psi^{m',\cF/\mathfrak{t}}_{\QAC-\Qb}(M;E,F)$, their composition is well-defined with 
$$
  A\circ B\in \Psi^{m+m',\cK/\mathfrak{k}}_{\QAC-\Qb}(M;E,G),
$$
where the index family $\cK$, with the convention that $\cE_{\ff_{i,0}}=\cF_{\ff_{i,0}}=\emptyset$ when $H_i$ is a  maximal hypersurface and  that $\cE_{00,+}=\cF_{00,+}=\bbN_0$, is  given by \begin{equation}
\begin{aligned}
\cK_{i0,\nu}&= (\cE_{i0,\nu}+\cF_{00,\nu})\overline{\cup}(\cE_{\ff_{i,\nu}}+\cF_{i0,\nu}) \overline{\cup} \overline{\bigcup_{k\ge1}} (\cE_{ik,\nu}+\cF_{k0,\nu}-h_{k,\nu}), \\
\cK_{0i,\nu}&= (\cE_{00,\nu} +\cF_{0i,\nu})\overline{\cup} (\cE_{0i,\nu}+ \cF_{\ff_{i,\nu}})\overline{\cup} \overline{\bigcup_{k\ge 1}} (\cE_{0k,\nu}+\cF_{ki,\nu}-h_{k,\nu}), \\
\cK_{ij,\nu}&= (\cE_{\ff_{i,\nu}}+\cF_{ij,\nu})\overline{\cup} (\cE_{ij,\nu}+ \cF_{\ff_{j,\nu}})\overline{\cup} \overline{\bigcup_{k\ge 0}} (\cE_{ik,\nu}+\cF_{kj,\nu}-h_k), \\
\cK_{\ff_{i,\nu}}&= (\cE_{\ff_{i,\nu}}+\cF_{\ff_{i,\nu}})\overline{\cup} \overline{\bigcup_{k\ge 0}} (\cE_{ik,\nu}+\cF_{ki,\nu}-h_{k,\nu}) \quad (\mbox{for $H_i$ not maximal or $\nu=+$}), \\
\cK_{00,0}&= \overline{\bigcup_{k\ge 0}}\left( \cE_{0k,0}+\cF_{k0,0}-h_{k,0} \right)
\end{aligned}
\label{ckqfb.12a}\end{equation}
and where the multiweight $\mathfrak{k}$, with the convention that $\mathfrak{s}_{\ff_{i,0}}=\mathfrak{t}_{\ff_{i,0}}=\infty$ when $H_i$ is maximal and that $\mathfrak{s}_{00,+}=\mathfrak{t}_{00,+}=0$, is given by 
\begin{equation}
\begin{aligned}
\mathfrak{k}_{i0,\nu}&= \min\{\mathfrak{s}_{i0,\nu}+\mathfrak{t}_{00,\nu}, (\mathfrak{s}_{\ff_{i,\nu}}\dot{+}\mathfrak{t}_{i0,\nu}), \min_{k\ge 1}\{(\mathfrak{s}_{ik,\nu}\dot{+}\mathfrak{t}_{k0,\nu}-h_{k,\nu})  \} \},
 \\
\mathfrak{k}_{0i,\nu}&= \min\{ \mathfrak{s}_{00,\nu}+\mathfrak{t}_{0i,\nu}, (\mathfrak{s}_{0i,\nu}\dot{+}\mathfrak{t}_{\ff_{i,\nu}}),  \min_{k\ge 1}\{
  (\mathfrak{s}_{0k,\nu}\dot{+}\mathfrak{t}_{ki,\nu}-h_{k,\nu})\}\}, \\
\mathfrak{k}_{ij,\nu}&= \min\{(\mathfrak{s}_{\ff_{i,\nu}}\dot{+}\mathfrak{t}_{ij,\nu}), (\mathfrak{s}_{ij,\nu}\dot{+}\mathfrak{t}_{\ff_{j,\nu}}), \min_{k\ge 0}\{\mathfrak{s}_{ik,\nu}\dot{+}\mathfrak{t}_{kj,\nu}-h_{k,\nu}\}\},  \\
\mathfrak{k}_{\ff_{i,\nu}}&= \min\{(\mathfrak{s}_{\ff_{i,\nu}}\dot{+}\mathfrak{t}_{\ff_{i,\nu}}),\min_{k\ge 0} \{ (\mathfrak{s}_{ik,\nu}\dot{+}\mathfrak{t}_{ki,\nu}-h_{k,\nu}) \} \}  \quad (\mbox{for $H_i$ not maximal or $\nu=+$}), \\
\mathfrak{k}_{00,0} &= \min_{k\ge 0}\{ \mathfrak{s}_{0k,0}\dot{+} \mathfrak{t}_{k0,0}-h_{k,0}  \}.
\end{aligned}
\label{ckqfb.12b}\end{equation}
If instead $\cE=\cF=\emptyset$ except possibly at boundary hypersurfaces intersecting the lift of $\diag_M\times [0,\infty)$, where they could be given by $\bbN_0$, then for $A\in\Psi^{m,\cE/\mathfrak{s}}_{\QAC-\Qb,\cn}(M;F,G)$ and $B\in\Psi^{m',\cF/\mathfrak{t}}_{\QAC-\Qb,\cn}(M;E,F)$, we have that
$$
    A\circ B\in \Psi^{m+m',\cK/\mathfrak{k}}_{\QAC-\Qb,\cn}(M;E,G)
$$
with index family $\cK$ and multiweight $\mathfrak{k}$  still given by \eqref{ckqfb.12a} and \eqref{ckqfb.12b}.
\label{ckqfb.12}\end{theorem}
\begin{proof}
The proof is similar to the one of Theorem~\ref{ckqfb.10}.
\end{proof}

\section{Symbol maps} \label{ksm.0}

In this section, we will introduce various symbol maps for $\k,\QFB$ operators that will play an important role in the construction of parametrices.  First, the principal symbol of an operator $A\in\Psi^{m,\cE/\mathfrak{s}}_{\k,\QFB,\cn}(M;E,F)$ can be defined using its Schwartz kernel $\kappa_A$.  Indeed, by definition, $\kappa_A$ has  conormal singularities at the lifted diagonal $\diag_{\k,\QFB}$, so has a principal symbol 
$$
    \sigma_m(\kappa_A)\in \cS^{[m]}(N^*\diag_{\k,\QFB};\End(E,F)),
$$
hence we can define the principal symbol $\kridx{{}^{\k,\phi}\sigma_m}{sigmaQFBk}{$\k,\QFB$ principal symbol}(A)$ of $A$ to be $\sigma_m(\kappa_A)$.  By Lemma~\ref{kqfb.9}, there is a natural identification
$$
    N^*\diag_{\k,\QFB}\cong {}^{\k,\QFB}T^*M_{\k,\QFB},
$$
so that ${}^{\k,\phi}\sigma_m(A)$ can be seen as an element of $\cS^{[m]}({}^{\k,\QFB}T^*M_{\k,\QFB};\End(E,F))$.  The principal symbol induces a short exact sequence 
\begin{equation}
\xymatrix{
0 \ar[r] & \Psi^{m-1,\cE/\mathfrak{s}}_{\k,\QFB,\cn}\ar[r] & \Psi^{m,\cE/\mathfrak{s}}_{\k,\QFB,\cn}(M;E,F) \ar[r]^-{{}^{\k,\phi}\sigma_m} & \cS^{[m]}({}^{\k,\QFB}T^*M_{\k,\QFB};\End(E,F)) \ar[r] & 0.
}
\label{ksm.1}\end{equation}
To construct good parametrices, we need however other symbols, referred to as normal operators, capturing the asymptotic behavior of $\k,\QFB$ operators in various regimes.  More precisely, each boundary hypersurface of $M^2_{\k,\QFB}$ intersecting the lifted diagonal $\diag_{\k,\QFB}$ induces by restriction a normal operator.  The simplest to describe is the normal operator associated to $H_{00,0}$,
\begin{equation}
  N_{0,0}(A):= \left.\kappa_A\right|_{H_{00,0}}\quad \mbox{for} \; A\in \Psi^{m,\cE/\mathfrak{s}}_{\k,\QFB}(M;E,F) \quad \mbox{with} \; \Re\cE_{00,0}\ge 0, \; \mathfrak{s}(H_{00,0})>0.
\label{ksm.2}\end{equation} 
Indeed, since $H_{00,0}$ is naturally identified with $M^2_{\QFB}$, $N_{0,0}(A)$ can be interpreted as a $\QFB$ operator.  In terms of the small calculus, there is in fact a short exact sequence 
\begin{equation}
\xymatrix{
0 \ar[r] & x_{00,0}\Psi^m_{\k,\QFB}(M;E,F) \ar[r] & \Psi^m_{\k,\QFB}(M;E,F) \ar[r]^-{N_{0,0}} & \Psi^m_{\QFB}(M;E,F) \ar[r] &0,
}
\label{ksm.3}\end{equation}
where $x_{00,0}\in\CI(M^2_{\k,\QFB})$ is a boundary defining function for $H_{00,0}$.  
\begin{proposition}
For $A\in \Psi^{m,\cE/\mathfrak{s}}_{\k,\QFB,\cn}(M;F,G)$ and $B\in\Psi^{m',\cF/\mathfrak{t}}_{\k,\QFB,\cn}(M;E,F)$ with index families $\cE, \cF$ and multiweights $\mathfrak{s}, \mathfrak{t}$ as in Theorem~\ref{ckqfb.10} such that
$$
    \inf \Re(\cE_{00,0})\ge 0, \quad \inf\Re(\cF_{00,0})\ge0, \quad \mathfrak{s}(H_{00,0})>0, \quad \mathfrak{t}(H_{00,0})>0
$$
and 
$$
 \inf \Re(\cE_{0k,0}+ \cF_{k0,0})>h_k,  \quad \mathfrak{s}(H_{0k,0})+\mathfrak{t}(H_{k0,0})>h_k
$$
for all $k\in\{1,\ldots,\ell\}$, we have that
$$
       N_{0,0}(A\circ B)= N_{0,0}(A)\circ N_{0,0}(B)
$$
with composition on the right as $\QFB$ operators.  
\label{ksm.4}\end{proposition} 
\begin{proof}
By Theorem~\ref{ckqfb.10}, the composition $A\circ B$ is well-defined and its restriction to $H_{00,0}$ makes sense.  This restriction comes from the pushforward of the restriction of 
$$
     (\pi^3_{\k,\QFB,L})^*A\cdot (\pi^3_{\k,\QFB,R})^*B
$$
to $H_{000,0}$, so that $N_{0,0}(A\circ B)$ is given by the composition of $N_{0,0}(A)$ and $N_{0,0}(B)$ induced by $H_{000,0}$ seen as a triple space for $H_{00,0}$.  Since $H_{000,0}$ is naturally identified with the $\QFB$ triple space $M^3_{\QFB}$, the result follows.  
\end{proof}

For $i\ne 0$, the normal operator $\kridx{N_{i,0}}{Normi0}{$\k,\QFB$ normal operator}(A)$ associated to the boundary hypersurface $\ff_{i,0}$ corresponds instead to a suspended family of $\QFB$ operators, at least within the small calculus.  Indeed, as on $H_{i,0}$, the fiber bundle $\phi_i: H_i\to S_i$ induces a fiber bundle 
$$
      \phi_{\ff_{i,0}}: \ff_{i,0}\to S_i\rttimes [0,\frac{\pi}2]
$$
given by $\phi_{\ff_{i,0}}:= \phi_{i,0}\circ \pi_{\k,R}|_{\ff_{i,0}}= \phi_{i,0}\circ \pi_{\k,L}|_{\ff_{i,0}}$.  This is in fact just a pull-back, via the blow-down map
$$
 \beta_{i,0}: S_i\rttimes [0,\frac{\pi}2]\to S_i\times [0,\frac{\pi}2]
$$
and the projection $\pr_1: S_i\times [0,\frac{\pi}2]\to S_i$
 of the corresponding fiber bundle $\phi_{\ff_i}: \ff_i\to S_i$ on $\ff_i$ of \eqref{sm.18a}.
Thus, as for $\ff_i$ on $M^2_{\QFB}$, where we have the canonical identification \eqref{sm.18}, there is within the small calculus a canonical identification between the space of conormal distributions
$$
   \Psi^m_{\ff_{i,0}}(H_{i,0};E,F):= \{N_{i,0}(A) \; | \; A\in \Psi^m_{\k,\QFB}(M;E,F)\}
$$
and the space of suspended conormal distributions
$$
  \Psi^m_{\QFB, {}^{\phi}N(S_i\rttimes[0,\frac{\pi}2])}(H_{i,0}/(S_i\rttimes [0,\frac{\pi}2]);E,F).
$$
Since $\phi_{\ff_{i,0}}$ is just the pull-back of $\phi_{\ff_i}$, we can still use Lemma~\ref{sm.19} to conclude that the identification
$$
   \Psi^m_{\ff_{i,0}}(H_{i,0};E,F)\cong \Psi^m_{\QFB, {}^{\phi}N(S_i\rttimes[0,\frac{\pi}2])}(H_{i,0}/(S_i\rttimes [0,\frac{\pi}2]);E,F)$$
is also an identification as operators, namely the composition on the left induced by Theorem~\ref{ckqfb.10} corresponds to the composition as $\QFB$ suspended operators on the right.  There is in particular a short exact sequence 
\begin{equation}
\xymatrix{
0\ar[r] & x_{\ff_{i,0}}\Psi^m_{\k,\QFB}(M;E,F) \ar[r] & \Psi^m_{\k,\QFB}(M;E,F)\ar[r]^-{N_{i,0}} & \Psi^m_{\ff_{i,0}}(H_i;E,F) \ar[r] &0,
}
\label{ksm.5}\end{equation}  
where $x_{\ff_{i,0}}\in\CI(M^2_{\k,\QFB})$ is a boundary defining function for $\ff_{i,0}$.  More generally, we have the following composition result in the large calculus.  
\begin{proposition}
For $A\in\Psi^{m,\cE/\mathfrak{s}}_{\k,\QFB,\cn}(M;E,G)$ and $B\in \Psi^{m',\cF/\mathfrak{t}}_{\k,\QFB,\cn}(M;E,F)$ with index families $\cE,\cF$ and multiweight $\mathfrak{s},\mathfrak{t}$ as in Theorem~\ref{ckqfb.10} such that
$$
\inf\Re \cE_{\ff_{i,0}}\ge0, \quad \inf\Re \cF_{\ff_{i,0}}\ge0, \quad \mathfrak{s}(\ff_{i,0})>0, \quad \mathfrak{t}(\ff_{i,0})>0
$$
and 
$$
    \inf\Re(\cE_{ik,0}+ \cF_{ki,0}-h_k)>0, \quad \mathfrak{s}(H_{ik,0})+\mathfrak{t}(H_{ki,0})-h_k>0 
$$
for all $k\in\{0,1,\ldots,\ell\}$, we have that
$$
    N_{i,0}(A\circ B)= N_{i,0}(A)\circ N_{i,0}(B)
$$
with composition on the right induced by the triple space $\ff_{i,T,0}$.
\label{ksm.6}\end{proposition}   
\begin{proof}
By Theorem~\ref{ckqfb.10}, the composition $A\circ B$ is well-defined and its restriction to $\ff_{i,0}$ makes sense.  This restriction comes from the pushforward of the restriction of 
$$
      (\pi^3_{\k,\QFB,L})^* A \cdot (\pi^3_{\k,\QFB,R})^* B
$$
to $\ff_{i,T,0}$, which is just the composition of $N_{i,0}(A)$ and $N_{i,0}(B)$ using $\ff_{i,T,0}$ as a triple space. 
\end{proof}

For $i\ne 0$, the normal operator $\kridx{N_{i,+}}{Normi+}{$\k,\QFB$ normal operator}(A)$ associated to $\ff_{i,+}$ corresponds instead to some sort of suspended version of a $\k,\QFB$ operator, but as for $\ff_{i,0}$, the corresponding double space is more complicated.  In fact, there is a natural fiber bundle $\phi_{\ff_{i,+}}: \ff_{i,+}\to S_i$ given by
\begin{equation}
    \phi_{\ff_{i,+}}:= \phi_{i,+}\circ \pi_{\k,R}|_{\ff_{i,+}}= \phi_{i,+}\circ \pi_{\k,L}|_{\ff_{i,+}}.
\label{ffi.1}\end{equation}
 For $s\in S_i$, the fiber $\phi_{\ff_{i,+}}^{-1}(s)$ is not quite
$$
    \overline{{}^{\phi}N_sS}_i\times (\phi_i^{-1}(s))^2_{\k,\QFB}
$$
for two reasons.  Firstly, for $H_j>H_i$ the blow-ups of $\Phi_{j,+}$ and $\Phi_{j,0}$ corresponds to blowing up $\{0\}\times \Phi^s_{j,\nu}$ in 
$$
\overline{{}^{\phi}N_sS}_i\times [(\phi_i^{-1}(s))^2_{\rp}\rttimes [0,\infty)],
$$
not $\overline{{}^{\phi}N_sS}_i\times \Phi^s_{j,\nu}$,  where $\Phi_{j,\nu}^s\subset [(\phi_i^{-1}(s))^2_{\rp}\rttimes [0,\infty)]$ is the analog of $\Phi_{j,\nu}$ in $M^2_{\rp}\rttimes [0,\infty)$.  Secondly, the blow-up of $\Phi_{i,0}$ creates a new face at $\k=0$, some sort of adiabatic limit as $\k\searrow 0$.

Of course, for $\delta>0$ fixed, a slice of $\ff_{i,+}$ at $\k=\delta$ is canonically identified with $\ff_i$ in $M^2_{\QFB}$, so within the small calculus, the normal operator $N_{i,+}(A)$  corresponds for $\k=\delta$ to a family of suspended $\QFB$ operators via the identification \eqref{sm.18} and Lemma~\ref{sm.19}.  Again in this case, we have a general composition result.  

\begin{proposition}
For $A\in\Psi^{m,\cE/\mathfrak{s}}_{\k,\QFB,\cn}(M;E,G)$ and $B\in \Psi^{m',\cF/\mathfrak{t}}_{\k,\QFB,\cn}(M;E,F)$ with index families $\cE,\cF$ and multiweight $\mathfrak{s},\mathfrak{t}$ as in Theorem~\ref{ckqfb.10} such that
$$
\inf\Re \cE_{\ff_{i,+}}\ge 0, \quad \inf\Re \cF_{\ff_{i,+}}\ge 0, \quad \mathfrak{s}(\ff_{i,+})>0, \quad \mathfrak{t}(\ff_{i,+})>0
$$
and 
$$
    \inf\Re(\cE_{ik,+}+ \cF_{ki,+}-h_k)>0, \quad \mathfrak{s}(H_{ik,+})+\mathfrak{t}(H_{ki,+})-h_k>0 
$$
for all $k\in\{0,1,\ldots,\ell\}$, we have that
$$
    N_{i,+}(A\circ B)= N_{i,+}(A)\circ N_{i,+}(B)
$$
with composition on the right induced by the triple space $\ff_{i,T,+}$.
\label{ksm.6b}\end{proposition}  
\begin{proof}
The proof is similar to the proofs of Proposition~\ref{ksm.4} and Proposition~\ref{ksm.6}.
\end{proof}

For the $\QAC-\Qb$ calculus, one can define a principal symbol and normal operators in a similar way.  For us, the most important one will the be the normal operator $\kridx{N_{ii,0}}{Normii0}{$\QAC-\Qb$ normal opertor}(A)$ associated to $H^{\Qb}_{ii,0}$ for $H_i$ a maximal boundary hypersurface.  There is the following analog of Propositions~\ref{ksm.4},\ref{ksm.6} and \ref{ksm.6b}.

\begin{proposition}
For $A\in\Psi^{m,\cE/\mathfrak{s}}_{\QAC-\Qb,\cn}(M;E,G)$ and $B\in \Psi^{m',\cF/\mathfrak{t}}_{\QAC-\Qb,\cn}(M;E,F)$ with index families $\cE,\cF$ and multiweight $\mathfrak{s},\mathfrak{t}$ as in Theorem~\ref{ckqfb.10} such that
$$
\inf\Re \cE_{ii,0}\ge 0, \quad \inf\Re \cF_{ii,0}\ge 0, \quad \mathfrak{s}(H_{ii,0})>0, \quad \mathfrak{t}(H_{ii,0})>0
$$
and 
$$
    \inf\Re(\cE_{ik,0}+ \cF_{ki,0}-h_{k,0})>0, \quad \mathfrak{s}(H_{ik,0})+\mathfrak{t}(H_{ki,0})-h_{k,0}>0 
$$
for all $k\in\{1,\ldots,\ell\}$, we have that
$$
    N_{ii,0}(A\circ B)= N_{ii,0}(A)\circ N_{ii,0}(B)
$$
with composition on the right induced by the triple space $H_{iii,0}^{\Qb}$ corresponding to the lift of $H^{\rp}_{iii,0}$ to $M^3_{\QAC-\Qb}$.
\label{ksm.12}\end{proposition}

\begin{proof}
The proof is similar to the proofs of Proposition~\ref{ksm.4} and Proposition~\ref{ksm.6}.
\end{proof} 

\section{Resolvent of a $\QFB$ Dirac operator in the low energy limit} \label{qfble.0}

We have now all the ingredients to provide a pseudodifferential characterization of the low energy limit of the resolvent of a $\QFB$ Dirac operators.  The overall strategy is to follow the approach of \cite{GH2,KR0} suitably adapted to $\k,\QFB$ operators.    
Thus, let $\eth_{\QFB}$ be a Dirac operator as in \S~\ref{do.0}, so that Assumptions~\ref{do.1}, \ref{su.7}, \ref{su.2}, \ref{do.46} and \ref{do.26} hold.  For Assumption~\ref{do.26}, we suppose more precisely that it holds for $\delta=-\frac12$.  For each boundary hypersurface $H_i$, it will be convenient as well to assume that those assumptions hold for each member of the vertical $\eth_{v,i}$ seen as a Dirac operator.  By Remark~\ref{do.26b}, we can suppose  that 
\begin{equation}
  \mu_L= \mu_R+1>\frac32.
\label{qfble.1}\end{equation}
Let $\gamma\in \CI(M;\End(E))$ be self-adjoint with respect to the $\QFB$ metric $g_{\QFB}$ and the bundle metric of $E$ and suppose that 
\begin{equation}
\gamma^2=\Id_E, \quad \eth_{\QFB}\gamma+ \gamma\eth_{\QFB}=0.
\label{qfble.2}\end{equation}
In terms of \eqref{su.3b} and \eqref{do.15b}, suppose also that $\gamma$ anti-commutes with $\eth_{v,i}$, $\eth_{h,i}$, $c$ and $\eth_{S_i}$.  In this section, we will consider the first order $\k,\QFB$ operator
\begin{equation}
\eth_{\k,\QFB}:= \eth_{\QFB}+ \k \gamma.
\label{qfble.3}\end{equation}
As in \eqref{do.2}, it will be convenient to consider the conjugated operator
\begin{equation}
D_{\k,\QFB}:= x^{-\mathfrak{w}}\eth_{\k,\QFB}x^{\mathfrak{w}}= D_{\QFB}+\k\gamma.
\label{qfble.4}\end{equation}
By \eqref{qfble.2}, notice that 
$$
   \eth^2_{\k,\QFB}=\eth_{\QFB}^2+\k^2\Id_E,
$$
so that for $\k>0$ fixed, we see, proceeding as in the proof of Corollary~\ref{mp.33}, that $\eth^2_{\k,\QFB}$ and $\eth_{\k,\QFB}$ are invertible in the small $\QFB$ calculus.  However, when $\k=0$, $\eth_{\k,\QFB}=\eth_{\QFB}$ is no longer fully elliptic, but at least it is Fredholm in some sense by Corollary~\ref{do.53} with inverse in the weakly conormal large $\QFB$ calculus by Corollary~\ref{hd.4}.  

This and the invertibility for $\k>0$ can be combined to invert $\eth_{\k,\QFB}$ within the $\k,\QFB$ calculus, for the moment provided we make the following assumption.

\begin{assumption}
For each boundary hypersurface $H_i$ of $M$, the normal operator $N_{i,+}(D_{\k,\QFB})$ is invertible with inverse 
$$
    N_{i,+}(D_{\k,\QFB})^{-1}\in \Psi^{-1,\cE_i/\mathfrak{s}_i}_{\k,\phi-{}^{\phi}NS_i,\cn}(H_i/S_i;E),
$$
where $\cE_i$ is a $\k,\QFB$ nonnegative index family except at $H_{jj,0}$ for $H_j\ge H_i$ where we have that 
$$
     \cE_i|_{H_{jj,0}}=\bd_j+\bbN_0,
$$
and where $\mathfrak{s}_i$ is a $\k,\QFB$ positive multiweight, except at $H_{jj,0}$ for $H_j\ge H_i$ where 
$$
          \mathfrak{s}(H_{jj,0})>\bd_i.
$$ 
Moreover, for $H_j\ge H_i$, $N_{i,+}(D_{\k,\QFB})^{-1}$ agrees with $N_{j,0}(D_{\k,\QFB})^{-1}$ given by Assumption~\ref{do.46} on $\ff_{i,+}\cap\ff_{j,0}$ and with $N_{j,+}(D_{\k,\QFB})^{-1}$ for $H_j\ne H_i$.  Finally, the term $A_{ij}$ of order $\bd_j$ at $\ff_{i,+}\cap H_{jj,0}$ is such that $\widetilde{\Pi}_{h,j}A_{ij}= A_{ij}\widetilde{\Pi}_{h,j}=A_{ij}$. 
\label{qfble.5}\end{assumption}
\begin{remark}
Notice that $\ff_{i,+}\cap H_{0j,0}=\ff_{i,+}\cap H_{j0,0}=\emptyset$ for all $i>0$ and $j\ge 0$.
\label{qfble.6}\end{remark}

To give a good description of $(D_{\k,\QFB}+\gamma \k)^{-1}$ within the $\k,\QFB$ calculus, we will construct first a good parametrix following roughly the strategy of \cite{GH1,GH2, GS, KR0}.  Indeed, with all our hypotheses, we know how to invert each of the normal operators of $D_{\k,\QFB}$, a good start to construct such a parametrix.  First, thanks to Corollary~\ref{hd.4} with $\delta=-\frac12$, we know how to invert 
\begin{equation}
      N_{0,0}(D_{\k,\QFB})=D_{\QFB}.
\label{qfble.7}\end{equation}
More precisely, consider the quasi-fibered cusp operator
$$
   D_{\QFC}= v^{-\frac12}D_{\QFB}v^{-\frac12}= v^{-1}D_{\QFB,-\frac12}.
$$
This operator is formally self-adjoint with respect to $L^2_b(M;E)$.  By Corollary~\ref{do.53} with $\delta=-\frac12$, it is Fredholm.  If $\Pi_{\ker_{L^2_b}D_{\QFC}}$ denotes the orthogonal projection onto its $L^2_b$ kernel, then by Corollary~\ref{hd.4}, there exists $G_{-\frac12}\in \Psi^{-1,\cG/\mathfrak{g}}_{\QFB,\cn}(M;E)$ such that
\begin{equation}
G_{-\frac12}D_{\QFC}=\Id -\Pi_{\ker_{L^2_b}D_{\QFC}}\quad \mbox{and} \quad D_{\QFC}G_{-\frac12}=\Id -\Pi_{\ker_{L^2_b}D_{\QFC}}.
\label{qfble.7}\end{equation}     
If the $L^2_b$ kernel of $D_{\QFC}$ is trivial, then one can simply take $v^{-\frac12}G_{-\frac12}v^{-\frac12}$ to invert
\eqref{qfble.7}, since 
$$
\begin{aligned}
D_{\QFB}(v^{-\frac12}G_{-\frac12}v^{-\frac12}) &= v^{\frac12}D_{\QFC}G_{-\frac12}v^{-\frac12}= v^{\frac12}\Id v^{-\frac12}=\Id,  \\
(v^{-\frac12}G_{-\frac12}v^{-\frac12})D_{\QFB}& = v^{-\frac12}G_{-\frac12}D_{\QFC}v^{\frac12}= v^{-\frac12}\Id v^{\frac12}=\Id.
\end{aligned}
$$
If instead $\Pi_{\ker_{L^2_b}D_{\QFC}}\ne 0$, then we only have that 
$$
\begin{aligned}
D_{\QFB}(v^{-\frac12}G_{-\frac12}v^{-\frac12}) &=\Id - v^{\frac12}\Pi_{\ker_{L^2_b}D_{\QFC}}v^{-\frac12}, \\
(v^{-\frac12}G_{-\frac12}v^{-\frac12})D_{\QFB} & =\Id-v^{-\frac12}\Pi_{\ker_{L^2_b}D_{\QFC}}v^{\frac12}.
\end{aligned}
$$
Following the strategy of \cite{GH2}, see also \cite{GS} and \cite{KR0}, this can be improved as follows.  Let $\{\varphi\}_{j=1}^J$ be an orthonormal basis of the kernel of $D_{\QFC}$ in $L^2_b(M;E)$, so that
$$
  \Pi_{\ker_{L^2_b}D_{\QFC}}= \sum_{j=1}^J (\pr_L^*\varphi_j)\pr_R^*(\varphi_j \nu_b)
$$
for the $b$-density $\nu_b= x^{2\mathfrak{w}}dg_{\QFB}$ with respect to which $D_{\QFC}$ is formally self-adjoint.  By Corollary~\ref{do.48b} with $\delta=-\frac12$, or equivalently by Corollary~\ref{do.48} with $\delta=-\frac12$ and\eqref{qfble.1},
$$
     \varphi_j\in v^{\nu} \cA_{\QFB}(M;E) \quad \forall \nu<\mu_R+\frac12,
$$
so that
$$
          \varphi_j\in \cA^{\mathfrak{s}}_{\QFB,-}(M;E) 
$$
with multiweight $\mathfrak{s}$ given by $\mathfrak{s}(H_i)=\mu_R+\frac12$ for each $i$.  In fact, taking into account Corollary~\ref{id.1}, we can say that there is $\widetilde{\mu}_R\ge \mu_R$ such that 
\begin{equation}
 \varphi_j\in v^{\nu} \cA_{\QFB}(M;E) \quad \forall \nu<\widetilde{\mu}_R+\frac12.
\label{id.3}\end{equation}

In particular, since $\widetilde{\mu}_R>\frac12$, $\{v^{-\frac12}\varphi_j\}_{j=1}^J$ is a basis of the kernel of $D_{\QFB}$ in $L^2_b(M;E)$.  If $\{\psi_j\}_{j=1}^J$ is an orthonormal basis of $\ker_{L^2_b}D_{\QFB}$, then 
\begin{equation}
   \Pi_{\ker_{L^2_b}D_{\QFB}}= \sum_{j=1}^J (\pr_L^* \psi_j)\pr_R^*(\psi_j\nu_b)
 \label{qfble.8}\end{equation}
 and
$$
  \psi_i= \sum_{j=1}^J \alpha_{ij} v^{-\frac12}\varphi_j
$$
for some matrix $\alpha_{ij}$.  In particular, $\psi_i\in v^{\nu}\cA_{\QFB}(M;E)$ for all $\nu<\widetilde{\mu}_R$.  If $\alpha^{ij}$ denotes the inverse of the matrix $\alpha_{ij}$, then
$$
   v^{-\frac12}\varphi_i= \sum_{j=1}^{J} \alpha^{ij}\psi_j,
$$  
so that 
\begin{equation}
\begin{aligned}
\Pi_{\ker_{L^2_b}D_{\QFB}}(v^{\frac12}\varphi_j)&= \sum_{k=1}^J \left( \int_M \psi_k v^{\frac12}\varphi_j\nu_b \right)\psi_k=
  \sum_{k=1}^J \left( \int_M \left( \sum_{i=1}^J \alpha_{ki}v^{-\frac12}\varphi_i \right)v^{\frac12}\varphi_j \nu_b \right) \psi_k \\
  &= \sum_{k=1}^J \alpha_{kj}\psi_k.
\end{aligned}
\label{qfble.9}\end{equation}
This implies that 
\begin{equation}
\psi_j^{\perp}:= v^{\frac12}\varphi_j-\Pi_{\ker_{L^2_b}D_{\QFB}}(v^{\frac12}\varphi_j)\in v^{\nu}\cA_{\QFB}(M;E)\;\forall\nu<\widetilde{\mu}_R.
\label{qfble.10}\end{equation}
\begin{lemma}
There exists $$
\displaystyle \chi_k\in M_{\QFC,-\frac12}\cap \left( \bigcap_{\nu<\nu_R} v^{\nu+\frac12}(\cA_{\QFB,2}(M;E)\cap\cA_{\QFB}(M;E))\right)
$$ 
with $\nu_R:=\min\{ \mu_R, \widetilde{\mu}_R-1 \}$ such that
$$
       D_{\QFC}\chi_k= v^{-\frac12}\psi_k^{\perp}.
$$
\label{qfble.11}\end{lemma}
\begin{proof}
Since $\widetilde{\mu}_R\ge\mu_R>\frac12$ by Assumption~\ref{do.26}, notice from \eqref{qfble.10} that 
$$
  v^{-\frac12}\psi_k^{\perp}\in \bigcap_{\nu<\widetilde{\mu}_R}v^{\nu-\frac12}\cA_{\QFB,2}(M;E)\subset L^2_b(M;E).
$$
Now, $\{v^{-\frac12}\varphi_j\}_{j=1}^J$ is a basis of $\ker_{L^2_b}D_{\QFB}$, which means that $\psi_k^{\perp}$ is orthogonal to $v^{-\frac12}\varphi_j$, that is, $v^{-\frac12}\psi_k^{\perp}$ is orthogonal to $\varphi_j$.  This means that $\Pi_{\ker_{L^2_b}D_{\QFC}}(v^{-\frac12}\psi_k^{\perp})=0$.  Thus, it suffices to take 
$$
    \chi_k:= G_{-\frac12}(v^{-\frac12}\psi_k^{\perp})\in M_{\QFC,-\frac12},
$$
for then
$$
\begin{aligned}
D_{\QFC}\chi_k &= D_{\QFC}G_{-\frac12}(v^{-\frac12}\psi_k^{\perp})= (\Id-\Pi_{\ker_{L^2_b}D_{\QFC}})(v^{-\frac12}\psi_k^{\perp}) \\
 &= v^{-\frac12}\psi_k^{\perp},
\end{aligned}
$$
while by \eqref{qfble.10}, Corollary~\ref{con.2} and Corollary~\ref{hd.4} and \eqref{qfble.1}, we have that 
$$
    \chi_k\in \bigcap_{\nu<\nu_R}v^{\nu+\frac12}\cA_{\QFB,2}(M;E).
$$
Using boundedness of classical pseudodifferential operators on Hölder spaces as in \cite[Proposition~3.27]{MazzeoEdge}, we can deduce a $L^{\infty}$ version of Corollary~\ref{con.2} to deduce as well that
$$
    \chi_k\in \bigcap_{\nu<\nu_R}v^{\nu+\frac12}\cA_{\QFB}(M;E)
$$
as claimed.
\end{proof}

Thanks to \eqref{qfble.9} we see that
\begin{equation}
\begin{aligned}
D_{\QFB} v^{-\frac12} &\left(  G_{-\frac12} + \sum_{j=1}^{J} ( \pr_L^*\chi_j \pr_R^*(\varphi_j\nu_b)+  \pr_L^*\varphi_j \pr_R^*(\chi_j\nu_b) ) \right)v^{-\frac12} \\
&\hspace{5cm}= \Id -v^{\frac12}\Pi_{\ker_{L^2_b}D_{\QFC}}v^{-\frac12}+ \sum_{j=1}^J \pr_L^*\psi_j^{\perp} \pr_R^*(v^{-\frac12}\varphi_j \nu_b) \\
&\hspace{5cm}=\Id + \sum_{j=1}^J \left( -\pr_L^*(v^{\frac12}\varphi_j)\pr_R^*(v^{-\frac12}\varphi_j\nu_b) + \pr_L^*\psi_j^{\perp}\pr_R^*(v^{-\frac12}\varphi_j\nu_b) \right)\\
&\hspace{5cm}= \Id- \sum_{j=1}^J \pr_L^*(\Pi_{\ker_{L^2_b}D_{\QFB}}(v^{\frac12}\varphi_j))\pr_R^*(v^{-\frac12}\varphi_j\nu_b) \\
& \hspace{5cm}= \Id -\sum_{j=1}^J \sum_{k=1}^J \alpha_{kj}\pr_L^*(\psi_k) \sum_{i=1}^J \alpha^{ji}\pr_R^*(\psi_i \nu_b) \\
& \hspace{5cm}= \Id -\sum_{j=1}^J \pr_L^*\psi_j \pr_R^*(\psi_j\nu_b)=\Id-\Pi_{\ker_{L^2_b}D_{\QFB}}.
\end{aligned}
\label{qfble.12}\end{equation}
To invert $D_{\k,\QFB}$ at $H_{00,0}$, this suggests to consider the approximate inverse 
\begin{equation}
  Q_0:= \k^{-1}\gamma G^{-1}_{00,0}+ G^0_{00,0}
\label{qfble.13}\end{equation}
with
$$
\begin{gathered}
G_{00,0}^{-1}:= \sum_{j=1}^J \pr_L^*(\psi_j) \pr_R^*(\psi_j\nu_b)= \Pi_{\ker_{L^2_b}D_{\QFB}}, \\
G_{00,0}^0:= v^{-\frac12}\left( G_{-\frac12} +\sum_{j=1}^J \left( \pr_L^*\chi_j\pr_R^*(\varphi_j\nu_b)+ \pr_L^*\varphi_j \pr_R^*(\chi_j \nu_b) \right)  \right) v^{-\frac12}.
\end{gathered}
$$
On $M^2_{\QFB}\times [0,\infty)_{\k}$, it is such that 
$$
     (D_{\QFB}+\gamma\k)Q_0= \Id +R_0 \quad \mbox{with} \quad R_0=\k\gamma G^0_{00,0}.
$$

By Lemma~\ref{kqfb.18b} and the fact that $2\mu_R>1$, we can regard $Q_0$ as an element of $\Psi^{-1,\cQ_0/\mathfrak{q}_0}_{\k,\QFB,\cn}(M;E)$ with index family $\cQ_0$ given by
\begin{equation}
\cQ_0|_{H_{00,0}}=\bbN_0-1, \quad \cQ_0|_{\ff_{i,0}}= \cR_0|_{\ff_{i,+}}=\bbN_0 \quad \cQ_{0}|_{H_{ii,0}}=\cQ_0|_{H_{ii,+}}=\bbN_0+\bd_i,  \quad i>0,
\label{qfble.14}\end{equation} 
and the empty set elsewhere, while the multiweight $\mathfrak{q}_0$ is such that for $i>0$ and $j>0$,
\begin{equation}
\begin{gathered}
\mathfrak{q}_0(H_{i0,0})=\nu_R, \quad \mathfrak{q}_0(H_{i0,+})=\nu_R, \quad \mathfrak{q}_0(H_{0i,0})=\bd_i+1+\nu_R, \quad \mathfrak{q}_0(H_{0i,+})=\bd_i+1+\nu_R, \\
 \mathfrak{q}_0(H_{ij,0})=\mathfrak{q}_0(H_{ij,+})>\left\{\begin{array}{ll} \bd_i, & i=j, \\ \bd_j+1, & i\ne j, \end{array}  \right. \quad \mathfrak{q}_0(H_{00,0})=\infty \quad \mathfrak{q}_0(\ff_{i,0})= \mathfrak{q}_0(\ff_{i,+})>0.
 \end{gathered}
\label{qfble.15}\end{equation}
Correspondingly, the error term $R_0=\k \gamma G^0_{00,0}$ can be seen as an element of $\Psi^{-1,\cR_0/\mathfrak{r}_0}_{\k,\QFB,\cn}(M;E)$ with index family $\cR_0$ such that 
\begin{equation}
\cR_0|_{H_{00,0}}=\bbN_0+1, \quad \cR_0|_{\ff_{i,0}}=\bbN_0+1, \quad \cR_0|_{\ff_{i,+}}=\bbN_0, \quad \cR_{0}|_{H_{ii,0}}=\bbN_0+\bd_i+1, \quad \cR_0|_{H_{ii,+}}=\bbN_0+\bd_i,
\label{qfble.16}\end{equation}
and the empty set elsewhere, and with  multiweight $\mathfrak{r}_0$ such that for $i>0$ and $j>0$,
\begin{equation}
\begin{gathered}
\mathfrak{r}_0(H_{00,0})=\infty, \quad \mathfrak{r}_0(\ff_{i,0})>1, \quad \mathfrak{r}_0(H_{i0,0})=\nu_R+1, \quad \mathfrak{r}_0(H_{0i,0})=\bd_i+2+\nu_R, \quad \mathfrak{r}_0(H_{ij,0})> \left\{\begin{array}{ll} \bd_i+1, & i=j \\ \bd_j+2, & i\ne j,  \end{array}  \right. \\
 \mathfrak{r}_0(\ff_{i,+})>0, \quad \mathfrak{r}_0(H_{i0,+})=\nu_R, \quad \mathfrak{r}_0(H_{0i,+})=\bd_i+1+\nu_R, \quad \mathfrak{r}_0(H_{ij,+})> \left\{\begin{array}{ll} \bd_i, & i=j \\ \bd_j+1, & i\ne j.  \end{array}  \right.
\end{gathered}
\label{qfble.17}\end{equation}
This shows that $\cQ_0$ also inverts $D_{\k,\QFB}$ at $\ff_{i,0}$.  Since $\widetilde{\mu}_R\ge\mu_R>\frac12$ and since $\varphi_k$ is of order $v^{\nu+\frac12}$ for all $\nu<\widetilde{\mu}_R$ while $\psi_k$ is of order $v^{\nu}$ for all $\nu<\widetilde{\mu}_R$ and $\chi_k$ is of order $v^{\nu+\frac12}$ for all $\nu<\nu_R$, we see  also that the only term of order $0$ of $Q_0$ at $\ff_{i,0}$ comes from $v^{-\frac12}G_{-\frac12}v^{-\frac12}$ in $G^0_{00,0}$.  Similarly, the terms of order $\bd_i$ at $H_{ii,0}$ comes exclusively from $v^{-\frac12}G_{-\frac12}v^{-\frac12}$ in $G^0_{00,0}$.  In particular, the term $q_i$ of order $\bd_i$ at $H_{ii,0}$ of $Q_0$ is such that $\widetilde{\Pi}_{h,i}q_i= q_i\widetilde{\Pi}_{h,i}=q_i$.
However, at $H_{ii,0}$, the error term does not vanish at order $\bd_i+1$.  Moreover, the error term $R_0$ does not vanish rapidly at $H_{ij,+}$.  The lack of decay at $H_{ij,+}$ can be easily remedied. Indeed, let $\varphi\in \CI(\bbR)$ be a cut-off function which is equal to zero on $(-\infty,1]$ and equal to $1$ on $[2,\infty)$, and consider the function
\begin{equation}
    \varphi_{\epsilon}(r)= \varphi\left( \frac{r}{\epsilon} \right)
\label{qfble.17b}\end{equation}
for $\epsilon>0$.  Then for $\epsilon>0$ sufficiently small, it suffices to replace $Q_0$ by 
$$
Q_0'=\varphi_{\epsilon}\left(\frac{v}{\k}\right)Q_0\varphi_{\epsilon}\left(\frac{v}{\k}\right).
$$
By \eqref{su.10}, we see that 
\begin{equation}
 [D_{\QFB},\varphi_{\epsilon}\left(\frac{v}{\k}\right)]= \varphi_{\epsilon}'\left(\frac{v}{\k}\right)\frac{1}{\k} [D_{\QFB},v]\in \frac{v^2}{\k}\CI(M\times [0,\infty)_{\k};\End(E)),
\label{qfble.18}\end{equation}
which ensures that after cutting off, the good decay of the error term at $H_{0i,0}$ is preserved since $\varphi_{\epsilon}'\left(\frac{v}{\k}\right)$ is supported away from $H_{0i,0}$ (when $v$ is seen as a left variable), while at $H_{i0,0}$, the decay is preserved  by \eqref{qfble.18} and \eqref{qfble.15}.  However, in the process, part of the term of order $0$ of $Q_0$ at $\ff_{i,0}$ is removed, so that the new error term has a term of order 0 at $\ff_{i,0}$.  But by Remark~\ref{new.3}, we can add back this term and  cut it off  so that $Q_{0}'$ and $Q_0$ have the same term of order zero at $\ff_{i,0}$  without compromising the decay of the error term that was achieved at the other boundary hypersurfaces.  Moreover, the cut-off function used to add back this term of order $0$ at $\ff_{i,0}$ can be chosen to be constant on the fibers of the lift of the fiber bundle  \eqref{ps.5} to $H_{ii,0}$.  Consequently, this discussion shows that we can construct $Q_1\in\Psi^{-1,\cQ_1,\mathfrak{q}_1}_{\k,\QFB,\cn}(M;E)$ out of $Q_0$ with $\cQ_1$ and $\mathfrak{q_1}$ as $\cQ_0$ and $\mathfrak{q}_0$, except at $H_{ij,+}$ for $(i,j)\ne (0,0)$  and at $\ff_{i,+}$ where
$$
   \cQ_1|_{H_{ij,+}}=\emptyset, \quad \cQ_1|_{\ff_{i,+}}=\bbN_0, \quad \mbox{and} \quad \mathfrak{q}_1(H_{ij,+})=\infty, \quad \mathfrak{q}_1(\ff_{i,+})=\infty,
$$  
in such a way that 
$$
  D_{\k,\QFB}Q_1=\Id-R_1
$$
with $R_1\in \Psi^{-1,\cR_1/\mathfrak{r}_1}_{\k,\QFB,\cn}(M;E)$, where $\cR_1$ and $\mathfrak{r}_1$ are as $\cR_0$ and $\mathfrak{r}_0$, except at $H_{ij,+}$ where 
$$
      \cR_1|_{H_{ij,+}}=\emptyset \quad \mbox{and} \quad \mathfrak{r}_1(H_{ij,+})=\infty \quad \mbox{for} \; (i,j)\ne (0,0).
$$
Thanks to \eqref{qfble.18} and the fact the cut-off function used to add back part of the term order $0$ at $\ff_{i,0}$ is constant along the fibers of the lift of the fiber bundle \eqref{ps.5} to $H_{ii,0}$, notice that the term $r_i$ of order $\bd_i+1$ of $R_1$ at $H_{ii,0}$ is still such that $r_i\widetilde{\Pi}_{h,i}=r_i$ since $R_1$ is still such that $R_1D_{\QFB}$ has no term of order $b_i+1$ at $H_{ii,0}$.  In fact, recursively adding  terms in $\k\Psi^{-2}_{\k,\QFB}(M;E)$, $\k\Psi^{-3}_{\k,\QFB}(M;E)$ etc. supported near the lifted diagonal to $Q_1$, we can decrease the order of the error term as much as we want.  Taking an asymptotic sum of these corrections and adding them to $Q_1$, we may therefore supposed that in fact $R_1\in \Psi^{-\infty,\cR_1/\mathfrak{r}_1}_{\k,\QFB,\cn}(M;E)$ is of order $-\infty$.  

Using Assumption~\ref{qfble.5}, we can improve this parametrix as follows.

\begin{proposition}
There exists $Q_2\in \Psi^{-1,\cQ_2/\mathfrak{q}_2}_{\k,\QFB,\cn}(M;E)$ and $R_2\in\Psi^{-\infty,\cR_2/\mathfrak{r}_2}_{\k,\QFB}(M;E)$ such that 
$$
       D_{\k,\QFB}Q_2= \Id-R_2
$$
with $N_{i,+}(R_2)=0$ for each $i$, where $\cQ_{2},\mathfrak{q}_{2}, \cR_{2}$ and $\mathfrak{r}_{2}$ satisfy the same properties as $\cQ_{1},\mathfrak{q}_{1}, \cR_{1}$ and $\mathfrak{r}_{1}$.

Furthermore, the term $q_{2,i}$ of order $b_i$ at $H_{ii,0}$ of $Q_2$ is such that $\widetilde{\Pi}_{h,i}q_{2,i}=q_{2,i}\widetilde{\Pi}_{h,i}=q_{2,i}$ with decay of order $\mu_R$ and $\bd_j+1+\mu_R$ at $H_{j0,0}$ and $H_{0j,0}$ for $H_j\le H_i$.  Similarly, the term $r_{2,i}$ of order $\bd_i+1$ at $H_{ii,0}$ of $R_2$ still such that $r_{2,i}= r_{2,i}\widetilde{\Pi}_{h,i}$ with decay of order $\mu_R+1$ and $\bd_j+2+\mu_R$ at $H_{j0,0}$ and $H_{0j,0}$ for $H_j\le H_i$.  Finally, $\widetilde{\Pi}_{h,i}Q_2$ has the same asymptotic behavior as $Q_2$ at $H_{ii,0}$ up to a term of order $\bd_i+1$ and a weakly conormal section vanishing faster than $x_{H_{ii,0}}^{\bd_i+1+\nu}$ for some $\nu>0$.
\label{qfble.19}\end{proposition}
\begin{proof}
By Assumption~\ref{qfble.5} and Remark~\ref{qfble.6}, it suffices to choose $Q_2$ to be like $Q_1$, except at $\ff_{i,+}$ for all $i$ where we choose $Q_2$ such that $N_{i,+}(Q_2)=N_{i,+}(D_{\k,\QFB})^{-1}$. The possible better decay of $q_{2,i}$ and $r_{2,i}$ at $H_{j0,0}$ and $H_{0j,0}$ follows from the fact that these terms comes exclusively from $v^{-\frac12}G_{-\frac12}v^{-\frac12}$ in $G^0_{00,0}$.
\end{proof}

To improve the parametrix further, we need to remove the term of order $\bd_i+1$ at $H_{ii,0}$ of the error term.  Now, in terms of the model operator \eqref{su.5}, we need to consider the corresponding model operator for $D_{\QFB,\k}$,
\begin{equation}
  D_{\cC_{\phi_i}}+ \k\gamma= D_{v,i}+ x^{-\mathfrak{w}}\widetilde{\eth}_{\cC_i}x^{\mathfrak{w}}+ \k\gamma.
\label{qfble.20}\end{equation}
Since the term $r_{2,i}$ of $R_2$ of order $\bd_{i}+1$ at $H_{ii,0}$ is such that 
$$
      r_{2,i}\widetilde{\Pi}_{h,i}=r_{2,i},
$$
we can first invert the corresponding model operator induced from $\widetilde{\Pi}_{h,i}(x^{-\mathfrak{w}}\widetilde{\eth}_{\cC_i}x^{\mathfrak{w}}+\k\gamma)\widetilde{\Pi}_{h,i}$, namely, by \eqref{do.12} and \eqref{do.25},
\begin{equation}
  v\left( c v\frac{\pa}{\pa v} + \left( D_{S_i}+\frac{c}2) \right) \right)+ \k\gamma.
\label{qfble.21}\end{equation}  
As in \cite{GH1} or \cite{KR0}, it is useful to rewrite the operator \eqref{qfble.20} with respect to the variable $\kappa= \frac{\k}{v}$,
\begin{equation}
   D_{\cC_{\phi_i}}+\k\gamma= D_{v,i}+ \k D_{\cC_i},  \quad D_{\cC_i}:= -c\frac{\pa}{\pa\kappa}+ \frac{1}{\kappa}(D_{S_i}+\frac{c}2)+ \gamma.
\label{qfble.22}\end{equation}
The operator $D_{\cC_i}$ can be seen as an operator acting on sections of $\ker D_{v,i}$ on the cone $S_i\times [0,\infty)_{\kappa}$ with cone metric $d\kappa^2+\kappa^2g_{S_i}$.  If $H_i$ is maximal, $H_i=S_i$ and $\phi_i:H_i\to S_i$ is the identity map, the natural pseudodifferential calculus in which this operator can be inverted is the one induced by the front face $H^{\Qb}_{ii,0}$ in $M^2_{\QAC-\Qb}$ with composition induced by the lift $H^{\Qb}_{iii,0}$ of $H^{\rp}_{iii,0}$ to $M^3_{\QAC-\Qb}$.  More generally, we must consider the manifold with fibered corners $S_i\times [0,1)$ and the pseudodifferential calculus induced by the corresponding front face $H^{\Qb}_{ii,0}$ in $(S_i\times [0,1))^2_{\QAC-\Qb}$.  To show that $D_{\cC_i}$ can be inverted within this calculus, we will proceed by induction on $i$, assuming without loss of generality that the boundary hypersurfaces $H_1,\ldots, H_{\ell}$ of $M$ are listed in such a way that 
$$
    H_i<H_j\quad \Longrightarrow \quad i<j.
$$
In the process, we will prove the following result by induction on $i$, which in fact is really the result we are interested in.
\begin{proposition}
For $i\ge 0$ fixed, there exists $Q_{3,i}\in \Psi^{-1,\cQ_{3,i}/\mathfrak{q}_{3,i}}_{\k,\QFB,\cn}(M;E)$ and 
$R_{3,i}\in \Psi^{-\infty,\cR_{3,i}/\mathfrak{r}_{3,i}}$ such that 
$$
   D_{\k,\QFB}Q_{3,i}=\Id-R_{3,i}
$$
with $\cQ_{3,i},\mathfrak{q}_{3,i}, \cR_{3,i}$ and $\mathfrak{r}_{3,i}$ satisfying the same properties as 
 $\cQ_{1},\mathfrak{q}_{1}, \cR_{1}$ and $\mathfrak{r}_{1}$, except that 
 $$
  \mathfrak{q}_{3,i}(H_{00,0})>0, \quad \mathfrak{r}_{3,i}(H_{00,0})>1 \quad \mbox{and} \quad \mathfrak{r}_{3,i}(H_{0j,0})=\bd_j+1+\mu_R \quad \mbox{for} \; 0<j\le i.
 $$
Moreover, for $0<j\le i$, $R_{3,i}$ has no term of order $\bd_{j}+1$ at $H_{jj,0}$.  Finally, for $j>i$,  
the term $q_{2,j}$ of order $b_j$ at $H_{jj,0}$ of $Q_{3,i}$ is such that $\widetilde{\Pi}_{h,j}q_{2,j}=q_{2,j}\widetilde{\Pi}_{h,j}=q_{2,j}$ with decay of order $\mu_R$ and $\bd_k+1+\mu_R$ at $H_{k0,0}$ and $H_{0k,0}$ for $H_k\le H_j$, while the term $r_{2,j}$ of order $\bd_j+1$ at $H_{jj,0}$ of $R_{3,i}$ is such that $r_{2,j}=r_{2,j}\widetilde{\Pi}_{h,j}$ with multiweight at $H_{jj,0}\cap H_{k0,0}$ and $H_{jj,0}\cap H_{0k,0}$ for $H_k<H_j$ given by $\mu_R+1$ and $\bd_k+2+\mu_R$.
\label{qfble.23}\end{proposition}

For $i=0$, notice that the statement of Proposition~\ref{qfble.23} holds with $Q_{3,0}:=Q_2$ and $R_{3,i}:=R_2$.  Assuming now that the proposition holds for $i-1$ for some $i\in\{1,\ldots,\ell\}$, we will prove that the proposition also holds for $i$ by inverting the model operator $D_{\cC_i}$.  First notice that the parametrix $Q_{3,i-1}$ given by Proposition~\ref{qfble.23} for $i-1$ yields a parametrix for $D_{\cC_i}$ by restriction to $H_{ii,0}$.  More precisely, $Q_{3,i-1}$ has a term of order $\bd_i$, which in terms of $b$-densities corresponds to a term of order $-1$.  Hence, $\k Q_{3,i-1}$ is of order $0$ at $H_{ii,0}$ in terms of $b$-densities.  Let 
$$
        Q_{\cC_i}:= \k Q_{3,i-1}|_{H_{ii,0}}
$$
be its restriction.  On the other hand, $R_{3,i-1}$ has a term of order $\bd_i+1$ at $H_{ii,0}$, that is, of order 0 in terms of $b$-densities.  let 
$$
   R_{\cC_i}= \widetilde{\Pi}_{h,i}R_{3,i-1}|_{H_{ii,0}}
$$
be its restriction to $H_{ii,0}$.  Then 
\begin{equation}
\begin{aligned}
\widetilde{\Pi}_{h,i}-R_{\cC_i} &= \widetilde{\Pi}_{h,i}(\Id-R_{3,i-1})|_{H_{ii,0}}= \widetilde{\Pi}_{h,i}(D_{\k,\QFB}Q_{3,i-1})|_{H_{ii,0}}\\
  &= (\k D_{\cC_i}Q_{3,i-1})|_{H_{ii,0}}= (D_{\cC_i}\k Q_{3,i-1})|_{H_{ii,0}} \\
  &= D_{\cC_i}Q_{\cC_i},
\end{aligned}
\label{qfble.24}\end{equation}
that is, $Q_{\cC_i}$ is a parametrix of $D_{\cC_i}$ with 
\begin{equation}
   D_{\cC_i}Q_{\cC_i}= \widetilde{\Pi}_{h,i}-R_{\cC_i}.
\label{qfble.25}\end{equation}
Now, the operator $D_{\cC_i}$ being an operator geometrically associated to the cone $(S_i\times [0,\infty)_{\kappa}, d\kappa^2+ \kappa^2g_{S_i})$, it can be thought as a wedge-$\QAC$ operator.  If we let $u_i$ be a lift of a boundary defining function of $S_i\times \{0\}$ in $S_i\times [0,\infty]_{\frac{\k}{v_i}}$ to
$$
 [S_{i}\times [0,\infty]_{\frac{\k}{v_i}}; \pa_{i-1}S_i\times \{0\},\ldots, \pa_1S_i\times \{0\}],
$$
this means that $u_i D_{\cC_i}$ and $u_i^{\frac12}D_{\cC_i}u_i^{\frac12}$ are edge-$\QAC$ operators.  In particular, 
$u_i^{-\frac12}Q_{\cC_i} u_{i}^{-\frac12}$ is a parametrix for $u_i^{\frac12}D_{\cC_i}u_{i}^{\frac12}$,
\begin{equation}
(u_i^{\frac12}D_{\cC_i}u_{i}^{\frac12})(u_i^{-\frac12}Q_{\cC_i} u_{i}^{-\frac12})=\widetilde{\Pi}_{h,i}- u_i^{\frac12}R_{\cC_i}u_i^{-\frac12}.
\label{qfble.26}\end{equation} 
\begin{lemma}
The edge-$\QAC$ self-adjoint extension of $u_i^{\frac12}D_{\cC_i}u_i^{\frac12}$ acting formally on $L^2_b(S_i\times[0,\infty];\ker D_{v,i})$ is Fredholm.
\label{qfble.27}\end{lemma}
\begin{proof}
By Proposition~\ref{qfble.23} for $i-1$, $u_i^{\frac12} R_{\cC_i}u_i^{-\frac12}$ decays enough at each boundary hypersurfaces of $H_{ii,0}$ to ensure that $u_i^{\frac12}R_{\cC_i}u_i^{-\frac12}$ is compact when acting on $L^2_b(S_i\times [0,\infty];\ker D_{v,i})$, so that $u_i^{\frac12}D_{\cC_i}q^{\frac12}_i$ has a right inverse modulo compact operators.  Taking the adjoint gives an inverse on the left modulo compact operators on the edge-$\QAC$ Sobolev space of order $1$ associated to $L^2_b(S_i\times [0,\infty];\ker D_{v,i})$, showing as claimed that $u_i^{\frac12}D_{\cC_i}u_i^{\frac12}$ is Fredholm as an edge-$\QAC$ operators.  
\end{proof}
Knowing that $u_i^{\frac12}D_{\cC_i}u_i^{\frac12}$ is Fredholm, it is not hard to see that it is in fact invertible as the next lemma shows.
\begin{lemma}
The Fredholm operator of Lemma~\ref{qfble.27} is in fact an isomorphism.  
\label{qfble.28}\end{lemma}
\begin{proof}
We will follow the strategy of the proof of \cite[Lemma~8.9]{KR0}.  Using that $\gamma$ anti-commutes with $-c\frac{\pa}{\pa \kappa}+ \frac{1}{\kappa}(D_{S_i}+\frac{c}2)$ and that $c$ anti-commutes with $D_{S_i}$, we compute that 
$$
\kappa^2D^2_{\cC_i}= -\left(\kappa\frac{\pa}{\pa\kappa} \right)^2 + 2\kappa\frac{\pa}{\pa\kappa}+ \left( D_{S_i}^2+cD_{S_i}-\frac34 \right)+\kappa^2,  
$$
that is,
$$
\kappa^{-1}(\kappa^2D_{\cC_i}^2)\kappa=-\left(\kappa\frac{\pa}{\pa\kappa} \right)^2 + \left( D_{S_i}^2+cD_{S_i}+\frac14 \right)+\kappa^2.
$$
Setting $\widetilde{D}_{S_i}:= cD_{S_i}$, this becomes
\begin{equation}
\kappa^{-1}(\kappa^2D_{\cC_i}^2)\kappa=-\left(\kappa\frac{\pa}{\pa\kappa} \right)^2 + \left( \widetilde{D}_{S_i}^2+\widetilde{D}_{S_i}+\frac14 \right)+\kappa^2.
\label{qfble.29}\end{equation}
Now, $\widetilde{D}_{S_i}$ is formally self-adjoint since $(\widetilde{D}_{S_i})^*= (cD_{S_i})^*=-D_{S_i}c=cD_{S_i}=\widetilde{D}_{S_i}$.  Moreover, we see from \eqref{do.24} that 
\begin{equation}
  I(D_{b,i},\lambda)= c(\lambda-\widetilde{D}_{S_i}+\frac12).
\label{qfble.30}\end{equation}
Hence, Assumption~\ref{do.26} with $\delta=-\frac12$ implies that $\widetilde{D}_{S_i}$ has no eigenvalue in the range $(-\mu_R-\frac12,\mu_L-\frac12)=(-\mu_R-\frac12,\mu_R+\frac12)$.  Now, if $\lambda$ is an eigenvalue of $\widetilde{D}_{S_i}$ with eigensection $\widetilde{\sigma}$, then $\widetilde{\sigma}$ is an eigensection of $\widetilde{D}_{S_i}^2+ \widetilde{D}_{S_i}+\frac14$ with eigenvalue $(\lambda+\frac12)^2$.  Since $\lambda\notin (-\mu_R-\frac12,\mu_R+\frac12)$, we see in particular that 
\begin{equation}
(\lambda+\frac12)^2\ge \mu_R^2>\frac14.
\label{qfble.31}\end{equation}
Hence, decomposing the kernel of the operator \eqref{qfble.29} in terms of the eigenspaces of $\widetilde{D}_{S_i}$, we obtain the modified Bessel equation
\begin{equation}
\left( -\left( \kappa \frac{\pa}{\pa \kappa}\right)^2+\alpha^2+\kappa^2 \right)f=0 \quad \mbox{with} \quad \alpha^2= (\lambda+\frac12)^2\ge \mu_R^2>\frac14.
\label{qfble.32}\end{equation}
A basis of solutions of this equation is given by the modified Bessel functions $K_{\alpha}$ and $I_{\alpha}$.  The function $I_{\alpha}$ grows exponentially as $\kappa\to \infty$ and tends to zero as $\kappa\searrow 0$, while the solution $K_{\alpha}$ blows up like $\kappa^{-|\alpha|}$ as $\kappa\searrow 0$ and decays exponentially at infinity.  Thus, the modified Bessel equation \eqref{qfble.32} has no non-trivial solution in
$$
        u_i^{-\mu_R}L^2_b(S_i\times [0,\infty];\ker D_{v,i}),
$$  
where $u_i$ is seen as a boundary defining function for $S_i\times \{0\}$ in $S_i\times [0,\infty]_{\kappa}$.  Consequently, the operator $\kappa^{-1}(\kappa^2D_{\cC_i}^2)\kappa$ has trivial kernel in 
$$
     u_i^{-\mu_R}L^2_b(S_i\times [0,\infty];\ker D_{v,i}),
$$  
hence, that $D_{\cC_i}$ has trivial kernel in 
$$
\kappa u_i^{-\mu_R}L^2_b(S_i\times [0,\infty];\ker D_{v,i}),
$$
thus trivial kernel in
$$
u_i^{1-\mu_R}L^2_b(S_i\times [0,\infty];\ker D_{v,i})
$$
since elements of the kernel either grow or decrease exponentially as $\kappa\to \infty$.  Since $\mu_R>\frac12$, this implies that $D_{\cC_i}$ has trivial kernel in 
$$
     u_i^{\frac12}L^2_b(S_i\times [0,\infty];\ker D_{v,i}).
$$
In particular, $u^{\frac12}_iD_{\cC_i}u_i^{\frac12}$ has trivial kernel in $L^2_b(S_i\times [0,\infty];\ker D_{v,i})$.  Since $u^{\frac12}_iD_{\cC_i}u_i^{\frac12}$ is formally self-adjoint, this means that the Fredholm operator of Lemma~\ref{qfble.27} also has trivial cokernel, hence is an isomorphism.
\end{proof}

Let  $G_{\cC_i}$ be the operator such that $u_i^{-\frac12}G_{\cC_i}u_i^{-\frac12}$ is  the inverse of the Fredholm operator of Lemma~\ref{qfble.28}.  We can use the parametrix $Q_{\cC_i}$ to give a pseudodifferential characterization of the inverse.
\begin{lemma}
On the face $H^{\Qb}_{ii,0}$ of $(S_i\times [0,1))^2_{\QAC-\Qb}$, we have that  
$$
u_i^{-\frac12}G_{\cC_i}u_i^{-\frac12} \in \Psi^{-1,\cG_i/\mathfrak{g}_i}_{\QAC-\Qb,\cn}(H_{ii,0}^{\Qb};\ker D_{v,i})
$$ 
with $\QAC-\Qb$ nonnegative  index family $\cG_i$ except at $H^{\Qb}_{jj,0}$ for $H_j<H_i$, where instead  
$$
     \inf\Re(\cG_i|_{H^{\Qb}_{jj,0}})\ge\bd_j
$$
with
\begin{equation}
\cG_i|_{\ff_{j,+}^{\Qb}}=\bbN_0, \quad \cG_i|_{H^{\Qb}_{kj,+}}= \cG_i|_{H_{j0,+}^{\Qb}}=\cG_i|_{H^{\Qb}_{0j,+}}= \cG_i|_{H_{j0,0}^{\Qb}}=\cG_i|_{H^{\Qb}_{0j,0}}= \emptyset, 
\label{qfble.33b}\end{equation}
for $H_k\le H_i$ and $H_j\le H_i$,
and with $\QAC-\Qb$ positive multiweight $\mathfrak{g}_i$ except at $H^{\Qb}_{jj,0}$ for $H_j<H_i$, where instead  
$$
     \mathfrak{g}_i(H^{\Qb}_{jj,0})>\bd_j,
$$
and   such that
\begin{equation}
\mathfrak{g}_i(\ff_{j,+}^{\Qb})=\mathfrak{g}_i(H^{\Qb}_{kj,+})= \mathfrak{g}_i(H_{j0,+}^{\Qb})=\mathfrak{g}_i(H^{\Qb}_{0j,+})= \infty, \quad
\mathfrak{g}_i(H^{\Qb}_{j0,0})=\mu_R+\frac12, \quad \mathfrak{g}_i(H^{\Qb}_{0j,0})= \bd_j+\frac32 +\mu_R
\label{qfble.33c}\end{equation}
for $H_{j}\le H_i$ and $H_k\le H_i$. Furthermore, for $H_{j}< H_i$, $G_{\cC_i}$ and $Q_{\cC_i}$ have the same top order terms at $\ff_{i,+}^{\Qb}$, $\ff^{\Qb}_{j,+}$, $\ff^{\Qb}_{j,0}$,  $H^{\Qb}_{jj,0}$ and  $H^{\Qb}_{00,0}$.
\label{qfble.33}\end{lemma}
\begin{proof}
The parametrix $u_{i}^{-\frac12}Q_{C_i}u_{i}^{-\frac12}$ for the edge-$\QAC$ operator of Lemma~\ref{qfble.28} shows that  its models at  $\ff^{\Qb}_{i,+}$, $\ff^{\Qb}_{j,+}$, $\ff^{\Qb}_{j,0}$ and $H^{\Qb}_{00,0}$ for $H_j<H_i$ are invertible. In fact, Assumption~\ref{qfble.5} ensures that $u_i^{-\frac12}D_{\cC_i}u_i^{\frac12}$ is fully elliptic at $\ff^{\Qb}_{j,+}$ for $H_j\le H_i$.  To invert at $H^{\Qb}_{00,0}$ and at $\ff^{\Qb}_{j,0}$ for $H_j<H_i$, we can just use Assumption~\ref{do.26} and take the inverse Fourier transform as in Lemma~\ref{do.39} to get the right model inverse, since $u_i^{\frac12}\gamma u_i^{\frac12}$ does not contribute to the models related to the edge part of the operator.  Thus, we can replace $Q_{\cC_i}$ by $\widetilde{Q}_{\cC_i}$ such that $u_i^{-\frac12}\widetilde{Q}_{\cC_i}u_i^{-\frac12}\in \Psi^{-1,\widetilde{\cQ}_i/\widetilde{\mathfrak{q}}_i}_{\QAC-\Qb,\cn}(H^{\Qb}_{ii,0};\ker D_{v,i})$ with $\widetilde{\cQ}_i$ a $\QAC-\Qb$ nonnegative index family and $\widetilde{\mathfrak{q}}_i$ a $\QAC-\Qb$ positive index family, except at $H^{\Qb}_{jj,0}$ for $H_j<H_i$ where instead 
$$
     \widetilde{\mathfrak{q}}_i(H^{\Qb}_{jj,0})>\bd_j  \quad \mbox{and} \quad \inf\Re(\widetilde{\cQ}_i|_{H^{\Qb}_{ii,0}})\ge \bd_j,
$$
with
\begin{equation}
(u_i^{\frac12}D_{\cC_i}u_i^{\frac12})(u_i^{-\frac12}\widetilde{Q}_{\cC_i}u_i^{-\frac12})= \widetilde{\Pi}_{h,i}- u_i^{\frac12}\widetilde{R}_{\cC_i}u_i^{-\frac12}
\label{qfble.33d}\end{equation}
where $u_i^{\frac12}\widetilde{R}_{\cC_i}u_i^{-\frac12}\in \Psi^{-\infty,\widetilde{\mathfrak{r}}_i}_{\QAC-\Qb,\cn}(H^{\Qb}_{ii,0};\ker D_{v,i})$ with $\widetilde{\mathfrak{r}}_i$ a $\QAC-\Qb$ positive index family.  Moreover, $u_i^{-\frac12}\widetilde{Q}_{\cC_i}u_i^{-\frac12}$ has the same top order terms as $u_i^{-\frac12}Q_{\cC_i}u_i^{-\frac12}$ at $\ff^{\Qb}_{i,+}$, $\ff^{\Qb}_{j,+}$ for $H_j<H_i$ and $H^{\Qb}_{00,0}$.  

Using \eqref{qfble.33d} and its adjoint, we can apply the standard sandwich argument of \cite{MazzeoEdge} to obtain that
$$
u_i^{-\frac12}G_{\cC_i}u_i^{-\frac12}= u_i^{-\frac12}\widetilde{Q}_{\cC_i}u_i^{-\frac12}+ (u_i^{-\frac12}\widetilde{Q}^*_{\cC_i}u_i^{-\frac12})(u_i^{\frac12}\widetilde{R}_{\cC_i}u_i^{-\frac12})+ u_i^{-\frac12}\widetilde{R}_{\cC_i}^*G_{\cC_i}\widetilde{R}_{\cC_i}u_i^{-\frac12},
$$ 
or equivalently that
\begin{equation}
G_{\cC_i}= \widetilde{Q}_{\cC_i}+  \widetilde{Q}^*_{\cC_i} \widetilde{R}_{\cC_i}+  \widetilde{R}_{\cC_i}^*G_{\cC_i} \widetilde{R}_{\cC_i}.
\label{qfble.34}\end{equation}
Since $u_i^{-\frac12} \widetilde{R}^*_{\cC_i}u_i^{\frac12}$ and $u_i^{\frac12} \widetilde{R}_{\cC_i}u_i^{-\frac12}$ are $\mathfrak{r}$ residual for some multiweight $\mathfrak{r}$ and that $u_i^{-\frac12}G_{\cC_i}u_i^{-\frac12}$ is bounded on $L^2_b(S_i\times [0,\infty];\ker D_{v,i})$, we conclude that 
\begin{equation}
u_i^{-\frac12}G_{\cC_i}u_i^{-\frac12}\in \Psi^{-1,\cG_i/\mathfrak{g}_i}_{\QAC-\Qb,\cn}(H_{ii,0}^{\Qb};\ker D_{v,i})
\label{qfble.35}\end{equation}
for some index family $\cG_i$ and multiweight $\mathfrak{g}_i$ as claimed, but possibly not satisfying \eqref{qfble.33b} and \eqref{qfble.33c}. But using \eqref{qfble.35} and the rapid decay of $\widetilde{R}_{\cC_i}$  at $\ff_{j,+}^{\Qb}$ and $H^{\Qb}_{kj,+}$ for $H_j\le H_i$ and $H_k\le H_j$, we can apply a bootstrap argument via \eqref{qfble.34} using Theorem~\ref{ckqfb.12} to deduce that $G_{\cC_i}$ decays rapidly at $H_{kj,+}^{\Qb}$, $H_{0j,+}^{\Qb}$ and $H_{j0,+}^{\Qb}$ for $H_j\le H_i$, $H_k\le H_i$, and that it has a $\QAC-\Qb$ conormal  polyhomogeneous expansion at $\ff_{j,+}$ in nonnegative integer powers of $\rho_{\ff_{i,+}}$ as claimed. Using \eqref{qfble.26} instead of \eqref{qfble.33d}, we obtain
\begin{equation}
G_{\cC_i}= Q_{\cC_i}+  Q^*_{\cC_i} R_{\cC_i}+  R_{\cC_i}^*G_{\cC_i} R_{\cC_i},
\label{qfble.33e}\end{equation} 
which we can use to infer the claimed better decay at $H^{\Qb}_{j0,0}$ and $H^{\Qb}_{0j,0}$ in \eqref{qfble.33c}.

\end{proof}
We can now complete the inductive proof of Proposition~\ref{qfble.23}.
\begin{proof}[Proof of Proposition~\ref{qfble.23}]
As already observed, the Proposition holds for $i=0$, so it suffices to show that if it holds for $i-1\le \ell-1$, then it holds for $i$.  For this, we need to improve the parametrix $Q_{3,i-1}$ to get rid of the term of order $\bd_i+1$ at $H_{ii,0}$.  Let $r_i$ be the term of order $\bd_i+1$ of $R_{3,i-1}$ at $H_{ii,0}$.  To get rid of $\widetilde{\Pi}_{h,i}r_i$, we can use Lemmas~\ref{qfble.28} and \ref{qfble.33} and consider on $H_{ii,0}$ the term
$$
   q_i:= \k^{-1} G_{\cC_i}\widetilde{\Pi}_{h,i}r_i
$$
of order $\bd_i$ at $H_{ii,0}$, that is, of order $-1$ in terms of $b$-densities, since
$$
     (\k D_{\cC_i})q_i= D_{\cC_i}G_{\cC_i}\widetilde{\Pi}_{h,i}r_2=\widetilde{\Pi}_{h,i}r_2.
$$
Adding $q_i$ to $Q_{3,i-1}$ then gives the parametrix $Q_{3,i}'$ with error term almost as claimed, namely, its error terms still has possibly a term $r_i'$ of order $\bd_i+1$ at $H_{ii,0}$ but it is such that $ \widetilde{\Pi}_{h,i}r_i'=0$.  Hence, considering again the model operator \eqref{qfble.20}, it suffices to add a term
$$
    q_i':= D_{v,i}^{-1}r_i'
$$  
of order $\bd_i+1$ at $H_{ii,0}$ to $Q_{3,i}'$ to obtain the parametrix $Q_{3,i}$ with the desired parametrix.  Indeed, $q_i'$ is well-defined thanks to Corollary~\ref{hd.4} with $\delta=-\frac12$ and the analog of \eqref{qfble.12}  applied to the members of the vertical family $D_{v,i}$.   Extending $q_i$ using a right variable like $x_i'$ near $H_{j0,0}$ for $0<j\le i$ ensures that we can take $\mathfrak{r}_{3,i}(H_{j0,0})=\mathfrak{r}_0(H_{j0,0})=\nu_R+1$.  Similarly, extending $q_i$ and $q_i'$ using a right variable near $H_{jk,0}$ for $0<j$ and $0<k$ with $j\ne k$ ensures that we have $\mathfrak{r}_{3,i}(H_{jk,0})>\bd_k+2$. However, near $H_{0j,0}$ for $0<j\le i$, the extension of $q_i$ introduce in principle an error term with weight $\bd_j+1+\mu_R\le \bd_j+2+\nu_R$, hence a possible loss of decay at this boundary hypersurface.
\end{proof}
\begin{remark}
This proof by induction gives at the same time a proof of Lemmas~\ref{qfble.27}, \ref{qfble.28} and \ref{qfble.33} for each $i$.
\label{qfble.36}\end{remark}
In particular, Proposition~\ref{qfble.23} for $i=\ell$ yields the following parametrix.
\begin{proposition}
There exists $Q_4\in\Psi^{-1,\cQ_4,\mathfrak{q}_4}$ and $R_4\in \dot{\Psi}^{-\infty}_{\k,\QFB}(M;E)$  such that 
$$
  D_{\k,\QFB}Q_4=\Id-R_4
$$
with $\cQ_4$ and $\mathfrak{q}_4$ satisfying the same properties as $\cQ_1$ and $\mathfrak{q}_1$ except that
$$
  \mathfrak{q}_4(H_{00,0})>0,
$$ 
where $\dot{\Psi}^{-\infty}_{\k,\QFB}(M;E)$ is the subspace of $\Psi^{-\infty}_{\k,\QFB}(M;E)$ of Definition~\ref{kqfb.21}.
\label{qfble.37}\end{proposition}
\begin{proof}
Start with the parametrix $Q_{3,\ell}$ of Proposition~\ref{qfble.23} with error term $R_{3,\ell}$.  By the decay rates of $R_{3,\ell}$,  the composition results of Theorem~\ref{ckqfb.10} and the fact that
$$
    \mu_R+\widetilde{\mu}_R\ge 2\mu_R>1,
$$
taking $\mathfrak{r}_{3,\ell}$ possibly slightly smaller at some boundary hypersurfaces, we see that there exists $\delta>0$ such that for each hypersurface $H$ of $M^2_{\k,\QFB}$,
$$
      R_{3,\ell}= \cO(x_{H}^{\nu}) \; \forall \nu<\mu \quad \mbox{at} \; H \quad \Longrightarrow \quad R_{3,\ell}^k= \cO(x_H^{\nu+k\delta}) \; \forall\nu<\mu, \quad \forall k\in\bbN_0.
$$
Hence, we can take an asymptotic sum 
$$
S\sim \sum_{k=1}^{\infty}R_{3,\ell}^k
$$
with $S\in \Psi^{-\infty,\cS/\mathfrak{s}}_{\k,\QFB}(M;E)$, where $\cS$ and $\mathfrak{s}$ satisfy the same properties as $\cR_{3,\ell}$ and $\mathfrak{r}_{3,\ell}$.  By construction,
$$
   R_4:= \Id-(\Id-R_{3,\ell})(\Id+S)\in \dot{\Psi}^{-\infty}_{\k,\QFB}(M;E),
$$
so it suffices to take $Q_4= Q_{3,\ell}(\Id +S)$ so that
$$
     D_{\k,\QFB}Q_4= \Id-R_4
$$
as claimed.
\end{proof}
This can be used to give the following pseudodifferential characterization of the inverse of $D_{\k,\QFB}$.
\begin{theorem}
There exists $G_{\k,\QFB}\in \Psi^{-1,\cG/\mathfrak{g}}_{\k,\QFB,\cn}(M;E)$ such that
$$
     D_{\k,\QFB}G_{\k,\QFB}=\Id \quad \mbox{and} \quad G_{\k,\QFB}D_{\k,\QFB}=\Id,
$$
where $\cG$ is an index family given by the empty set at $H_{ij,+}$  for all $(i,j)\ne (0,0)$ and at $H_{ij,0}$ for $i\ne j$ and by 
$$
 \cG|_{\ff_{i,+}}=\cG|_{\ff_{i,0}}=\bbN_0, \quad \cG|_{H_{ii,0}}=\left\{\begin{array}{ll} \bd_i+\bbN_0, & i>0, \\
                                                                 \bbN_0-1, & i=0,\end{array}\right.
$$
and where $\mathfrak{g}$ is a multiweight such that 
$$
\begin{array}{ll}
\mathfrak{g}(H_{ij,+})=\infty, &  \forall (i,j)\ne (0,0), \\  \mathfrak{g}(H_{ij,0})>\bd_j+1, & i\ne j, 
\end{array}
$$ 
and for $i>0$,
$$
  \mathfrak{g}(\ff_{i,+})=\infty, \quad \mathfrak{g}(H_{ii,0})>\bd_i, \quad \mathfrak{g}(\ff_{i,0})>0, \quad \mathfrak{g}(H_{i0,0})=\nu_R, \quad \mathfrak{g}(H_{0i,0})=\bd_i+1+\nu_R, \quad \mathfrak{g}(H_{00,0})>0.
$$
\label{qfble.38}\end{theorem}
\begin{proof}
Let $Q_4$ be the parametrix of Proposition~\ref{qfble.37}.  Since its error term decays rapidly when $\k\searrow 0$, there exists $\delta>0$ such that $\Id-R_4(\k)$ is invertible with inverse given by $\Id+ S_4(\k)$ for all $\k<\delta$ with
$$
      S_4(\k)= \sum_{j=1}^{\infty} R_4(\k)^j \quad \forall \k\in [0,\delta),
$$ 
an element of $\dot{\Psi}^{-\infty}_{\k,\QFB}(M;E)$ for $\k<\delta$.  Hence, for $\k<\delta$, we can set
$$
      G_{\k,\QFB}= Q_4(\Id+S_4)  \quad \Longrightarrow D_{\k,\QFB}G_{\k,\QFB}=\Id.
$$
By Theorem~\ref{ckqfb.10}, $G_{\k,\QFB}\in \Psi^{-1,\cG/\mathfrak{g}}_{\k,\QFB,\cn}(M;E)$ as claimed.  
For $\k\ge \delta$, we can invert $D_{\k,\QFB}$ simply in the small $\QFB$ calculus, that is, with inverse in $\Psi^{-1}_{\QFB}(M;E)$ for $\k>0$ fixed.  This gives the desired inverse for all $\k\ge 0$.  Taking the adjoint of $G_{\k,\QFB}$, we see that
\begin{equation}
 G^*_{\k,\QFB} D_{\k,\QFB}=\Id,
\label{qfble.39}\end{equation}
which implies that
$$
 G_{\k,\QFB}^*=  G_{\k,\QFB}^*\Id= G_{\k,\QFB}^* (D_{\k,\QFB}G_{\k,\QFB})= \Id G_{\k,\QFB}=G_{\k,\QFB},
$$
that is, $G_{\k,\QFB}$ is self-adjoint with
$$
       G_{\k,\QFB}D_{\k,\QFB}=\Id.
$$
\end{proof}

\section{Inverse of a non-fully elliptic suspended Dirac $\QFB$ operator } \label{ift.0}

In the parametrix construction of \S~\ref{do.0}, one of the key assumptions is Assumption~\ref{do.46} about the invertibility of some normal operators and a pseudodifferential characterization of the inverses essentially asserting that these inverses exactly fit where they should on the front faces of the $\QFB$ double space.  Now, these normal operators correspond themselves to families of non-fully elliptic suspended Dirac $\QFB$ operators.  The purpose of the present section is to explain how the results of \S~\ref{qfble.0} can be used to show that these operators are indeed invertible with inverse as described in Assumption~\ref{do.46}.  As in \cite[\S~9]{KR0}, the strategy is to first take the Fourrier transform in the suspension parameters of such a non-fully elliptic suspended Dirac $\QFB$ operator and to apply the results of \S~\ref{qfble.0} to give a fine pseudodifferential characterization of the corresponding inverse.  Taking the inverse Fourier transform, one can then show through careful arguments that the inverse of the original operator has a pseudodifferential characterization as in Assumption~\ref{do.46}.  

With this objective in mind, let $\eth_{\QFB}$ be a Dirac operator as in \S~\ref{qfble.0} so that Assumptions~\ref{do.1}, \ref{su.7},  \ref{su.2}, \ref{do.46} hold and Assumption~\ref{do.26} holds for $\delta=-\frac12$.  For each boundary hypersurface $H_i$, assume again that those assumptions also hold for each member of the vertical family $\eth_{v,i}$ seen as a Dirac operator.   In this section, we will consider a corresponding $\bbR^q$ suspended Dirac operator
\begin{equation}
\eth_{\sus}= \eth_{\QFB}+ \eth_{\bbR^q},
\label{ift.1}\end{equation}
where $\eth_{\bbR^q}$ is a family of Euclidean Dirac operators on $\bbR^q$ parametrized by $M$ and anti-commuting with $\eth_{\QFB}$.  If $\{e_1,\ldots, e_q\}$ is the canonical basis of $\bbR^q$, then we suppose that 
$$
       \eth_{\bbR^q}= \sum_{j=1}^q \cl(e_j)\nabla_{e_j},
$$
with $\nabla$ the pull-back of the Clifford connection of $E$ to its pull-back on $M\times \bbR^q$ and with $\cl(e_j)$ denoting Clifford multiplication by $e_j$.  More explicitly, we suppose that the Clifford module structure of $E$ lifts to a Clifford module structure on its pull-back on $M\times \bbR^q$ for the Clifford bundle corresponding to the Cartesian product metric 
$$
   g_{\QFB}+ g_{\bbR^q} \quad \mbox{on}  \quad M\times \bbR^q,
$$
where $g_{\bbR^q}$ is the canonical Euclidian metric on $\bbR^q$.  If we take the Fourier transform in the $\bbR^q$ factor, the suspended operator \eqref{ift.1} becomes a family of operators 
\begin{equation}
\hat{\eth}_{\sus}(\xi)= \eth_{\QFB}+ i\cl(\xi),  \quad \xi\in\bbR^q.
\label{ift.2}\end{equation}
For $\xi\ne 0$, we can use  spherical coordinates and write 
\begin{equation}
\hat{\eth}_{\sus}= \eth_{\QFB}+ \k \gamma \quad \mbox{with} \quad \k=|\xi| \quad \mbox{and} \quad \gamma= \frac{i}{|\xi|}\cl(\xi).
\label{ift.3}\end{equation}
For $\frac{\xi}{|\xi|}\in \bbS^{q-1}$ fixed, this is precisely an operator of the form \eqref{qfble.3} with $\gamma$ anti-commuting with $\eth_{\QFB}$ and such that $\gamma^2=\Id_E$.  To describe the inverse of $\eth_{\sus}$, we will need to make the following assumption.
\begin{assumption}
For each fixed $\frac{\xi}{|\xi|}\in\bbS^{q-1}$, we suppose that Assumption~\ref{qfble.5} holds.  Correspondingly, in terms of \eqref{su.3b}, for each boundary hypersurface $H_i$ of $M$ and for each member of the family 
$$
         N_i(\eth_{\QFB})+ \eth_{\bbR^q}= \eth_{v,i}+\eth_{h,i}+\eth_{\bbR^q}
$$
specified by $s\in S_i$ and  
seen as a ${}^{\phi}N_sS_i\oplus \bbR^q$ suspended Dirac operator, suppose that Assumption~\ref{qfble.5} holds for $\gamma$ the Clifford multiplication by an element of the unit sphere of ${}^{\phi}N_sS_i\oplus \bbR^q$. 
\label{ift.4}\end{assumption}

Thanks to this assumption, Theorem~\ref{qfble.38} applies to give a uniform description of the inverse $\hat{\eth}_{\sus}^{-1}(\xi)$ as $\xi\to 0$.  To obtain a pseudodifferential characterization of the inverse of $\eth_{\sus}^{-1}$, it suffices then to determine how the description of $\hat{\eth}_{\sus}^{-1}(\xi)$ translates when we take the inverse Fourier transform in $\xi$.  

As in \S~\ref{qfble.0}, it is convenient to describe instead the inverse of the conjugated operator
$$
    D_{\sus}= x^{-\w}\eth_{\sus}x^{\w}= D_{\QFB}+ \eth_{\bbR^{q}}.
$$

Before stating the main result of this section, we need also to describe the natural double space on which $\eth_{\sus}^{-1}$ admits a nice pseudodifferential characterization.  Indeed, since $\eth_{\sus}$ is not assumed to be fully elliptic, its inverse is not in general a suspended $\QFB$ operator in the sense of Definition~\ref{sm.9}.   Thus, let $\varrho$ be a total boundary defining function for $M^2_{\rp}$ and let $V_{\varrho}:=M^2_{\rp}\times \varrho\bbR^q$ denote the vector bundle of rank $q$ over $M^2_{\rp}$ trivialized by the sections $\varrho e_1,\ldots,\varrho e_q$.  As sections of $V_{\varrho}$, these do not vanish on $\pa M^2_{\rp}$, though as sections of $M^2_{\rp}\times \bbR^q\to M^2_{\rp}$, they do.  
Let $\overline{V}_{\varrho}= M^2_{\rp}\times \overline{\varrho\bbR^q}$ be the fiberwise radial compactification of the vector bundle $V_{\varrho}$.  In terms of $\overline{V}_{\varrho}$ and using the notation of \eqref{ds.1}, the double space we need to consider is 
\begin{equation}
\kridx{\widetilde{M}^2_{\phi-\sus(V_{\varrho})}}{M2QFBsus}{suspended QFB double space}:= [ M^2_{\rp}\times \overline{\varrho\bbR^q};\Phi_1\times\{0\},\ldots, \Phi_{\ell}\times \{0\}].
\label{ift.5}\end{equation}
Denote by $H^{\sus}_{ij}$ the lift of the boundary hypersurface $\overline{V}_{\varrho}|_{H_{ij}^{\rp}}$ of $\overline{V}_{\varrho}$ to $\widetilde{M}^2_{\phi-\sus(V_{\varrho})}$ for $i,j\in\{0,1,\ldots,\ell\}$ with $(i,j)\ne (0,0)$.  let also $H^{\sus}_{\infty}$ denotes the lift to $\widetilde{M}^2_{\phi-\sus(V_{\varrho})}$ of the boundary hypersurface $M^2_{\rp}\times \pa(\overline{\varrho\bbR^q})$ and denote by $\ff^{\sus}_i$ the boundary hypersurface of $\widetilde{M}^2_{\phi-\sus(V_{\varrho})}$ created by the blow-up of $\Phi_i\times\{0\}$ in $\overline{V}_{\varrho}$.   Because of the factor $\varrho$ in our definition of $V_{\varrho}$ and $\widetilde{M}^2_{\phi-\sus(V_{\varrho})}$, notice that the space of suspended operators in \eqref{sm.9b} can be described using $\widetilde{M}^2_{\phi-\sus(V_{\varrho})}$ as follows, 
\begin{multline}
\kridx{\Psi^m_{\QFB,\bbR^q}}{PsiQFBsus}{suspended QFB pseudodifferential operators (small calculus)}(M;E,F)= \{\kappa\in I^m(\widetilde{M}_{\phi-\sus(V_{\varrho})}; \diag_{\sus(V_{\varrho})}; \widetilde{\pr}_1^*(\Hom(E,F)\otimes \pr_R^*{}^{\QFB}\Omega(M))\otimes \widetilde{\pr}_2^*(\varrho^{-q}\Omega_{\varrho\bbR^q}) ) \quad | \\
\kappa\equiv 0 \quad \mbox{at} \quad H^{\sus}_{\infty}\; \mbox{and} \; H^{\sus}_{ij}\quad \forall i,j \in \{0,\ldots,\ell\}, \; (i,j)\ne (0,0)\},
\label{ift.6}\end{multline}
where $\diag_{\sus(V_{\varrho})}$ is the lift of $\diag_{\rp}\times \{0\} \subset M^2_{\rp}\times \overline{\varrho\bbR^q}$ to $\widetilde{M}^2_{\phi-\sus(V_{\varrho})}$ with $\diag_{\rp}$ the lifted diagonal in $M^2_{\rp}$ and
$$
   \widetilde{\pr}_1: \widetilde{M}^2_{\phi-\sus(V_{\varrho})}\to M^2 \quad \mbox{and}\quad \widetilde{\pr}_2: \widetilde{M}^2_{\phi-\sus(V_{\varrho})}\to \overline{\varrho\bbR^q}
$$
are the maps induced by the blow-down maps $\widetilde{M}^2_{\phi-\sus(V_{\varrho})}\to \overline{V}_{\varrho}$ and $M^2_{\rp}\to M^2$ and the natural projections $\overline{V}_{\varrho}\to M^2_{\rp}$ and $\overline{V}_{\varrho}\to \overline{\varrho\bbR^q}$, while $\Omega_{\varrho\bbR^q} = \varrho^{q}\Omega_{\bbR^q}$ is the natural Euclidean density bundle on $\overline{\varrho\bbR^q}$.  

If $\cE$ is an indicial family and $\mathfrak{s}$ is a multiweight for the manifold with corners $\widetilde{M}^2_{\phi-\sus(V_{\varrho})}$, one can more generally consider the enlarged spaces of suspended operators 
\begin{equation}
\Psi^{-\infty,\cE/\mathfrak{s}}_{\QFB,\bbR^q}(M;E,F)= \cA^{\cE/\mathfrak{s}}_{\phg}(\widetilde{M}^2_{\phi-\sus(V_{\varrho})}; \widetilde{\pr}_1^*(\Hom(E,F)\otimes \pr_R^*{}^{\QFB}\Omega(M))\otimes \widetilde{\pr}_2^*(\varrho^{-q}\Omega_{\varrho\bbR^q}) )
\label{ift.7}\end{equation}
and 
\begin{equation}
\kridx{\Psi^{m,\cE/\mathfrak{s}}_{\QFB,\bbR^q}}{PsiQFBsusEs}{suspended QFB pseudodifferential operators (large calculus)}(M;E,F)= \Psi^m_{\QFB,\bbR^q}(M;E;F)+ \Psi^{-\infty,\cE/\mathfrak{s}}_{\QFB,\bbR^q}(M;E,F),  \quad m\in\bbR.
\label{ift.8}\end{equation}
We can also define the weakly conormal version of these spaces.  To do so, we can first define the space of weakly conormal functions on $\widetilde{M}_{\phi-\sus(V_{\varrho})}$, 
\begin{multline}
\sA_{\phi-\sus(V\varrho)}(\widetilde{M}^2_{\phi-\sus(V_{\varrho})})= \\
 \{ \kappa\in L^\infty(\widetilde{M}^2_{\phi-\sus(V_{\varrho})}) \; | \;
\forall p,r,s\in\bbN_0, \quad \forall X_1,\ldots, X_{p+r}\in \cV_{\QFB}(M),  \forall Y_1,\ldots,Y_s\in \{ e_1,\ldots, e_q\}, \\
\widetilde{\pr}_1^*(\pi_L^*X_1)\cdots \widetilde{\pr}_1^*(\pi_L^*X_p)\widetilde{\pr}_1^*(\pi_R^*X_{p+1})\cdots\widetilde{\pr}_1^*(\pi_R^*X_{p+r}) \widetilde{\pr}_2^*Y_1\cdots \widetilde{\pr}_2^*Y_s \kappa\in L^\infty(\widetilde{M}^2_{\phi-\sus(V_{\varrho})})\}.
\label{ift.9}\end{multline}
 as well as the weighted versions 
\begin{equation}
\sA^{\mathfrak{s}}_{\phi-\sus(V_{\varrho})}(\widetilde{M}^2_{\phi-\sus(V_{\varrho})})= \rho^{\mathfrak{s}}\sA_{\phi-\sus(V_\varrho)}(\widetilde{M}^2_{\phi-\sus(V_{\varrho})})
\label{ift.10}\end{equation}
and
\begin{equation}
\sA^{\mathfrak{s}}_{\phi-\sus(V_{\varrho}),-}(\widetilde{M}^2_{\phi-\sus(V_{\varrho})})= \bigcap_{\mathfrak{t}<\mathfrak{s}}
\sA^{\mathfrak{t}}_{\phi-\sus(V_{\varrho})}(\widetilde{M}^2_{\phi-\sus(V_{\varrho})}).
\label{ift.11}\end{equation}
If $\cB$ is any boundary hypersurface of $\widetilde{M}^2_{\phi-\sus(V_{\varrho})}$ and $\rho_{\cB}$ is a boundary defining function, one can consider the space
$$
 \sA_{\phi-\sus(V_{\varrho})}(\cB)= \{ \kappa\in L^\infty(\cB)\; | \; \widetilde{\kappa}\in \sA_{\phi-\sus(V_{\varrho})}(\widetilde{M}^2_{\phi-\sus(V_{\varrho})}) \},
$$
where $\widetilde{\kappa}$ is a smooth extension off $\cB$ of $\kappa$, and correspondingly the weighted spaces 
$$
  \sA^{\mathfrak{s}}_{\phi-\sus(V_{\varrho})}(\cB)= \rho^{\mathfrak{s}}\sA_{\phi-\sus(V_{\varrho})}(\cB) \quad \mbox{and} \quad \sA^{\mathfrak{s}}_{\phi-\sus(V_{\varrho}),-}(\cB)= \bigcap_{\mathfrak{t}<s}\sA^{\mathfrak{t}}_{\phi-\sus(V_{\varrho})}(\cB)
$$
for $\mathfrak{s}$ a multiweight associated to the manifold with corners $\cB$.  Proceeding as in \eqref{wc.5}, we can then define the space 
\begin{equation}
  \sA^{\cE/\mathfrak{s}}_{\phi-\sus(V_{\varrho}),\phg}(\widetilde{M}^2_{\phi-\sus(V_{\varrho})};F)
\label{ift.12}\end{equation}
of partially polyhomogeneous weakly conormal sections of a vector bundle $F\to \widetilde{M}^2_{\phi-\sus(V_{\varrho})}$, where $\cE$ is an index family and $\mathfrak{s}$ is a multiweight both associated to the manifold with corners $\widetilde{M}^2_{\phi-\sus(V_{\varrho})}$.  With this notation understood, we can then define the weakly conormal version of the spaces \eqref{ift.7} and \eqref{ift.8} by
\begin{equation}
\Psi^{-\infty,\cE/\mathfrak{s}}_{\QFB,\bbR^q,\cn}(M;E,F)= \cA^{\cE/\mathfrak{s}}_{\phi-\sus(V_{\varrho}),\phg}(\widetilde{M}^2_{\phi-\sus(V_{\varrho})}; \widetilde{\pr}_1^*(\Hom(E,F)\otimes \pr_R^*{}^{\QFB}\Omega(M))\otimes \widetilde{\pr}_2^*(\varrho^{-q}\Omega_{\varrho\bbR^q}) )
\label{ift.13}\end{equation}
and 
\begin{equation}
\kridx{\Psi^{m,\cE/\mathfrak{s}}_{\QFB,\bbR^q,\cn}}{PsiQFBqEscn}{weakly conormal suspended QFB pseudodifferential operators}(M;E,F) = \Psi^m_{\QFB,\bbR^q}(M;E,F)+ \Psi^{-\infty,\cE/\mathfrak{s}}_{\QFB,\bbR^q,\cn}(M;E,F),  \quad m\in\bbR.
\label{ift.14}\end{equation}
We are ready to state the main result of this section.
\begin{theorem}
If Assumptions~\ref{do.1}, \ref{su.7}, \ref{su.2},  \ref{do.46}, \ref{ift.4} hold and if Assumption~\ref{do.26} holds with $\delta=-\frac12$, then the inverse $D_{\sus}^{-1}$ of $D_{\sus}$ is an element of 
$
   \Psi^{-1,\check{\cG}/\check{\mathfrak{g}}}_{\QFB,\bbR^q,\cn}(M;E)
$
for an index family $\check{\cG}$ given by
$$
\check{\cG}|_{\ff^{\sus}_i}=\bbN_0, \quad \check{\cG}|_{H^{\sus}_{ii}}=\bbN_0+\bd_i+q, \quad \check{\cG}|_{H^{\sus}_{\infty}}=\bbN_0+q-1
$$
for all $i\in\{1,\ldots,\ell\}$ and by the empty set elsewhere, while $\check{\mathfrak{g}}$ is a multiweight such that for all $i,j\in\{1,\ldots,\ell\}$ with $i\ne j$, 
$$
\begin{aligned}
&\check{\mathfrak{g}}(\ff^{\sus}_i)>0, \quad \check{\mathfrak{g}}(H^{\sus}_{ii})>\bd_j+q, \quad \check{\mathfrak{g}}(H^{\sus}_{\infty})>q,  \\ 
&\check{\mathfrak{g}}(H_{i0}^{\sus})= q+\nu_R, \quad \check{\mathfrak{g}}(H^{\sus}_{0i})=\bd_i+q+1+\nu_R, \quad \check{\mathfrak{g}}(H^{\sus}_{ij})>\bd_j+1+q.
\end{aligned}
$$
Moreover, the term of order $q-1$ at $H^{\sus}_{\infty}$ of $D_{\sus}^{-1}$ corresponds to the term of order $-1$ at $H_{00,0}$ of $G_{\k,\QFB}=(D_{\k,\QFB})^{-1}$ in Theorem~\ref{qfble.38}, namely it corresponds to $\gamma G^{-1}_{00,0}=\gamma\Pi_{\ker_{L^2_b}D_{\QFB}}$ in \eqref{qfble.13}.   
\label{ift.15}\end{theorem}

By the discussion above, the inverse is given by
\begin{equation}
D_{\sus}^{-1}=\frac{1}{(2\pi)^q}\int_{\bbR^q} e^{ix\cdot \xi} \hat{D}^{-1}_{\sus}(\xi) d\xi, 
\label{ift.16}\end{equation}
where $\hat{D}^{-1}_{\sus}(\xi)$ is the inverse of the family of operators 
$$
  \hat{D}_{\sus}(\xi)= D_{\QFB}+ i\cl(\xi)
$$   
and $x\in\bbR^q$ is seen as the variable dual to $\xi$.  By Theorem \ref{qfble.38}, we have a pseudodifferential characterization of $\hat{D}^{-1}_{\sus}(\xi)$ all the way down to $\xi\to 0$.  To give a pseudodifferential characterization of the inverse Fourier transform \eqref{ift.16}, we can first notice that inverting $D_{\sus}$ symbolically, we can construct a parametrix $Q\in\Psi^{-1}_{\QFB,\bbR^q}(M;E)$ such that
$$
     D_{\sus}Q=\Id-R \quad \mbox{with} \quad R\in \Psi^{-\infty}_{\QFB,\bbR^q}(M;E).
$$
Taking its Fourier transform in the factor $\bbR^q$ gives a parametrix $\hat{Q}(\xi)$ such that 
$$
       \hat{D}_{\sus}(\xi)\hat{Q}(\xi)=\Id-\hat{R}(\xi)
$$ 
with $\hat{R}(\xi)$ such that $\hat{R}(\k \eta)\in \Psi^{-\infty}_{\k,\QFB}(M;E)$ for all $\eta\in\bbS^{q-1}$, in fact $\hat{R}(\xi)$ even descends to a smooth section on $M^2_{\QFB}\times [0,\infty)_{\k}$.  Thus, this means that 
\begin{equation}
   \hat{D}^{-1}_{\sus}(\xi)=  \hat{D}^{-1}_{\sus}(\xi)(\hat{D}_{\sus}(\xi)\hat{Q}(\xi)+ \hat{R}(\xi)   )= \hat{Q}(\xi)  + \hat{D}^{-1}_{\sus}(\xi)\hat{R}(\xi).
\label{new.1}\end{equation}
Since $\hat{R}(\xi)$ is of order $-\infty$, so is $\hat{D}^{-1}_{\sus}(\xi)\hat{R}(\xi)$, which means that $\hat{D}^{-1}_{\sus}(\xi)$ has the same symbol as $\hat{Q}(\xi)$.  Since for $\hat{Q}(\xi)$, we already know that the inverse Fourier transform is given by $Q$, this means that we can completely characterize the symbol of the inverse Fourier transform \eqref{ift.16}.  That is, we can characterize the inverse Fourier transform \eqref{ift.16} up to an operator of order $-\infty$.  Thus, what is left to do is to make sense of $D^{-1}_{\sus}$ as a conormal distribution on $\widetilde{M}^{2}_{\phi-\sus(V_{\varrho})}$.  In particular, for all practical purposes, we can ignore the conormal singularities of $\hat{D}_{\sus}(\xi)$ along the diagonal while taking the inverse Fourier transform, since the image of those conormal singularities has already been described.  It corresponds to the conormal singularities of $Q$ along the lifted diagonal.  Moreover, since $\hat{R}(\xi)$ decays rapidly as $|\xi|\to \infty$, $\Id-\hat{R}(\xi)$ is invertible for $|\xi|$ large with inverse 
$$
 (\Id-\hat{R}(\xi))^{-1}= \sum_{k=0}^{\infty} \hat{R}(\xi),  \quad \mbox{so that} \quad    \hat{D}^{-1}_{\sus}(\xi)=\hat{Q}(\xi)\lrp{ \sum_{k=0}^{\infty} \hat{R}(\xi)} \quad \mbox{for} \; |\xi| \; \mbox{large}.
$$ 
In particular, $\hat{D}^{-1}_{\sus}(\xi)\hat{R}(\xi)$ decays rapidly as  $|\xi|\to \infty$. By \eqref{new.1}, this means that the asymptotic behavior of $\hat{D}^{-1}_{\sus}(\xi)$ as $|\xi|\to \infty$ is the same as the one of $\hat{Q}(\xi)$.  The inverse Fourier transform of this asymptotic expansion yields conormal singularities at the lifted diagonal $\diag_{\sus(V_{\varrho})}$ in $\widetilde{M}^2_{\phi-\sus(V_{\varrho})}$ correspond to those of $Q$.  Thus, for all practical purposes, we can suppose that $\hat{D}_{\sus}^{-1}(\xi)$ decays rapidly at $|\xi|\to\infty$, since its asymptotic expansion there is already taken care of, namely it is the same as the one of $\hat{Q}(\xi)$, an operator for which we know already how to describe the inverse Fourier transform.  

Now, the derivatives from the left and from the right of $\QFB$ differential operators in \eqref{ift.9} clearly commute with the inverse Fourier transform, while a derivative with respect to $ e_j$ becomes multiplication  by  $i\xi_j$ inside the inverse Fourier transform, which is under control, since as just mentioned, we can pretend that $\hat{D}_{\sus}$ decays rapidly as $|\xi|\to\infty$.  Near $H_{00,0}$, but away from the other boundary hypersurfaces, the Fourier transform converts the polyhomogeneous expansion at $H_{00,0}$ into a polyhomogeneous expansion at $H^{\sus}_{\infty}$ with term of order $\k^{\ell}=|\xi|^{\ell}$ at $H_{00,0}$ turning into a term of order $\rho_{\infty}^{q+\ell}$ at $H^{\sus}_{\infty}$, where $\rho_{\infty}$ denotes a boundary defining function of $H^{\sus}_{\infty}$.  This implies that the top order term of order $-1$ at $H_{00,0}$ induces a term of order $\rho_{\infty}^{q-1}$ at $H_{\infty}^{\sus}$.  

For the part of the expansion at $H_{00,0}$ which is not polyhomogeneous, but only weakly conormal, we also need to check that the decay at $H_{00,0}$ gets translated by the inverse Fourier transform into appropriate decay at $H^{\sus}_{\infty}$. Let $\varphi\in\CI_c([\bbR^q;\{0\}])$ be a cut-off function which is equal to $1$ near $\pa [\bbR^q;\{0\}]$, the boundary hypersurface created by the blow-up of the origin.  If $(\varphi\cdot)$ denotes the operation of multiplication by $\varphi$, then using the fact that multiplication by $x_j$ becomes $i\frac{\pa}{\pa\xi_j}$ inside the inverse Fourier transform, we see that for $k\in\bbN_0$,  there is a bounded map
$$
\cF^{-1}\circ(\varphi\cdot): |\xi|^k\cA_{b,2}([\bbR^q;\{0\}])\to (1+|x|^2)^{-\frac{k+\frac{q}2}{2}}L^2(\bbR^q)=  (1+|x|^2)^{-\frac{k+q}2}L^2_b(\overline{\bbR^q})
$$  
for $q$ even as well as a bounded map
$$
\cF^{-1}\circ(\varphi\cdot): |\xi|^{k-\frac12}\cA_{b,2}([\bbR^q;\{0\}])\to (1+|x|^2)^{-\frac{k+\frac{q-1}2}{2}}L^2(\bbR^q)=  (1+|x|^2)^{-\frac{k-\frac12+q}2}L^2_b(\overline{\bbR^q})
$$  
for $q$ odd.

Using interpolation theory, this induces a map
$$
\cF^{-1}\circ(\varphi\cdot): |\xi|^{\epsilon}\cA_{b,2}([\bbR^q;\{0\}])\to (1+|x|^2)^{-\frac{\epsilon+\frac{q}2}{2}}L^2(\bbR^q)=  (1+|x|^2)^{-\frac{\epsilon+q}2}L^2_b(\overline{\bbR^q})
$$    
for all $\epsilon\ge 0$.   Since 
$$
   \cF(x_j \frac{\pa}{\pa x_k} f)= \frac{\pa}{\pa \xi_j}\lrp{\xi_j \cF(f)},
$$ 
we see that conormality is preserved by the inverse Fourier transform, so that $\cF^{-1}\circ(\varphi\cdot)$ induces a map
$$
   \cF^{-1}\circ(\varphi\cdot): |\xi|^{\epsilon} \cA_{b,2}([\bbR^q;\{0\}])\to (1+|x|^2)^{-\frac{\epsilon+q}2}\cA_{b,2}(\overline{\bbR^q})
$$
for $\epsilon\ge 0$.  This shows that the inverse Fourier transform \eqref{ift.16} has the claimed behavior near $H^{\sus}_{\infty}$, but away from the other boundary hypersurfaces.  Away from $\ff_{i,0}$ and $\ff_{i,+}$, we can take advantage of the rapid decay at $H_{kj,+}$ to make the change of variable 
$$
      \widetilde{\xi}=\frac{\xi}{\varrho},  \quad \widetilde{x}= \varrho x
$$
in the inverse Fourier transform \eqref{ift.16}, so that 
$$
\lrp{\frac{1}{(2\pi)^q}\int_{\bbR^q} e^{ix\cdot \xi} \hat{D}^{-1}_{\sus}(\xi) d\xi} dx= \lrp{\frac{1}{(2\pi)^q}\int_{\varrho^{-1}\bbR^q} e^{i\widetilde{x}\cdot \widetilde{\xi}} \hat{D}^{-1}_{\sus}(\varrho\widetilde{\xi}) d\widetilde{\xi},}d\widetilde{x}
$$
with $\widetilde{x}$ the natural variable in the fibers of $V_{\varrho}= M^2_{\rp}\times \varrho\bbR^q$, so that $d\widetilde{x}$ induces a nonvanishing section of $\Omega_{\varrho\bbR^q}= \varrho^{q}\Omega_{\bbR^q}$.  This shows that the Fourier transform will have the claimed behavior away from the lift of $V_{\varrho}|_{\Phi_i}$ for each $i$ on $\widetilde{M}^2_{\phi-\sus}(V_{\varrho})$.  

Near $\ff_{i,+}$, we will make use  of the following lemma.

\begin{lemma}
The expansion of $ \hat{D}^{-1}_{\sus}(\xi)$ at $\ff_{i,+}$ is in powers of $\rho_{\ff_{i,+}}\k$, where $\rho_{\ff_{i,+}}$ is a boundary defining function for $\ff_{i,+}$.   
\label{ift.17}\end{lemma}
\begin{proof}
We will adapt the approach of the proof of \cite[Lemma~9.2]{KR0}. First, the inverse of $N_{i,+}(\hat{D}_{\sus}(\xi))$ can be extended smoothly off $\ff_{i,+}$ to an operator $Q_0$ such that its expansion at $\ff_{i,+}$ is in powers of $\rho_{\ff_{i,+}}\k$.  This gives a parametrix $Q_0\in\Psi^{-1,\cQ_0/\mathfrak{q}_0}_{\k,\QFB,\cn}(M;E)$ such that
$$
    \hat{D}_{\sus}(\xi)Q_0= \Id-R_0
$$
with $R_0\in\Psi^{0,\cR_0/\mathfrak{r}_0}_{\k,\QFB,\cn}(M;E)$ also having an expansion at $\ff_{i,+}$ in powers of $\rho_{\ff_{i,+}}\k$, where $\cQ_0$ is an index family such that
$$
\begin{gathered}
\cQ_0|_{\ff_{j,+}}=\bbN_0, \quad j\in\{1,\ldots,\ell\}, \quad \cQ_0|_{\ff_{j,0}}=\bbN_0, \quad \cQ_0|_{H_{jj,0}}=\bd_j+\bbN_0, \quad H_j\ge H_i, \\
 \cQ_0|_{H_{jk,+}}=\emptyset, \quad j,k\in\{ 0,\ldots,\ell\}, \; (j,k)\ne (0,0), \\
  \cQ_0|_{H_{jk,0}}=\emptyset, \quad H_j\ge H_i, \; H_k\ge H_i,  \; j\ne k, \; (j,k)\ne (0,0), \\
\end{gathered}
$$
and $\mathfrak{q}_0$ is a $\k,\QFB$ positive multiweight, except at $H_{jj,0}$ for $H_j\ge H_i$, where 
$$
      \mathfrak{q}_0(H_{jj,0})>\bd_j,
$$
while $\cR_0$ is an index family such that 
$$
\begin{gathered}
   \cR_0|_{\ff_{i,+}}=\bbN_0+1, \quad  \cR_0|_{\ff_{j,+}}=\bbN_0, \quad j\ne i, \\
      \cR_0|_{\ff_{j,0}}=\bbN_0+1, \quad \cR_0|_{H_{jj,0}}=\bbN_0+\bd_j+1, \quad H_j\ge H_i, \\
    \cR_0|_{H_{jk,+}}=\emptyset, \quad  j,k\in\{ 0,\ldots,\ell\},  \; (j,k)\ne (0,0), \\
     \cR_0|_{H_{jk,0}}=\emptyset, \quad H_j\ge H_i, \; H_k\ge H_i, \; j\ne k, \; (j,k)\ne (0,0), 
\end{gathered}    
$$
and $\mathfrak{r}_1$ is a $\k,\QFB$ positive multiweight.  A priori, the expansion of $R_0$ in  powers of $\rho_{\ff_{i,+}}\k$ at $\ff_{i,+}$ has a term of order $0$, but since $N_{i,+}(Q_0)$ is the inverse of $N_{i,+}(\hat{D}_{\sus}(\xi))$, this top order term vanishes, so that we can take $\cR_0|_{\ff_{i,+}}=\bbN_0+1$.  On the other hand, since the expansion at $\ff_{i,+}$ is in powers of $\rho_{\ff_{i,+}}\k$, notice that the term of order $1$ of $R_0$ at $\ff_{i,+}$, namely $N_{i,+}(R_0\rho_{\ff_{i,+}}^{-1})$, has top order term of order $\bd_j+2$ at $H_{jj,0}\cap\ff_{i,+}$ for $H_j\ge H_i$ instead of just of order $\bd_j+1$.  Similarly, it has top order term of order $2$ at $\ff_{j,0}\cap\ff_{i,+}$ for $H_j\ge H_i$ instead of just of order $1$.  At $H_{kj,0}\cap\ff_{i,+}$ for $H_k\ge H_i$ and $H_j\ge H_i$ with $j\ne k$, it also vanishes up to order $\bd_j+2+\epsilon$ for some $\epsilon>0$.  This suggests to consider $Q_1\in\Psi^{-1,\cQ_1/\mathfrak{q}_1}_{\k,\QFB,\cn}(M;E)$ having an expansion in powers of $\rho_{\ff_{i,+}}\k$ at $\ff_{i,+}$ such that
\begin{equation}
  N_{i,+}(Q_1\rho^{-1}_{\ff_{i,+}}):= N_{i,+}(Q_0)N_{i,+}(R_0\rho^{-1}_{\ff_{i,+}}),
\label{ift.18}\end{equation} 
where $\cQ_1$ is an index family such that  
$$
\begin{gathered}
\cQ_1|_{\ff_{i,+}}=\bbN_0+1, \quad \cQ_1|_{\ff_{j,+}}=\bbN_0, \; j\ne i, \\
\quad \cQ_1|_{\ff_{j,0}}=\bbN_0+1, \quad \cQ_1|_{H_{jj,0}}=\bd_j+1+\bbN_0, \quad H_j\ge H_i, \\
 \cQ_1|_{H_{jk,+}}=\emptyset, \quad j,k\in\{ 0,\ldots,\ell\}, \; (j,k)\ne (0,0), \\
\cQ_1|_{H_{jk,0}}=\emptyset, \quad H_j\ge H_i, \; H_k\ge H_i, \; j\ne k, \; (j,k)\ne (0,0), 
\end{gathered}
$$
and $\mathfrak{q}_1$ is a $\QFB$ positive multiweight with $\mathfrak{q}_1(\ff_{j,0})>1$ for $H_j\ge H_i$.  By \eqref{ift.18}, we see that
$$
   \hat{D}_{\sus}(\xi)(Q_0+Q_1)= \Id-R_1
$$
with $R_1\in\Psi^{0,\cR_1/\mathfrak{r}_1}_{\k,\QFB,\cn}(M;E)$ having an expansion in powers of $\rho_{\ff_{i,+}}\k$ at $\ff_{i,+}$, where $\mathfrak{r}_1$ is a $\k,\QFB$ positive multiweight and $\cR_1$ is an index family satisfying the same properties as the index family $\cR_0$, except that 
$$
       \cR_1|_{\ff_{i,+}}=\bbN_0+2.
$$
Since the expansion at $\ff_{i,+}$ is in powers of $\rho_{\ff_{i,+}}\k$ at $\ff_{i,+}$, this means that $N_{i,+}(R_1\rho_{\ff_{i,+}}^{-2})$ has top order terms of order $\bd_j+3$ at $H_{jj,0}\cap\ff_{i,+}$ and $3$ at $\ff_{j,0}\cap\ff_{i,+}$ for $H_j\ge H_i$ respectively, while it vanishes at order $\bd_k+3+\epsilon$ at $H_{jk,0}$ for some $\epsilon>0$  for $H_j\ge H_i$ and $H_k\ge H_i$ with $j\ne k$.  At this point, the construction of $Q$ can be iterated, namely we can more generally define recursively $Q_m\in\Psi^{-1,\cQ_m/\mathfrak{q}_m}_{\k,\QFB,\cn}(M;E)$ such that 
$$
      N_{i,+}(Q_m\rho_{\ff_{i,+}}^{-m}):= N_{i,+}(Q_0)N_{i,+}(R_{m-1}\rho_{\ff_{i,+}}^{-m})
$$  
and 
$$
 \hat{D}_{\sus}(\xi)(Q_0+\cdots +Q_m)= \Id-R_m
$$
with $R_m\in\Psi^{0,\cR_m/\mathfrak{r}_m}_{\k,\QFB,\cn}(M;E)$, where $Q_m$ and $R_m$ have both expansion in powers of $\rho_{\ff_{i,+}}\k$ at $\ff_{i,+}$ and vanishing respectively at orders $(\rho_{\ff_{i,+}}\k)^m $ and 
$(\rho_{\ff_{i,+}}\k)^{m+1} $ there, where $\cQ_m$ is an index family such that 
$$
\begin{gathered}
\cQ_m|_{\ff_{i,+}}=\bbN_0+m, \quad \cQ_m|_{\ff_{j,+}}=\bbN_0, \; j\ne i, \\
 \cQ_m|_{\ff_{j,0}}=\bbN_0+m, \quad \cQ_m|_{H_{jj,0}}=\bd_j+m+\bbN_0,  \quad H_j\ge H_i, \\
  \cQ_m|_{H_{jk,+}}=\emptyset, \quad j,k\in\{ 0,\ldots,\ell\}, \; (j,k)\ne (0,0), \\
  \cQ_m|_{H_{jk,0}}=\emptyset, \quad H_j\ge H_i, \; H_k\ge H_i, \; j\ne k, \; (j,k)\ne (0,0), 
\end{gathered}
$$
and where $\mathfrak{q}_m$ is a $\k,\QFB$ positive multiweight with $\mathfrak{q}_m(\ff_{j,0})>m$ and $\mathfrak{q}_m(H_{jk,0})>\bd_k+m$ for $H_j\ge H_i$ and $H_k\ge H_i$, while $\mathfrak{r}_m$ is a $\k,\QFB$ positive multiweight and $\cR_m$ is an index family having the same properties as $\cR_0$, except that 
$$
    \cR_{m}|_{\ff_{i,+}}=\bbN_0+m+1.  
$$
If $Q\in\Psi^{-1,\cQ_0/\mathfrak{q}}_{\k,\QFB,\cn}(M;E)$ is a Borel sum of the $Q_j$ at $\ff_{i,+}$, then its expansion at $\ff_{i,+}$ is in powers of $\rho_{\ff_{i,+}}\k$ and such that
$$
       \hat{D}_{\sus}(\xi)Q= \Id-R
$$
with $R\in\Psi^{0,\cR/\mathfrak{r}}_{\k,\QFB,\cn}(M;E)$ vanishing rapidly at $\ff_{i,+}$, where $\mathfrak{q}$ and $\mathfrak{r}$ are multiweights satisfying the same properties as $\mathfrak{q}_0$ and $\mathfrak{r}_0$ respectively, and  where $\cR$ is an index family satisfying the same properties as $\cR_0$, except that $\cR|_{\ff_{i,+}}=\emptyset$.  Since 
$$
      \hat{D}_{\sus}^{-1}(\xi)=\hat{D}_{\sus}^{-1}(\xi)( \hat{D}_{\sus}(\xi)Q+R)= Q+ \hat{D}_{\sus}^{-1}(\xi)R,
$$
we see from Theorem~\ref{ckqfb.10} that $\hat{D}_{\sus}^{-1}(\xi)$ has the same expansion as $Q$ at $\ff_{i,+}$, from which the result follows.

\end{proof}

We can now complete the proof of Theorem~\ref{ift.15}.

\begin{proof}[End of the proof of Theorem~\ref{ift.15}]
Thanks to Lemma~\ref{ift.17}, we can take the inverse Fourier transform for each term in the expansion of $\hat{D}_{\sus}^{-1}(\xi)$ at $\ff_{i,+}$.  For the top order term, notice that this will simply yield the inverse of $D_{\sus}|_{\ff^{\sus}_i}$ on $\ff^{\sus}_i$.  In fact, instead of computing the restriction of \eqref{ift.16} to $\ff_{i,+}$, namely
$$
\frac{1}{(2\pi)^q}\int_{\bbR^q} e^{ix\cdot \xi} N_{i,+}(\hat{D}^{-1}_{\sus}(\xi)) d\xi, 
$$
we can first also take the Fourier transform of $N_{i,+}(\hat{D}^{-1}_{\sus}(\xi))$ in the suspension parameters of $\ff_{i,+}$.  If $u_i\in\bbR^{\bd_i+1}$ denotes locally (with respect to $S_i$) such a suspension parameter and $\eta_i$ denotes the dual variable, then taking the Fourier transform yields a family of $\k,\QFB$ operators
$$
      (\hat{D}_{\sus}|_{\ff^{\sus}_i})^{-1}(\xi,\eta_i)
$$
parametrized by $s\in S_i$ and with respect to the fibers of the fiber bundle $\phi_i:H_i\to S_i$.  Hence, by Assumption~\ref{ift.4}, we can apply Theorem \ref{qfble.38}  to get a suitable microlocal characterization of each member of the family.  Taking the inverse Fourier transform in $\xi$ and $\eta_i$, we can then simply apply the lower depth version of Theorem~\ref{ift.15}, which by induction on the depth of $M$ we can assume holds already for $(\hat{D}_{\sus}|_{\ff^{\sus}_i})^{-1}(\xi,\eta_i)$, showing that
$$
\frac{1}{(2\pi)^q}\int_{\bbR^q} e^{ix\cdot \xi} N_{i,+}(\hat{D}^{-1}_{\sus}(\xi)) d\xi= \frac{1}{(2\pi)^{q+\bd_i+1}}\int_{\bbR^q} e^{ix\cdot \xi}e^{iu_i\cdot\eta_i} (\hat{D}_{\sus}|_{\ff_i^{\sus}})^{-1}(\xi,\eta_i) d\xi d\eta_i
$$
is indeed $\lrp{D_{\sus}|_{\ff^{\sus}_i}}^{-1}$ on $\ff_i^{\sus}$ and that it has the claimed behavior of the statement of Theorem~\ref{ift.15}.  For the higher order terms in the expansion of $\hat{D}_{\sus}(\xi)^{-1}$ at $\ff_{i,+}$, notice that they will have nice expansions at $\ff_{j,+}$ for $H_j\ge H_i$ by Lemma~\ref{ift.17}.  Hence, we can again take first the Fourier transform in $u_i$, then the inverse Fourier transform in $(\xi,\eta_i)$ to compute their contributions to \eqref{ift.16}.  They yield a corresponding expansion for $D_{\sus}^{-1}$ at $\ff_i^{\sus}$ similar in behavior to the one coming from the term of order $0$.  However, because of the extra decay in $\rho_{\ff_{i,+}}\k$, notice that they yield only terms of order $\bd_j+1+q$ or higher at $H_{jj}^{\sus}$ for $H_j\ge H_i$.  

Removing this contribution from the expansion at $\ff_{i,+}$, we can thus assume that $D_{\sus}(\xi)^{-1}$ vanishes rapidly at $\ff_{i,+}$.  In particular, because we removed the top order term of the expansion of $\hat{D}_{\sus}(\xi)^{-1}$ at $\ff_{i,0}$, notice that we can assume that these terms at $\ff_{i,0}$ are of order at least $\rho_{\ff_{i,0}}$.  Since there are no `dual boundary face' corresponding to $\ff_{i,0}$ in $\widetilde{M}^{2}_{\phi-\sus(V_{\varrho})}$, this extra decay is very helpful, since we can simply use it to see that they contribute as a weakly conormal section vanishing at least to order $\bd_j+1+q$ at $H_{jj}^{\sus}$ for $H_j\ge H_i$.
\end{proof}

\section{Resolvent of a suspended $\QFB$ Dirac operator in the low energy limit} \label{sle.0}

To complete the induction on the depth for the constructions of the parametrices of \S~\ref{do.0} and \S~\ref{qfble.0}, we need to provide a pseudodifferential characterization of the low energy limit of the resolvent of a suspended $\QFB$ Dirac operator.  Indeed, such characterization will ensure that Assumption~\ref{qfble.5} is satisfied in higher depth cases.  To achieve this, our strategy will consist in suitably adapting the approach of \S~\ref{qfble.0} to suspended $\k,\QFB$ operators. 

  Thus, let $\eth_{\sus}=\eth_{\QFB}+\eth_{\bbR^q}$ be the $\bbR^q$ suspended Dirac operator \eqref{ift.1} of \S~\ref{ift.0} and suppose that Assumptions~\ref{do.1}, \ref{su.7}, \ref{su.2}, \ref{do.46} and \ref{ift.4} hold, as well as Assumption~ \ref{do.26} for $\delta=-\frac12$.  For each boundary hypersurface $H_i$, assume as well that the corresponding assumptions hold for each member of the vertical $\eth_{v,i}$ seen as a Dirac operator.  Let $\gamma\in\CI(M;\End(E))$ be self-adjoint with respect to the $\QFB$ metric $g_{\QFB}$ and the bundle metric of $E$ such that
\begin{equation}
   \gamma^2=\Id_E, \quad \eth_{\QFB}\gamma+ \gamma\eth_{\QFB}=0= \eth_{\bbR^q}\gamma+ \gamma\eth_{\bbR^q}.
\label{sle.1}\end{equation}
As in \S~\ref{qfble.0}, suppose also that $\gamma$ anti-commutes with $\eth_{v,i}$, $\eth_{h,i}$ and $\eth_{S_i}$ in \eqref{su.3b} and \eqref{do.15b}.  For $\k\ge 0$, we can then consider the family of operators
\begin{equation}
  \eth_{\k,\sus}:= \eth_{\sus}+ \k \gamma.
\label{sle.2}\end{equation}
By \eqref{sle.1}, notice that
$$
      \eth_{\k,\sus}^2= \eth^2_{\sus}+ \k^2\Id_{E},
$$
so that for $\k>0$ fixed, we see from Corollary~\ref{mp.4} that $\eth^2_{\k,\sus}$ and $\eth_{\k,\sus}$ are invertible as fully elliptic suspended $\QFB$ operators.  For $\k=0$, we know instead from Theorem~\ref{ift.15} that $\eth_{0,\sus}=\eth_{\sus}$ is invertible, but in an enlarged calculus of $\QFB$ pseudodifferential operators.  These two ways of inverting can be combined to invert the family of operators $\eth_{\k,\sus}$ within a suspended version of the $\k,\QFB$ pseudodifferential calculus.

The double space of this calculus can be defined as follows.  
Let $\bvarrho\in\CI(M_{\rp}\rttimes [0,\infty))$ be a total boundary defining function for $M_{\rp}\rttimes [0,\infty)$ and consider the trivial vector bundle of rank $q$
\begin{equation}
  V_{\bvarrho}:= (M^2_{\rp}\rttimes[0,\infty))\times \bvarrho\bbR^q \to (M^2_{\rp}\rttimes[0,\infty)).
\label{ksm.7}\end{equation}
More precisely, if $e_1,\ldots,e_q$ is the canonical basis of $\bbR^q$, then $\bvarrho e_1,\ldots,\bvarrho e_q$ gives a basis of non-vanishing sections trivializing $V_{\bvarrho}$.  Let 
$$
   \overline{V}_{\bvarrho}= (M^2_{\rp}\rttimes [0,\infty))\times \overline{\bvarrho\bbR^q}
$$ 
be the fiberwise radial compactification of the vector bundle $V_{\bvarrho}$.  Then, using the notation of Definition~\ref{kqfb.5},  consider the double space
\begin{equation}
\begin{gathered}
  M^2_{\k,\phi-q}:= [\overline{V}_{\bvarrho};\Phi_{0,0}\times\{0\}, \Phi_{1,+}\times\{0\},\Phi_{1,0}\times\{0\}, \ldots, \Phi_{\ell,+}\times\{0\}, \Phi_{\ell,0}\times\{0\}] \\
  \mbox{with blow-down map} \quad \beta_{\k,\phi-q}: M^2_{\k,\phi-q}\to (M^2_{\rp}\rttimes [0,\infty))\times \overline{\bvarrho\bbR^q},
\end{gathered}  
\label{ksm.8}\end{equation}
where $\Phi_{0,0}$ is the lift of $M^2\times \{0\}$ in $M^2\times [0,\infty)$ to $M^2_{\rp}\rttimes [0,\infty)$.  Let us denote by $\ff^{\sus}_{i,\nu}$ the boundary hypersurface created by the blow-up of $\Phi_{i,\nu}\times \{0\}$ for $i\ge 0$ and $\nu=\{0,+\}$ with $(i,\nu)\ne (0,+)$.  On $M^2_{\rp}\rttimes [0,\infty)$, let us denote by $H^{\rp}_{ij,0}$ the lift of $H_i\times H_j\times \{0\}$ in $M^2\times [0,\infty)$ for $i\ge 0$ and $j\ge 0$, and by $H^{\rp}_{ij,+}$ the lift of $H_i\times H_j\times [0,\infty)$ for $(i,j)\ne (0,0)$.  Correspondingly, on $M^2_{\k,\phi-q}$, let us denote by $H^{\sus}_{ij,\nu}$ the boundary hypersurface corresponding to the lift of $H^{\rp}_{ij,\nu}\times \overline{\bvarrho\bbR^q}$ for $i\ge 0$, $j\ge 0$, $\nu\in\{0,+\}$ and $(i,j,\nu)\ne (0,0,+)$.  Finally let $H^{\sus}_{00,+}$ be the boundary hypersurface of $M^2_{\k,\phi-q}$ corresponding to the lift of $(M^2_{\rp}\rttimes [0,\infty))\times \pa(\overline{\bvarrho\bbR^q})$ in $\overline{V}_{\bvarrho}$.  
Notice that the face created by the blow-up of $\Phi_{0,0}\times\{0\}$ in $[\overline{V}_{\bvarrho}; \Phi_{0,0}\times\{0\}]$ is naturally identified with the space $\overline{V}_{\varrho}$ used to define $\widetilde{M}^2_{\phi-\sus(V_{\varrho})}$ in \eqref{ift.5}, showing that the boundary hypersurface $\ff^{\sus}_{0,0}$ in $M^2_{\k,\phi-q}$ is naturally identified with $\widetilde{M}^2_{\phi-\sus(V_{\varrho})}$,
\begin{equation}
      \ff^{\sus}_{0,0}\cong \widetilde{M}^2_{\phi-\sus(V_{\varrho})}.
\label{ksm.8b}\end{equation}
Similarly, a slice of $M^2_{\k,\phi-q}$ at a fixed $\k>0$ is naturally identified with $\widetilde{M}^2_{\phi-\sus(V_{\varrho})}$.  In fact, we will need the following suspended version of Lemma~\ref{kqfb.18b}.
\begin{lemma}
There is a natural surjective $b$-submersion
\begin{equation}
  M^2_{\k,\phi-q}\setminus (H^{\sus}_{00,0}\cup H^{\sus}_{00,+})\to \widetilde{M}^2_{\phi-\sus(V_{\varrho})}\setminus H^{\sus}_{\infty}.
\label{ksm.8d}\end{equation}
\label{ksm.8c}\end{lemma}
\begin{proof}
Since we only blow up corners, notice first that the blow-down map 
\begin{equation}
(M^2_{\rp}\rttimes [0,\infty) )\times \varrho\bbR^q\to (M^2_{\rp}\times [0,\infty) )\times \varrho\bbR^q
\label{ksm.8e}\end{equation}
is a surjective $b$-submersion.  Now, there is a natural identification 
$$
  (M^2_{\rp}\rttimes [0,\infty) )\times \varrho\bbR^q= [(M^2_{\rp}\rttimes [0,\infty) )\times \bvarrho\bbR^q; \Phi_{0,0}\times \{0\}]\setminus H
$$
with $H$ the lift of the boundary hypersurface $H^{\rp}_{00,0}\times \bvarrho\bbR^q$ to $[(M^2_{\rp}\rttimes [0,\infty) )\times \bvarrho\bbR^q; \Phi_{0,0}\times \{0\}]$.  Applying Lemma~\ref{bf.3}, we thus see that \eqref{ksm.8e} lifts to the natural surjective $b$-submersion \eqref{ksm.8d}.
\end{proof}

Let $\diag_{\k,\rp}$ be the lift of $\diag_M\times [0,\infty)$ to $M^2_{\rp}\rttimes [0,\infty)$ and let $\diag_{\k,\phi-q}$ be the lift of $\diag_{\k,\rp}\times \{0\}$ in $\overline{V}_{\bvarrho}$ to $M^2_{\k,\phi-q}$.  Thanks to the factor $\bvarrho$ and the blow-ups in \eqref{ksm.8}, one can then define the space of $\bbR^q$ suspended $\k,\QFB$ operators of order $M$ by 
\begin{equation}
\kridx{\Psi^m_{\k,\phi-q}}{PsiQFBksus}{$\bbR^q$ suspended $\k,\QFB$ pseudodifferential operators}(M) = \{\kappa\in I^{m-\frac14}(M^2_{\k,\phi-q};\diag_{\k,\phi-q}; {}^{\k,\phi-q}\Omega)\; | \; \kappa\equiv 0 \; \mbox{at}\; \pa M^2_{\k,\phi-q}\setminus \ff_{\k,\phi-q}\},
\label{ksm.9}\end{equation}
where $\ff_{\k,\phi-q}$ is the union of the boundary hypersurfaces intersecting $\diag_{\k,\phi-q}$ and 
$$
   \kridx{{}^{\k,\phi-q}\Omega}{OQFBksus}{suspended $\k,\QFB$ density bundle}:= \nu_R^*({}^{\QFB}\Omega)\otimes \nu_q^*(\bvarrho^{-q}\Omega_{\bvarrho\bbR^q})
$$
with $\nu_R$ the composition of the natural maps
$$
\xymatrix{
M^2_{\k,\phi-q} \ar[r]^-{\beta_{\k,\phi-q}} & \overline{V}_{\bvarrho}\to (M^2_{\rp}\rttimes [0,\infty))\ar[r] & M^2_{\rp} \ar[r] & M^2 \ar[r]^-{\pr_R} & M
}
$$
and $\nu_q$ the composition of the natural maps
$$
\xymatrix{
M^2_{\k,\phi-q} \ar[r]^-{\beta_{\k,\phi-q}} & \overline{V}_{\bvarrho}\ar[r] & \overline{\bvarrho\bbR^q}
}
$$
and with $\Omega_{\bvarrho\bbR^q}= \bvarrho^{q}\Omega_{\bbR^q}$ the natural scattering density bundle on $\overline{\bvarrho\bbR^q}$.  

If $M$ is a fiber $\phi_i^{-1}(s)$ of a fiber bundle $\phi_i: H_i\to S_i$ for a boundary hypersurface $H_i$ of a manifold with fibered corners $W$ with base $S_i$ of dimension $q$, then $M^2_{\k,\phi-q}$ is naturally diffeomorphic to the fiber $\phi_{\ff_{i,+}}^{-1}(s)$ of the fiber bundle $\phi_{\ff_{i,+}}: \ff_{i,+}\to S_i$ of \eqref{ffi.1} for the corresponding $\k,\QFB$ double space $W^2_{\k,\QFB}$.  Thus, using the $\k,\QFB$ pseudodifferential calculus of $W$ and Proposition~\ref{ksm.6b} to restrict it to $\phi_{\ff_{i,+}}^{-1}(s)$, one can define an enlarged space of $\bbR^q$ suspended $\k,\QFB$ operators
\begin{equation}
            \kridx{\Psi^{m,\cE/\mathfrak{s}}_{\k,\phi-q}}{PsiQFBksusEs}{$\bbR^q$ suspended $\k,\QFB$ pseudodifferential operators (large calculus)}(M;E,F)
\label{ksm.10}\end{equation}
for $E,F$ vector bundles on $M$ and  $\cE$ and $\mathfrak{s}$ an indicial family and a multiweight for $M^2_{\k,\phi-q}$.  
We also denote by $\dot{\Psi}^{-\infty}_{\k,\phi-q}(M;E,F)$ the subspace of $\Psi^{-\infty}_{\k,\phi-q}(M;E;F)$ consisting of those operators with Schwartz kernels vanishing rapidly at all boundary hypersurfaces of $M^2_{\k,\phi-q}$ and admitting a smooth expansion in powers of $\k^{-1}$ as $\k\to \infty$.
We can more generally define, still through the $\k,\QFB$ calculus of $W$, a weakly conormal version 
\begin{equation}
 \kridx{\Psi^{m,\cE/\mathfrak{s}}_{\k,\phi-q,\cn}}{PsiQFBksusEscn}{weakly conormal $\bbR^q$ suspended $\k,\QFB$ pseudodifferential operators}(M;E,F)
 \label{ksm.11}\end{equation} 
of \eqref{ksm.10}.    A multiweight $\mathfrak{s}$ will be said to be \textbf{$\k,\phi-q$ positive} if 
$$
        \mathfrak{s}(H_{ij,0}^{\sus})> h_j+q, \; \mathfrak{s}(H_{ij,+}^{\sus})=\infty\quad  \forall i,j,  \quad \mathfrak{s}(\ff^{\sus}_{i,0})>0, \; \mathfrak{s}(\ff^{\sus}_{i,+})=\infty \quad \; i>0, \quad \mathfrak{s}(\ff^{\sus}_{0,0})>0,
$$
where we recall that $h_0=0$ and $ h_j=\bd_j+1$ for $j\ne 0$.  Similarly, an index family $\cE$ will be said to be \textbf{$\k,\phi-q$ nonnegative} if 
$$
        \inf\Re(\cE(H^{\sus}_{ij,0}))> h_j+q, \; \cE(H^{\sus}_{ij,+})=\emptyset \;  \forall i,j,  \; \inf\Re(\cE(\ff^{\sus}_{i,0}))\ge 0, \; \cE(\ff^{\sus}_{i,+})=\bbN_0 \; \forall i>0 \; \inf\Re(\cE(\ff^{\sus}_{0,0}))\ge 0. 
$$

By restriction of the composition of the $\k,\QFB$ pseudodifferential calculus of $W$ described above, we have the following composition result.  
\begin{theorem}
Let $E,F$ and $G$ be vector bundles over $M$.  Suppose that $\cE$ and $\cF$ are index families and $\mathfrak{s}$ and $\mathfrak{t}$ are multiweights  on $M^2_{\k,\phi-q}$ such that for each $i>0$, 
\begin{equation}
   \min\{\mathfrak{s}(H^{\sus}_{0i,+}),\min\Re\cE|_{H^{\sus}_{0i,+}}\}+\min\{\mathfrak{t}(H^{\sus}_{i0,+}),\min\Re \cF|_{H^{\sus}_{i0,+}}\}> h_i+q=1+\dim S_i+q.
\label{cks.1a}\end{equation}
Then given $A\in\Psi^{m,\cE/\mathfrak{s}}_{\k,\phi-q}(M;F;G)$ and $B\in \Psi^{m',\cF/\mathfrak{t}}_{\k,q-\phi}(M;E,F)$, their composition is well-defined with 
$$
  A\circ B\in \Psi^{m+m',\cK/\mathfrak{k}}_{\k,\phi-q}(M;E,G),
$$
where, using the convention that $h_0=0$, that
$$\cE|_{\ff^{\sus}_{0,+}}=\cF|_{\ff^{\sus}_{0,+}}=\bbN_0 \quad  \mbox{and that} \quad \mathfrak{s}(\ff^{\sus}_{0,+})=\mathfrak{t}(\ff^{\sus}_{0,+})=0,
$$ 
the  index family $\cK$ is for $i,j\in\{0,1,\ldots,\ell\}$ and $\nu\in \{0,+\}$  given by 
\begin{equation}
\begin{aligned}
\cK|_{H^{\sus}_{ij,\nu}}&= (\cE|_{\ff^{\sus}_{i,\nu}}+\cF|_{H^{\sus}_{ij,\nu}})\overline{\cup} (\cE|_{H^{\sus}_{ij,\nu}}+ \cF|_{\ff^{\sus}_{j,\nu}})\overline{\cup} \overline{\bigcup_{k\ge 0}} (\cE|_{H^{\sus}_{ik,\nu}}+\cF|_{H^{\sus}_{kj,\nu}}-h_k-q),\\
\cK|_{\ff^{\sus}_{i,\nu}}&= (\cE|_{\ff^{\sus}_{i,\nu}}+\cF|_{\ff^{\sus}_{i,\nu}})\overline{\cup} \overline{\bigcup_{k\ge 0}} (\cE|_{H^{\sus}_{ik,\nu}}+\cF|_{H^{\sus}_{ki,\nu}}-h_k-q), \quad (i,\nu)\ne (0,+),
 \end{aligned}
\label{cks.1b}\end{equation}
while  the multiweight $\mathfrak{k}$ is given by 
\begin{equation}
\begin{aligned}
\mathfrak{k}(H^{\sus}_{ij,\nu})&= \min\{(\mathfrak{s}(\ff^{\sus}_{i,\nu})\dot{+}\mathfrak{t}(H^{\sus}_{ij,\nu})), (\mathfrak{s}(H^{\sus}_{ij,\nu})\dot{+}\mathfrak{t}(\ff^{\sus}_{j,\nu})), \min_{k\ge 0}\{\mathfrak{s}(H^{\sus}_{ik,\nu})\dot{+}\mathfrak{t}(H^{\sus}_{kj,\nu})-h_k-q\}\},    \\
\mathfrak{k}(\ff^{\sus}_{i,\nu})&= \min\{(\mathfrak{s}(\ff^{\sus}_{i,\nu})\dot{+}\mathfrak{t}(\ff^{\sus}_{i,\nu})),\min_{k\ge 0} \{ (\mathfrak{s}(H^{\sus}_{ik,\nu})\dot{+}\mathfrak{t}(H^{\sus}_{ki,\nu})-h_k-q) \} \}, \quad (i,\nu)\ne (0,+). \end{aligned}
\label{cks.1c}\end{equation}
If instead $\cE=\cF=\emptyset$, except possibly  at $\overline{\ff}^{\sus}_{i,\nu}$ for $(i,\nu)\ne(0,+)$, where it could possibly be $\bbN_0$, then for $A\in\Psi^{m,\cE/\mathfrak{s}}_{\k,\QFB,\cn}(M;F,G)$ and 
$B\in\Psi^{m,\cF/\mathfrak{t}}_{\k,\QFB,\cn}(M;E,F)$ only weakly conormal $\k,\QFB$ pseudodifferential operators,
$$
       A\circ B\in \Psi^{m+m',\cK/\mathfrak{k}}_{\k,\QFB,\cn}(M;E,G)
$$
with indicial family $\cK$ and multiweight $\mathfrak{k}$ still given by \eqref{cks.1b} and \eqref{cks.1c}.
\label{cks.1}\end{theorem}

We will also need a suspended version of the $\QAC-\Qb$ calculus.  The corresponding double space is obtained from \eqref{ksm.8} by omitting the blow-up of $\Phi_{j,0}\times\{0\}$ for $H_j$ maximal, namely it is given by
\begin{equation}
 M^2_{\k,\Qb-q}:= [ \overline{V}_{\bvarrho}; \Phi_{0,0}\times\{0\}, \Phi_{1,+}\times\{0\},  \Phi_{1,0}\times\{0\}, \ldots,  \Phi_{k,+}\times\{0\},  \Phi_{k,0}\times\{0\}, \Phi_{k+1,+}\times\{0\},\ldots,  \Phi_{\ell,+}\times\{0\}]
\label{qbq.1}\end{equation}
with blow-down map
$$
        \beta_{\k,\Qb-q}: M^2_{\k,\Qb-q}\to \overline{V}_{\bvarrho}.
$$
We denote by $\ff^{\Qb-q}_{i,\nu}$ the boundary hypersurface created by the blow-up of $\Phi_{i,\nu}\times\{0\}$ for 
$$
(i,\nu)\notin\{(0,+), (k+1,0),\ldots,(\ell,0)\}.
$$  
Let us also denote by $H^{\Qb-q}_{ij,\nu}$ the boundary hypersurface corresponding to the lift of $H^{\rp}_{ij,\nu}\times \overline{\bvarrho\bbR^q}$ for $i\ge 0$, $j\ge 0$ and $\nu\in \{0,+\}$ such that $(i,j,\nu)\ne (0,0,+)$.  Finally, let $H^{\Qb-q}_{00,+}$ be the boundary hypersurface of $M^2_{\k,\Qb-q}$ corresponding to the lift of $(M^2_{\rp}\rttimes [0,\infty))\times \pa(\overline{\bvarrho \bbR^q})$.  Let $\diag_{\k,\Qb-q}$ be the lift of $\diag_{\k,\rp}\times\{0\}$ in $\overline{V_{\bvarrho}}$ to $M^2_{\k,\Qb-q}$.  In parallel to \eqref{ksm.9}, we can define the space of operators
\begin{equation}
\Psi^{m}_{\k,\Qb-q}(M)= \{ \kappa\in I^{m-\frac14}(M^2_{\k,\Qb-q};\diag_{\k,\Qb-q}; {}^{\k,\Qb-q}\Omega) \; | \; \kappa\equiv 0 \; \mbox{at} \; \pa M^2_{\k,\Qb-q}\setminus \ff_{\k,\Qb-q} \},
\label{qbq.2}\end{equation}
where $\ff_{\k,\Qb-q}$ is the union of the boundary hypersurfaces of $M^2_{\k,\Qb-q}$ intersecting $\diag_{\k,\Qb-q}$ and 
\begin{equation}
{}^{\k,Qb-q}\Omega := \varpi_R^*({}^{\QAC-\Qb}\Omega)\otimes \varpi_q^{*}(\bvarrho^{-q}\Omega_{\bvarrho\bbR^q}),
\label{qbq.3}\end{equation}
where $\varpi_R$ is the composition of the natural maps
$$
\xymatrix{
    M^2_{\k,\Qb-q} \ar[r]^-{\beta_{\k,\Qb-q}} & \overline{V_{\bvarrho}} \ar[r] & M^2_{\rp}\rttimes [0,\infty) \ar[r] & M_{\k,\QAC}
}    
$$
with the last map being the lift of $\pr_R\times \Id: M^2\times [0,\infty)\to M\times [0,\infty)$ and where $\varphi_q$ is the composition of the natural maps
$$
\xymatrix{
M^2_{\k,\Qb-q}\ar[r]^-{\beta_{\k,\Qb-q}} & \overline{V_{\bvarrho}} \ar[r] & \overline{\bvarrho\bbR^q}.
}
$$
If $M$ is a fiber $\phi_i^{-1}(s)$ of a fiber bundle $\phi_i: H_i\to S_i$ for a non-maximal boundary hypersurface $H_i$ of a manifold with fibered corners $W$ with base $S_i$ of dimension $q$ and if the maximal boundary hypersurfaces of $W$ have corresponding  fiber bundle given by the identity map, then $M^2_{\k,\Qb}$ is naturally diffeomorphic to the fiber $\phi^{-1}_{\ff^{\Qb}_{i,+}}(s)$ of $\phi_{\ff^{\Qb}_{i,+}}: \ff^{\Qb}_{i,+}\to S_i$ in \eqref{ffi.1} for the corresponding $\QAC-\Qb$ double space $W^2_{\QAC-\Qb}$.  Thus, using the $\QAC-\Qb$ pseudodifferential calculus of $W$ and restricting to $\phi^{-1}_{\ff^{\Qb}_{i,+}}(s)$, we can define an enlarged space 
\begin{equation}
      \Psi^{m,\cE/\mathfrak{s}}_{\k,\Qb-q}(M;E,F)  
\label{qbq.4}\end{equation}
with $E$ and $F$ vector bundles on $M$ and $\cE$ and $\mathfrak{s}$ an indicial family and a multiweight on $M^2_{\k,\Qb-q}$.   We can more generally define, still through the $\QAC-\Qb$ calculus of $W$, a weakly conormal version
\begin{equation}
 \Psi^{m,\cE/\mathfrak{s}}_{\k,\Qb-q,\cn}(M;E,F)  
\label{qbq.5}\end{equation}
of \eqref{qbq.4}.  

 A multiweight $\mathfrak{s}$ will be said to be \textbf{$\k,\Qb-q$ positive} if  
 $$
        \mathfrak{s}(H^{\Qb-q}_{ij,0})> \left\{ \begin{array}{ll} 
      q, & H_j \; \mbox{maximal,} \\
       h_j+q, &  \mbox{otherwise},  \end{array} \right. \; \mathfrak{s}(H^{\Qb-q}_{ij,+})=\infty \quad  \forall i,j,  \quad \mathfrak{s}(\ff^{\Qb-q}_{i,0})>0 \; (H_i \; \mbox{non-maximal}),
$$
$\mathfrak{s}(\ff^{\Qb-q}_{i,+})=\infty$ for all $i>0$ and $ \mathfrak{s}(\ff^{\Qb-q}_{0,0})>0$.   Similarly, we say that an index family $\cE$ is \textbf{$\Qb-q$ nonnegative} if  
 $$
 \begin{aligned}
    \cE(H^{\Qb-q}_{00,\infty})= \bbN_0, \quad  \cE(H_{ij,+})=\emptyset \; \forall i,j, \quad    \inf \Re(\cE(H^{\Qb-q}_{ij,0})) \left\{ \begin{array}{ll} 
      >q, & H_j \; \mbox{maximal}, i\ne j, \\
       \ge q, & H_j \; \mbox{maximal}, i= j, \\
      \ge  h_j+q, &  \mbox{otherwise},  \end{array} \right.    \\ 
       \quad \inf\Re(\cE(\ff^{\Qb-q}_{i,0}))\ge 0\; (H_i \; \mbox{non-maximal}), \;   \inf\Re(\cE(\ff^{\Qb-q}_{0,0}))\ge 0, \quad \cE(\ff^{\Qb-q}_{i,+})=\bbN_0  \quad \forall i>0.
\end{aligned}       
$$

By restriction of the composition result of Theorem~\ref{ckqfb.12}, we obtain the following.
\begin{theorem}
Let $E,F$ and $G$ by vector bundles over $M$.  Suppose that $\cE$ and $\cF$ are index families and $\mathfrak{s}$ and $\mathfrak{t}$ are multiweights on $M^2_{\k,\Qb-q}$ such that for each $i>0$,
\begin{equation}
   \min\{\mathfrak{s}(H^{\Qb-q}_{0i,+}),\min\Re\cE|_{H^{\Qb-q}_{0i,+}}\}+\min\{\mathfrak{t}(H^{\Qb-q}_{i0,+}),\min\Re \cF|_{H^{\Qb-q}_{i0,+}}\}> h_i+q=1+\dim S_i+q.
\label{qbq.6a}\end{equation}
Then given $A\in \Psi^{m,\cE/\mathfrak{s}}_{\k,\Qb-q}(M; F,G)$ and $B\in \Psi^{m',\cF/\mathfrak{t}}_{\k,\Qb-q}(M;E,F)$, their composition is well-defined with
$$
   A\circ B\in \Psi^{m+m',\cK/\mathfrak{k}}_{\k,\Qb-q}(M;E,G),
$$
where, using the conventions that $h_0=0$, $\cE|_{\ff^{\Qb-q}_{0,+}}=\cF|_{\ff^{\Qb-q}_{0,+}}=\bbN_0$, $\mathfrak{s}(\ff^{\Qb-q}_{0,+})=\mathfrak{t}(\ff^{\Qb-q}_{0,+})=0$ and that
$\cE|_{\ff^{\Qb-q}_{i,0}}=\cF|_{\ff^{\Qb-q}_{i,0}}= \emptyset$ and $\mathfrak{s}(\ff^{\Qb-q}_{i,0})=\mathfrak{t}(\ff^{\Qb-q}_{i,0})=\infty$ when $H_i$ is a maximal boundary hypersurface, the index family $\cK$ 
is for $i,j\in\{0,1,\ldots,\ell\}$ and $\nu\in \{0,+\}$  given by 
\begin{equation}
\begin{aligned}
\cK|_{H^{\Qb-q}_{ij,\nu}}&= (\cE|_{\ff^{\Qb-q}_{i,\nu}}+\cF|_{H^{\Qb-q}_{ij,\nu}})\overline{\cup} (\cE|_{H^{\Qb-q}_{ij,\nu}}+ \cF|_{\ff^{\Qb-q}_{j,\nu}})\overline{\cup} \overline{\bigcup_{k\ge 0}} (\cE|_{H^{\Qb-q}_{ik,\nu}}+\cF|_{H^{\Qb-q}_{kj,\nu}}-h_{k,\nu}-q),\\
\cK|_{\ff^{\Qb-q}_{i,\nu}}&= (\cE|_{\ff^{\Qb-q}_{i,\nu}}+\cF|_{\ff^{\Qb-q}_{i,\nu}})\overline{\cup} \overline{\bigcup_{k\ge 0}} (\cE|_{H^{\Qb-q}_{ik,\nu}}+\cF|_{H^{\Qb-q}_{ki,\nu}}-h_{k,\nu}-q), \quad (i,\nu)\ne (0,+),
 \end{aligned}
\label{qbq.6b}\end{equation}
in the latter case with $H_i$ not maximal when $\nu=0$,
while  the multiweight $\mathfrak{k}$ is given by 
\begin{equation}
\begin{aligned}
\mathfrak{k}(H^{\Qb-q}_{ij,\nu})&= \min\{(\mathfrak{s}(\ff^{\Qb-q}_{i,\nu})\dot{+}\mathfrak{t}(H^{\Qb-q}_{ij,\nu})), (\mathfrak{s}(H^{\Qb-q}_{ij,\nu})\dot{+}\mathfrak{t}(\ff^{\Qb-q}_{j,\nu})), \min_{k\ge 0}\{\mathfrak{s}(H^{\Qb-q}_{ik,\nu})\dot{+}\mathfrak{t}(H^{\Qb-q}_{kj,\nu})-h_{k,\nu}-q\}\},  \\
\mathfrak{k}(\ff^{\Qb-q}_{i,\nu})&= \min\{(\mathfrak{s}(\ff^{\Qb-q}_{i,\nu})\dot{+}\mathfrak{t}(\ff^{\Qb-q}_{i,\nu})),\min_{k\ge 0} \{ (\mathfrak{s}(H^{\Qb-q}_{ik,\nu})\dot{+}\mathfrak{t}(H^{\Qb-q}_{ki,\nu})-h_{k,\nu}-q) \} \}, \quad (i,\nu)\ne (0,+), \end{aligned}
\label{qbq.6c}\end{equation}
with again in the latter case $H_i$ not maximal when $\nu=0$.
If instead $\cE=\cF=\emptyset$, except possibly  at boundary hypersurfaces intersecting $\diag_{\k,\Qb-q}$, where it could possibly be $\bbN_0$, then for $A\in\Psi^{m,\cE/\mathfrak{s}}_{\k,\Qb-q,\cn}(M;F,G)$ and 
$B\in\Psi^{m,\cF/\mathfrak{t}}_{\k,\Qb-q,\cn}(M;E,F)$ only weakly conormal $\k,\Qb-q$ pseudodifferential operators,
$$
       A\circ B\in \Psi^{m+m',\cK/\mathfrak{k}}_{\k,\Qb-q,\cn}(M;E,G)
$$
with indicial family $\cK$ and multiweight $\mathfrak{k}$  still given by \eqref{qbq.6b} and \eqref{qbq.6c}.

\label{qbq.6}\end{theorem}

Coming back to  $\eth_{\k,\sus}$, to invert it within the $\k,\phi-q$ calculus, it is convenient, as in \S~\ref{qfble.0}, to consider the conjugated operator
\begin{equation}
    D_{\k,\sus}:= x^{-\mathfrak{w}}\eth_{\k,\sus}x^{\mathfrak{w}}= D_{\sus}+\k\gamma= D_{\QFB}+\eth_{\bbR^q}+\k\gamma.
\label{ksm.13}\end{equation}
We know already how to invert $D_{\k,\sus}$ at various front faces of $M^2_{\k,\phi-q}$.  At $\ff^{\sus}_{0,0}$ the model to invert is 
\begin{equation}
                N^{\sus}_{0,0}(D_{\k,\sus}):= D_{\sus},
\label{ksm.14}\end{equation}
and its inverse, described in Theorem~\ref{ift.15}, fits nicely on $\ff^{\sus}_{0,0}$ via the identification \eqref{ksm.8b}. Similarly, at $\ff^{\sus}_{i,0}$ for $i>0$, we see from \eqref{do.41} that the model to invert is 
\begin{equation}
   N^{\sus}_{i,0}(D_{\k,\sus}):= D_{v,i}+ \eth_{h,i}+\eth_{\bbR^q}.
\label{ksm.14}\end{equation}
As we will see and similarly to what was happening in \S~\ref{qfble.0}, inverting $D_{\k,\sus}$ at $\ff^{\sus}_{0,0}$ will yield automatically an inverse at $\ff_{i,0}^{\sus}$ for $i>0$.   On the other hand, for $i>0$, we will need to rely on the suspended version of Assumption~\ref{qfble.5}.

\begin{assumption}
For each boundary hypersurface $H_i$ of $M$, the normal operator
\begin{equation}
    N^{\sus}_{i,+}(D_{\k,\sus}):= \left.  D_{\k,\sus} \right|_{\ff^{\sus}_{i,+}}
\label{ksm.15a}\end{equation}
is invertible with inverse 
$$
   N^{\sus}_{i,+}(D_{\k,\sus})^{-1}\in \Psi^{-1,\cE_i/\mathfrak{s}_i}_{\ff^{\sus}_{i,+},\cn}(H_i;E)
$$
with  $\Psi^{-1,\cE_i/\mathfrak{s}_i}_{\ff^{\sus}_{i,+},\cn}(H_i;E)$ the space of operators obtained by restriction of \eqref{ksm.11} to $\ff^{\sus}_{i,+}$, where $\cE_i$ is a $\k,\phi-q$ nonnegative  index family, except at $H^{\sus}_{jj,0}\cap \ff^{\sus}_{i,+}$ for $H_j\ge H_i$, where we have that 
$$
   \cE_{i}|_{H^{\sus}_{jj,0}}=\bd_j+q+\bbN_0,
$$ 
and where $\mathfrak{s}_i$ is a $\k,\phi-q$ positive multiweight, except at $H^{\sus}_{jj,0}\cap \ff^{\sus}_{i,+}$ for $H_j\ge H_i$, where 
$$
    \mathfrak{s}(H^{\sus}_{jj,0})>\bd_i+q.
$$
Moreover, for $H_j\ge H_i$, $N^{\sus}_{i,+}(D_{\k,\sus})^{-1}$ agrees with $N^{\sus}_{j,0}(D_{\k,\sus})^{-1}$ on $\ff^{\sus}_{i,+}\cap\ff^{\sus}_{j,0}$, while for $H_j\ne H_i$, it agrees with
$N^{\sus}_{j,+}(D_{\k,\sus})^{-1}$ on $\ff^{\sus}_{i,+}\cap \ff^{\sus}_{j,+}$.  Finally, for $H_j\ge H_i$, the term $A_{ij}$ of order $\bd_j+q$ at $\ff^{\sus}_{i,+}\cap H^{\sus}_{jj,0}$ is such that $\widetilde{\Pi}_{h,j}A_{ij}= A_{ij}\widetilde{\Pi}_{h,j}=A_{ij}$.  
\label{ksm.15}\end{assumption}

\begin{remark}
As in Remark~\ref{qfble.6}, we have that $\ff^{\sus}_{i,+}\cap H^{\sus}_{0j,0}=\ff^{\sus}_{i,+}\cap H^{\sus}_{j0,0}=\emptyset$ for all $i>0$ and $j\ge 0$.  
\label{ksm.16}\end{remark}

Starting with $N^{\sus}_{0,0}(D_{\k,\sus})^{-1}$, we can via Lemma~\ref{ksm.8c} and the identification \eqref{ksm.8b} pull it back to $M^2_{\k,\phi-q}$, yielding a parametrix $Q_0\in\Psi^{-1,\cQ_0/\mathfrak{q}_0}_{\k,\phi-q,\cn}(M;E)$   such that
$$
     (D_{\k,\sus}Q_0)=\Id+R_0 \quad \mbox{with}\quad R_0=\k\gamma Q_0,
$$  
where $\cQ_0$ is the  index family given by
\begin{equation}
\begin{gathered}
\cQ_0|_{\ff^{\sus}_{0,0}}=\cQ_0|_{\ff^{\sus}_{i,0}}= \cQ_0|_{\ff^{\sus}_{i,+}}=\bbN_0 \quad i>0,  \\
 \cQ_{0}|_{H^{\sus}_{ii,0}}=\cQ_0|_{H^{\sus}_{ii,+}}=\bbN_0+h_i-1+q,  \quad i\ge0,
\end{gathered}
\label{ksm.17}\end{equation}
and the empty set elsewhere,
while $\mathfrak{q}_0$ is a multiweight such that for $i>0$ and $j>0$,
\begin{equation}
\begin{gathered}
\mathfrak{q}_0(H^{\sus}_{i0,0})=q + \nu_R, \quad \mathfrak{q}_0(H^{\sus}_{i0,+})=q+\nu_R, \quad \mathfrak{q}_0(H^{\sus}_{0i,0})=\bd_i+1+q+\nu_R, \quad \mathfrak{q}_0(H^{\sus}_{0i,+})=\bd_i+1+q+\nu_R, \\
 \mathfrak{q}_0(H^{\sus}_{ij,0})=\mathfrak{q}_0(H^{\sus}_{ij,+})>\left\{\begin{array}{ll} \bd_i+q, & i=j, \\ \bd_j+1+q, & i\ne j, \end{array}  \right. \quad \mathfrak{q}_0(H^{\sus}_{00,0})= \mathfrak{q}_0(H^{\sus}_{00,+})>q \quad \mathfrak{q}_0(\ff^{\sus}_{i,0})= \mathfrak{q}_0(\ff^{\sus}_{i,+})>0
 \end{gathered}
\label{ksm.18}\end{equation}
and $\mathfrak{q}(\ff^{\sus}_{0,0})>0$.
Correspondingly, the error term $R_0$ can be seen as an element of $\Psi^{-1,\cR_0/\mathfrak{r}_0}_{\k,\phi-q,\cn}(M;E)$ 
 with index family $\cR_0$ given by  
  \begin{equation}
  \begin{gathered}
\cR_0|_{\ff^{\sus}_{0,0}}= \cR_0|_{\ff^{\sus}_{i,0}}=\bbN_0+1, \quad \cR_0|_{\ff_{i,+}}=\bbN_0, \quad i>0,  \\
 \cR_{0}|_{H^{\sus}_{ii,0}}=\bbN_0+h_i+q, \quad \cR_0|_{H^{\sus}_{ii,+}}=\bbN_0+h_i-1+q, \quad i\ge 0,
 \end{gathered}
\label{ksm.19}\end{equation}
and the empty set elsewhere, and with  multiweight $\mathfrak{r}_0$ such that for $i>0$ and $j>0$,
\begin{equation}
\begin{gathered}
\mathfrak{r}_0(H^{\sus}_{00,0})>q, \quad  \mathfrak{r}_0(\ff^{\sus}_{i,+})>0, \quad \mathfrak{r}_0(\ff^{\sus}_{i,0})>1, \quad \mathfrak{r}_0(\ff^{\sus}_{0,0})>1,   \\
\quad \mathfrak{r}_0(H^{\sus}_{i0,0})=q+\nu_R+1, \quad \mathfrak{r}_0(H^{\sus}_{0i,0})=h_i+1+q+\nu_R,  \quad
\mathfrak{r}_0(H^{\sus}_{i0,+})=q+\nu_R, \quad \mathfrak{r}_0(H^{\sus}_{0i,+})=\bd_i+1+q+\nu_R,  \\
 \quad \mathfrak{r}_0(H^{\sus}_{ij,0})> \left\{\begin{array}{ll} h_i, & i=j \\ h_i+1, & i\ne j,  \end{array}  \right. \quad \mathfrak{r}_0(H^{\sus}_{ij,+})> \left\{\begin{array}{ll} \bd_i+q, & i=j \\ \bd_j+1+q, & i\ne j.  \end{array}  \right.
\end{gathered}
\label{ksm.20}\end{equation}
This shows in particular that $Q_0$ inverts $D_{\k,\sus}$ also at $\ff^{\sus}_{i,0}$ for $i>0$.  The error term however does not decay rapidly at $H^{\sus}_{ij,+}$ for $i\ge 0$ and $j\ge 0$.  To enforce such a rapid decay, it suffices to cut-off $Q_0$ as follows.  Let $r$ be the canonical radial function of $\bbR^q$, so that $\frac{r}{\bvarrho}$ is the radial function of $\bvarrho\bbR^q$.  Seen as a function on $M^2_{\k,\phi-q}$, notice that $\frac{1}{\sqrt{1+r^2}}$ vanishes at the boundary hypersurfaces that to not intersect the lifted diagonal $\diag_{\k,\phi-q}$.  Hence, it suffices to replace $Q_0$ by 
$$
    Q_1:= \varphi_{\epsilon}\lrp{\frac{1}{\k \sqrt{1+r^2} }}Q_0
$$ 
for $\epsilon>0$ small with $\varphi_{\epsilon}$ the cut-off function of \eqref{qfble.17b},
so that $Q_1\in\Psi^{-1,\cQ_1/\mathfrak{q}_1}_{\k,\phi-q,\cn}(M;E)$ with 
 $\cQ_1$ the  index family given by
\begin{equation}
\begin{gathered}
\cQ_1|_{\ff^{\sus}_{0,0}}=\cQ_1|_{\ff^{\sus}_{i,0}}= \cQ_1|_{\ff^{\sus}_{i,+}}=\bbN_0 \quad i>0,  \\
 \cQ_{1}|_{H^{\sus}_{ii,0}}=\bbN_0+h_i-1+q,  \quad i\ge0,
\end{gathered}
\label{ksm.21}\end{equation}
and by the empty set elsewhere,
while $\mathfrak{q}_1$ is a multiweight such that $i>0$ and $j>0$,
\begin{equation}
\begin{gathered}
\mathfrak{q}_1(H^{\sus}_{i0,0})=q+\nu_R,  \quad \mathfrak{q}_1(H^{\sus}_{0i,0})=\bd_i+1+q+\nu_R, \\
 \mathfrak{q}_1(H^{\sus}_{ij,0})>\left\{\begin{array}{ll} \bd_i+q, & i=j, \\ \bd_j+1+q, & i\ne j, \end{array}  \right. \quad \mathfrak{q}_1(H^{\sus}_{00,0})>q \quad \mathfrak{q}_1(\ff_{i,0})= \mathfrak{q}_1(\ff_{i,+})>0, \quad \mathfrak{q}_1(\ff^{\sus}_{0,0})>0,
 \end{gathered}
\label{ksm.22}\end{equation}
and given by $\infty$ elsewhere.

Since  $ \varphi_{\epsilon}\lrp{\frac{1}{\k \sqrt{1+r^2} }}$ is supported away from $\ff^{\sus}_{i,0}$ for $i\ge 0$ and since
$$
\left[D_{\k,\sus}, \varphi_{\epsilon}\lrp{\frac{1}{\k \sqrt{1+r^2} }}  \right]= \left[\eth_{\bbR^q}, \varphi_{\epsilon}\lrp{\frac{1}{\k \sqrt{1+r^2} }}  \right]= \varphi_{\epsilon}'\lrp{\frac{1}{\k \sqrt{1+r^2} }} \frac{1}{\k \sqrt{1+r^2}} \mathcal{O}(\frac{1}{r}),
$$
we see that 
\begin{equation}
    D_{\k,\sus}Q_1= \Id+R_1
\label{ksm.23}\end{equation}
with error term $R_1\in\Psi^{-1,\cR_1/\mathfrak{r}_1}_{\k,\phi-q,\cn}(M;E)$, where $\cR_1$ is an index set satisfying the same properties as $\cR_0$, but with 
$$
     \cR_1|_{H^{\sus}_{ij,+}}=\emptyset \quad \mbox{for all} \; i,j,
$$  
 and with $\mathfrak{r}_1$ a multiweight satisfying the same properties as $\mathfrak{r}_0$, but with
 $$
     \mathfrak{r}_1(H^{\sus}_{ij,+})=\infty  \quad \mbox{for all} \; i,j.
 $$ 
Our way of cutting off also ensures that the term $r_i$ of order $\bd_i+1+q$ of $R_1$ at $H^{\sus}_{ii,0}$ is still such that $r_i\widetilde{\Pi}_{h,i}=r_i$.  As in \S~\ref{qfble.0}, adding a term in $\k\Psi^{-2}_{\k,\phi-q}(M;E)$ supported near the lifted diagonal to $Q_1$, we can assume that $R_1\in\Psi^{-2,\cR_0/\mathfrak{r}_0}_{\k,\phi-q,\cn}(M;E)$.  Iterating this argument by adding to $Q_1$ a term in $\k\Psi^{-3}_{\k,\phi-q}(M;E)$, $\k\Psi^{-4}_{\k,\phi-q}(M;E)$ and so forth supported near the lifted diagonal and taking a Borel sum, we can assume in fact that $R_1\in \Psi^{-\infty,\cR_1/\mathfrak{r}_1}_{\k,\phi-q,\cn}(M;E)$.   Using Assumption~\ref{ksm.15}, we can improve this parametrix as follows.
\begin{proposition}
There exist $Q_2 \in\Psi^{-1,\cQ_2/\mathfrak{q}_2}_{\k,\phi-q,\cn}(M;E)$ and $R_2\in\Psi^{-\infty,\cR_2/\mathfrak{r}_2}_{\k,\phi-q,\cn}(M;E)$ such that
$$
      D_{\k,\sus}Q_2= \Id-R_2,
$$
where $\cQ_2,\mathfrak{q}_2, \cR_2$ and $\mathfrak{r}_2$ satisfy the same properties as $\cQ_1,\mathfrak{q}_1, \cR_1$ and $\mathfrak{r}_1$ with the exception that
$$
            \cR_2|_{\ff^{\sus}_{i,+}}=\bbN_0+1 \quad \mbox{for all} \; i>0.
$$
\label{ksm.23b}\end{proposition}  
\begin{proof}
By Assumption~\ref{ksm.15} and Remark~\ref{ksm.16}, the parametrix $Q_2$ can be obtained from $Q_1$ by changing it on $\ff^{\sus}_{i,+}$ for $i>0$ in such a way that 
$$
      N^{\sus}_{i,+}(Q_2)= N^{\sus}_{i,+}(D_{\k,\sus})^{-1}.
$$
\end{proof}  

To improve the parametrix further, we need to remove the term of order $h_i+q$ of the error term at $H^{\sus}_{ii,0}$ for $i\ge 0$.  This means that we need to choose more carefully the term $A_i$ of $Q_2$ of order $h_i+q-1$ at $H^{\sus}_{ii,0}$.  Since this term is such that $\widetilde{\Pi}_{h,i}A\widetilde{\Pi}_{h,i}=A_i$, the model operator to invert involves only the part of $D_{\QFB}$ acting on the range of $\widetilde{\Pi}_{h,i}$.  More precisely, for $i=0$, 
\begin{equation}
\widetilde{\Pi}_{h,0}:=\Pi_{\ker_{L^2}D_{\QFB}},
\label{ksm.24}\end{equation}
so that the model operator to invert is $\widetilde{\Pi}_{h,0}(\eth_{\bbR^q}+\k\gamma)\widetilde{\Pi}_{h,0}$ acting on the range of $\Pi_{\ker_{L^2}D_{\QFB}}$.  However, since we want to consider the operator $\eth_{\bbR^q}$ on $H^{\sus}_{00,0}$, not on $\ff^{\sus}_{0,0}$, we should write it as 
$$
     \widetilde{\Pi}_{h,0}\eth_{\bbR^q}\widetilde{\Pi}_{h,0}= \k \eth_{\bvarrho \bbR^q,0},
$$
where $\eth_{\bvarrho\bbR^q,0}:= \k^{-1} \widetilde{\Pi}_{h,0}\eth_{\bbR^q}\widetilde{\Pi}_{h,0}$ is the naturally associated operator on $\bvarrho\bbR^q$.  Hence, the model to invert on $H^{\sus}_{00,0}$ is 
$$
    \k (\eth_{\bvarrho\bbR^q,0}+\gamma). 
$$
Since 
\begin{equation}
   (\eth_{\bvarrho\bbR^q,0}+\gamma)^2= \eth_{\bvarrho\bbR^q,0}^2+ \gamma^2= \eth_{\bvarrho\bbR^q,0}^2+ \Id_E,
\label{ksm.24b}\end{equation}
this operator is invertible as a scattering operator, in fact, as a suspended operator.  Now, as in the proof of Proposition~\ref{do.40}, the leading order behavior of $(\eth_{\bvarrho\bbR^q,0}+\gamma)^{-1}$ at $\ff^{\sus}_{0,0}\cap H^{\sus}_{00,0}$ is the inverse Fourier transform of its principle symbol, which by Theorem~\ref{ift.15}, taking into account the factor $\k$ (and keeping in mind that the term $\gamma$ in Theorem~\ref{ift.15} comes from the principal symbol of $\eth_{\bbR^q}$ as described in \eqref{ift.3} and is distinct from the term $\gamma$ in \eqref{ksm.24b}) agrees with the top order term of $N^{\sus}_{0,0}(D_{\k,\sus})^{-1}$ at $\ff^{\sus}_{0,0}\cap H^{\sus}_{00,0}$.  Hence, by changing the parametrix $Q_2$ by requiring that  its term of order $q-1$ at $H^{\sus}_{00,0}$ be $\k^{-1}(\eth_{\bvarrho\bbR^q,0}+\gamma)^{-1}$, we can find a parametrix $Q_3\in \Psi^{-1,\cQ_2/\mathfrak{q}_2}_{\k,\phi-q,\cn}(M;E)$ and $R_3\in \Psi^{-\infty,\cR_2/\mathfrak{r}_2}_{\k,\phi-q,\cn}(M;E)$ such that
$$
       D_{\k,\sus}Q_3=\Id-R_3
$$   
with $R_3$ having no term of order $q$ at $H^{\sus}_{00,0}$.  To get rid of the term of order $\bd_i+1+q$ of the error term at $H^{\sus}_{ii,0}$ for $i>0$, we can proceed essentially as in the proof of Theorem~\ref{qfble.38}.  More precisely, near $H^{\sus}_{ii,0}$, we have, compared to \eqref{qfble.22}, to add a term to $D_{\cC_i}$, namely the model to invert is
\begin{equation}
D_{\cC_i,\sus}= D_{\cC_i}+ \frac{1}{\kappa} \eth_{v\bbR^q,i}= -c\frac{\pa}{\pa \kappa}+ \frac{1}{\kappa}\lrp{D_{S_i}+\eth_{v\bbR^q,i}+\frac{c}2} \quad \mbox{with}\; \eth_{v\bbR^q,i}:= v^{-1}\widetilde{\Pi}_{h,i}\eth_{\bbR^q}\widetilde{\Pi}_{h,i}.
\label{ksm.25}\end{equation}
By Lemmas~\ref{qfble.28}, \ref{qfble.33} and Remark~\ref{qfble.36}, we know that $D_{\cC_i}$ is invertible with inverse admitting a nice pseudodifferential characterization.  Nevertheless, to conclude that its suspended version $D_{\cC_i,\sus}$ is also invertible with inverse admitting a nice pseudodifferential characterization, we will need as in \S~\ref{qfble.0} to proceed by induction on $i$, again assuming without loss of generality that the boundary hypersurfaces $H_1, \ldots, H_{\ell}$ of $M$ are listed in a way compatible with the partial order, that is, in such a way that 
$$
       H_i< H_j \quad \Longrightarrow \quad i<j.
$$
First, the analogue of Proposition~\ref{qfble.23} is the following.
\begin{proposition}
For $i\ge 0$ fixed, there exists $Q_{3,i}\in\Psi^{-1,\cQ_{3,i}/\mathfrak{q}_{3,i}}_{\k,\phi-q,\cn}(M;E)$ and $R_{3,i}\in\Psi^{-\infty,\cR_{3,i}/\mathfrak{r}_{3,i}}_{\k,\phi-q,\cn}(M;E)$ such that
$$
      D_{\k,\sus}Q_{3,i}=\Id-R_{3,i}
$$
with $R_{3,i}$ having no term of order $h_j+q$ at $H^{\sus}_{jj,0}$ for $0\le j\le i$ and with $\cQ_{3,i}, \mathfrak{q}_{3,i}, \cR_{3,i}$ and $\mathfrak{r}_{3,i}$ satisfying the same properties as $\cQ_{2}, \mathfrak{q}_{2}, \cR_{2}$ and $\mathfrak{r}_{2}$.  On the other hand, for $j>i$, the term $q_{3,ij}$ of order $\bd_j+q$ at $H^{\sus}_{jj,0}$ of $Q_{3,i}$ is such that $\widetilde{\Pi}_{h,j}q_{3,ij}= q_{3,ij}\widetilde{\Pi}_{h,j}= q_{3,ij}$, while the term $r_{3,ij}$ of order $\bd_j+1+q$ at $H^{\sus}_{jj,0}$ of $R_{3,i}$ is such that $r_{3,ij}=r_{3,ij}\widetilde{\Pi}_{h,j}$.
\label{ksm.26}\end{proposition}

For $i=0$, notice that we can just take $Q_{3,0}=Q_3$ and $R_{3,0}=R_3$ to see that Proposition~\ref{ksm.26} holds.  Thus, to prove Proposition~\ref{ksm.26}, we can proceed by induction and assume it holds for $i-1$ to show it holds for $i$.  The key step will be to invert the model operator $D_{\cC_i,\sus}$ at $H^{\sus}_{ii,0}$.  First, notice that the parametrix of Proposition~\ref{ksm.26} for $i-1$ yields a parametrix for $D_{\cC_i,\sus}$ by restriction to $H^{\sus}_{ii,0}$.  Indeed, $Q_{3,i-1}$ has a term of order $h_i-1+q$ at $H^{\sus}_{ii,0}$, which in terms of $\QAC-\Qb$ densities as defined in \eqref{kqfb.23b} and the density $\Omega_{\bvarrho\bbR^q}$, can be seen as a term of order $-1$.  This means that $\k Q_{3,i-1}$ is of order $0$ at $H^{\sus}_{ii,0}$ in terms of the densities ${}^{\QAC-\Qb}\Omega$ and $\Omega_{\bvarrho\bbR^q}$.  Similarly, $R_{3,i-1}$ has a term of order $h_i+q$ at $H^{\sus}_{ii,0}$, that is, of order $0$ in terms of ${}^{\QAC-\Qb}\Omega$ and $\Omega_{\bvarrho\bbR^q}$.  Thus, if we let 
$$
   Q_{\cC_i,\sus}:= \left.\k Q_{3,i-1}\right|_{H^{\sus}_{ii,0}} \quad \mbox{and} \quad R_{\cC_i,\sus}:= \left. \widetilde{\Pi}_{h,i}R_{2,i-1}\right|_{H^{\sus}_{ii,0}}
$$
be these terms of order $h_i+q$, we see that
\begin{equation}
\begin{aligned}
\widetilde{\Pi}_{h,i}-R_{\cC_i,\sus} &= \left. \widetilde{\Pi}_{h,i}(\Id-R_{2,i-1})\right|_{H^{\sus}_{ii,0}}= \left.  \widetilde{\Pi}_{h,i}( D_{\k,\sus}Q_{3,i-1})\right|_{H^{\sus}_{ii,0}} \\
&= \left. (\k D_{\cC_i,\sus} Q_{3,i-1})\right|_{H^{\sus}_{ii,0}}= \left.(D_{\cC_i,\sus}\k Q_{3,i-1})\right|_{H^{\sus}_{ii,0}} \\
&= D_{\cC_i,\sus}Q_{\cC_i,\sus}.
\end{aligned}
\label{ksm.27}\end{equation}
In other words, $Q_{\cC_i,\sus}$ is a parametrix for $D_{\cC_i,\sus}$.  Now, $D_{\cC_i,\sus}$ is an operator geometrically associated to the cone
\begin{equation}
(S_i\times v\bbR^q\times [0,\infty)_{\k}, d\kappa^2+\kappa^2(g_{S_i}+g_{v\bbR^q})),
\label{ksm.28}\end{equation}
where $g_{v\bbR^q}=v^2g_{\bbR^q}$ is the natural euclidian metric on $v\bbR^q$.   Taking $u_i$ as in \eqref{qfble.26}, this suggests to consider the operator
\begin{equation}
   u_i^{\frac12}(D_{\cC_i,\sus})u_i^{\frac12},
\label{ksm.29}\end{equation}
which can be thought as an edge-$\QAC$ operator, but with $S_i$ replaced by  $S_i\times v\bbR^q$.  It acts formally on $L^2_{b,v\bbR^q}(S_i\times v\bbR^q\times [0,\infty];\ker D_{v,i})$, the $L^2$ space specified by a choice of $b$-metric on $S_i\times [0,\infty]$ and the Euclidean metric $g_{v\bbR^q}$ on the factor $v\bbR^q$.  In particular, $u_i^{-\frac12}Q_{\cC_i,\sus}u_i^{-\frac12}$ is a parametrix for $u_i^{\frac12}D_{\cC_i,\sus}u_i^{\frac12}$, 
\begin{equation}
 (u_i^{\frac12}D_{\cC_i,\sus}u_i^{\frac12})(u_i^{-\frac12}Q_{\cC_i,\sus}u_i^{-\frac12})=\widetilde{\Pi}_{h,i}- u_i^{\frac12}R_{\cC_i,\sus}u_i^{-\frac12}.
\label{ksm.30}\end{equation}
\begin{lemma}
The edge-$\QAC$ self-adjoint extension of $u_i^{\frac12}D_{\cC_i,\sus}u_i^{\frac12}$ associated to  
$$
L^2_{b,v\bbR^q}(S_i\times v\bbR^q\times [0,\infty];\ker D_{v,i})
$$  
is Fredholm.
\label{ksm.31}\end{lemma}
\begin{proof}
By Proposition~\ref{ksm.26} for $i-1$, since $\nu_R+1\ge \mu_R>\frac12$, we see that the error term $u_i^{\frac12}R_{\cC_i}u_i^{-\frac12}$ is compact when acting on $L^2_{b,v\bbR^q}(S_i\times v\bbR^q\times [0,\infty];\ker D_{v,i})$, so that $u_i^{\frac12}D_{\cC_i,\sus} u_i^{\frac12}$ has a right inverse modulo compact operators.  Taking the adjoint gives a left inverse modulo compact operators acting on the Sobolev space
$$
  \{ \sigma \in L^2_{b,v\bbR^q}(S_i\times v\bbR^q\times [0,\infty];\ker D_{v,i}) \; | \; D_{\cC_i,\sus}\sigma\in L^2_{b,v\bbR^q}(S_i\times v\bbR^q\times [0,\infty];\ker D_{v,i})\},
$$
from which the result follows.  
\end{proof}

As for $u_i^{\frac12}D_{\cC_i}u_i^{\frac12}$, one can check that this Fredholm operator is in fact invertible.  
\begin{lemma}
The Fredholm operator of Lemma~\ref{ksm.31} is in fact an isomorphism.
\label{ksm.32}\end{lemma}
\begin{proof}
The idea is to replace $u_i^{\frac12}D_{\cC_i}u_i^{\frac12}$ by $u_i^{\frac12}D_{\cC_i,\sus}u_i^{\frac12}$ and $\widetilde{D}_{S_i}$ by $\widetilde{D}_{S_i,\sus}:=c(D_{S_i}+ \eth_{v\bbR^q})$ in the proof of Lemma~\ref{qfble.28}.  Indeed, with these changes, we can essentially proceed as in this proof, the only new feature to add being an argument  showing that the spectrum of $\widetilde{D}_{S_i,\sus}$, which compared to the spectrum of $\widetilde{D}_{S_i}$ has a continuous part,  is still disjoint from the interval $(-\mu_R-\frac12,\mu_R+\frac12)$.  As in the proof of Lemma~\ref{qfble.28}, we know from \eqref{qfble.30} and Assumption~\ref{do.26} that $\widetilde{D}_{S_i}$ has no eigenvalue in the range $(-\mu_R-\frac12,\mu_R+\frac12)$, so that
$$
      \inf \Spec(\widetilde{D}_{S_i}^2)\ge (\mu_R+\frac12)^2.
$$  
Since $\widetilde{\eth}_{v\bbR^2}:=c \eth_{v\bbR^q}$ anti-commutes with $\widetilde{D}_{S_i}$, notice on the other hand that 
$$
        \widetilde{D}^2_{S_i,\sus}= (\widetilde{D}_{S_i}+ \widetilde{\eth}_{v\bbR^q})^2= \widetilde{D}_{S_i}^2+ \widetilde{\eth}^2_{v\bbR^q}.
$$
Since $\Spec(\widetilde{\eth}_{v\bbR^2}^2)= [0,\infty)$ and that $\widetilde{\eth}_{v\bbR^q}^2$ commutes with $\widetilde{D}_{S_i}^2$, we thus see that
$$
       \inf\Spec(\widetilde{D}^2_{S_i,\sus})\ge \inf \Spec(\widetilde{D}_{S_i}^2)\ge (\mu_R+\frac12)^2,
$$
so that the spectrum of the formally self-adjoint operator $\widetilde{D}_{S_i,\sus}$ must be disjoint from $(-\mu_R-\frac12,\mu_R+\frac12)$ as claimed.  With this result, we can use separation of variables as before to show that $u_i^{\frac12}D_{\cC_i,\sus}u_i^{\frac12}$ is invertible.

\end{proof}

Let $G_{\cC_i,\sus}$ be the operator such that $u_i^{-\frac12}G_{\cC_i,\sus}u_i^{-\frac12}$ is the inverse of the isomorphism of Lemma~\ref{ksm.32}.  Using the parametrix $Q_{\cC_i,\sus}$, we obtain the following characterization of $G_{\cC_i,\sus}$.
\begin{lemma}
On the face $H_{ii}^{\Qb-q}$ of $(S_i\times [0,1))^2_{\k,\Qb-q}$, we have that 
$$
u_i^{-\frac12}G_{\cC_i,\sus}u_i^{-\frac12} \in \Psi^{-1,\cG_i/\mathfrak{g}_i}_{\k,\Qb-q,\cn}(H_{ii,0}^{\Qb-q};\ker D_{v,i})
$$ 
with   $\Psi^{-1,\cG_i/\mathfrak{g}_i}_{\k,\Qb-q,\cn}(H_{ii,0}^{\Qb-q};\ker D_{v,i})$ the space of operators obtained by restriction to $H_{ii}^{\Qb}$ of \eqref{qbq.4} with $M$ replaced by $M_{S_i}= S_i\times [0,1)$, where  $\cG_i$ is a  $\k,\Qb-q$ nonnegative  index family, except at $H^{\Qb-q}_{jj,0}$ for $H_j<H_i$, where instead  
$$
     \cG_i|_{H^{\Qb}_{jj,0}}=\bd_j+q+\bbN,
$$
and such that
\begin{equation}
\cG_i|_{\ff_{j,+}^{\Qb-q}}=\bbN_0, \quad  \cG_i|_{H^{\Qb-q}_{00,+}}= \cG_i|_{H^{\Qb-q}_{kj,+}}= \cG_i|_{H_{j0,+}^{\Qb-q}}=\cG_i|_{H^{\Qb-q}_{0j,+}}= \cG_i|_{H_{j0,0}^{\Qb-q}}=\cG_i|_{H^{\Qb-q}_{0j,0}}= \emptyset, 
\label{ksm.33b}\end{equation}
for $H_k\le H_i$ and $H_j\le H_i$,
and with $\k,\Qb-q$ positive multiweight $\mathfrak{g}_i$ except at $H^{\Qb-q}_{jj,0}$ for $H_j<H_i$, where instead  
$$
     \mathfrak{g}_i(H^{\Qb-q}_{jj,0})>\bd_j+q,
$$
and   such that
\begin{equation}
\mathfrak{g}_i(\ff_{j,+}^{\Qb-q})= \mathfrak{g}_i(H^{\Qb-q}_{00,+})=  \mathfrak{g}_i(H^{\Qb-q}_{kj,+})= \mathfrak{g}_i(H_{j0,+}^{\Qb-q})=\mathfrak{g}_i(H^{\Qb-q}_{0j,+})= \infty, 
\label{ksm.33c}\end{equation}
and
\begin{equation}
\mathfrak{g}_i(H^{\Qb-q}_{j0,0})=\mu_R+q+\frac12, \quad \mathfrak{g}_i(H^{\Qb-q}_{0j,0})= h_j+q+\frac12 +\mu_R
\label{ksm.33cc}\end{equation}
for $H_{j}\le H_i$ and $H_k\le H_i$. Furthermore, for $H_{j}< H_i$, $G_{\cC_i,\sus}$ and $Q_{\cC_i,\sus}$ have the same top order terms at $\ff_{i,+}^{\Qb-q}$, $\ff^{\Qb}_{j,+}$, $\ff^{\Qb-q}_{j,0}$,  $H^{\Qb-q}_{jj,0}$ and  $H^{\Qb-q}_{00,0}$.

\label{ksm.33}\end{lemma}
\begin{proof}
 This lemma is a suspended version of Lemma~\ref{qfble.33}.  As such, it can be proved using a suspended version of the proof of Lemma~\ref{qfble.33}.  First, the parametrix $u_i^{-\frac12}Q_{\cC_i,\sus}u_i^{-\frac12}$ shows that the models at $\ff^{\Qb-q}_{i,+}$, $\ff^{\Qb-q}_{j,+}$, $\ff^{\Qb-q}_{j,0}$, $\ff^{\Qb-q}_{0,0}$ and $H^{\Qb-q}_{00,0}$ for $H_j<H_i$ are invertible.  In fact, Assumption~\ref{ksm.15} ensures that $u_i^{\frac12}(D_{\cC_i,\sus})u_i^{\frac12}$ is fully elliptic at $\ff^{\Qb-q}_{j,+}$ for $H_j\le H_i$.  To invert $H^{\Qb-q}_{00,0}$ and at $\ff_{j,0}^{\Qb-q}$ for $H_j<H_i$, we can again rely on Assumption~\ref{do.26} and applied a suspended version of Lemma~\ref{do.39} to obtain the correct model inverse.  In this case, the inverse at $\ff^{\Qb-q}_{0,0}$ will corresponds to the principal symbol of the inverse at $H^{\Qb-q}_{00,0}$ ( after taking the Fourier transform).
Thus, we can replace $Q_{\cC_i,\sus}$ by $\widetilde{Q}_{\cC_i,\sus}$ such that $u_i^{-\frac12}\widetilde{Q}_{\cC_i,\sus}u^{-\frac12}_i\in  \Psi^{-1,\widetilde{\cQ}_i/\widetilde{\mathfrak{q}}_i}_{\k,\Qb-q}(H_{ii,0}^{\Qb-q};\ker D_{v,i})$ with
\begin{equation}
  (u_i^{\frac12}D_{\cC_i,\sus}u_i^{\frac12})(u_i^{-\frac12}\widetilde{Q}_{\cC_i,\sus}u_i^{-\frac12})= \widetilde{\Pi}_{h,i}- u_i^{\frac12}\widetilde{R}_{\cC_i,\sus}u_i^{\frac12}
\label{ksm.33d}\end{equation}
for some $u_i^{\frac12}\widetilde{R}_{\cC_i,\sus}u_i^{-\frac12}\in  \Psi^{-\infty,\widetilde{\mathfrak{r}}_i}_{\k,\Qb-q}(H_{ii,0}^{\Qb-q};\ker D_{v,i})$, where $\widetilde{\cQ}_i$ is a $\k,\Qb-q$ nonnegative index family and $\widetilde{\mathfrak{q}}$ is a $\k,\Qb-q$ positive multiweight, except at $H^{\Qb-q}_{jj,0}$ for $H_j<H_i$ where instead
$$
 \widetilde{\mathfrak{q}}(H^{\Qb-q}_{jj,0})>\bd_j+q
$$
 and where $\widetilde{\mathfrak{r}}_i$ is a $\k,\Qb-q$ positive multiweight.  Using the standard sandwich argument of \cite{MazzeoEdge}, we see  that
$$
u_i^{-\frac12}G_{\cC_i,\sus}u_i^{-\frac12}= u_i^{-\frac12}\widetilde{Q}_{\cC_i,\sus}u_i^{-\frac12}+ (u_i^{-\frac12}\widetilde{Q}^*_{\cC_i,\sus}u_i^{-\frac12})(u_i^{\frac12}\widetilde{R}_{\cC_i,\sus}u_i^{-\frac12})+ u_i^{-\frac12}\widetilde{R}_{\cC_i,\sus}^*G_{\cC_i}\widetilde{R}_{\cC_i,\sus}u_i^{-\frac12},
$$ 
or equivalently that
\begin{equation}
G_{\cC_i,\sus}= \widetilde{Q}_{\cC_i,\sus}+ \widetilde{Q}^*_{\cC_i,\sus}\widetilde{R}_{\cC_i,\sus}+ \widetilde{R}_{\cC_i,\sus}^*G_{\cC_i,\sus}\widetilde{R}_{\cC_i,\sus}.
\label{ksm.34}\end{equation}
Since the operators $u_i^{-\frac12}\widetilde{R}^*_{\cC_i,\sus}u_i^{\frac12}$ and $u_i^{\frac12}\widetilde{R}_{\cC_i,\sus}u_i^{-\frac12}$ are $\mathfrak{r}$ residual for some multiweight $\mathfrak{r}$ and that $u_i^{-\frac12}G_{\cC_i,\sus}u_i^{-\frac12}$ is bounded on $L^2$, we conclude that 
\begin{equation}
u_i^{-\frac12}G_{\cC_i,\sus}u_i^{-\frac12}\in \Psi^{-1,\cG_i/\mathfrak{g}_i}_{\k,\Qb-q,\cn}(H_{ii,0}^{\Qb-q};\ker D_{v,i})
\label{ksm.35}\end{equation}
for some $\k,\Qb-q$ index family $\cG_i$ and some $\k,\Qb-q$ multiweight $\mathfrak{g}_i$ as claimed, but possibly not satisfying \eqref{ksm.33b} and \eqref{ksm.33c}.  However, relying on the rapid decay of $\widetilde{R}_{\cC_i,\sus}$ and $\widetilde{R}_{\cC_i,\sus}^*$ at $\ff^{\Qb-q}_{j,+}$, $H^{\Qb-q}_{00,+}$ and $H^{\Qb-q}_{jk,+}$ for $H_j\le H_i$, and $H_k\le H_i$, we can use a boothstrap argument relying on \eqref{ksm.34} to conclude that the index family $\cG_i$ and the multiweight $\mathfrak{g}_i$ can be chosen as claimed at those faces.  Using \eqref{ksm.30} instead of \eqref{ksm.33d}, we obtain
\begin{equation}
G_{\cC_i,\sus}= Q_{\cC_i,\sus}+  Q^*_{\cC_i} R_{\cC_i,\sus}+  R_{\cC_i,\sus}^*G_{\cC_i,\sus} R_{\cC_i\sus},
\label{ksm.33e}\end{equation} 
which we can use to infer the claimed better decay at $H^{\Qb-q}_{j0,0}$ and $H^{\Qb-q}_{0j,0}$ in \eqref{ksm.33cc}.

\end{proof}

This lemma can be used as follows to prove Proposition~\ref{ksm.26}.
\begin{proof}[Proof of Proposition~\ref{ksm.26}]
Since the proposition holds for $i=0$, it suffices to show that it holds for $i\in\{1,\ldots,\ell\}$ if it holds for $i-1$.  To show this, we need to improve the parametrix $Q_{3,i-1}$ by removing the term $r_i$ of order $\bd_i+1+q$ of its error term $R_{3,i-1}$ at $H^{\Qb-q}_{ii,0}$.  As a first step, we can eliminate $\widetilde{\Pi}_{h,i}r_i$ by considering the term
$$
   q_i:= \k^{-1}G_{\cC_i,\sus}\widetilde{\Pi}_{h,i}r_i
$$
of order $\bd_i+q$ at $H^{\Qb-q}_{ii,0}$, that is, of order $-1$ in terms of the densities ${}^{\QAC-\Qb}\Omega$ and $\Omega_{\bvarrho\bbR^q}$, where $G_{\cC_i,\sus}$ is the operator of Lemma~\ref{ksm.33}.   By construction, 
$$
     (\k D_{\cC_i,\sus}q_i)= D_{\cC_i,\sus}G_{\cC_i,\sus}\widetilde{\Pi}_{h,i}r_i= \widetilde{\Pi}_{h,i}r_i.
$$
Adding $q_i$ as a term of order $\bd_i+q$ at $H^{\Qb-q}_{ii,0}$ to $Q_{3,i-1}$ then gives a parametrix $Q_{3,i}'$ with error term almost as claimed, namely with term $r_i'$ of order $\bd_i+1+q$ at $H^{\Qb-q}_{ii,0}$ such that $\widetilde{\Pi}_{h,i}r_i'=0$.  Hence, this suggests to consider the term 
$$
      q_i':=D_{v,i}^{-1}r_i',
$$ 
which is well-defined by Corollary~\ref{hd.4} with $\delta=-\frac12$ and the analog of \eqref{qfble.12}  applied to the members of the vertical family $D_{v,i}$.  Adding it as a term of order $\bd_i+1+q$ to $Q_{3,i}'$ yields a parametrix $Q_{3,i}$ with error term as claimed.

\end{proof}

\begin{remark}
The completion of this proof by induction gives at the same time a proof for each $i$ of Lemmas~\ref{ksm.31}, \ref{ksm.32} and \ref{ksm.33}.
\label{ksm.36}\end{remark}

Using Proposition~\ref{ksm.26} for $i=\ell$ yields the following improved parametrix.
\begin{proposition}
There exists $Q_4\in\Psi^{-1,\cQ_4/\mathfrak{q}_4}_{\k,\phi-q,\cn}(M;E)$ and $R_4\in \dot{\Psi}^{-\infty}_{\k,\phi-q}(M;E)$  such that 
$$
   D_{\k,\sus}Q_4= \Id-R_4
$$
with $\cQ_4$ and $\mathfrak{q}_4$ satisfying the same properties as $\cQ_2$ and $\mathfrak{q}_2$, where $\dot{\Psi}^{-\infty}_{\k,\phi-q}(M;E)$ is the space introduced below \eqref{ksm.10}.
\label{ksm.37}\end{proposition}
\begin{proof}
Let us start with the parametrix $Q_{3,\ell}$ of Proposition~\ref{ksm.26} for $i=\ell$.  By the decay rates of the error term $R_{3,\ell}$ and the composition result of Theorem~\ref{cks.1}, there is $\delta>0$ such that for each boundary hypersurface $H$ of $M^2_{\k,\phi-q}$, there is $\mu>0$ such that 
$$
     R_{3,\ell}^{k}= \mathcal{O}(x^{\nu+k\delta}_H) \quad \forall \nu<\mu, \; \forall k\in\bbN_0,
$$
where $x_H$ is a boundary defining function for $H$.  This means that we can take an asymptotic sum
$$
      S\sim \sum_{k=1}^{\infty} R^k_{3,\ell}
$$
with $S\in \Psi^{-\infty,\mathcal{S}/\mathfrak{s}}_{\k,\Qb-q,\cn}(M;E)$, where $\cS$ and $\mathfrak{s}$ are satisfying the same properties as $\cR_{3,\ell}$ and $\mathfrak{r}_{3,\ell}$.  Essentially by definition of $S$, notice that 
$$
   R_4:=\Id-(\Id-R_{3,\ell})(\Id+S)\in \dot{\Psi}^{-\infty}_{\k,\phi-q}(M;E),
$$
so 
$$
     Q_4: Q_{3,\ell}(\Id+S)
$$
is the desired parametrix, since 
$$
D_{\k,\sus}Q_4= D_{\k,\sus}Q_{3,\ell}(\Id+S)= (\Id-R_{3,\ell})(\Id+S)=\Id-R_4.
$$

\end{proof}

Finally, we can give the following pseudodifferential characterization of the inverse $D_{\k,\sus}$.
\begin{theorem}
There exists $G_{\k,\sus}\in \Psi^{-1,\cG/\mathfrak{g}}_{\k,\phi-q,\cn}(M;E)$ such that 
$$
   D_{\k,\sus}G_{\k,\sus}= G_{\k,\sus}D_{\k,\sus}=\Id,
$$
where $\cG$ is an index family given by the empty set at $H^{\sus}_{ij,+}$ for all $i\ge 0$ and $j\ge 0$ and at $H_{ij,0}^{\sus}$ for $i\ne j$, and elsewhere given by
$$
   \cG|_{\ff^{\sus}_{i,+}}= \cG|_{\ff^{\sus}_{i,0}}=\bbN_0 \quad \mbox{and} \quad \cG|_{H^{\sus}_{ii,0}}=h_i+q-1+\bbN_0,
$$
and where $\mathfrak{g}$ is a multiweight such that 
$$
  \mathfrak{g}(H^{\sus}_{ij,+})=\infty \; \forall i,j, \quad \mathfrak{g}(H^{\sus}_{ij,0})>h_j+q, \; \mbox{for} \; i>0, j>0 \; \mbox{with} \; i\ne j, \quad \mathfrak{g}(H^{\sus}_{ii,0})>h_i -1+q \; \forall i\ge 0,
$$
and 
$$
\mathfrak{g}(\ff^{\sus}_{i,+})=\infty, \quad \mathfrak{g}(\ff^{\sus}_{i,0})>0, \quad  \mathfrak{g}(\ff^{\sus}_{0,0})>0, \quad \mathfrak{g}(H^{\sus}_{i0,0})=\nu_R, \quad \mathfrak{g}(H^{\sus}_{0i,0})=h_i+q+\nu_R
$$
for $i>0$.
\label{ksm.38}\end{theorem}
\begin{proof}
The proof is the same as the one of Theorem~\ref{qfble.38}, except that we use the parametrix $Q_4$ of Proposition~\ref{ksm.37} instead of the parametrix of Proposition~\ref{qfble.37} and the composition result of Theorem~\ref{cks.1} instead of the one of Theorem~\ref{ckqfb.10}.
\end{proof}

\section{Results for Hodge-deRham $\QFB$ operators} \label{HdR.0}

The parametrix construction of \S~\ref{do.0} crucially relies on 5 assumptions, namely on Assumptions~\ref{do.1}, \ref{su.7}, \ref{su.2}, \ref{do.46} and \ref{do.26}.  Among those, Assumptions~\ref{do.1} and \ref{su.2} are mild geometric assumptions relatively easy to check or to enforce.  Assumption~\ref{do.26} is not as straightforward, but can be dealt with efficiently thanks to the recent work of Albin and Gell-Redman \cite{AG}.  For the remaining two assumptions, namely Assumptions~\ref{su.7} and \ref{do.46}, they can be seen to hold in depth 2 for important classes of examples thanks to \cite{HHM2004} and \cite{KR0}.  However, to check those conditions in higher depth, we need to use an inductive argument involving the parametrix construction of \S~\ref{do.1}, but also those of \S~\ref{qfble.0} for the low energy limit of the resolvent of a $\QFB$ Dirac operator, of \S~\ref{ift.0} for the inverse of a non-fully elliptic suspended Dirac $\QFB$ operator and of \S~\ref{sle.0} for the resolvent of a suspended $\QFB$ Dirac operator in the low energy limit.  Each of these parametrix constructions relies on assumptions which in general can only be checked through an inductive argument involving the depth of the underlying manifold with fibered corners.  The relations between these various parametrix constructions are summarized in the following diagram.
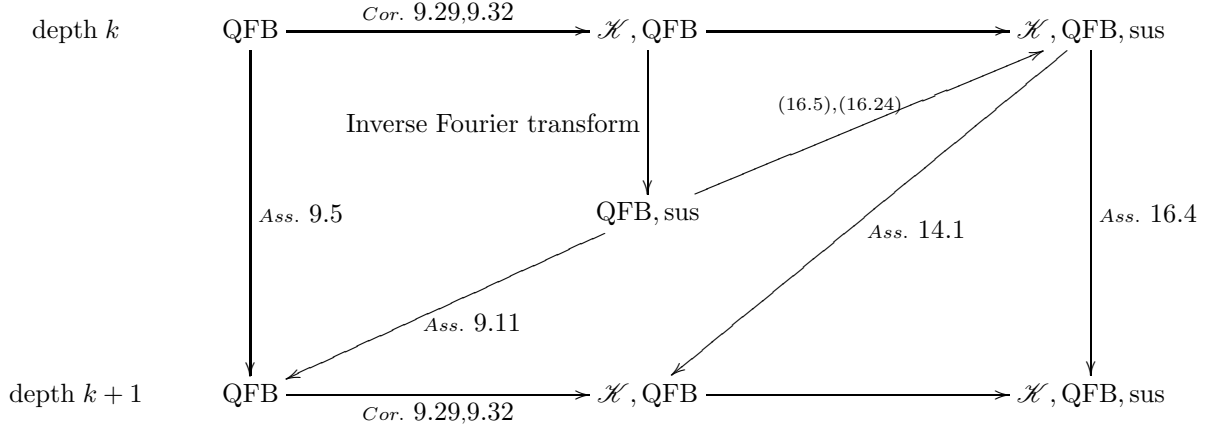
\begin{equation}
\xymatrix{
\mbox{depth} \; k & \QFB \ar[dddd]^-{Ass.~\ref{su.7}}\ar[rrrr]^-{Cor.~\ref{id.1},\ref{hd.4}} & & & & \k,\QFB \ar[dd]_-{\mbox{Inverse Fourier transform}}  \ar[rrrr] & & & & \k,\QFB, \sus \ar[dldllddl]^-{Ass.~\ref{qfble.5}}\ar[dddd]^-{Ass.~\ref{ksm.15}} \\
 & & & & &  & & &  & \\
 &  &  & & & \QFB, \sus \ar[ddllll]^-{Ass.~\ref{do.46}}\ar[uurrrr]^-{\eqref{ksm.8b}, \eqref{ksm.14}}  & & & &  \\
  & & & & & &  & & & \\
 \mbox{depth} \; k+1 & \QFB \ar[rrrr]_-{Cor.~\ref{id.1},\ref{hd.4}} & & & & \k,\QFB  \ar[rrrr] & & & & \k,\QFB, \sus 
}
\label{HdR.1}\end{equation}  

Focusing on the Hodge-deRham operator, we will now implement such an inductive argument to obtain results free of inductive hypotheses.  To that end, it will be convenient in Assumption~\ref{do.26} to take $\delta=-\frac12$, in which case
\begin{equation}
    \mu_L=\mu_R+1>\frac32
\label{HdR.2}\end{equation}
as observed already in \eqref{qfble.1}.  In fact, to enforce Assumptions~\ref{su.7} and \ref{do.46} at each inductive step, we will require that $\mu_R>1$, or else that $\mu_R>\frac12$ and $\widetilde{\mu}_R:=\widetilde{\mu_L}-1>1$ in Corollary~\ref{id.1}.  Thus, let $g_{\QFB}$ be a $\QFB$ metric on a manifold with fibered corners $M$ satisfying Assumption~\ref{do.1}. Let $\eth_{\QFB}$ be the corresponding Hodge-deRham operator.  In this particular setting, we will replace Assumptions~\ref{su.7} and \ref{do.26} by the following simpler assumptions.  

\begin{assumption}
For $H_i$ a boundary hypersurface of $M$, the vertical family $\eth_{v,i}$ is such that its fiberwise $L^2$ kernels in $L^2_{\QFB}(H_i/S_i;E)$ form a finite rank vector bundle $\ker \eth_{v,i}$ over $S_i$.
\label{HdR.3}\end{assumption}
\begin{assumption}
For $H_i$ a boundary hypersurface of $M$, we assume that the flat vector bundle $\cH^*_{L^2}(H_i/S_i)\to S_i$ is such that 
\begin{equation}
      \left| q-\frac{\bd_i}2 \right|\le 1  \quad \Longrightarrow \quad \ker_{L^2_{w}}\mathfrak{d}_{S_i} \; \mbox{is trivial in degree $q$},
\label{HdR.4a}\end{equation} 
where $\mathfrak{d}_{S_i}$ is the wedge Hodge-deRham operator occurring in \eqref{do.17}.  Similarly, for each $H_j<H_i$ and for each fiber $F_{ji}$ of $\phi_{ji}: \pa_jS_i\to S_j$, the associated wedge Hodge-deRham operator $\mathfrak{d}_{F_{ji}}$ acting on forms taking values in $\cH^*_{L^2}(H_i/S_i)$ is such that 
\begin{equation}
      \left| q-\frac{\dim F_{ji}}2 \right|\le 1  \quad \Longrightarrow \quad \ker_{L^2_{w}}\mathfrak{d}_{F_{ji}} \; \mbox{is trivial in degree $q$}.
\label{HdR.4b}\end{equation}
If $H_i$ is maximal and $H_j$ submaximal with $H_j<H_i$, \eqref{HdR.4b} can in fact be replaced by the weaker assumption
\begin{equation}
      \left| q-\frac{\dim F_{ji}}2 \right|< 1  \quad \Longrightarrow \quad \ker_{L^2_{w}}\mathfrak{d}_{F_{ji}} \; \mbox{is trivial in degree $q$}.
\label{HdR.4d}\end{equation}
\label{HdR.4}\end{assumption}
\begin{remark}
This assumption can be seen as an enlarged Witt condition on the flat vector bundle 
$$
\cH^*_{L^2}(H_i/S_i)\to S_i.
$$
\label{HdR.4c}\end{remark}
To be able to apply Corollary~\ref{id.1} in each inductive step, we will also need to make the following assumption.  For $H_j$ a boundary hypersurface, let $\mathfrak{P}_j$ be the subset of degrees where the fibers of $\phi_j: H_j\to S_j$ have a non-trivial reduced $L^2$-cohomology.  Then for $H_i>H_j$,
let 
$$
\cP_{ji}: \CI(F_{ji}; \Lambda^*({}^wT^*F_{ji})\otimes \cH^*_{L^{2}}(H_j/S_j))\to \CI(F_{ji}; \Lambda^*({}^wT^*F_{ji})\otimes \cH^*_{L^{2}}(H_j/S_j)) 
$$ 
be the projections on sections of total degrees $q$ (that is, the sum of the horizontal degree in the factor $\Lambda^*({}^wT^*F_{ji})$ and the vertical degree in the factor $\cH^*_{L^2}(H_i/S_i)$ is equal to $q$)  such that $q$ or $q+1$ is in $\mathfrak{P}_j$.
\begin{assumption}
 For $H_i>H_j$, 
\begin{equation}
      \left| q-\frac{\dim F_{ji}}2 \right|\le \frac32  \quad \Longrightarrow \quad  \cP_{ij} (\ker_{L^2_{w}} \mathfrak{d}_{F_{ji}})_q \quad \mbox{is trivial,}
\label{HdR.5a}\end{equation}
where $(\ker_{L^2_{w}} \mathfrak{d}_{F_{ji}})_q$ is the $L^2$ kernel of $ \mathfrak{d}_{F_{ji}}$ restricted to forms with horizontal degree $q$.
\label{HdR.5}\end{assumption}
We can now formulate the main result of this section.
\begin{theorem}
Let $g_{\QFB}$ be a $\QFB$ metric on a $\QFB$ manifold $M$ for which Assumption~\ref{do.1} holds.  Suppose that Assumptions~\ref{HdR.3}, \ref{HdR.4} and \ref{HdR.5} hold for the corresponding Hodge-deRham operator $\eth_{\QFB}$.  Then there exists a $\QFB$ metric $\widetilde{g}_{\QFB}$ on $M$ such that Theorem~\ref{do.28} holds with $\delta=-\frac12$ for $\eth_{\QFB}$.  Moreover, by Corollary~\ref{id.1},  Corollary~\ref{decay.1} in the introduction holds in this case. 
\label{HdR.6}\end{theorem}
\begin{proof}
We need to check that the five assumptions of Theorem~\ref{do.28} hold.  By hypothesis, Assumptions~\ref{do.1} holds, and since $\eth_{\QFB}$ is a Hodge-deRham operator, this implies Assumption~\ref{su.2}.  

On the other hand, the first part of Assumption~\ref{do.26} requires that the wedge operator $cD_{D_i}-\frac12$ be essentially self-adjoint.  By Assumption~\ref{HdR.4}, Remark~\ref{e.3h} and Corollary~\ref{e.11}, this can be achieved by suitably modifying each of the model metrics $g_{S_i}$.  To enforce the second part of Assumption~\ref{do.26}, namely invertibility of the indicial family $cI(D_{b,i},\lambda)$ for certain values of $\lambda$, it suffices then to modify the $\QFB$ metric in such a way that the model metric $g_{S_i}$ is replaced by $\epsilon g_{S_i}$ for some $\epsilon>0$ sufficiently small.  Indeed, scaling $g_{S_i}$ in such a way, we can make the positive eigenvalues of $\delta^{S_i}d^{S_i}$ and $d^{S_i}\delta^{S_i}$ as large as we want, so combined with Lemma~\ref{do.27d} and Assumption~\ref{HdR.4}, this means that we can require the indicial family $cI(D_{b,i},\lambda)$ to be invertible for a range of $\Re \lambda$ as large as claimed in Assumption~\ref{do.26}.

Now, proceeding by induction on the depth of $M$, we can suppose that Theorem~\ref{HdR.6} holds on $\QFB$ manifolds of lower depth.  Hence, using Assumptions~\ref{HdR.4} and \ref{HdR.5} together with Lemma~\ref{do.27d} (possibly still modifying the $\QFB$ metric to scale away as above some of the indicial roots), we can apply Corollary~\ref{id.1} to each member of the vertical family $\eth_{v,i}$, which together with Assumption~\ref{HdR.3} shows that Assumption~\ref{su.7} holds for each vertical family $\eth_{v,i}$.  When $H_i$ is maximal and $H_j$ is submaximal with $H_j<H_i$, we use \eqref{HdR.4d} instead of \eqref{HdR.4b} to deduce Assumption~\ref{su.7} since we can appeal to \cite[Corollary~3.16]{KR0} and Lemma~\ref{do.27d}.  Hence, proceeding by recursion on the depth of $M$ as in the diagram of \eqref{HdR.1}, we can show that Theorems~\ref{qfble.38}, \ref{ift.15} and Theorem~\ref{ksm.38} hold with $\widetilde{\mu}_R>1$, that is, with 
$$
\nu_R=\min\{\mu_R,\widetilde{\mu}_R-1\}>0
$$
for each member of the vertical family $\eth_{v,i}$.  In particular, Theorem~\ref{ift.15} with $\nu_R>0$ implies that Assumption~\ref{do.46} holds for each vertical family.  This shows that the five assumptions of Theorem~\ref{do.28} hold, from which the result follows. 
\end{proof}

As in \cite{KR0}, this can be used to obtain a pseudodifferential characterization of the low energy limit of the resolvent of the Hodge Laplacian.  
\begin{corollary}
Let $\eth_{\QFB}$ be the Hodge-deRham operator of the $\QFB$ metric $\widetilde{g}_{\QFB}$ of Theorem~\ref{HdR.6} and assume that $\bd_i=\dim S_i>1$ for each boundary hypersurface $H_i$ of $M$.    Then there exists $G^2_{\k,\QFB}\in \Psi^{-2,\cG/\mathfrak{g}}_{\k,\QFB,\cn}(M;E)$ such that 
$$
       (D^2_{\QFB}+\k^2)G^2_{\k,\QFB}= G^{2}_{\k,\QFB}(D^2_{\QFB}+\k^2)=\Id,
$$
where $\cG$ is an index family given by the empty set, except at $H_{00,0}$ and $\ff_{i,\nu}$ for $\nu\in\{0,+\}$ and $H_i$ a boundary hypersurface, where
$$
       \cG(H_{00,0})= \bbN_0-2, \quad \cG(\ff_{i,\nu})=\bbN_0,
$$
and where $\mathfrak{g}$ is a multiweight such that 
$$
               \mathfrak{g}(H_{ij,+})=\infty  \quad \forall (i,j)\ne (0,0) \quad \mbox{and} \quad \mathfrak{g}(H_{ij,0})>\bd_j, \quad \mbox{for} \; i\ne j, \; i\ne 0, j\ne 0,
$$
while for $i>0$, 
$$
\mathfrak{g}(\ff_{i,+})=\infty, \quad \mathfrak{g}(H_{ii,0})=\bd_i-1, \quad \mathfrak{g}(\ff_{i,0})>0, \quad \mathfrak{g}(H_{i0,0})=\nu_R-1, \quad \mathfrak{g}(H_{0i,0})=\bd_i+\nu_R, \quad \mathfrak{g}(H_{00,0})>-1.
$$
\label{HdR.7}\end{corollary}
\begin{proof}
We can start as in the proof of \cite[Corollary~8.13]{KR0} and consider the operator 
$$
    \widetilde{\eth}_{\QFB}:= \left(\begin{matrix} \eth_{\QFB} & 0 \\ 0 & -\eth_{\QFB}  \end{matrix} \right) 
$$
acting on sections of $E\oplus E$ together with the self-adjoint operator
$$
  \widetilde{\gamma}:= \left( \begin{matrix} 0 & -\sqrt{-1}  \\ \sqrt{-1} & 0 \end{matrix} \right).
$$
In this case, \eqref{qfble.2} holds for $\widetilde{\eth}_{\QFB}$ and $\widetilde{\gamma}$.  Let $\widetilde{D}_{\QFB}:= x^{-\w}\widetilde{\eth}_{\QFB} x^{\w}$ be the corresponding conjugated operator as in \eqref{do.2}.  Then by Theorem~\ref{ksm.38} applied to $N_{i,+}(\widetilde{D}_{\QFB}+ \widetilde{\gamma}\k)$, Assumption~\ref{qfble.5} holds for $\widetilde{D}_{\QFB}$, so that Theorem~\ref{qfble.38} gives an inverse $\widetilde{G}_{\k,\QFB}$.  Using the composition result of Theorem~\ref{ckqfb.10} to take the square of $\k G_{\k,\QFB}$, so that
$$
      \widetilde{G}^2_{\k,\QFB}= \k^{-2}(\k\widetilde{G}_{\k,\QFB})^2,
$$
we obtain an inverse for $(\widetilde{D}_{\QFB}^2+ \k^2)$.  If $P_1: E\oplus E\to E$ is the projection on the first factor, we get an inverse of the claimed form by taking 
$$
         G^2_{\k,\QFB}:= P_1\widetilde{G}^2_{\k,\QFB}P_1.
$$
\end{proof}

A simple way to ensure that Assumptions~\ref{HdR.3}, \ref{HdR.4} and \ref{HdR.5} hold is as follows.  
\begin{corollary}
Let $g_{\QFB}$ be a metric on a $\QFB$ manifold $M$ satisfying Assumption~\ref{do.1}.  Whenever the fibers of $\phi_i:H_i\to S_i$ have non-trivial reduced $L^2$ cohomology with respect to the induced $\QFB$ metrics, suppose that the stratified space $\widehat{S}_i$ corresponding to $S_i$ is a quotient of $\bbS^{\bd_i}$ by the action of a finite subgroup of the orthogonal group $O(\bd_i+1)$ (this holds automatically when $g_{\QFB}$ is a $\QALE$ metric), and that $\dim S_i\ge 3$, as well as $\dim F_{ji}=\dim S_i-\dim S_j-1>3$ for each boundary hypersurface $H_j<H_i$.  Then Theorem~\ref{HdR.6} holds for $g_{\QFB}$.  Moreover, if $\dim S_i>3$ whenever the fibers of $\phi_i: H_i\to S_i$ have non-trivial reduced $L^2$ cohomology, then we can also assume that Corollary~\ref{id.1} holds with $\widetilde{\mu}_L=\mu_L=\mu_R+1\ge \frac{\bd_i+1}2>2$ for the $\QFB$ metric $\widetilde{g}_{\QFB}$ of Theorem~\ref{HdR.6}. 
\label{HdR.8}\end{corollary}
\begin{proof}
Notice first that Assumption~\ref{HdR.3} together with the hypotheses of Corollary~\ref{HdR.8} ensures that Assumptions~\ref{HdR.4} and \ref{HdR.5} hold.  Indeed, the flat vector bundle $\ker \eth_{v,i}$ lifts to be trivial on the universal cover of $S_i$, so $\ker_{L^2_w}\mathfrak{d}_{S_i}$ must be trivial, except possibly in degrees $0$ and $\dim S_i$.  Similarly, $\ker_{L^2_{w}}\mathfrak{d}_{F_{ji}}$ must be trivial, except possibly in degrees $0$ and $\dim F_{ji}$.  This shows that Assumptions~\ref{HdR.4} and \ref{HdR.5} hold when $\dim S_i\ge 3$ and $\dim F_{ji}>3$ for each $i$ and $j$.  

To prove the corollary, it suffices then to check that Assumption~\ref{HdR.3} holds.  To do so, we can proceed by induction on the depth of $M$ and assume that Corollary~\ref{HdR.8} holds for $\QFB$ manifolds of lower depth.  In  particular, applying it to the members of the vertical family $\eth_{v,i}$ shows that they have finite dimensional $L^2$ kernel.  Since these spaces are canonically identified with reduced $L^2$-cohomology, their dimensions only depends on the quasi-isometric class of the corresponding $\QFB$ metrics.  Hence, by the Ehresmann lemma of \cite[Corollary~A.6]{KR3}, this dimension is locally constant on $S_i$, showing that $\ker\eth_{v,i}$ is indeed a finite rank a vector bundle.  This shows that Assumption~\ref{HdR.3} holds for $g_{\QFB}$.  
\end{proof}

An important example where Corollary~\ref{HdR.8} holds is in the setting provided by the hyperK\"ahler metric on the moduli space of $\SU(2)$ monopoles of charge $3$ on $\bbR^3$, see \cite{KR2} for further details. At the cost of imposing restrictions in the degrees where the reduced $L^2$-cohomology of the fibers of $\phi_i: H_i\to S_i$ lies, we can relax the hypotheses of Corollary~\ref{HdR.8} as follows.
\begin{corollary}
Let $g_{\QFB}$ be a $\QFB$ metric on a $\QFB$ manifold $M$ satisfying Assumption~\ref{do.1}.   Suppose that, except possibly in middle degree, the fibers of $\phi_i:H_i\to S_i$ have trivial reduced $L^2$ cohomology with respect to the induced $\QFB$ metrics.   When it is non-trivial in middle degree, suppose that the stratified space $\widehat{S}_i$ corresponding to $S_i$ is a quotient of $\bbS^{\bd_i}$ by the action of a finite subgroup of the orthogonal group $O(\bd_i+1)$ (this holds automatically when $g_{\QFB}$ is a $\QALE$ metric), and that $\dim S_i\ge 3$, as well as $\dim F_{ji}=\dim S_i-\dim S_j-1\ge 3$ for each boundary hypersurface $H_j<H_i$.  Then Theorem~\ref{HdR.6} holds.  Moreover, if, except in middle degree, $g_{\QFB}$ has trivial reduced $L^2$-cohomology, then Corollary~\ref{id.1} holds with $\widetilde{\mu}_L=\mu_L=\mu_R+1>2$ for the $\QFB$ metric $\widetilde{g}_{\QFB}$ of Theorem~\ref{HdR.6}.
\label{HdR.9}\end{corollary}
\begin{proof}
The proof is the same as the one of Corollary~\ref{HdR.8}, except that to check that Assumption~\ref{HdR.5} holds, we will use instead the hypotheses on reduced $L^2$ cohomology.  Indeed, thanks to these hypotheses, $\ker_{L^2_w}(\cP_{ji}\mathfrak{d}_{F_{ji}}\cP_{ji})$ is possibly non-trivial only in degrees $\frac{\dim F_{ji}\pm 1}2$.  But since $\ker_{L^2_w} \mathfrak{d}_{F_{ji}}$ is possibly non-trivial only in degrees $0$ and $\dim F_{ji}$ and since $\dim F_{ji}\ge 3$, this shows that  $\ker_{L^2_w}(\cP_{ji}\mathfrak{d}_{F_{ji}}\cP_{ji})$ is trivial in all degrees, showing that Assumption~\ref{HdR.5} holds in this case.
\end{proof}
An important example where Corollary~\ref{HdR.9} holds is in the setting provided by the hyperK\"ahler metric of Nakajima on the Hilbert scheme of $n$ points on $\bbC^2$, see \cite{KR2} for all the details.

\printindex
\bibliography{QFBop}
\bibliographystyle{amsplain}

\end{document}